\documentclass[12pt]{amsart}
\usepackage{float}
\usepackage{subfigure}
\usepackage{tikz}
\usepackage{tkz-euclide}
\usetikzlibrary{shapes.geometric}
\tikzset{
	rline/.style ={color = red, line width =1pt}
}
\tikzset{
	bline/.style ={color = blue, line width =1pt}
}
\tikzset{
	gline/.style ={color = green, line width =1pt}
}
\usepackage[colorlinks,
            linkcolor=black,
            anchorcolor=blue,
            citecolor=black]{hyperref}
\usepackage{graphicx}
\usepackage{amsmath}
\usepackage{amssymb}
\usepackage{setspace}
\usepackage{amsthm}
\allowdisplaybreaks[4]
\newtheorem{thm}{Theorem}[section]
\newtheorem{cor}[thm]{Corollary}

\newtheorem{lem}[thm]{Lemma}
\newtheorem{prop}[thm]{Proposition}
\theoremstyle{definition}
\newtheorem{defn}[thm]{Definition}
\theoremstyle{remark}
\newtheorem{rem}[thm]{Remark}
\theoremstyle{conclusion}

\theoremstyle{conjecture}

\numberwithin{equation}{section}

\newcommand{\eps}{\varepsilon}

\newcommand{\dd}{\mathrm{d}}

\begin{document}
\title[The method of scaling spheres and blowing-up analysis]{Method of scaling spheres: Liouville theorems in inner or outer (unbounded or bounded) generalized radially convex domains, blowing-up analysis on domains with not necessarily $C^1$ boundary and other applications}

\author{Wei Dai, Guolin Qin}

\address{School of Mathematical Sciences, Beihang University (BUAA), Beijing 100191, P. R. China, and Key Laboratory of Mathematics, Informatics and Behavioral Semantics, Ministry of Education, Beijing 100191, P. R. China}
\email{weidai@buaa.edu.cn}

\address{Institute of Applied Mathematics, Chinese Academy of Sciences, Beijing 100190, and University of Chinese Academy of Sciences, Beijing 100049, P. R. China}
\email{qinguolin18@mails.ucas.ac.cn}

\thanks{Wei Dai is supported by the NNSF of China (No. 12222102 and No. 11971049) and the Fundamental Research Funds for the Central Universities.}

\begin{abstract}
In this paper, we aim to introduce the method of scaling spheres (MSS) as a unified approach to Liouville theorems on general domains in $\mathbb{R}^{n}$, and apply it to establish Liouville theorems on arbitrary unbounded or bounded MSS applicable domains in $\mathbb{R}^{n}$ for general ($\leq n$-th order) PDEs and integral equations without translation invariance or with singularities. The set of MSS applicable domains includes any unbounded or bounded generalized radially convex domains and any complementary sets of the closures of generalized radially convex domains, which is invariant under Kelvin transforms and is the maximal collection of simply connected domains in $\mathbb{R}^{n}$ such that the MSS works. For instance, $\mathbb{R}^{n}$, $\mathbb{R}^{n}_{+}$, balls, unbounded cones, bounded cone-like domains, exterior domains of a bounded cone-like domain (in the unbounded cone), convex domains, star-shaped domains and all the complements of their closures are MSS applicable domains. One should note that, MSS applicable domains is to the MSS what convex domains (at least in one direction) is to the method of moving planes. In the method of scaling spheres, we only need to dilate or shrink a subset of the sphere w.r.t. the scale of the radius and one fixed point (usually the singular point or the center of the MSS applicable domain). As applications, we derive a priori estimates from the boundary H\"{o}lder estimates for Dirichlet or Navier problems of ($\leq n$-th order) elliptic equations by applying the blowing-up argument on domains with blowing-up cone boundary (BCB domains for short). After the blowing-up procedure, the BCB domains allow the limiting shape of the domain to be a cone (half space is a cone). While the classical blowing-up techniques in previous works (c.f. e.g. \cite{CL4,DQ3,DQ0,GS1,MP,PQS,RW}) work on $C^{1}$-smooth domains, we are able to apply blowing-up analysis on more general BCB domains on which the boundary H\"{o}lder estimates hold (can be guaranteed by uniform exterior cone property etc). Consequently, by using the Leray-Schauder fixed point theorem, we derive existence of positive solutions to Lane-Emden equations in BCB domains. Since there are no smoothness conditions on the boundary of MSS applicable domain and BCB domain, our results clearly reveal that the existence and nonexistence of solutions mainly rely on topology (not smoothness) of the domain.
\end{abstract}
\maketitle {\small {\bf Keywords:} The method of scaling spheres, Liouville theorems, generalized radially convex domains, MSS applicable domains, domains with blowing-up cone boundary, H\'{e}non-Hardy type equations, a priori estimates, existence of solutions. \\

{\bf 2020 MSC} Primary: 35B53; Secondary: 35J30, 35J91.}

\tableofcontents

\section{Introduction}

\subsection{The method of scaling spheres: a unified approach to Liouville theorems for PDEs, IEs \& Systems on general domains}

Since Joseph Liouville's celebrated papers at 1844 and 1855, a large amount of works were devoted to Liouville type theorems. Nevertheless, there are still \emph{numerous open problems} on Liouville theorems for classical solutions. Even for equations involving classical second order Laplace operators, there are \emph{barely a few} Liouville type results on \emph{general domains} (such as cone-like domains) except on $\mathbb{R}^{n}$, $\mathbb{R}^{n}_{+}$ and bounded star-shaped domains with regular boundary (such that the formula of integration by parts holds).

\medskip

Liouville type theorems on PDEs or integral equations (IEs) (i.e., nonexistence of nontrivial nonnegative solutions) in the whole space $\mathbb{R}^n$ and in the half space $\mathbb{R}^n_+$ have been extensively studied. These Liouville theorems, in conjunction with the blowing-up and re-scaling arguments, are crucial in establishing a priori estimates and hence existence of positive solutions to non-variational boundary value problems for a class of elliptic equations on bounded domains or on Riemannian manifolds with boundaries (see e.g. \cite{BM,CDQ,CL4,DD,DQ3,DQ0,GS1,MR,MP,PQS,RW}). In previous works, blowing-up techniques require the limiting shape of the domain is either the whole space or a half space and hence reduce \emph{the a priori bounds} to \emph{the boundary H\"{o}lder estimates} and \emph{Liouville theorems in $\mathbb{R}^{n}$ and $\mathbb{R}^n_+$}. However, even if the boundary H\"{o}lder estimates can be derived from \emph{weaker regularity assumption} on the boundary such as \emph{Lipschitz continuity}, \emph{uniform exterior cone condition} or more weaker conditions (see e.g. \eqref{pde-boundary} and \eqref{pde-boundary2}, c.f. Theorems 3.3 and 3.11 in \cite{LZLH} and Corollaries 9.28 and 9.29 in \cite{GT}), the \emph{$C^{1}$-smoothness of the boundary} is still required to guarantee that, the limiting shape of the domain is \emph{a half space} after the blowing-up procedure. Therefore, it is desirable for us to develop a unified approach to establishing Liouville type theorems on general (convex or non-convex) simply connected domains in $\mathbb{R}^{n}$, and hence derive a priori estimates and existence of solutions via blowing-up analysis with \emph{only the boundary H\"{o}lder estimates but not necessarily $C^{1}$-smoothness of the boundary}. At least, the limiting shape of the domain should be allowed to be a \emph{cone-like domain} after the blowing-up procedure, so that blowing-up analysis could be carried out on such domains provided that the boundary H\"{o}lder estimates hold (which is guaranteed by Lipschitz domains, domains with uniform exterior cone property and so on $\cdots$).

\medskip

In this paper, our first main purpose is to introduce the \emph{method of scaling spheres} (\emph{MSS} for short) as \emph{a unified approach} to Liouville theorems for general domains in $\mathbb{R}^{n}$ and apply it to establish Liouville theorems for nonnegative solutions to general ($\leq n$-th order) PDEs (and related integral equations) \emph{without translation invariance or with singularities}:
\begin{equation}\label{PDE}
  (-\Delta)^{s}u=f(x,u(x)), \qquad \forall \,\, x\in\Omega,
\end{equation}
where $\Omega$ is an arbitrary unbounded or bounded MSS applicable domain in $\mathbb{R}^{n}$ with center $P$, $n\geq1$, $0<s\leq 1$ or $s\geq1$ is an integer, and $2s\leq n$.

\medskip

In order to give the definition of \emph{MSS applicable domain} in $\mathbb{R}^{n}$, we need the following definition of \emph{generalized radially convex domain}.
\begin{defn}[Generalized radially convex domain]\label{g-rcd}
A unbounded or bounded domain $\Omega\subseteq\mathbb{R}^{n}$ is called \emph{generalized radially convex domain} (\emph{g-radially convex domain} for short) or \emph{generalized star-shaped domain}, if there exists $x_{0}\in\overline{\Omega}$ such that the line segment $[x_{0},x]:=\{y\in\mathbb{R}^{n}\mid\,y=(1-t)x_{0}+tx,\,t\in[0,1]\}\subseteq\overline{\Omega}$ for any $x\in\Omega$. The point $x_{0}\in\overline{\Omega}$ is called the \emph{radially convex center} of $\Omega$.
\end{defn}

\medskip

One can see that \emph{unbounded or bounded convex domain} and \emph{star-shaped domain} are g-radially convex domain. The empty set $\emptyset$ can be regarded as a g-radially convex domain with arbitrary $P\in\mathbb{R}^{n}$ as its radially convex center.

\medskip

Now we give the definition of \emph{MSS applicable domain} in $\mathbb{R}^{n}$.
\begin{defn}[MSS applicable domain in $\mathbb{R}^{n}$]\label{mss-a}
A simply connected domain $\Omega\subseteq\mathbb{R}^{n}$ is called \emph{MSS applicable domain}, provided that $\Omega$ or $\mathbb{R}^{n}\setminus\overline{\Omega}$ is a g-radially convex domain. That is, MSS applicable domain is \emph{inner or outer g-radially convex domain}. A point $P\in\mathbb{R}^{n}$ is called the \emph{center of the MSS applicable domain} $\Omega$, provided that $P$ is the radially convex center of $\Omega$ if $\Omega$ is a g-radially convex domain or $P$ is the radially convex center of $\mathbb{R}^{n}\setminus\overline{\Omega}$ if $\mathbb{R}^{n}\setminus\overline{\Omega}$ is a g-radially convex domain.
\end{defn}

One can see that there is \emph{no smoothness conditions} on the boundary of a MSS applicable domain, and hence \emph{inward or outward cusps} are admissible to appear on the boundary. The set of MSS applicable domains contains \emph{quite a wide range of bounded or unbounded simply connected domains} in $\mathbb{R}^{n}$ and is \emph{invariant} under the map $\Omega\mapsto\mathbb{R}^{n}\setminus\overline{\Omega}$, arbitrary \emph{linear transforms} and \emph{translations}. Both $\mathbb{R}^{n}$ and $\emptyset$ are MSS applicable sets with arbitrary $P\in\mathbb{R}^{n}$ as their centers.

\begin{rem}[Basic properties of MSS applicable domain in $\mathbb{R}^{n}$]\label{rem0}
From Definitions \ref{g-rcd} and \ref{mss-a}, one can easily verify that, if $\Omega$ is a g-radially convex domain with radially convex center $P$, then $\mathbb{R}^{n}\setminus\overline{\widetilde{\Omega}}$ is a g-radially convex domain with radially convex center $P$; if $\mathbb{R}^{n}\setminus\overline{\Omega}$ is a g-radially convex domain with radially convex center $P$, then $\widetilde{\Omega}$ is a g-radially convex domain with radially convex center $P$, where $\widetilde{\Omega}$ is the transformed domain of $\Omega$ under the Kelvin transform $\widetilde{x}:=\frac{x-P}{|x-P|^{2}}+P$. That is, the set of MSS applicable domains is \emph{invariant under Kelvin transforms} (w.r.t. the center). Consequently, if $\Omega$ is a MSS applicable domain such that $\mathbb{R}^{n}\setminus\overline{\Omega}$ is a g-radially convex domain, then $\Omega$ must be \emph{unbounded}. That is, $\Omega$ is a \emph{bounded MSS applicable domain} if and only if $\Omega$ is a \emph{bounded g-radially convex domain}.
\end{rem}

\begin{defn}[Equivalent definition of MSS applicable domain]\label{e-mss-a}
The set of MSS applicable domains is \emph{generated} by the set of g-radially convex domains under the \emph{Kelvin transforms} (w.r.t. the radially convex center $P$).
\end{defn}

For a MSS applicable domain $\Omega$, it is possible that both $\Omega$ and $\mathbb{R}^{n}\setminus\overline{\Omega}$ are g-radially convex (e.g., $\mathbb{R}^{n}$, $\emptyset$, the unbounded cones $\mathcal{C}$ and so on $\cdots$).

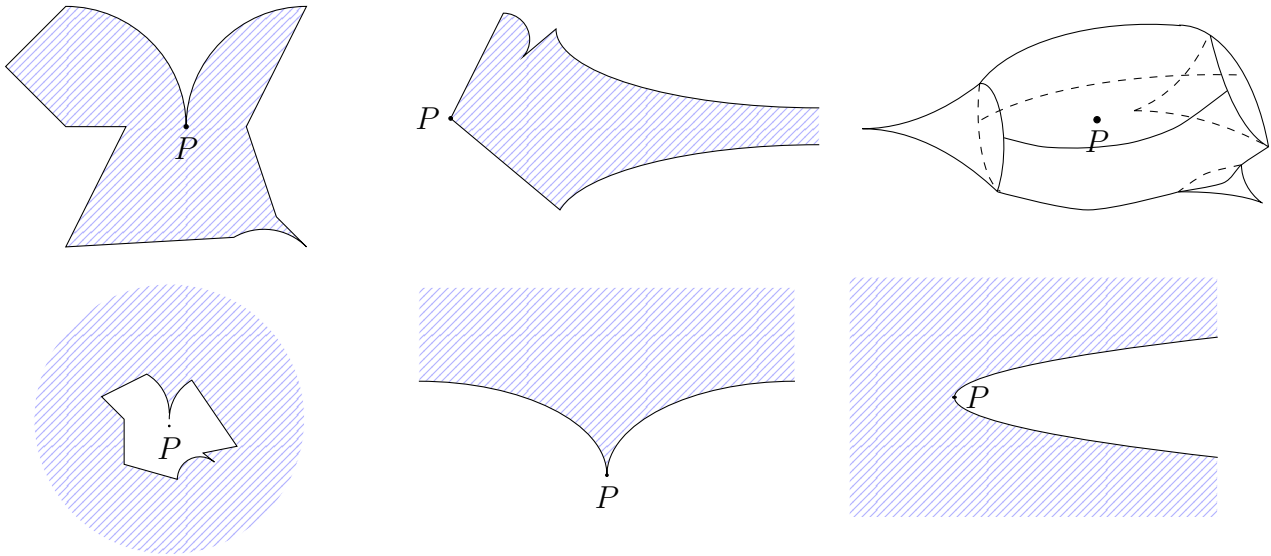
\begin{figure}[H]\label{F1}
	\centering	
	\subfigure {
		\begin{minipage}[b]{0.24\linewidth}
			\centering
			\begin{tikzpicture}[scale=0.8]
				\filldraw[pattern color=blue!30,  pattern=north east lines, draw=black]   (0,0) arc (180:90:2)--(1,0)--(1.5,-1.5)--(2,-2) arc(45:120:1)--(-2,-2)--(-1,0)--(-2,0)--(-3,1)--(-2,2) arc(90:0:2);
				\filldraw[fill=black, draw=black] (0,0) circle (1pt) node[below] {$P$};
			\end{tikzpicture}
			\label{F1-1}
		\end{minipage}
	}\hspace{20pt}
	\subfigure {
		\begin{minipage}[b]{0.24\linewidth}
			\centering
			\begin{tikzpicture}[scale=0.7]
				\filldraw[pattern color=blue!30,  pattern=north east lines, draw=white]  (0,0)--(1,2) arc(90:-40:0.5)--(2,1.7) arc (-180:-90: 5 and 1.5)--(7,-0.5) arc (90:170: 5 and 1.5)--(0,0);
				\draw   (0,0)--(1,2) arc(90:-40:0.5)--(2,1.7) arc (-180:-90: 5 and 1.5) (7,-0.5) arc (90:170: 5 and 1.5)--(0,0);
				\filldraw[fill=black, draw=black] (0,0) circle (1pt) node[left] {$P$};
			\end{tikzpicture}
			\label{F1-2}
		\end{minipage}
	}\hspace{40pt}
\subfigure {
	\begin{minipage}[b]{0.24\linewidth}
		\centering
		\begin{tikzpicture}[scale=1.2]
			\draw [dashed]  (1.3, 0.5) arc (160:260:0.3 and 0.9)  (3.01, 0.2)parabola bend (3.01, 0.2) (4.5, -0.2)(3.01, 0.2) arc(-70:-22:1.41  and 1.5)  (1.32,0.1) arc(150:84: 3 and 1);
			\draw (0,0) parabola bend (0,0) (1.5,-0.7)  (0,0) parabola bend (0,0) (1.3, 0.5) arc (160:80:2 and 1)  (1.3, 0.5) arc (94:-45:0.25 and 0.7)  (4.5, -0.2) arc (260:215:1 and 3) (4.5, -0.2) arc (19:90:1.05 and 2) (3.5, -0.7)arc (94:30:1 and 0.25) (4.2,-0.4)arc (180:225:0.75 and 0.59);
			\draw  plot[smooth] coordinates{(1.5,-0.7)   (2.5,-0.9)(3.5, -0.7) };
			\draw  plot[smooth] coordinates{ (4.5, -0.2)(4.2,-0.4)   (4,-0.6)(3.5, -0.7)  };
			\draw  plot[smooth] coordinates{ (1.57,-0.1) (2,-0.2)(2.5,-0.21)(3,-0.15) (3.5,0.03) (4.04,0.42)};
			\draw[dashed] plot[smooth] coordinates{  (4.2,-0.4)   (3.8,-0.5)(3.5, -0.7)  };
			\filldraw[fill=black, draw=black] (2.6,0.1) circle (1pt) node[below] {$P$};
		\end{tikzpicture}
		\label{F1-0}
	\end{minipage}
}\hspace{30pt}
	\subfigure {
		\begin{minipage}[b]{0.24\linewidth}
			\centering
			\begin{tikzpicture}[scale=0.3]
				
				\filldraw[pattern color=blue!30,  pattern=north east lines, draw=white,   even odd rule] (0,0) arc (180:120:2)--(3,-1.2) --(1.5,-1.5)--(2,-1.9) arc(50:180:1)--(-2,-2) --(-2,0)--(-3,1)--(-1,2) arc(60:0:2)  (0:0) circle(6cm);
				\draw (0,0) arc (180:120:2) --(3,-1.2)--(1.5,-1.5)--(2,-1.9) arc(50:180:1)--(-2,-2) --(-2,0)--(-3,1)--(-1,2) arc(60:0:2);
				\filldraw[fill=black, draw=black] (0,-0.3) circle (1pt) node[below] {$P$};
			\end{tikzpicture}
			\label{F1-3}
		\end{minipage}
	}\hspace{20pt}
	\subfigure {
		\begin{minipage}[b]{0.24\linewidth}
			\centering
			\begin{tikzpicture}[scale=0.5]
				\filldraw[pattern color=blue!30,  pattern=north east lines, draw=white]  (0,0)  arc ( 180: 90: 5 and 2.5)--(5,5)--(-5, 5)--(-5,2.5) arc (90:0: 5 and 2.5)--(0,0);
				\draw   (0,0)  arc ( 180: 90: 5 and 2.5) (-5,2.5) arc (90:0: 5 and 2.5)--(0,0);
				\filldraw[fill=black, draw=black] (0,0) circle (1pt) node[below] {$P$};
			\end{tikzpicture}
			\label{F1-4}
		\end{minipage}
	}\hspace{30pt}
\subfigure {
		\begin{minipage}[b]{0.24\linewidth}
			\centering
			\begin{tikzpicture}[rotate=90, xscale=0.4, yscale=0.7]
				\filldraw[pattern color=blue!30,  pattern=north east lines, draw=white]   (-2,-3)--(-4,-3)--(-4,4)--(4,4)--(4,-3)-- (2,-3) parabola bend (0,2) (-2,-3);
				\draw (2,-3) parabola bend (0,2) (-2,-3);
				\filldraw[fill=black, draw=black] (0,2) circle (1pt) node[right] {$P$};
			\end{tikzpicture}
			\label{F1-5}
		\end{minipage}
	}
	\centering
	\caption{ MSS applicable domains (examples)}
\end{figure}

\emph{Cone-like domains} are \emph{special examples} of MSS applicable domains.
\begin{defn}[Cone-like domains $\mathcal{C}$]\label{cone}
Let $\Sigma$ be an \emph{arbitrary sub-domain} of the unit sphere $S^{n-1}\subset\mathbb{R}^{n}$ with center $0$, the notation $$\mathcal{C}_{P,\Sigma}:=\bigcup\limits_{0<r<+\infty}r\Sigma+P$$
denotes the \emph{unbounded cone} with vortex $P\in\mathbb{R}^{n}$ and cross-section $\Sigma$;
$$\mathcal{C}^{\Omega}_{P,\Sigma}:=\mathcal{C}_{P,\Sigma}\bigcap\Omega$$
denotes the \emph{bounded cone-like domain} with vortex $P$, cross-section $\Sigma$ and bounded g-radially convex domain $\Omega$; and $$\mathcal{C}^{\Omega,e}_{P,\Sigma}:=\mathcal{C}_{P,\Sigma}\setminus\overline{\mathcal{C}^{\Omega}_{P,\Sigma}}$$
denotes the \emph{unbounded fan-shaped exterior domain} of $\mathcal{C}^{\Omega}_{P,\Sigma}$ (in the unbounded cone $\mathcal{C}_{P,\Sigma}$) with vortex $P$, cross-section $\Sigma$ and bounded MSS applicable domain $\Omega$. We use the notations $\mathcal{C}^{R}_{P,\Sigma}:=\mathcal{C}_{P,\Sigma}\bigcap B_{R}(P)$ and $\mathcal{C}^{R,e}_{P,\Sigma}:=\mathcal{C}_{P,\Sigma}\setminus\overline{\mathcal{C}^{R}_{P,\Sigma}}$ for short. A domain $\mathcal{C}$ is called a \emph{cone-like domain} if $\mathcal{C}=\mathcal{C}_{P,\Sigma}$, $\mathcal{C}^{\Omega}_{P,\Sigma}$ or $\mathcal{C}^{\Omega,e}_{P,\Sigma}$.

\medskip

For any $1\leq k\leq n$, we define the \emph{$\frac{1}{2^{k}}$-unit sphere} by
$$S^{n-1}_{2^{-k}}:=O\left(\{x=(x_{1},\cdots,x_{n})\in S^{n-1}\mid x_{1},\cdots,x_{k}>0\}\right),$$
where $O$ denotes an \emph{arbitrary orthogonal transform} on $\mathbb{R}^{n}$. Both the $\frac{1}{2^{k}}$-space $\mathbb{R}^{n}_{2^{-k}}:=\mathcal{C}_{0,S^{n-1}_{2^{-k}}}$ and the $\frac{1}{2^{k}}$-ball $B^{2^{-k}}_{R}(0):=\mathcal{C}^{R}_{0,S^{n-1}_{2^{-k}}}$ are (unbounded or bounded) cone-like domains with cross-section $\Sigma=S^{n-1}_{2^{-k}}$ and vertex $0$.
\end{defn}

\begin{figure}[H]\label{F2}
		\centering	
		\subfigure[ $\mathcal{C}_{P,\Sigma}$]{
	 \begin{minipage}[b]{0.24\linewidth}
				\centering
					\begin{tikzpicture}
					\node[below] at (0:0){$P$};
                    \filldraw[fill=gray!10, draw=gray!10] (0:0)--(55:8) arc (55:80:8) (80:8)--(0:0)--cycle;
					\draw (0:0) -- (55:6) (0:0) -- (80:6);
					\node[]at(62.5:5) {$\mathcal C_{P,\Sigma}$};
					\draw[dashed] (55:6)--(55:8) (80:6)--(80:8);
					\draw (55:2.5) arc (55:80:2.5);
                    \node[below] at (65:2.5){$\Sigma$};
				\end{tikzpicture}
				
				\label{F2-1}
 \end{minipage}
		}\hspace{15pt}
		\subfigure [ $\mathcal{C}^{R}_{P,\Sigma}$]{
	 \begin{minipage}[b]{0.24\linewidth}
				\centering
				\begin{tikzpicture}
				\node at (230:0.3) {$P$};
				
				\filldraw[fill=gray!10] (0:0)--(55:4.5) arc (55:80:4.5) (80:4.5)--(0:0)--cycle;
				\draw (55:2.5) arc (55:80:2.5);
                \node[below] at (65:2.5){$\Sigma$};
				\node[]at(62.5:4) {$\mathcal C_{P,\Sigma}^R$};
				\draw[|<->|,shift={(300:0.2)}] (0:0) -- (55:4.5) node[midway, below] {$R$};
			\end{tikzpicture}
				\label{F2-2}
  \end{minipage}
		}\hspace{15pt}
	\subfigure[ $\mathcal{C}^{R,e}_{P,\Sigma}$] {
 \begin{minipage}[b]{0.24\linewidth}
			\centering
			\begin{tikzpicture}
				\filldraw[fill=gray!10, draw=gray!10] (55:3) arc (55:80:3) (80:3)--(80:8) arc (80:55:8)--(55:3)--cycle;
				\draw  (55:3) arc (55:80:3);
				\draw (55:3)--(55:6) (80:3)--(80:6);
				\node at (230:0.3){$P$};
				\draw[dashed] (0:0) -- (55:3) (0:0) -- (80:3) (55:6)--(55:8) (80:6)--(80:8);
				\node[]at(62.5:5) {$\mathcal{C}^{R,e}_{P,\Sigma}$};
                \draw[|<->|,shift={(300:0.2)}] (0:0) -- (55:3) node[midway, below] {$R$};
				\draw[dashed] (55:2.5) arc (55:80:2.5);
                \node[below] at (65:2.5){$\Sigma$};
			\end{tikzpicture}
			\label{F2-3}
  \end{minipage}
	}%
		\centering
		\caption{ Unbounded cone $\mathcal{C}_{P,\Sigma}$, bounded cone $\mathcal{C}^{R}_{P,\Sigma}$ and its exterior domain $\mathcal{C}^{R,e}_{P,\Sigma}$}
	\end{figure}
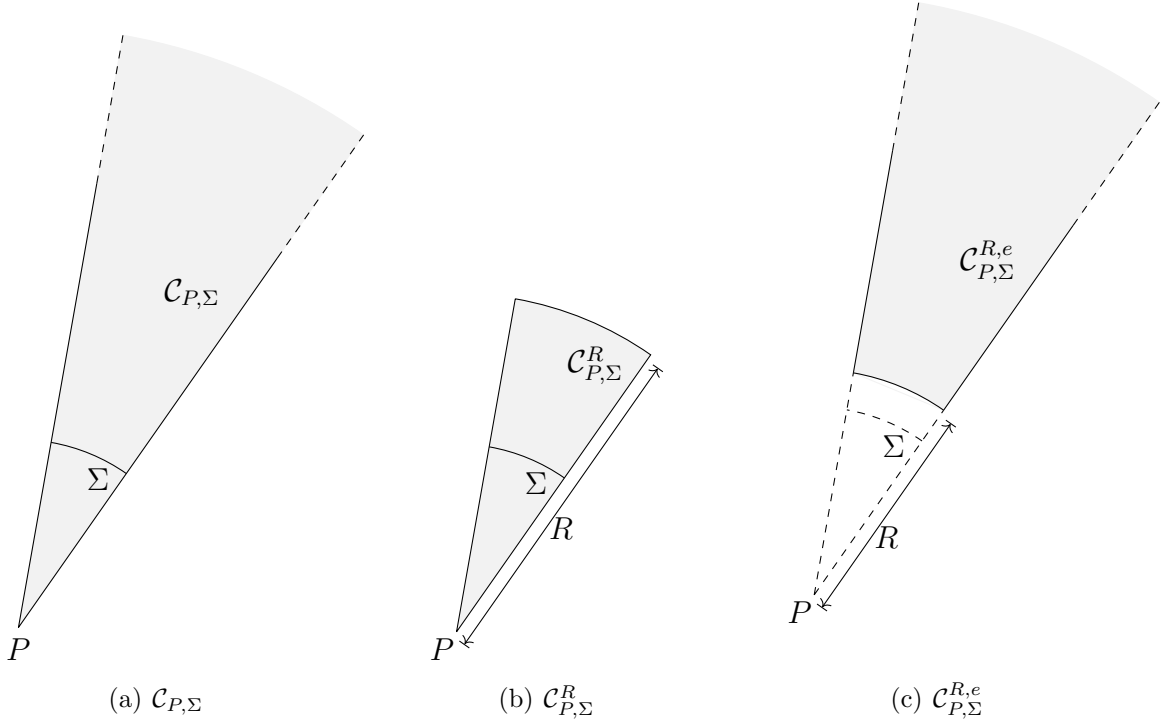

\begin{rem}\label{rem2}
For example, $\mathbb{R}^{n}\setminus\{0\}$, $\mathbb{R}^{n}_{+}$, $B_{R}(0)\setminus\{0\}$, $\mathcal{C}^{R}_{P,\Sigma}$, $\mathcal{C}^{R,e}_{P,\Sigma}$ and $\mathbb{R}^{n}\setminus\overline{B_{R}(0)}$, $\mathbb{R}^{k}\times\mathcal{C}$ with unbounded cone $\mathcal{C}\subseteq\mathbb{R}^{n-k}$ and $k=0,\cdots,n$ and so on $\cdots$ are cone-like domains. In particular, $S^{n-1}$ is a $\frac{1}{2^{k}}$-unit sphere with $k=0$, $\mathbb{R}^{n}_{+}:=\{x\in\mathbb{R}^{n}\mid \, x_{n}>0\}$ is a $\frac{1}{2^{k}}$-space with $k=1$, $\{x\in \mathbb{R}^{n}\mid x_{1},\cdots,x_{k}>0\}$ is a $\frac{1}{2^{k}}$-space with $k\geq1$ and $\{x\in B_{R}(0)\mid x_{1},\cdots,x_{k}>0\}$ is a $\frac{1}{2^{k}}$-ball with $k\geq1$.
\end{rem}

\begin{rem}\label{rem3}
From Definition \ref{mss-a}, one can see that MSS applicable domains include $\mathbb{R}^{n}$, $\mathbb{R}^{n}_{+}$, ball $B_{R}(0)$, cone-like domains $\mathcal{C}$ (unbounded cone $\mathcal{C}_{P,\Sigma}$, bounded cone-like domain $\mathcal{C}^{\Omega}_{P,\Sigma}$, exterior domain $\mathcal{C}^{\Omega,e}_{P,\Sigma}$ of a bounded cone-like domain (in the unbounded cone)), $\mathbb{R}^{k}\times\mathcal{C}$ with cone-like domain $\mathcal{C}\subseteq\mathbb{R}^{n-k}$ and $k=0,\cdots,n$, convex domain, star-shaped domain and \emph{all the complements of their closures} and so on $\cdots$.
\end{rem}

\begin{rem}\label{rem1}
Let $\Omega$ be a MSS applicable domain with center $P$. If $\Omega$ is a g-radially convex domain, then for any $0<r<\rho_{\Omega}:=\sup\limits_{x\in\Omega}|x-P|$, one has $S_{r}(P)\bigcap\Omega\neq\emptyset$ and the bounded cone $\mathcal{C}^{r}_{P,\Sigma^{r}_{\Omega}}=\mathcal{C}_{P,\Sigma^{r}_{\Omega}}\bigcap B_{r}(P)\subset\Omega$, where $S_{r}(P):=\partial B_{r}(P)$ and the cross-section $\Sigma^{r}_{\Omega}:=r^{-1}\left((S_{r}(P)\bigcap\Omega)-P\right)\subseteq S^{n-1}$. Though it is possible that $\Sigma^{r}_{\Omega}$ is a union of several disjoint connected cross-sections, i.e., $\Sigma^{r}_{\Omega}=\bigcup\limits_{i=1}^{k}\Sigma^{r}_{\Omega,i}$ for some $k\geq1$, we still use the natural notation $\mathcal{C}^{r}_{P,\Sigma^{r}_{\Omega}}$ to denote $\bigcup\limits_{i=1}^{k}\mathcal{C}^{r}_{P,\Sigma^{r}_{\Omega,i}}$ for the sake of simplicity. If $\mathbb{R}^{n}\setminus\overline{\Omega}$ is a g-radially convex domain, then by Definition \ref{g-rcd}, for any $d_{\Omega}:=\min\limits_{x\in\Omega}|x-P|<r<+\infty$, one has $S_{r}(P)\bigcap\Omega\neq\emptyset$ and the exterior fan-shaped domain $\mathcal{C}^{r,e}_{P,\Sigma^{r}_{\Omega}}:=\mathcal{C}_{P,\Sigma^{r}_{\Omega}}\setminus\overline{\mathcal{C}^{r}_{P,\Sigma^{r}_{\Omega}}}\subset\Omega$. In a word, there are \emph{infinite (bounded, unbounded, resp.) cones} in a (bounded, unbounded, resp.) MSS applicable domain.
\end{rem}

\begin{rem}[Equivalent definition of MSS applicable domain]\label{rem14}
A domain $\Omega$ is g-radially convex if and only if there exists $P\in\overline{\Omega}$ such that $\Sigma^{r}_{\Omega}$ is \emph{monotone decreasing} w.r.t. $r\in(0,\rho_{\Omega})$ in the sense of inclusion relationship. $\mathbb{R}^{n}\setminus\overline{\Omega}$ is g-radially convex if and only if there exists $P\in\mathbb{R}^{n}\setminus\Omega$ such that $\Sigma^{r}_{\Omega}$ is \emph{monotone increasing} w.r.t. $r\in(d_{\Omega},+\infty)$ in the sense of inclusion relationship. Consequently, we say $\Omega$ is a MSS applicable domain if and only if there exists $P\in\mathbb{R}^{n}$ such that $\Sigma^{r}_{\Omega}$ is either monotone increasing w.r.t. $r\in(d_{\Omega},+\infty)$ or monotone decreasing w.r.t. $r\in(0,\rho_{\Omega})$.
\end{rem}

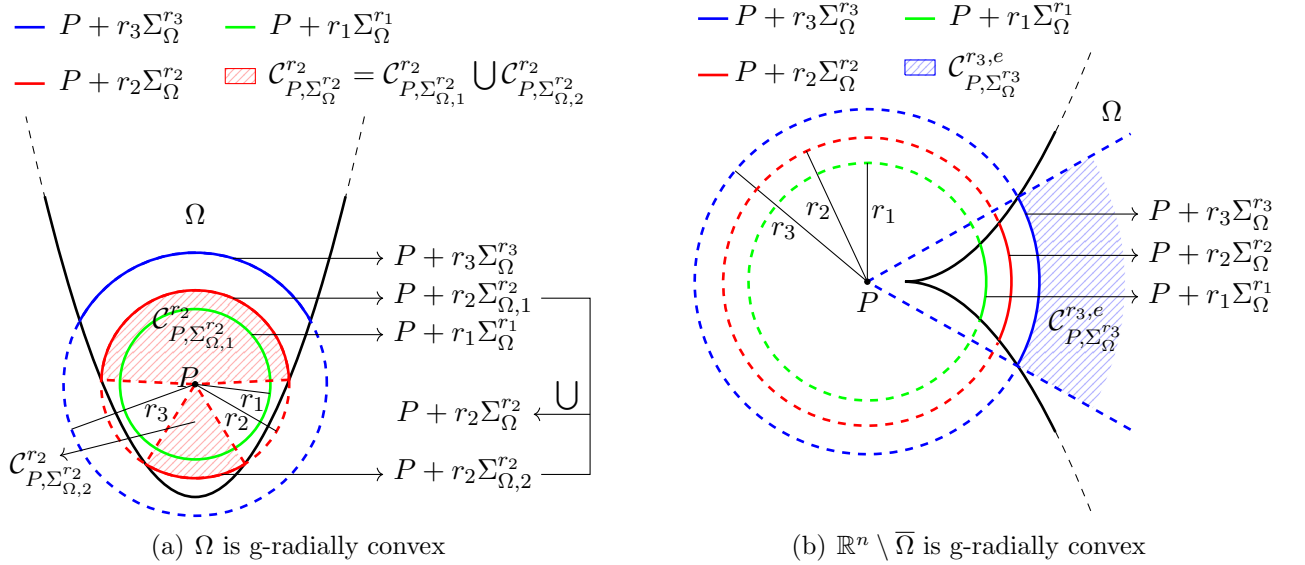
\begin{figure}[H]\label{F10}
	\centering	
	\subfigure[$\Omega$ is g-radially convex]{
		\begin{minipage}[b]{0.45\linewidth}
			\centering
			\begin{tikzpicture}[ scale=0.5]	\begin{small}
				\filldraw[pattern color=red!30,  pattern=north east lines, draw=white, even odd rule]  (-2.497,0.118)--(0,0)--(2.497,0.118)  arc (2.7:177.3:2.5)--(-2.497,0.118);
				\filldraw[pattern color=red!30,  pattern=north east lines, draw=white, even odd rule]  (-1.328,-2.118)--(0,0)--(1.328,-2.118) arc (-57.9:-122.1:2.5)--(-1.328,-2.118);
				\draw[line width=1pt] (4,5) parabola bend (0,-3) (-4,5)  ;\draw[dashed] (4.5,7.125) parabola bend (0,-3) (-4.5,7.125)  ;
				\draw[gline] (2,0) arc (0:360:2);
				\draw[rline]  (2.497,0.118) arc (2.7:177.3:2.5)  (1.328,-2.118) arc (-57.9:-122.1:2.5); \draw[rline,dashed] (2.5,0) arc (0:360:2.5)  (2.497,0.118)--(0,0)--(-2.497,0.118)   (1.328,-2.118)--(0,0)-- (-1.328,-2.118);
				\draw[bline] (3.06352,1.69248) arc (28.9:151.6:3.5);\draw[bline,dashed] (3.5,0) arc (0:360:3.5);
				\filldraw[fill=black, draw=black] (0,0) circle (2pt);\node at(-0.2,0.2) {$P$};
				\draw    (0,0)--(-7:2) (0,0)--(200:3.5)  (0,0)--(-30:2.5);\node[left] at (-13:2.2) {$r_1$};\node[right] at (205:1.8) {$r_3$};\node[left] at (-30:2) {$r_2$};\node[above] at (0,4) {$\Omega$};
				\draw[bline] (-4.8,9.5)--(-4,9.5);\node[right] at (-3.9,9.5) {$P+r_3\Sigma_\Omega^{r_3}$};
    \draw[rline]  (-4.8,8)--(-4,8);\node[right] at (-3.9,8) {$P+r_2\Sigma^{r_2}_{\Omega}$};
    \draw[gline]  (0.8,9.5)--(1.6,9.5);\node[right] at (1.7,9.5) {$P+r_1\Sigma_\Omega^{r_1}$};

    \filldraw[pattern color=red!30,  pattern=north east lines, draw=red, even odd rule]  (0.8,8)--(1.6,8)--(1.6,8.4)--(0.8,8.4)--(0.8,8);\node[right] at (1.7,8) {$\mathcal{C}^{r_2}_{P, \Sigma_\Omega^{r_2}}=\mathcal{C}^{r_2}_{P, \Sigma^{r_2}_{\Omega,1}}\bigcup \mathcal{C}^{r_2}_{P, \Sigma^{r_2}_{\Omega,2}}$};
				
				\draw[->] (1, 3.35)--(5, 3.35) node[right] {$P+r_3\Sigma_\Omega^{r_3}$};
				\draw[->] (1, 2.29)--(5, 2.29) node[right] {$P+r_2\Sigma^{r_2}_{\Omega,1}$};\draw[->] (0.7,-2.4)-- (5, -2.4);\node[right] at (5,-2.4) {$P+r_2\Sigma^{r_2}_{\Omega,2}$};
				\draw[->] (1.5, 1.32)--(5, 1.32) node[right] {$P+r_1\Sigma_\Omega^{r_1}$};
				\node[above] at (0,0.6) {$\mathcal{C}^{r_2}_{P, \Sigma^{r_2}_{\Omega,1}}$};
				\node[below] at (-3.8,-1.5) {$\mathcal{C}^{r_2}_{P, \Sigma^{r_2}_{\Omega,2}}$};\draw[->](0,-1)--(-3.6,-1.9);
				\draw (9.2, 2.29)--(10.5, 2.29)--(10.5, -2.4)--(9.2, -2.4); \draw[->] (10.5,-0.8)--(9,-0.8) node[left] {$P+r_2\Sigma_\Omega^{r_2}$};\node[right] at (9.3,-0.3) {$\bigcup$};\end{small}
			\end{tikzpicture}
			\centering
		\end{minipage}
	}\hspace{15pt}
	\subfigure [$\mathbb{R}^{n}\setminus\overline{\Omega}$ is g-radially convex]{
		\begin{minipage}[b]{0.45\linewidth}
			\centering
			\begin{tikzpicture} [  scale=0.5]	\begin{small}
				\filldraw[pattern color=blue!30,  pattern=north east lines, draw=white, even odd rule] (4,2.25) arc (29:-29:4.59)-- (4,-2.25)--(6,-3.375) arc(-29:29:6.88)-- (6,3.375)--(4,2.25);
				\draw[line width=1pt] (1,0) parabola bend (1,0) (5, 4)    (1,0) parabola bend (1,0) (5,-4)  ;
				\draw[dashed] (5, 4)  parabola bend (1,0) (6, 6.25)    (5, -4) parabola bend (1,0) (6, -6.25) ;
				\draw[bline] (4,2.25) arc (29:-29:4.59); \draw[bline,dashed] (30: 4.59) arc (30:330:4.59) (7,-3.9375)--(0,0)--(7,3.9375) ;\draw (0:0)--(140:4.59);
				\draw[rline] (3.5,1.5625) arc (24:-24:3.833); \draw[rline,dashed] (25:3.833) arc (25:335:3.833); \draw  (0:0)--(115:3.833);
				\draw[gline] (3,1) arc (18.4:-18.4:3.16); \draw[gline,dashed] (19:3.16) arc (19:341:3.16) ; \draw  (0:0)--(90:3.16);
				\filldraw[fill=black, draw=black] (0,0) circle (2pt) node[below] {$P$};
				\node[right] at (95:1.8) {$r_1$};\node[left] at (110:2) {$r_2$};\node[left] at (140:2.1) {$r_3$};\node[above] at (6.5,4) {$\Omega$};
				\draw[bline] (-4.5,7)--(-3.7,7);\node[right] at (-3.8,7) {$P+r_3\Sigma_\Omega^{r_3}$};
				\draw[rline]  (-4.5,5.5)--(-3.7,5.5);\node[right] at (-3.8,5.5) {$P+r_2\Sigma_\Omega^{r_2}$};
				\draw[gline]  (1,7)--(1.8,7);\node[right] at (1.9,7) {$P+r_1\Sigma_\Omega^{r_1}$};
				\filldraw[pattern color=blue!30,  pattern=north east lines, draw=blue, even odd rule]  (1,5.5)--(1.8,5.5)--(1.8,5.9)--(1,5.9)--(1,5.5);\node[right] at (1.9,5.5) {$\mathcal{C}^{r_3,e}_{P, \Sigma_\Omega^{r_3}}$};
				\draw[->] (4.2, 1.8)--(7.2, 1.8) node[right] {$P+r_3\Sigma_\Omega^{r_3}$};
				\draw[->] (3.77, 0.7)--(7.2, 0.7) node[right] {$P+r_2\Sigma_\Omega^{r_2}$};
				\draw[->] (3.15, -0.4)--(7.2,- 0.4) node[right] {$P+r_1\Sigma_\Omega^{r_1}$};
				\node[above] at (-20:6.2) {$\mathcal{C}^{r_3,e}_{P, \Sigma_\Omega^{r_3}}$};
		\end{small}	\end{tikzpicture}
			\centering
		\end{minipage}
	}
	\centering
	\caption{Geometric features of MSS applicable domain $\Omega$}
\end{figure}

The generalized nonlinear term $f(x,u)$ in equation \eqref{PDE} may have \emph{singularity} (at the center $P$ of the MSS applicable domain $\Omega$) and satisfies \emph{mild assumptions} (briefly speaking, $f(x,u)\geq C|x-P|^{a}u^{p}$ with $a>-2s$ and $p\geq1$ in $\mathcal{C}^{r}_{P,\Sigma}$ with $r\in(0,\rho_{\Omega})$ if $\Omega$ is a g-radially convex domain or in $\mathcal{C}^{r,e}_{P,\Sigma}$ with $r\in(d_{\Omega},+\infty)$ if $\mathbb{R}^{n}\setminus\overline{\Omega}$ is a g-radially convex domain, where $\Sigma\subset\Sigma^{r}_{\Omega}$ is a smaller cross-section). A typical case is \emph{H\'{e}non-Hardy type nonlinearity} $f(x,u)=|x-P|^{a}u^{p}$ with $a>-2s$ and $1\leq p<p_{c}(a):=\frac{n+2s+2a}{n-2s}$ ($p_{c}(a)=+\infty$ if $n=2s$). When $0<s<1$, the nonlocal fractional Laplacian $(-\Delta)^{s}$ is defined by (see e.g. \cite{CT,CLM,C,DQ1,DPV,Si})
\begin{equation}\label{nonlocal defn}
  (-\Delta)^{s}u(x)=C_{s,n} \, P.V.\int_{\mathbb{R}^n}\frac{u(x)-u(y)}{|x-y|^{n+2s}}\dd y:=C_{s,n}\lim_{\epsilon\rightarrow0}\int_{|y-x|\geq\epsilon}\frac{u(x)-u(y)}{|x-y|^{n+2s}}\dd y
\end{equation}
for any function $u\in C^{[2s],\{2s\}+\eps}_{loc}\cap\mathcal{L}_{s}(\mathbb{R}^{n})$ with $\eps>0$ arbitrarily small, where the constant $C_{s,n}=\left(\int_{\mathbb{R}^{n}}\frac{1-\cos(2\pi\zeta_{1})}{|\zeta|^{n+2s}}\dd\zeta\right)^{-1}$, $[x]$ denotes the largest integer $\leq x$, $\{x\}:=x-[x]$, and the function space
\begin{equation}\label{0-1}
  \mathcal{L}_{s}(\mathbb{R}^{n}):=\bigg\{u: \mathbb{R}^{n}\rightarrow\mathbb{R}\,\Big|\,\int_{\mathbb{R}^{n}}\frac{|u(x)|}{1+|x|^{n+2s}}\dd x<+\infty\bigg\}.
\end{equation}

\medskip

In recent years, fractional order operators have attracted more and more attentions. Besides various applications in fluid mechanics, molecular dynamics, relativistic quantum mechanics of stars (see e.g. \cite{CV,C}) and conformal geometry (see e.g. \cite{CG}), it also has many applications in probability and finance (see e.g. \cite{Ber,CT}). The fractional Laplacians can be understood as the infinitesimal generators of stable L\'{e}vy diffusion processes (see e.g. \cite{Ber}). The fractional Laplacians can also be defined equivalently (see \cite{CLM}) by Caffarelli and Silvestre's extension method (see \cite{CS}) for any $u\in C^{[2s],\{2s\}+\eps}_{loc}\cap\mathcal{L}_{s}(\mathbb{R}^{n})$. For more interesting works on fractional order nonlocal problems, please refer to \cite{BKN,CS,DDGW,DPV,DRV,DSV,FW,FS,FLS,RS,Si,TTV} and the references therein.

\begin{defn}\label{defn1}
In subcritical order case $2s<n$, we say that the nonlinear term $f$ has \emph{subcritical (critical, supercritical, resp.) growth} provided that
\begin{equation}\label{e1}
  \mu^{\frac{n+2s}{n-2s}}f(\mu^{\frac{2}{n-2s}}(x-P)+P,\mu^{-1}u)
\end{equation}
is \emph{strictly increasing (invariant, strictly decreasing, resp.)} with respect to $\mu\geq1$ or $\mu\leq1$ for all $(x,u)\in(\Omega\setminus\{P\})\times\mathbb{R}_{+}$.

\noindent In critical order case $2s=n$, we say that the nonlinear term $f$ satisfies \emph{subcritical condition (on singularity)}, provided that
\begin{equation}\label{e1}
  \mu^{n}f(\mu(x-P)+P,u)
\end{equation}
is \emph{strictly increasing} with respect to $\mu\geq1$ or $\mu\leq1$ for all $(x,u)\in(\Omega\setminus\{P\})\times\mathbb{R}_{+}$. We say $f$ satisfies \emph{critical condition (on singularity)} provided that $\mu^{n}f(\mu x,u)$ is \emph{invariant} with respect to $\mu$.
\end{defn}

\begin{rem}\label{rem4}
In the critical order case $2s=n$, being different from the \emph{subcritical growth condition on $u$} in previous works (see e.g. \cite{BM,CY,CDQ,CL,DD,DQ,Lin,LZ,WX}), there is \emph{no growth condition} on $u$ in Definition \ref{defn1} of the \emph{subcritical condition (on singularity)} (see also \cite{DQ3}). As a consequence, suppose without loss of generalities that $P=0$, then $f(x,u)=|x|^{a}g(u)$ with $a>-n$ satisfies the \emph{subcritical condition (on singularity)}, while $f(x,u)=|x|^{-n}g(u)$ satisfies the \emph{critical condition (on singularity)}, where $g(u)$ is arbitrary and may take various forms including: $u^{p}$ with $p\in \mathbb{R}$, $u^{p}\log(\tau+u^{q})$ with $p,q\in\mathbb{R}$ and $\tau\in[1,+\infty)$, $u^{p}e^{\kappa u}$ with $p\in \mathbb{R}$ and $\kappa\in\mathbb{R}$, $u^{p}\arctan(\kappa u)$, $u^{p}\cosh(\kappa u)$ and $u^{p}\sinh(\kappa u)$ with $p,\kappa\in\mathbb{R}$, or even $u^{p}\exp^{(m)}(\kappa u)$ with $p,\kappa\in \mathbb{R}$, $\exp^{(m)}:=\underbrace{\exp\cdots\exp}_{m \, \text{times}}$ ($m\geq1$) and $\exp(t):=e^{t}$, and so on. The subcritical growth condition on $u$ in previous works only admits the nonlinearity that has arbitrary order polynomial growth on $u$ (say, $f=u^{p}$ with $0<p<+\infty$), while the exponential type nonlinearity $f=e^{\kappa u}$ with $\kappa>0$ satisfies the critical growth condition on $u$ (see e.g. \cite{BM,CY,CDQ,CL,DD,DQ,Lin,LZ,WX}). The \emph{subcritical condition (on singularity)} reveals that the method of scaling spheres is essentially different from previous methods and especially suitable for dealing with problems with singularities.
\end{rem}

\begin{defn}\label{defn2}
For $1\leq q\leq +\infty$, a function $g(x,u)$ is called \emph{locally $L^q$-Lipschitz} on $u$ in $\Omega\times\overline{\mathbb{R}_{+}}$ ($\Omega\times\mathbb{R}_{+}$), provided that for any $u_{0}\in\overline{\mathbb{R}_{+}}$ ($u_{0}\in\mathbb{R}_{+}$) and $\omega\subseteq\Omega$ bounded, there exists a (relatively) open neighborhood $U(u_{0})\subset\overline{\mathbb{R}_{+}}$ ($U(u_{0})\subset\mathbb{R}_{+}$) and a non-negative function $h_{u_0}\in L^q(\omega)$ such that
     	$$\sup_{u,v\in U(u_{0})} \frac{|g(x,u)-g(x,v)|}{|u-v|}\leq h_{u_0}(x), \qquad  \forall\, x\in \omega.$$
\end{defn}

We need the following three assumptions on the nonlinear term $f(x,u)$.
\begin{itemize}
\item[$(\mathbf{f_{1}})$] The nonlinear term $f$ is \emph{non-decreasing} about $u$ in $(\Omega\setminus\{P\})\times\overline{\mathbb{R}_{+}}$, namely,
\begin{equation*}\label{e2}
  (x,u), \, (x,v)\in(\Omega\setminus\{P\})\times\overline{\mathbb{R}_{+}} \,\,\, \text{with} \,\,\, u\leq v \,\,\, \text{implies} \,\,\, f(x,u)\leq f(x,v).
\end{equation*}
\item[$(\mathbf{f_{2}})$] In the sense of Definition \ref{defn2}, $f(x,u)$ is locally $L^{q}$-Lipschitz on $u$ in $\Omega\times\overline{\mathbb{R}_{+}}$ for some $1\leq q\leq+\infty$. \\
\item[$(\mathbf{f_{3}})$] (i) If $\mathbb{R}^{n}\setminus\overline{\Omega}$ is a g-radially convex domain, there exist $d_{\Omega}<r<+\infty$, a smaller cross-section $\Sigma\subseteq\Sigma^{r}_{\Omega}$ (see Remark \ref{rem1} for definition of $\Sigma^{r}_{\Omega}$ and $d_{\Omega}$), constants $C>0$, $-2s<a<+\infty$, $1\leq p<p_{c}(a)$ if $f$ is subcritical, $p=p_{c}(a)$ if $f$ is critical and $p_{c}(a)<p<+\infty$ if $f$ is supercritical such that the nonlinear term
    \begin{equation*}
    f(x,u)\geq C|x-P|^{a}u^{p}, \qquad \forall\,\,  (x,u)\in\mathcal{C}^{r,e}_{P,\Sigma}\times\overline{\mathbb{R}_{+}};
    \end{equation*}
    (ii) if $\Omega$ is a g-radially convex domain, there exist $0<r<\rho_{\Omega}$, a smaller cross-section $\Sigma\subseteq\Sigma^{r}_{\Omega}$ (see Remark \ref{rem1} for definition of $\Sigma^{r}_{\Omega}$ and $\rho_{\Omega}$), constants $C>0$, $-2s<a<+\infty$, $1\leq p<p_{c}(a)$ if $f$ is subcritical, $p=p_{c}(a)$ if $f$ is critical and $p_{c}(a)<p<+\infty$ if $f$ is supercritical such that the nonlinear term
    \begin{equation*}
    f(x,u)\geq C|x-P|^{a}u^{p}, \qquad  \forall\,\,  (x,u)\in\mathcal{C}^{r}_{P,\Sigma}\times\overline{\mathbb{R}_{+}}.
    \end{equation*}
\end{itemize}

\begin{rem}\label{rem5}
The Hardy-H\'{e}non type nonlinearity $f(x,u)=|x-P|^{a}u^{p}$ with $a>-2s$ and $p\geq1$ is a typical model satisfying the assumptions $(\mathbf{f_{1}})$, $(\mathbf{f_{2}})$ and $(\mathbf{f_{3}})$.
\end{rem}

When $s\geq1$ is an integer, we assume the classical solution $u$ to PDEs \eqref{PDE} satisfies $u\in C^{2s}(\Omega)$ ($u\in C^{2s}(\Omega\setminus\{P\})$ if $\Omega$ is a g-radially convex domain, $P\in\Omega$ and $f$ satisfies $(\mathbf{f_{3}})$ with $a<0$), $(-\Delta)^{i}u\in C(\overline{\Omega})$ ($i=0,1,\cdots,s-1$) and the Navier boundary condition $u=-\Delta u=\cdots=(-\Delta)^{s-1}u=0$ on $\partial\Omega$. When $s\in(0,1)$, we assume the classical solution $u$ to PDEs \eqref{PDE} satisfies $u\in C^{[2s],\{2s\}+\eps}_{loc}(\Omega)\bigcap\mathcal{L}_{s}(\mathbb{R}^{n})\bigcap C(\overline{\Omega})$ ($u\in C^{[2s],\{2s\}+\eps}_{loc}(\Omega\setminus\{P\})\bigcap\mathcal{L}_{s}(\mathbb{R}^{n})\bigcap C(\overline{\Omega})$ if $\Omega$ is a g-radially convex domain, $P\in\Omega$ and $f$ satisfies $(\mathbf{f_{3}})$ with $a<0$) with $\eps>0$ arbitrarily small and the Dirichlet condition $u=0$ in $\mathbb{R}^{n}\setminus\Omega$.

\bigskip

In the typical case $f(x,u)=|x-P|^{a}u^{p}$, PDE \eqref{PDE} takes the form
\begin{equation}\label{HH}
  (-\Delta)^{s}u(x)=|x-P|^{a}u^{p}(x)
\end{equation}
with $0<s<+\infty$ are called the fractional order or higher order H\'{e}non, Lane-Emden, Hardy equations for $a>0$, $a=0$, $a<0$, respectively. These equations have numerous important applications in conformal geometry and Sobolev inequalities. In particular, in the case $a=0$, \eqref{HH} becomes the well-known Lane-Emden equation, which models many phenomena in mathematical physics and in astrophysics. We say equations \eqref{HH} have subcritical order if $0<2s<n$, critical order if $2s=n$ and supercritical order if $2s>n$. The nonlinear term in \eqref{HH} has critical growth if $p=p_{c}(a):=\frac{n+2s+2a}{n-2s}$ ($:=+\infty$ if $n\leq2s$) and has subcritical growth if $0<p<p_{c}(a)$.

\medskip

Liouville type theorems on nonnegative classical solutions to PDEs of type \eqref{PDE} or \eqref{HH} (and related IEs) in the whole space $\mathbb{R}^n$ and in the half space $\mathbb{R}^n_+$ have been extensively studied (see e.g. \cite{BP,CDQ0,CC,CDQ,CFL,CLL,CLZ,CLZC,DLQ,DQ1,DQ2,DQ3,DQ0,DQZ,DFSV,Farina,FW,GS,GL,GW,Lin,LZ1,MP,NY,PS,RW,TD,WX,ZCCY} and the references therein). Refer to e.g. \cite{MO} for Liouville theorems in Heisenberg group and \cite{CGS,CDQ0,CY,CL,CLL,CLO,CLZ,CFR,DLQ,DQ1,DQ,Lin,Li,LZ,SX,WX} for classification results of solutions in $\mathbb{R}^n$ and $\mathbb{R}^n_+$. For Liouville type theorems and related properties on systems of PDEs or IEs with respect to various types of solutions (e.g., stable, radial, weak, sign-changing, $\cdots$), please refer to e.g. \cite{BG,DDGW,FG,FS,LZ2,LB,LZou,M,PQS,RZ,S,Souto,SZ} and the references therein.

\medskip

For H\'{e}non-Hardy type equations \eqref{HH} in $\Omega$, in previous works, there are mainly \emph{three approaches} to derive Liouville type results for nonnegative classical solutions: \emph{the method of moving planes in conjunction with Kelvin transforms} (see e.g. \cite{CGS,CC,CY,CDQ,CL,CL4,CLL,CLO,DQ1,GNN,GS,Lin,WX} and the references therein), \emph{the integral estimates and feedback estimates arguments based on Pohozaev identities} (see, e.g. \cite{M,MP,PQS,PS,S,SZ} and the references therein), and \emph{the method of moving spheres} (see e.g. \cite{CL1,CLZ,DLQ,DQ,FKT,Li,LZ,LZ1,Pa,RZ} and the references therein). To establish Liouville type theorems, the method of moving planes (in conjunction with Kelvin transforms) requires $\Omega=\mathbb{R}^{n}$ or $\mathbb{R}^{n}_{+}$ and relies strongly on the structure of the equation (e.g., the potential term in nonlinearity must be \emph{strictly decreasing} as $|x|\rightarrow+\infty$ after Kelvin transform), and hence it usually requires $p>1$, in addition, $p<\frac{n+2s+a}{n-2s}<p_{c}(a):=\frac{n+2s+2a}{n-2s}$ if $a>0$, and $p>\frac{n+a}{n-2s}$ if $\Omega=\mathbb{R}^{n}_{+}$. The integral estimates and feedback estimates arguments based on Pohozaev identities usually requires the domain $\Omega=\mathbb{R}^{n}$ with special spatial dimension $n$ (say, $n=3$) or is bounded star-shaped with regular boundary such that the formula of integration by parts holds, $2s<n$ and \emph{boundedness assumptions} on $u$ if $a\neq0$. In order to derive classification results or Liouville type theorems, the method of moving spheres needs to assume $\Omega=\mathbb{R}^{n}$ or $\mathbb{R}^{n}_{+}$ and move spheres centered at \emph{every point} in $\mathbb{R}^{n}$ or $\partial\mathbb{R}^{n}_{+}$ \emph{synchronously} (i.e., either all stop at finite radius or all stop at $+\infty$) in conjunction with \emph{several key calculus lemmas}. As a consequence, the method of moving spheres usually requires \emph{the translation invariant property} of the problem and hence can hardly be employed to deal with various fractional or higher order problems \emph{with singularities} or \emph{without translation invariance} in general domains (e.g., the cases $a\neq0$ or $\Omega$ is general interior or exterior domain).

\medskip

In \cite{DQ0}, the authors introduced the \emph{method of scaling spheres} and applied it to establish Liouville theorems for PDEs \eqref{PDE} and related IEs in $\mathbb{R}^{n}$, $\mathbb{R}^{n}_{+}$, $B_{R}(0)$ and $\mathbb{R}^{n}\setminus\overline{B_{R}(0)}$ (see also e.g. \cite{CDQ0,DQ2,DQ3,DQZ}). One should note that, being different from the above three approaches, the key idea in \emph{the method of scaling spheres} is combining the \emph{lower bound estimates on the asymptotic behavior} of positive solutions w.r.t. \emph{singularities} or \emph{$\infty$} derived through the \emph{scaling spheres procedure} and the \emph{``Bootstrap" technique} with the precise estimates on \emph{the asymptotic behavior of Green's function} in a interior cone and the a priori upper bound estimates on \emph{the integrability} or on \emph{the asymptotic behavior} of positive solutions, which can finally lead to a \emph{sharp Liouville type result} by contradiction arguments. Besides improving most of the known Liouville type results to the \emph{best exponents} in a \emph{unified and simple} way, the \emph{method of scaling spheres} in \cite{DQ0} (c.f. also \cite{CDQ0,DQ2,DQ3,DQZ}) can be applied conveniently to various problems (including PDEs, IEs and systems) \emph{with singularities} or \emph{without translation invariance} in $\mathbb{R}^{n}$, $\mathbb{R}^{n}_{+}$, $B_{R}(0)$ and $\mathbb{R}^{n}\setminus\overline{B_{R}(0)}$ for any dimensions $n$. It can also be applied to various fractional or higher order problems in the cases that the method of moving planes in conjunction with Kelvin transforms do not work, and hence can cover the gap for $p$ between $\frac{n+2s+a}{n-2s}$ and the critical exponent $p_{c}(a):=\frac{n+2s+2a}{n-2s}$ when $a>0$. When carrying out the \emph{method of scaling spheres}, we do not need to impose any boundedness assumptions on $u$.

\medskip

Nevertheless, if we want to use blowing-up method on general domains with \emph{not necessarily $C^{1}$ boundary} (such as domains with uniform exterior cone property or Lipschitz domains), then it is necessary for us to prove Liouville theorems on more general (convex or non-convex) simply connected domains including cone-like domains. In this paper, we will introduce the method of scaling spheres as \emph{a unified approach} to Liouville theorems on \emph{arbitrary unbounded or bounded MSS applicable domains} in $\mathbb{R}^{n}$ for general ($\leq n$-th order) PDEs and IEs \emph{without translation invariance or with singularities}. In order to apply the method of scaling spheres in MSS applicable domains $\Omega$, we only dilate or shrink \emph{a subset $S_{\lambda}(P)\bigcap\Omega$} of the sphere $S_{\lambda}(P)=\partial B_{\lambda}(P)=\lambda S^{n-1}+P$ w.r.t. the scale of the radius $\lambda$ and \emph{one fixed point $P$} (usually the singular point or the center of MSS applicable domain $\Omega$). MSS applicable domain is the \emph{maximal collection} of simply connected domains in $\mathbb{R}^{n}$ such that the MSS works, which guarantee that the reflection w.r.t. $S_{\lambda}(P)$ of $\Omega\setminus\overline{B_{\lambda}(P)}$ is contained in $\Omega$ if $\Omega$ is a g-radially convex domain and the reflection w.r.t. $S_{\lambda}(P)$ of $\Omega\bigcap B_{\lambda}(P)$ is contained in $\Omega$ if $\mathbb{R}^{n}\setminus\overline{\Omega}$ is a g-radially convex domain. We \emph{decrease} the scale of radius $\lambda$ if $\Omega$ is a g-radially convex domain and \emph{increase} the scale of radius $\lambda$ if $\mathbb{R}^{n}\setminus\overline{\Omega}$ is a g-radially convex domain. One should note that, \emph{MSS applicable domains is to the MSS what convex domains (at least in one direction) is to the method of moving planes}. When applying the MSS, in the critical order case $n=2s$, we compare the value of solution $u$ at $x$ with the value of $u$ at the reflection $x^{\lambda}$ w.r.t. the sphere piece $S_{\lambda}(P)\bigcap\Omega$ directly instead of the value of Kelvin transform of $u$, which is essentially different from the method of moving planes/spheres. On establishing Louville type results, the method of scaling spheres can overcome \emph{all the restrictions} in the method of moving planes/spheres and the integral estimates and feedback estimates arguments based on Pohozaev identities. We believe that the MSS can also be used to prove \emph{lower bound estimates on asymptotic behavior} or \emph{radial monotonicity} of solutions and be applied to various problems (PDEs/IEs/Systems) without translation invariance or singularities in \emph{$\mathbb{R}^{n}$, metric spaces and manifolds}.

\medskip

In practice, the method of scaling spheres mainly consists of three steps and can be carried out in several convenient ways according to the feature of different problems. In this paper, we try to give adequate demonstrations and guidance on how to apply the method of scaling spheres on MSS applicable domains in different ways. We will apply \emph{the direct method of scaling spheres} to the fractional and second order PDEs \emph{without any help} of the \emph{integral representation formulae} of solutions (see Theorems \ref{pde-unbd}, \ref{pde-unbd-a}, \ref{pde-unbd-2}, \ref{pde-bd} and \ref{pde-bd-2} in subsection 1.2). To deal with the IEs, we apply \emph{the method of scaling spheres in integral forms} to prove Liouville theorems for IEs, then derive Liouville results for PDEs by the equivalence between IEs and PDEs (see subsection 1.3). For problems on a g-radially convex domain $\Omega$, Liouville theorems can be deduced from \emph{Kelvin transforms} and the Liouville type results on domain that is the complement of the closure of a g-radially convex domain, or from \emph{the method of scaling spheres in a local way} (see subsection 1.3), i.e., shrinking \emph{the sphere piece} $S_{\lambda}(P)\bigcap\Omega$ w.r.t. the point $P$ (usually the singular point or the center of the MSS applicable domain $\Omega$).

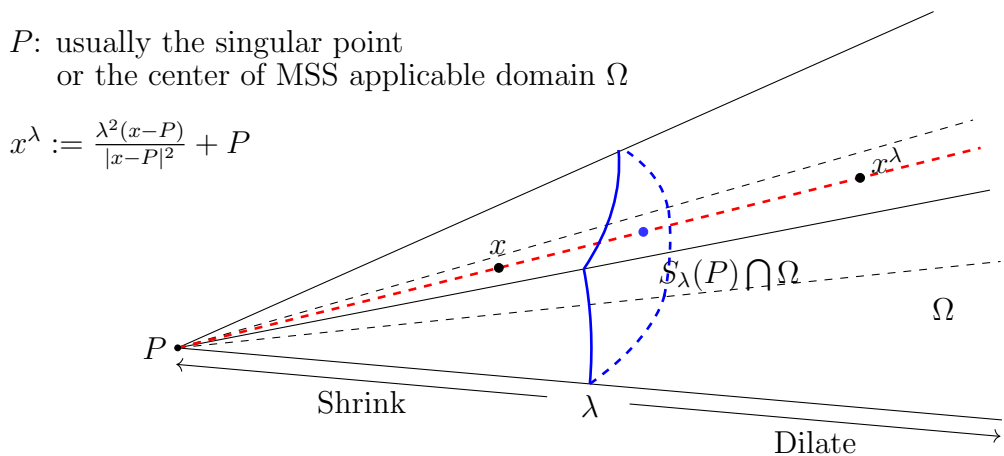
\begin{figure}[H]\label{F3}
	\centering	
	 	\begin{tikzpicture}[scale=1.1,rotate=10]
	 	
	 	\draw (0:0)--(-15:10);
	 	\draw (0:0)-- (1:10);
	 	\draw (0:0)-- (14:10);
	 	\draw[dashed] (6:0)--(6:10);
	 	\draw[dashed] (-4:0)--(-4:10);
	 	\draw[dashed, rline] (4:0)--(4:10);
	 	\filldraw[fill=blue!80,draw=blue!80 ] (4:5.8) circle(1.5pt);
	 	\filldraw[ ] (4:4) circle(1.5pt); \node[above]at(4.5:4){$x$};
	 	\filldraw[ ] (4:8.5) circle(1.5pt); \node[above]at(3:8.8){$x^\lambda$};
	 	\draw[bline]  (-15:5) arc (-15:1:5);
	 	\draw[bline] (1:5) arc (-30:-2.97:5 and 3);
	 	\draw[dashed,bline]  (-15:5) arc (-40:-10:   5 and 2);
	 	\draw[dashed,bline]  (-4:5.91) arc (-18:-1.95:   5 and 4);
	 	\draw[dashed,bline]  (6:6.16) arc (0:22:   5 and 2);
	 	\filldraw (0:0) circle(1pt);
	 	\node[left]at(0:0){$P$};
	 	\node[left] at(-3.4:7.7){$S_{\lambda}(P)\bigcap\Omega$};
	 	\node[right] at(-7:9){$\Omega$};
	 	\draw[<-, rotate=-15] (0,-0.2) -- (4.5,-0.2) node[ below] at (2.25,-0.2){Shrink};
	 	\draw[->, rotate=-15] (5.5,-0.2)--(10,-0.2) node[ below] at (7.75,-0.2){ Dilate};
	 	\node[below] at (-15:5){$\lambda$};
	 	\node[right] at (-1.7,2.8){$x^\lambda:=\frac{\lambda^2 (x-P)}{|x-P|^2}+P$};
	 	\node[right] at(-1.5,4) {$P$: usually the singular point};
	 	\node[right]at(-1,3.5) {  or the center of MSS applicable domain $\Omega$};
	 \end{tikzpicture}
	\centering
	\caption{The method of scaling spheres}
\end{figure}

\subsection{The direct method of scaling spheres on MSS applicable domains and applications}
Let $\Omega$ be a MSS applicable domain with center $P$, i.e., $\Omega$ or $\mathbb{R}^{n}\setminus\overline{\Omega}$ is a g-radially convex domain with radially convex center $P$. By establishing a \emph{small region principle} (Theorem \ref{NP-1}) with the help of the \emph{ABP estimates} (see e.g. \cite{FW,GT}), we apply the method of scaling spheres directly to the Dirichlet problem of fractional and second order PDEs \eqref{PDE} with $0<s\leq1$ (i.e., $u=0$ in $\mathbb{R}^{n}\setminus\Omega$).

\medskip

By applying the direct method of scaling spheres \emph{absolutely without the help} of \emph{integral representation formulae} of solutions, we derive the following Liouville theorem for nonnegative classical solutions in the case that $\mathbb{R}^{n}\setminus\overline{\Omega}$ is a g-radially convex domain.
\begin{thm}\label{pde-unbd}
    Assume $\mathbb{R}^{n}\setminus\overline{\Omega}$ is a g-radially convex domain, $n\geq\max\{2s,1\}$, $0<s\leq1$, $n\neq2$ if $s=1$ and $\Omega=\mathbb{R}^{n}$, $f$ is subcritical in the sense of Definition \ref{defn1} and satisfies $(\mathbf{f_{3}})$. Suppose $f$ satisfies $(\mathbf{f_{2}})$ with $q=+\infty$ when $s\in(0,1)$ and with $q=n$ when $s=1$. Then there is no classical solution to \eqref{PDE} that is positive in $\Omega$ ($\Omega\neq\mathbb{R}^{n}$ when $s=1$). Moreover, suppose $f(x,u)\geq0$ in $\Omega\times[0,+\infty)$, $f$ satisfies $(\mathbf{f_{1}})$ if $s=1$ and $\Omega=\mathbb{R}^{n}$, and $u$ is a nonnegative classical solution to the Dirichlet problem \eqref{PDE}, then $u\equiv0$ in $\Omega$.
\end{thm}

\medskip

When $s=1$, by using a \emph{narrow region principle for weak solutions} (see Theorem \ref{mp}) instead of the small region principle (Theorem \ref{NP-1}) derived through ABP estimates, we may also weaken the locally-$L^{q}$ Lipschitz condition in $(\mathbf{f_{2}})$ with $q=n$ to $q=\frac{n}{2}$. To this end, we need the following assumption.
\begin{itemize}
	\item [$(\mathbf{A})$]  There exists an $u_{0}\in[0,+\infty)$ such that, $f(x,u_{0})\in L^{1}(\Omega\bigcap B_{R}(P))$ ($\forall \, R>0$) if $\mathbb{R}^{n}\setminus\overline{\Omega}$ is a g-radially convex domain; $f(x,u_{0})\in L^{1}(\Omega\setminus B_{\eps}(P))$ ($\forall \, \eps>0$) if $\Omega$ is a bounded g-radially convex domain.	
\end{itemize}

\begin{thm}\label{pde-unbd-a}
    Assume $\mathbb{R}^{n}\setminus\overline{\Omega}$ is a g-radially convex domain, $s=1$, $f$ is subcritical in the sense of Definition \ref{defn1} and satisfies $(\mathbf{f_{3}})$. Suppose $P\in\mathbb{R}^{n}\setminus\overline{\Omega}$, $\partial\Omega$ is locally Lipschitz, $f$ satisfies $(\mathbf{A})$ and $(\mathbf{f_{2}})$ with $q=\frac{n}{2}$ if $n\geq3$ (with $q=1+\delta$ if $n=2$, where $\delta>0$ is arbitrarily small). Then there is no classical solution to \eqref{PDE} that is positive in $\Omega$. Moreover, suppose $f(x,u)\geq0$ in $\Omega\times[0,+\infty)$ and $u$ is a nonnegative classical solution to the Dirichlet problem \eqref{PDE}, then $u\equiv0$ in $\Omega$.
\end{thm}

\medskip

In the case $s\in(0,1)$, the assumption $(\mathbf{f_{2}})$ requires that $f(x,u)$ is locally $L^\infty$-Lipschitz on $u$ in $\Omega\times\overline{\mathbb{R}_{+}}$ and hence rules out the singularities of $f(x,u)$ as $x$ tends to $\partial \Omega$. However, when the boundary $\partial \Omega$ satisfies certain smoothness conditions (say, uniform exterior cone condition or more weaker condition \eqref{pde-boundary}), we can prove a new Liouville theorem for Dirichlet problem of PDEs \eqref{PDE} that allows $f(x,u)$ to blow up at $\partial\Omega$ by establishing a \emph{narrow region principle without using the ABP estimates} for fractional Laplacians (see Theorem \ref{NP-2}).

\medskip
	
To this end, we need the following definitions from \cite{LZLH}.
\begin{defn}\label{qgs}
Let $\{r_{k}\}_{k=0}^{+\infty}$ be a positive sequence. We call it a quasi-geometric sequence if there exist constants $0<\tau_{1}<\tau_{2}<1$ such that
\[\tau_{1}r_{k-1}\leq r_{k}\leq \tau_{2}r_{k-1}, \qquad \forall \,\, k\geq1.\]
\end{defn}
\begin{defn}\label{gc}
Let $\Omega\subset\mathbb{R}^{n}$ be a bounded domain and $x_{0}\in\partial\Omega$. We say that $\Omega$ satisfies the geometric condition $\mathbf{G}$ at $x_{0}$ if there exist a constant $\nu\in(0,1)$ and a quasi-geometric sequence $\{r_{k}\}_{k=0}^{+\infty}$ such that
\begin{equation}\label{gc0}
  r_{k}^{-n}\left| \left(B_{r_{k}}(x_{0})\setminus B_{r_{k+1}}(x_{0})\right)\bigcap \Omega^c \right|\geq\nu, \qquad \forall \,\, k\geq0.
\end{equation}
\end{defn}

\medskip

Let $\Omega\subset \mathbb R^n$ be a MSS applicable domain with $P\in \mathbb R^n\setminus \Omega$ as the center such that $\mathbb{R}^{n}\setminus\overline{\Omega}$ is a g-radially convex domain. Moreover, we assume that, for any $R>0$, the geometric condition $\mathbf{G}$ holds uniformly on $\partial\Omega\bigcap B_R(P)$, that is,
\begin{eqnarray}\label{pde-boundary}
	&& \forall \,\, R>0, \,\, \text{there exists} \,\, \nu_R>0 \,\, \text{s.t.}, \,\, \forall \,\, x\in\partial\Omega\bigcap B_R(P), \,\, \text{there are} \\ \nonumber && \text{quasi-geometric sequence} \,\, \{r_{k}\}_{k=0}^{+\infty} \,\,\text{and} \,\, \nu\geq\nu_{R} \,\, \text{s.t.}\,\, \eqref{gc0} \,\, \text{holds}.	
\end{eqnarray}
	
\medskip

Instead of assumption $(\mathbf{f_{2}})$, we need the following modified assumption $(\mathbf{f'_{2}})$ on $f(x,u)$.
\begin{itemize}
	\item [$(\mathbf{f'_{2}})$]  In the sense of Definition \ref{defn2}, there exists a function $g:\,[0,+\infty)$ with $\lim\limits_{t\rightarrow0^{+}}g(t)=+\infty$ such that $g(\text{dist}(x, \partial\Omega))\text{dist}(x, \partial\Omega)^{2s}f(x,u)$ is locally $L^\infty$-Lipschitz on $u$ in $\Omega\times\overline{\mathbb{R}_{+}}$. 	
\end{itemize}

\medskip

By establishing a narrow region principle without using the ABP estimate (Theorem \ref{NP-2}), we apply the method of scaling spheres directly to the Dirichlet problem of PDEs \eqref{PDE} with $0<s<1$ \emph{absolutely without any help} of the \emph{integral representation formulae} of solutions and derive the following Liouville theorem when \emph{$f(x,u)$ is not locally $L^\infty$-Lipschitz}.
\begin{thm}\label{pde-unbd-2}
Assume $\mathbb{R}^{n}\setminus\overline{\Omega}$ is a g-radially convex domain, $\Omega$ satisfies \eqref{pde-boundary} and $0<s<1$. Suppose $f$ is subcritical in the sense of Definition \ref{defn1} and satisfies the assumptions $(\mathbf{f'_{2}})$ and $(\mathbf{f_{3}})$. Then there is no classical solution to \eqref{PDE} that is positive in $\Omega$. Moreover, suppose $f(x,u)\geq0$ in $\Omega\times[0,+\infty)$ and $u$ is a nonnegative classical solution to the Dirichlet problem \eqref{PDE}, then $u\equiv0$ in $\Omega$.
\end{thm}	
	
\smallskip

Now we are to establish Liouville theorem for Dirichlet problem of PDEs \eqref{PDE} with $0<s\leq1$ via the \emph{direct method of scaling spheres} on g-radially convex domains.

\medskip

Suppose $\Omega\subset \mathbb R^n$ is a g-radially convex domain with $P\in\overline{\Omega}$ as the center. We consider the following fractional or second order PDEs of type \eqref{PDE} which may possibly be singular at the center $P$:
\begin{equation}\label{pdeb-1}
	\begin{cases}
		(-\Delta)^s u(x)=f(x,u(x)),\ \ \ & x\in\Omega,\\
		u(x)=0, \ \ \ & x\in\mathbb{R}^n\setminus (\Omega\cup \{P\}).
	\end{cases}
\end{equation}
where $n\geq \max\{2s,1\}$,  $s\in(0,1]$ and $u\in \mathcal{L}_{s}(\mathbb{R}^{n})\cap C^{[2s],\{2s\}+\eps}_{loc}(\Omega\setminus\{P\})\cap C(\overline{\Omega}\setminus\{P\})$ with arbitrarily small $\eps>0$ if $0<s<1$, $u\in C^{2}(\Omega\setminus\{P\})\cap C(\overline{\Omega}\setminus\{P\})$ if $s=1$. One should note that, the solution $u$ \emph{may have singularity} at $P$.

\medskip

Instead of assumption $(\mathbf{f_{2}})$, we need the following modified assumption $(\mathbf{\widehat{f}_{2}})$ on $f(x,u)$.
\begin{itemize}
	\item[$(\mathbf{\widehat{f}_{2}})$]  In the sense of Definition \ref{defn2}, $f(x,u)$ is locally $L^{q}$-Lipschitz on $u$ in $\left(\Omega\setminus B_\varepsilon(P)\right)\times\overline{\mathbb{R}_{+}}$ for any $\varepsilon>0$.
\end{itemize}

\medskip

One can easily check that if $f$ is supercritical in the sense of Definition \ref{defn1} and satisfies $(\mathbf{\widehat{f}_{2}})$ and $(\mathbf{f_{3}})$, then by Kelvin type transforms, we can transfer the supercritical Dirichlet problem \eqref{pdeb-1} in g-radially convex domain to a subcritical Dirichlet problem in MSS applicable domain (such that $\mathbb{R}^{n}\setminus\overline{\Omega}$ is a g-radially convex domain) with nonlinear term satisfying the assumptions $(\mathbf{f_{2}})$ and $(\mathbf{f_{3}})$. Thus we deduce immediately from Theorems \ref{pde-unbd} and \ref{pde-unbd-a} the following Liouville theorem on the supercritical Dirichlet problem \eqref{pdeb-1} for solutions that may possibly be singular at $P$ (including \emph{all classical solutions}).
\begin{thm}\label{pde-bd}
Assume $\Omega\subset \mathbb R^n$ is a g-radially convex domain with radially convex center $P$, $0<s\leq1$, $f$ is supercritical in the sense of Definition \ref{defn1} and satisfies $(\mathbf{f_{3}})$. Suppose $f$ satisfies $(\mathbf{\widehat{f}_{2}})$ with $q=+\infty$ when $s\in(0,1)$, and with $q=\frac{n}{2}$ if $n\geq3$ (with $q=1+\delta$ if $n=2$, where $\delta>0$ is arbitrarily small) provided that $\Omega$ is bounded, $\partial\Omega\setminus\{P\}$ is locally Lipschitz and $f$ satisfies $(\mathbf{A})$, or else with $q=n$ when $s=1$. If $\Omega$ is unbounded, assume that $u=o\left(\frac{1}{|x|^{n-2s}}\right)$ as $|x|\rightarrow+\infty$ with $x\in\Omega$. Then there is no positive solution to \eqref{pdeb-1} in $\Omega$. Moreover, suppose $f(x,u)\geq0$ in $\Omega\times[0,+\infty)$ and $u$ is a nonnegative solution to the Dirichlet problem \eqref{pdeb-1} that may be singular at $P$, then $u\equiv0$ in $\Omega$.
\end{thm}

For $0<s<1$, if the g-radially convex domain $\Omega$ satisfies the geometric property that the geometric condition $\mathbf{G}$ holds uniformly on $\partial\Omega\bigcap\left(B_R(P)\right)^{c}$ for any $R>0$, that is,
\begin{eqnarray}\label{pdeb-boundary}
	&& \forall \,\, R>0, \,\, \text{there exists} \,\, \nu_R>0 \,\, \text{s.t.}, \,\, \forall \,\, x\in\partial\Omega\bigcap\left(B_R(P)\right)^{c}, \,\, \text{there are} \\ \nonumber && \text{quasi-geometric sequence} \,\, \{r_{k}\}_{k=0}^{+\infty} \,\,\text{and} \,\, \nu\geq\nu_{R} \,\, \text{s.t.}\,\, \eqref{gc0} \,\, \text{holds},	
\end{eqnarray}
and suppose that $f$ is supercritical in the sense of Definition \ref{defn1}, and satisfies $(\mathbf{f_{3}})$ and
\begin{itemize}
	\item [$(\mathbf{\widehat{f}'_{2}})$]  In the sense of Definition \ref{defn2}, there exists a function $g:\,[0,+\infty)$ with $\lim\limits_{t\rightarrow0^{+}}g(t)=+\infty$ such that $g(\text{dist}(x, \partial\Omega))\text{dist}(x, \partial\Omega)^{2s}f(x,u)$ is locally $L^\infty$-Lipschitz on $u$  in $\Omega\setminus B_\varepsilon(P)\times\overline{\mathbb{R}_{+}}$ for any $\varepsilon>0$,
\end{itemize}	
then we can apply Kelvin type transforms and obtain from Theorem \ref{pde-unbd-2} the following Liouville theorem (when $f(x,u)$ itself is not locally $L^\infty$-Lipschitz) on the supercritical Dirichlet problem \eqref{pdeb-1} for solutions that \emph{may possibly be singular} at $P$.
 \begin{thm}\label{pde-bd-2}
Assume $\Omega\subset \mathbb R^n$ is a g-radially convex domain with radially convex center $P$ satisfying \eqref{pdeb-boundary} and $0<s<1$. If $\Omega$ is unbounded, assume that $u=o\left(\frac{1}{|x|^{n-2s}}\right)$ as $|x|\rightarrow+\infty$ with $x\in\Omega$. Suppose $f$ is supercritical in the sense of Definition \ref{defn1} and satisfies the assumptions $(\mathbf{\widehat{f}'_{2}})$ and $(\mathbf{f_{3}})$. Then there is no positive solution to \eqref{pdeb-1} in $\Omega$. Moreover, suppose $f(x,u)\geq0$ in $\Omega\times[0,+\infty)$ and $u$ is a nonnegative solution to the Dirichlet problem \eqref{pdeb-1} that may be singular at $P$, then $u\equiv0$ in $\Omega$.
\end{thm}

\begin{rem}\label{rem6}
One can observe that, by applying the method of scaling spheres directly to PDEs, we do not need the assumption $(\mathbf{f_{1}})$ except the case that $s=1$ and $\Omega=\mathbb{R}^{n}$ is regarded as the complement of g-radially convex set $\emptyset$. Thus it follows from Theorems \ref{pde-unbd}, \ref{pde-unbd-a}, \ref{pde-unbd-2}, \ref{pde-bd} and \ref{pde-bd-2} that PDEs \eqref{PDE} with $f(x,u)=|x-P|^{a}u^{-p}$, $0<s<1$ and $\Omega=\mathbb{R}^{n}$ possess no bounded positive solutions, where $p>0$ and $a>-2s$.
\end{rem}

\subsection{The method of scaling spheres in integral forms and in local way on MSS applicable domains and applications}

Consider the following integral equations (IEs):
\begin{equation}\label{IE}
  u(x)=\int_{\Omega}K^{s}_{\Omega}(x,y)f(y,u(y))\mathrm{d}y,
\end{equation}
where $\Omega$ is an arbitrary MSS applicable domain in $\mathbb{R}^{n}$ with center $P$ (i.e., $\Omega$ or $\mathbb{R}^{n}\setminus\overline{\Omega}$ is a g-radially convex domain), $u\in C(\overline{\Omega})$ if $\mathbb{R}^{n}\setminus\overline{\Omega}$ is a g-radially convex domain, $u\in C(\overline{\Omega}\setminus\{P\})$ if $\Omega$ is a g-radially convex domain, $f(x,u)\geq0$, $s>0$, $n\geq\max\{2s,1\}$ and the integral kernel $K^{s}_{\Omega}(x,y)=0$ if $x$ or $y\notin\Omega$.

\medskip

If $\mathbb{R}^{n}\setminus\overline{\Omega}$ is a g-radially convex domain with radially convex center $P$, we assume the integral kernel $K^{s}_{\Omega}(x,y)$ satisfies the following assumptions:
\begin{itemize}
  \item[$(\mathbf{H_{1}})$] There exist constants $C_{0}, C_1, C_2>0$ such that
  $$0<K^{s}_{\Omega}(x,y)\leq C_1\ln\left[\frac{C_0\left(1+|x-P|\right)\left(1+|y-P|\right)}{|x-y|}\right], \ \quad \forall \ x, y \in \Omega, \qquad \text{if} \,\, 2s=n,$$
  $$0<K^{s}_{\Omega}(x,y)\leq \frac{C_{2}}{|x-y|^{n-2s}}, \ \quad \forall \ x, y \in \Omega, \qquad \text{if} \,\, 2s<n.$$
  \item[$(\mathbf{H_{2}})$] There exist $d_{\Omega}<r<+\infty$, a smaller cross-section $\Sigma\subseteq\Sigma^{r}_{\Omega}$ satisfying $\mathcal{C}^{r,e}_{P,\Sigma}\subset\Omega$ and a constant $\theta\in\mathbb{R}$ such that for some $x_0\in \Omega$, it holds
  $$\liminf_{y\in \mathcal{C}_{P,\Sigma}, \, |y-P|\to +\infty} K^{s}_{\Omega}(x_0,y)|y-P|^{\theta}>0.$$
  Moreover, there exist constants $C_3>0, R_0>d_{\Omega}, 0<\sigma_1<\sigma_2$ such that $$K^{s}_{\Omega}(x, y)\geq \frac{C_{3}}{|x-y|^{n-2s}}>0$$
  for any $x, y  \in \mathcal{C}_{P,\Sigma}$ with $|x-P|>R_0$, $|y-P|>R_{0}$ and $\sigma_1|x-P|\leq |x-y|\leq \sigma_2|x-P|$.
  \item[$(\mathbf{H_{3}})$] For any $\lambda>d_\Omega$ and $x \in \Omega$ with $x^\lambda\in \Omega\cap B_\lambda(P)$, there hold $$K^{s}_{\Omega}(x, y)-\left(\frac{\lambda}{|x-P|}\right)^{n-2s}K^{s}_{\Omega}(x^\lambda, y)>0, \qquad \forall \,\, y\in\Omega\setminus \overline{B_{\lambda}(P)},$$
      and
      \begin{eqnarray*}
         && K^{s}_{\Omega}(x, y)-\left(\frac{\lambda}{|x-P|}\right)^{n-2s}K^{s}_{\Omega}(x^\lambda, y)\geq \left(\frac{\lambda^{2}}{|x-P|\cdot|y-P|}\right)^{n-2s}K^{s}_{\Omega}(x^\lambda, y^\lambda) \\
        && \qquad\qquad\qquad\qquad\qquad\qquad\qquad\qquad\qquad -\left(\frac{\lambda}{|y-P|}\right)^{n-2s}K^{s}_{\Omega}(x, y^\lambda)
      \end{eqnarray*}
      for any $y\in\Omega$ with $y^\lambda\in\Omega\bigcap B_{\lambda}(P)$, where $x^{\lambda}:=\frac{\lambda^{2}(x-P)}{|x-P|^{2}}+P$.
 \end{itemize}

\medskip

We will prove later in Theorem \ref{pG} and Remark \ref{rem8} that Green's functions $G^{s}_{\Omega}$ for Dirichlet problem of $(-\Delta)^{s}$ with $0<s\leq1$ on MSS applicable domain $\Omega$ such that $\mathbb{R}^{n}\setminus\overline{\Omega}$ is g-radially convex and for Navier problem of $(-\Delta)^{s}$ with $s\in\mathbb{Z}^{+}$ on $\mathcal{C}_{P,S^{n-1}_{2^{-k}}}$ and $\mathcal{C}^{R,e}_{P,S^{n-1}_{2^{-k}}}$ satisfy all the assumptions $(\mathbf{H_{1}})$, $(\mathbf{H_{2}})$ and $(\mathbf{H_{3}})$.

\medskip

By applying \emph{the method of scaling spheres in integral forms}, we can derive the following Liouville property for IEs \eqref{IE} with integral kernel $K^{s}_{\Omega}$ on MSS applicable domain $\Omega$ such that $\mathbb{R}^{n}\setminus\overline{\Omega}$ is a g-radially convex domain.
\begin{thm}\label{ie-unbd}
Assume $\Omega\subset \mathbb R^n$ is a MSS applicable domain with center $P$ such that $\mathbb{R}^{n}\setminus\overline{\Omega}$ is a g-radially convex domain, $s>0$, $n\geq\max\{2s,1\}$ and $u\in C(\overline{\Omega})$ is a nonnegative solution to IEs \eqref{IE}. Suppose the integral kernel $K^{s}_{\Omega}$ satisfies the assumptions $(\mathbf{H_{1}})$, $(\mathbf{H_{2}})$ and $(\mathbf{H_{3}})$, $f$ is subcritical in the sense of Definition \ref{defn1} and satisfies the assumptions $(\mathbf{f_{1}})$, $(\mathbf{f_{2}})$ with $q=\frac{n}{2s}$ if $n>2s$ and $q=1+\delta$ ($\delta>0$ arbitrarily small) if $n=2s$, and $(\mathbf{f_{3}})$ with the same cross-section $\Sigma$ as $(\mathbf{H_{2}})$. Then $u\equiv0$ in $\Omega$.
\end{thm}

\smallskip

Now assume that $\Omega$ is a g-radially convex domain with radially convex center $P$, $n\geq1$ and $0<2s<n$. Besides the assumption $(\mathbf{H_{1}})$ with $2s<n$, we assume the integral kernel $K^{s}_{\Omega}(x,y)$ satisfies the following assumptions:
\begin{itemize}
  \item[$(\mathbf{\widetilde{H}_{2}})$] There exist $0<r<\rho_{\Omega}$, a smaller cross-section $\Sigma\subseteq\Sigma^{r}_{\Omega}$ satisfying $\mathcal{C}^{r}_{P,\Sigma}\subset\Omega$ and a constant $\theta\in\mathbb{R}$ such that for some $x_0\in \Omega$, it holds
  $$\liminf_{y\in \mathcal{C}^{r}_{P,\Sigma}, \, |y-P|\to 0} K^{s}_{\Omega}(x_0,y)|y-P|^{-\theta}>0.$$
  Moreover, there exist constants $C_3>0, r_0<\rho_{\Omega}, 0<\sigma_1<\sigma_2$ such that $$K^{s}_{\Omega}(x, y)\geq \frac{C_{3}}{|x-y|^{n-2s}}>0$$
  for any $x, y\in \mathcal{C}_{P,\Sigma}$ with $|x-P|<r_0$, $|y-P|<r_{0}$ and $\sigma_1|y-P|\leq |x-y|\leq \sigma_2|y-P|$.
  \item[$(\mathbf{\widetilde{H}_{3}})$] For any $0<\lambda<\rho_\Omega$ and $x \in \Omega\cap B_\lambda(P)\setminus\{P\}$, there hold $$K^{s}_{\Omega}(x, y)-\left(\frac{\lambda}{|x-P|}\right)^{n-2s}K^{s}_{\Omega}(x^\lambda, y)>0, \qquad \forall \,\, y\in\Omega \,\,\,  \text{with} \,\, y^{\lambda}\in\Omega\setminus \overline{B_{\lambda}(P)},$$
      and
      \begin{eqnarray*}
         && K^{s}_{\Omega}(x, y)-\left(\frac{\lambda}{|x-P|}\right)^{n-2s}K^{s}_{\Omega}(x^\lambda, y)\geq \left(\frac{\lambda^{2}}{|x-P|\cdot|y-P|}\right)^{n-2s}K^{s}_{\Omega}(x^\lambda, y^\lambda) \\
        && \qquad\qquad\qquad\qquad\qquad\qquad\qquad\qquad\qquad -\left(\frac{\lambda}{|y-P|}\right)^{n-2s}K^{s}_{\Omega}(x, y^\lambda)
      \end{eqnarray*}
      for any $y\in\Omega\bigcap B_{\lambda}(P)\setminus\{P\}$, where $x^{\lambda}:=\frac{\lambda^{2}(x-P)}{|x-P|^{2}}+P$.
 \end{itemize}

\medskip

Instead of assumption $(\mathbf{f_{2}})$ with $q=\frac{n}{2s}$, we need the following modified assumption $(\mathbf{\widetilde{f}_{2}})$ on $f(x,u)$.
\begin{itemize}
	\item[$(\mathbf{\widetilde{f}_{2}})$]  In the sense of Definition \ref{defn2}, $f(x,u)$ is locally $L^{\frac{n}{2s}}$-Lipschitz on $u$ in $\Omega\setminus B_\varepsilon(P)\times\overline{\mathbb{R}_{+}}$ for any $\varepsilon>0$.
\end{itemize}

\medskip

By Theorem \ref{ie-unbd}, Kelvin transforms and \emph{the method of scaling spheres in local way}, we can derive the following Liouville property for nonnegative solutions (that may be singular at $P$) to IEs \eqref{IE} with $2s<n$ and integral kernel $K^{s}_{\Omega}$ on g-radially convex domain $\Omega$.
\begin{thm}\label{ie-bd}
Assume $\Omega\subset \mathbb R^n$ is a g-radially convex domain with radially convex center $P$, $0<2s<n$, $n\geq1$ and $u\in C(\overline{\Omega}\setminus\{P\})$ is a nonnegative solution to IEs \eqref{IE}. Suppose the integral kernel $K^{s}_{\Omega}$ satisfies the assumptions $(\mathbf{H_{1}})$, $(\mathbf{\widetilde{H}_{2}})$ and $(\mathbf{\widetilde{H}_{3}})$, $f$ satisfies the assumptions $(\mathbf{f_{1}})$, $(\mathbf{\widetilde{f}_{2}})$ and $(\mathbf{f_{3}})$ with the same cross-section $\Sigma$ as $(\mathbf{\widetilde{H}_{2}})$. Assume that either \\
(i) \, $f$ is supercritical in the sense of Definition \ref{defn1}, and $u=o\left(\frac{1}{|x|^{n-2s}}\right)$ as $|x|\rightarrow+\infty$ with $x\in\Omega$ if $\Omega$ is unbounded, \\
or\\
(ii) \, $\Omega$ is bounded, $f$ is critical or supercritical in the sense that $\mu^{\frac{n+2s}{n-2s}}f(\mu^{\frac{2}{n-2s}}(x-P)+P,\mu^{-1}u)$ is non-increasing w.r.t. $\mu\geq1$ or $\mu\leq1$ for all $(x,u)\in(\Omega\setminus\{P\})\times\mathbb{R}_{+}$, $f(x,u)>0$ provided that $u>0$ for a.e. $x\in\Omega\setminus\{P\}$ if $f$ is not supercritical in the sense of Definition \ref{defn1}; and $u(x)=o\left(\frac{1}{|x-P|^{\frac{n-2s}{2}}}\right)$ as $x\to P$ if $(\mathbf{f_{3}})$ is only fulfilled by $f$ in the case $p=p_{c}(a)$. \\
Then $u\equiv0$ in $\Omega$.
\end{thm}

\medskip

When $0<s\leq1$, for any $x\in\Omega$, the Green function $G^{s}_{\Omega}(x,y)$ for Dirichlet problem of $(-\Delta)^{s}$ on $\Omega$ solves
\begin{equation}
\begin{cases}(-\Delta)^{s} G^{s}_{\Omega}(x, y)=\delta(x-y), & y \in \Omega, \\ G^{s}_{\Omega}(x, y)=0, & y \in \mathbb{R}^{n}\setminus\Omega.\end{cases}
\end{equation}
When $s\geq1$ is an integer, for any $x\in\Omega$, the Green function $G^{s}_{\Omega}(x,y)$ for Navier problem of $(-\Delta)^{s}$ on $\Omega$ solves
\begin{equation}
\begin{cases}(-\Delta)^{s} G^{s}_{\Omega}(x, y)=\delta(x-y), & y \in \Omega, \\ G^{s}_{\Omega}(x, y)=(-\Delta) G^{s}_{\Omega}(x, y)=\cdots=(-\Delta)^{s-1} G^{s}_{\Omega}(x, y)=0, & y \in \partial \Omega.\end{cases}
\end{equation}
If $\Omega$ is unbounded, we assume naturally $\lim\limits_{|y|\to+\infty}G^{s}_{\Omega}(x,y)=0$ if $2s<n$ or $2s=n$ and $\mathbb{R}^{n}\setminus\overline{\Omega}$ is unbounded, and $\lim\limits_{|y|\to+\infty}G^{s}_{\Omega}(x,y)=c>0$ if $2s=n$ and $\mathbb{R}^{n}\setminus\overline{\Omega}$ is bounded.

\medskip

When $0<s\leq1$, we can show that the Green function $G^{s}_{\Omega}(x,y)$ for Dirichlet problem of $(-\Delta)^{s}$ satisfies all the assumptions $(\mathbf{H_{1}})$, $(\mathbf{H_{2}})$ and $(\mathbf{H_{3}})$, or $(\mathbf{\widetilde{H}_{2}})$ and $(\mathbf{\widetilde{H}_{3}})$ and deduce from Theorems \ref{ie-unbd} and \ref{ie-bd} the following theorem.
\begin{thm}\label{pG}
Assume $n\geq1$, $0<s\leq1$, $2s\leq n$ and there exists Green's function $G^{s}_{\Omega}(x,y)$ for Dirichlet problem of $(-\Delta)^{s}$ on MSS applicable domain $\Omega$, $\overline{\Omega}\neq\mathbb{R}^{n}$ if $n=2s$, then $G^{s}_{\Omega}(x,y)$ satisfies all the assumptions $(\mathbf{H_{1}})$, $(\mathbf{H_{2}})$ with arbitrary $r\in(d_{\Omega},+\infty)$, any cross-section $\Sigma$ such that $\overline{\Sigma}\subsetneq\Sigma^{r}_{\Omega}$ ($\Sigma=S^{n-1}$ if $\mathbb{R}^{n}\setminus\overline{\Omega}$ is bounded and nonempty), $\theta=n-2s+\gamma_{s}(\mathcal{C}_{P,\Sigma})\geq n-2s$ (see Theorem \ref{RMI-h} for definition of $\gamma_{s}(\mathcal{C})$) and $(\mathbf{H_{3}})$, or $(\mathbf{\widetilde{H}_{2}})$ with arbitrary $r\in(0,\rho_{\Omega})$, any cross-section $\Sigma$ such that $\overline{\Sigma}\subsetneq\Sigma^{r}_{\Omega}$, $\theta=\gamma_{s}(\mathcal{C}_{P,\Sigma})\geq 0$ and $(\mathbf{\widetilde{H}_{3}})$. Consequently, the Liouville results in Theorems \ref{ie-unbd} and \ref{ie-bd} hold true for IEs \eqref{IE} with $K^{s}_{\Omega}=G^{s}_{\Omega}$, i.e.,
\begin{equation}\label{IE1}
  u(x)=\int_{\Omega}G^{s}_{\Omega}(x,y)f(y,u(y))\mathrm{d}y.
\end{equation}
\end{thm}

\begin{rem}\label{rem7}
Smooth conditions on $\partial\Omega$ (e.g., Dirichlet-regular when $s\geq1$ is an integer, Lipschitz continuous or exterior cone condition when $s\in(0,1)$) can guarantee the existence of the Green function $G^{s}_{\Omega}(x,y)$ for Dirichlet problem of $(-\Delta)^{s}$ with $s\in(0,1]$ or Navier problem of $(-\Delta)^{s}$ with $s\in\mathbb{Z}^{+}$ exists on $\Omega$. For more sufficient conditions on the existence of Green's functions, refer to e.g. \cite{BH} and the references therein.
\end{rem}

\begin{rem}\label{rem8}
Let $\Omega$ be the standard $\frac{1}{2^{k}}$-space $\mathcal{C}_{0,\Sigma_{k}}$ or the standard $\frac{1}{2^{k}}$-ball $\mathcal{C}^{R}_{0,\Sigma_{k}}$ with $\Sigma_{k}:=\{x\in S^{n-1}\mid \, x_{1},\cdots,x_{k}>0\}\subset S^{n-1}$ ($k=1,\cdots,n$), and the fundamental solution for $(-\Delta)^{s}$ be denoted by $\Gamma_{s}(x,y):=\frac{c_{n,s}}{|x-y|^{n-2s}}$ if $0<2s<n$ and $\Gamma_{s}(x,y):=c_{s}\ln\left(\frac{1}{|x-y|}\right)$ if $2s=n$. The Green function $G_{\Omega}^{s}$ on $\Omega$ for Dirichlet problem of $(-\Delta)^{s}$ with $s=1$ or Navier problem of $(-\Delta)^{s}$ with $s\in\mathbb{Z}^{+}$ is given by
\begin{equation*}
  G^{s}_{\Omega}(x,y)=\sum_{j=0}^{k}(-1)^{j}\sum_{\mathcal{R}_{j}}\Gamma_{s}(x,\mathcal{R}_{j}y),   \qquad \text{if} \,\, \Omega=\mathcal{C}_{0,\Sigma_{k}};
\end{equation*}
\begin{equation*}
  G^{s}_{\Omega}(x,y)=\sum_{j=0}^{k}(-1)^{j}\sum_{\mathcal{R}_{j}}\Gamma_{s}(x,\mathcal{R}_{j}y)
  -\sum_{j=0}^{k}(-1)^{j}\sum_{\mathcal{R}_{j}}\Gamma_{s}\left(\frac{|y|}{R}x,\frac{|y|}{R}\mathcal{R}_{j}y^{R}\right),   \qquad \text{if} \,\, \Omega=\mathcal{C}^{R}_{0,\Sigma_{k}};
\end{equation*}
where $y^{R}:=\frac{R^{2}y}{|y|^{2}}$, and $\mathcal{R}_{j}$ ($j=0,\cdots,k$) denotes the transform such that there exists $1\leq i_{1}<i_{2}<\cdots<i_{j}\leq k$ satisfying $(\mathcal{R}_{j}y)_{i}=y_{i}$ ($i\in \{1,\cdots,n\}\setminus \{i_{1},\cdots,i_{j}\}$), $(\mathcal{R}_{j}y)_{i}=-y_{i}$ ($i\in\{i_{1},\cdots,i_{j}\}$).

\medskip

Let $\Omega$ be the general $\frac{1}{2^{k}}$-space $P+\mathbb{R}^{n}_{2^{-k}}:=\mathcal{C}_{P,S^{n-1}_{2^{-k}}}=P+O(\mathcal{C}_{0,\Sigma_{k}})$ or the general $\frac{1}{2^{k}}$-ball $B^{2^{-k}}_{R}(P):=\mathcal{C}^{R}_{P,S^{n-1}_{2^{-k}}}=P+O(\mathcal{C}^{R}_{0,\Sigma_{k}})$, where $P\in\mathbb{R}^{n}$, $O$ denotes an arbitrary orthogonal transform on $\mathbb{R}^{n}$ and $S^{n-1}_{2^{-k}}=O(\Sigma_{k})$. Then, the Green function $G_{\Omega}^{s}$ on $\Omega$ for Dirichlet problem of $(-\Delta)^{s}$ with $s=1$ or Navier problem of $(-\Delta)^{s}$ with $s\in\mathbb{Z}^{+}$ is given by
\begin{equation*}
  G^{s}_{\Omega}(x,y)=\sum_{j=0}^{k}(-1)^{j}\sum_{\mathcal{R}_{j}}\Gamma_{s}(O^{-1}(x-P),\mathcal{R}_{j}O^{-1}(y-P)), \qquad \text{if} \,\, \Omega=\mathcal{C}_{P,S^{n-1}_{2^{-k}}};
\end{equation*}
\begin{eqnarray*}
  && G^{s}_{\Omega}(x,y)=\sum_{j=0}^{k}(-1)^{j}\sum_{\mathcal{R}_{j}}\Gamma_{s}(O^{-1}(x-P),\mathcal{R}_{j}O^{-1}(y-P)) \\
  && \quad -\sum_{j=0}^{k}(-1)^{j}\sum_{\mathcal{R}_{j}}\Gamma_{s}\left(\frac{|y-P|}{R}O^{-1}(x-P),\frac{|y-P|}{R}\mathcal{R}_{j}\left[O^{-1}(y-P)\right]^{R}\right),   \quad\, \text{if} \,\, \Omega=\mathcal{C}^{R}_{P,S^{n-1}_{2^{-k}}}.
\end{eqnarray*}
The Green's function on the $\frac{1}{2^{k}}$-exterior domain of ball $\Omega:=\mathcal{C}^{R,e}_{P,S^{n-1}_{2^{-k}}}=\mathcal{C}_{P,S^{n-1}_{2^{-k}}}\setminus\overline{\mathcal{C}^{R}_{P,S^{n-1}_{2^{-k}}}}$ can be derived from the Green function on $\frac{1}{2^{k}}$-ball $\mathcal{C}^{R}_{P,S^{n-1}_{2^{-k}}}$ via the Kelvin transform $y^{R,P}:=\frac{R^{2}(y-P)}{|y-P|^{2}}+P$, i.e., $G^{s}_{\Omega}(x,y)=\left(\frac{R}{|x-y|}\right)^{n-2s}G^{s}_{\mathcal{C}^{R}_{P,S^{n-1}_{2^{-k}}}}(x^{R,P},y^{R,P})$.

\medskip

Through direct calculations, one can verify that the Green function $G^{s}_{\Omega}(x,y)$ with $\Omega=\mathcal{C}_{P,S^{n-1}_{2^{-k}}}$, $\mathcal{C}^{R}_{P,S^{n-1}_{2^{-k}}}$ or $\mathcal{C}^{R,e}_{P,S^{n-1}_{2^{-k}}}$ satisfies all the assumptions $(\mathbf{H_{1}})$, $(\mathbf{H_{2}})$  with arbitrary $r\in(d_{\Omega},+\infty)$, any cross-section $\Sigma$ such that $\overline{\Sigma}\subsetneq\Sigma^{r}_{\Omega}$, $\theta=n-2s+k$ and $(\mathbf{H_{3}})$ if $\Omega=\mathcal{C}_{P,S^{n-1}_{2^{-k}}}$ or $\mathcal{C}^{R,e}_{P,S^{n-1}_{2^{-k}}}$, $(\mathbf{\widetilde{H}_{2}})$ with arbitrary $r\in(0,\rho_{\Omega})$, any cross-section $\Sigma$ such that $\overline{\Sigma}\subsetneq\Sigma^{r}_{\Omega}$, $\theta=k$ and $(\mathbf{\widetilde{H}_{3}})$ if $\Omega=\mathcal{C}^{R}_{P,S^{n-1}_{2^{-k}}}$, $\mathcal{C}_{P,S^{n-1}_{2^{-k}}}$ or $\mathcal{C}^{R,e}_{P,S^{n-1}_{2^{-n}}}$.
\end{rem}

\medskip

For $0<s\leq1$, if $\Omega$ is a MSS applicable domain such that $\mathbb{R}^{n}\setminus\overline{\Omega}\neq\emptyset$ is bounded and $s=1$, or $\Omega$ is bounded and $G^{s}_{\Omega}$ exists, or $\Omega$ is a unbounded cone $\mathcal{C}_{P,\Sigma}$ such that $\overline{\mathcal{C}_{P,\Sigma}}\neq\mathbb{R}^{n}$ and the cross-section $\Sigma$ is $C^{1,1}$ if $s\in(0,1)$ and Dirichlet-regular if $s=1$, then from classification results on nonnegative harmonic functions in cones (c.f. e.g. \cite{A,BaBo}, see also \cite{BE,TTV} for related results), we can deduce the following equivalence between Dirichlet problem of PDEs \eqref{PDE} and IEs \eqref{IE1}.
\begin{thm}\label{equi}
Let $n\geq1$, $0<s\leq1$, $2s\leq n$, $\Omega$ be a MSS applicable domain with center $P$ and $u$ be a nonnegative classical solution to Dirichlet problem of PDEs \eqref{PDE}. Assume the following three cases: $(i)$ $\Omega$ is bounded and $G^{s}_{\Omega}$ exists, or $(ii)$ $s=1$, $u\in C^{0,1}(\overline{\Omega})$ and $\Omega$ satisfies that $\partial\Omega$ is Lipschitz and $\mathbb{R}^{n}\setminus\overline{\Omega}\neq\emptyset$ is bounded, or $(iii)$ $\Omega$ is a unbounded cone $\mathcal{C}_{P,\Sigma}$ such that $\overline{\mathcal{C}_{P,\Sigma}}\neq\mathbb{R}^{n}$, the cross-section $\Sigma$ is $C^{1,1}$ if $s\in(0,1)$ and is Dirichlet-regular if $s=1$. Suppose $f\geq0$ satisfies $(i)$ in $(\mathbf{f_{3}})$ if $\Omega=\mathcal{C}_{P,\Sigma}$ or $(i)$ in $(\mathbf{f_{3}})$ with $\Sigma=S^{n-1}$ if $\mathbb{R}^{n}\setminus\overline{\Omega}\neq\emptyset$ is bounded when $s=1$, then $u$ solves the IEs \eqref{IE1}, and vice versa.
\end{thm}

\smallskip

As an immediate consequence of Theorems \ref{pG} and \ref{equi}, we derive the following Liouville theorem on classical solution to Dirichlet problem of PDEs \eqref{PDE}.
\begin{thm}\label{f2PDE}
Let $n\geq1$, $0<s\leq1$, $2s\leq n$, $\Omega$ be a MSS applicable domain with center $P$ and $u$ be a nonnegative classical solution to Dirichlet problem of PDEs \eqref{PDE}. Assume the following three cases: $(i)$ $\Omega$ is bounded and $G^{s}_{\Omega}$ exists, or $(ii)$ $s=1$, $u\in C^{0,1}(\overline{\Omega})$ and $\Omega$ satisfies that $\partial\Omega$ is Lipschitz and $\mathbb{R}^{n}\setminus\overline{\Omega}\neq\emptyset$ is bounded, or $(iii)$ $\Omega$ is a unbounded cone $\mathcal{C}_{P,\Sigma}$ such that $\overline{\mathcal{C}_{P,\Sigma}}\neq\mathbb{R}^{n}$, the cross-section $\Sigma$ is $C^{1,1}$ if $s\in(0,1)$ and is Dirichlet-regular if $s=1$. Suppose $f\geq0$ satisfies $(\mathbf{f_{1}})$, $(i)$ in $(\mathbf{f_{3}})$ if $\Omega=\mathcal{C}_{P,\Sigma}$ or $(i)$ in $(\mathbf{f_{3}})$ with $\Sigma=S^{n-1}$ if $\mathbb{R}^{n}\setminus\overline{\Omega}\neq\emptyset$ is bounded when $s=1$, and $(ii)$ in $(\mathbf{f_{3}})$ if $\Omega$ is bounded. \\
In case $(i)$ $\Omega$ is bounded, suppose $2s<n$, $f$ satisfies $(\mathbf{\widetilde{f}_{2}})$ and is critical or supercritical in the sense that $\mu^{\frac{n+2s}{n-2s}}f(\mu^{\frac{2}{n-2s}}x,\mu^{-1}u)$ is non-increasing w.r.t. $\mu\geq1$ or $\mu\leq1$ for all $(x,u)\in(\Omega\setminus\{P\})\times\mathbb{R}_{+}$. \\
In case $(ii)$ $s=1$, $u\in C^{0,1}(\overline{\Omega})$, $\partial\Omega$ is Lipschitz and $\mathbb{R}^{n}\setminus\overline{\Omega}\neq\emptyset$ is bounded or case $(iii)$ $\Omega=\mathcal{C}_{P,\Sigma}$ with $\overline{\mathcal{C}_{P,\Sigma}}\neq\mathbb{R}^{n}$, when $n>2s$, suppose either $f$ is subcritical in the sense of Definition \ref{defn1} and satisfies the assumption $(\mathbf{f_{2}})$ with $q=\frac{n}{2s}$, or $f$ is supercritical in the sense of Definition \ref{defn1}, satisfies $(\mathbf{\widetilde{f}_{2}})$ and $(ii)$ in $(\mathbf{f_{3}})$, $\mathbb{R}^{n}\setminus\overline{\Omega}$ is unbounded (when $s=1$) and $u=o\left(\frac{1}{|x|^{n-2s}}\right)$ as $|x|\rightarrow+\infty$ with $x\in\Omega$; when $n=2s$, suppose $f$ is subcritical in the sense of Definition \ref{defn1} and satisfies the assumption $(\mathbf{f_{2}})$ with $q=1+\delta$ ($\delta>0$ arbitrarily small). \\
Then $u\equiv0$ in $\Omega$.
\end{thm}

\smallskip

Now we consider Navier problem of PDEs \eqref{PDE} with $s\in\mathbb{Z}^{+}$ on $\frac{1}{2^{k}}$-space $\mathcal{C}_{P,S^{n-1}_{2^{-k}}}$ or $\frac{1}{2^{k}}$-ball $\mathcal{C}^{R}_{P,S^{n-1}_{2^{-k}}}$. From Remark \ref{rem8} and Theorems \ref{ie-unbd} and \ref{ie-bd}, we obtain immediately the following theorem.
\begin{thm}\label{nie}
Assume $s\geq1$ is an integer such that $2s\leq n$, $\Omega$ is $\frac{1}{2^{k}}$-space $\mathcal{C}_{P,S^{n-1}_{2^{-k}}}$, $\frac{1}{2^{k}}$-ball $\mathcal{C}^{R}_{P,S^{n-1}_{2^{-k}}}$ or $\frac{1}{2^{k}}$-exterior domain of ball $\mathcal{C}^{R,e}_{P,S^{n-1}_{2^{-k}}}$ with $R>0$ and $k=1,\cdots,n$, $G^{s}_{\Omega}$ is the Green function for Navier problem of $(-\Delta)^{s}$ on $\Omega$. Then the Liouville results in Theorems \ref{ie-unbd} and \ref{ie-bd} hold true for IEs \eqref{IE} with $K^{s}_{\Omega}=G^{s}_{\Omega}$, i.e.,
\begin{equation}\label{IEN}
  u(x)=\int_{\Omega}G^{s}_{\Omega}(x,y)f(y,u(y))\mathrm{d}y.
\end{equation}
\end{thm}

\smallskip

By assuming the super poly-harmonic property, we can deduce the following equivalence between the Navier problem of PDEs \eqref{PDE} with $s\in\mathbb{Z}^{+}$ and IEs \eqref{IEN}.
\begin{thm}\label{equin}
Assume $s\geq1$ is an integer such that $2s\leq n$, $\Omega$ is $\frac{1}{2^{k}}$-space $\mathcal{C}_{P,S^{n-1}_{2^{-k}}}$ or $\frac{1}{2^{k}}$-ball $\mathcal{C}^{R}_{P,S^{n-1}_{2^{-k}}}$ with $k=1,\cdots,n$. Suppose $u$ is a nonnegative classical solution to Navier problem of PDEs \eqref{PDE} with $f\geq0$ satisfying $(\mathbf{f_{3}})$ if $\Omega=\mathcal{C}_{P,S^{n-1}_{2^{-k}}}$, and $(-\Delta)^{k}u\geq0$ in $\Omega$ ($k=1,\cdots,s-1$), then $u$ solves the IEs \eqref{IEN}, and vice versa.
\end{thm}

\smallskip

As an direct consequence of Theorems \ref{nie} and \ref{equin}, we derive the following Liouville theorem on classical solution to Navier problem of PDEs \eqref{PDE}.
\begin{thm}\label{NPDE}
Assume $s\geq1$ is an integer such that $2s\leq n$, $\Omega$ is $\frac{1}{2^{k}}$-space $\mathcal{C}_{P,S^{n-1}_{2^{-k}}}$ or $\frac{1}{2^{k}}$-ball $\mathcal{C}^{R}_{P,S^{n-1}_{2^{-k}}}$ with $k=1,\cdots,n$, $u$ is a nonnegative classical solution to Navier problem of PDEs \eqref{PDE} with $f\geq0$ satisfying $(\mathbf{f_{1}})$,  $(i)$ in $(\mathbf{f_{3}})$ if $\Omega=\mathcal{C}_{P,S^{n-1}_{2^{-k}}}$ and $(ii)$ in $(\mathbf{f_{3}})$ if $\Omega=\mathcal{C}^{R}_{P,S^{n-1}_{2^{-k}}}$, and $(-\Delta)^{k}u\geq0$ in $\Omega$ ($k=1,\cdots,s-1$). \\
If $\Omega=\mathcal{C}^{R}_{P,S^{n-1}_{2^{-k}}}$, suppose $2s<n$, $f$ satisfies $(\mathbf{\widetilde{f}_{2}})$ and is critical or supercritical in the sense that $\mu^{\frac{n+2s}{n-2s}}f(\mu^{\frac{2}{n-2s}}x,\mu^{-1}u)$ is non-increasing w.r.t. $\mu\geq1$ or $\mu\leq1$ for all $(x,u)\in(\Omega\setminus\{P\})\times\mathbb{R}_{+}$. \\
If $\Omega=\mathcal{C}_{P,S^{n-1}_{2^{-k}}}$, when $n>2s$, suppose either $f$ is subcritical in the sense of Definition \ref{defn1} and satisfies the assumption $(\mathbf{f_{2}})$ with $q=\frac{n}{2s}$, or $f$ is supercritical in the sense of Definition \ref{defn1} and satisfies $(\mathbf{\widetilde{f}_{2}})$ and $(ii)$ in $(\mathbf{f_{3}})$, and $u=o\left(\frac{1}{|x|^{n-2s}}\right)$ as $|x|\rightarrow+\infty$ with $x\in\Omega$; when $n=2s$, suppose $f$ is subcritical in the sense of Definition \ref{defn1} and satisfies the assumption $(\mathbf{f_{2}})$ with $q=1+\delta$ ($\delta>0$ arbitrarily small). \\
Then $u\equiv0$ in $\Omega$.
\end{thm}

\smallskip

All conclusions in Theorems \ref{pde-unbd}, \ref{pde-unbd-a}, \ref{pde-unbd-2}--\ref{pde-bd-2}, \ref{ie-unbd}--\ref{pG}, \ref{equi}--\ref{NPDE} for PDEs or IEs hold true for Hary-H\'{e}non type nonlinear term $f(x,u)=|x-P|^{a}u^{p}$.
\begin{cor}\label{HHcor}
If $2s<n$, assume $a>-2s$, $1\leq p<p_{c}(a)$ if $f$ has subcritical growth on $u$, $p=p_{c}(a)$ if $f$ has critical growth on $u$ and $p_{c}(a)<p<+\infty$ if $f$ has supercritical growth on $u$. If $2s=n$, assume $a>-n$ if $f$ satisfies subcritical condition on singularity, $a=-n$ if $f$ is satisfies critical condition on singularity. All conclusions in Theorems \ref{pde-unbd}, \ref{pde-unbd-a}, \ref{pde-unbd-2}--\ref{pde-bd-2}, \ref{ie-unbd}--\ref{pG}, \ref{equi}--\ref{NPDE} for PDEs \eqref{PDE}, \eqref{pdeb-1} and IEs \eqref{IE}, \eqref{IE1}, \eqref{IEN} hold true for Hary-H\'{e}non type nonlinear term $f(x,u)=|x-P|^{a}u^{p}$.
\end{cor}

\begin{rem}\label{rem10}
When carrying out the method of scaling spheres in integral forms to IEs, we can also replace the Hardy-H\'{e}non type lower bound on $f$ in (i) of assumption $(\mathbf{f_{3}})$ by $f(x,u)\geq C dist(x,\Gamma)^{a}u^{p}$ for some hyper-surface $\Gamma\subset\mathbb{R}^{n}$ with $dim\Gamma=k$ ($1\leq k\leq n-1$), where $1\leq p<p_{c}(a)=\frac{n+2s+2a}{n-2s}$ ($=+\infty$ if $n=2s$) if $-\min\{n-k,2s\}<a\leq0$ and $1\leq p<\frac{n+2s}{n-2s}$ ($1\leq p<+\infty$ if $n=2s$) if $0\leq a<+\infty$. It is clear from our proof in Section 3 that, the Liouville results for IEs on MSS applicable domain $\Omega$ such that $\mathbb{R}^{n}\setminus\overline{\Omega}$ is a g-radially convex domain in Theorems \ref{ie-unbd}, \ref{pG} and \ref{nie} still hold true and can be proved in similar way.
\end{rem}

\smallskip

In most cases that the Green function can be accurately expressed, the conditions $\mathbf{(H_{1})}$, $\mathbf{(H_{2})}$ and $\mathbf{(H_{3})}$, or $\mathbf{(\widetilde{H}_{2})}$ and $\mathbf{(\widetilde{H}_{3})}$ can be verified. The following are some typical examples.
\begin{cor}\label{cor-g}
Let the integral kernel $K^{s}_{\Omega}$ be the fundamental solution $\Gamma_{s}(x,y)=G^{s}_{\mathbb{R}^{n}}(x,y)$ for $(-\Delta)^{s}$ with $s\in\left(0,\frac{n}{2}\right)$ or the Green function $G^{s}_{\Omega}(x,y)$ for Dirichlet problems of $(-\Delta)^{s}$ with $s\in\mathbb{Z}^{+}$, $2s\leq n$ and $\Omega=\mathbb{R}^{n}_{+}$, $B_{R}(P)$ or $\mathbb{R}^{n}\setminus\overline{B_{R}(P)}$ ($\forall\,R>0$ and $P\in\mathbb{R}^{n}$). Then $K^{s}_{\Omega}$ satisfies the conditions $\mathbf{(H_{1})}$, $\mathbf{(H_{2})}$ and $\mathbf{(H_{3})}$, or $\mathbf{(\widetilde{H}_{2})}$ and $\mathbf{(\widetilde{H}_{3})}$. Consequently, the Liouville results in Theorems \ref{ie-unbd} and \ref{ie-bd} hold true for IEs \eqref{IE}. Moreover, if the equivalence between PDEs and IEs \eqref{IE} holds, then Liouville theorems also hold true for related PDEs.
\end{cor}

\smallskip

The known Liouville theorems on nonnegative classical solutions to PDEs \eqref{PDE} (and related IEs) with $n\geq2s$, in the whole space $\mathbb{R}^n$ and in the half space $\mathbb{R}^n_+$ when $f$ satisfies subcritical condition and in bounded star-shaped domain with regular boundary when $f$ satisfies supercritical condition, in e.g. \cite{BP,CDQ0,CC,CFL,CLL,CLZ,CLZC,DQ1,DQ2,DQ3,DQ0,DQZ,DFSV,Farina,FW,GS,GS1,GL,GW,Lin,LZ1,MP,NY,PS,RW,TD,WX,ZCCY} can be deduced as corollaries from our Liouville results on MSS applicable domains in Theorems \ref{pde-unbd}, \ref{pde-unbd-a}, \ref{pde-unbd-2}--\ref{pde-bd-2}, \ref{ie-unbd}--\ref{pG}, \ref{equi}--\ref{NPDE} and Corollaries \ref{HHcor} and \ref{cor-g} in subsections 1.2-1.3. Furthermore, since cone-like domain is a typically special example of MSS applicable domains, we have the following corollary on Liouville theorems in cone-like domains.
\begin{cor}[Liouville theorems in cone-like domains]\label{cone-cor}
Assume $2s\leq n$ and $\Omega$ is a cone-like domain $\mathcal{C}$, i.e., $\Omega=\mathcal{C}_{P,\Sigma}$, $\mathcal{C}^{R}_{P,\Sigma}$ or $\mathcal{C}^{R,e}_{P,\Sigma}$ with $P\in\mathbb{R}^{n}$, $\Sigma\subseteq S^{n-1}$ and $R>0$. Under the same assumptions, all conclusions in Theorems \ref{pde-unbd}, \ref{pde-unbd-a}, \ref{pde-unbd-2}--\ref{pde-bd-2}, \ref{ie-unbd}--\ref{pG}, \ref{equi}--\ref{NPDE} and Corollaries \ref{HHcor} and \ref{cor-g} for PDEs \eqref{PDE}, \eqref{pdeb-1} and IEs \eqref{IE}, \eqref{IE1}, \eqref{IEN} in cone-like domain $\Omega$ hold true.
\end{cor}

In \cite{CFR}, Ciraolo, Figalli and Roncoroni established the classification results for critical anisotropic $p$-Laplacian equations in $\mathbb{R}^{k}\times\mathcal{C}$ with convex unbounded cone $\mathcal{C}\subset\mathbb{R}^{n-k}$ and $k=0,\cdots,n$. Note that $\mathbb{R}^{k}\times\mathcal{C}$ with unbounded cone $\mathcal{C}\subset\mathbb{R}^{n-k}$ is a special unbounded cone in $\mathbb{R}^{n}$, more general, $\mathbb{R}^{k}\times\mathcal{C}$ with cone-like domain $\mathcal{C}\subset\mathbb{R}^{n-k}$ is a special MSS applicable domain in $\mathbb{R}^{n}$, our Liouville theorems in subsections 1.2 and 1.3 for PDEs \eqref{PDE}, \eqref{pdeb-1} and IEs \eqref{IE}, \eqref{IE1}, \eqref{IEN} hold true on $\mathbb{R}^{k}\times\mathcal{C}$ with cone-like domain $\mathcal{C}\subset\mathbb{R}^{n-k}$.

\subsection{Blowing-up analysis on BCB domains: a priori estimates of solutions}

In previous works, the blowing-up technique reduced the a priori estimates of solutions on bounded domains or on Riemannian manifolds with $C^{1}$ boundaries to the boundary H\"{o}lder estimates and Liouville type theorems on PDEs or integral equations (IEs) in the whole space $\mathbb{R}^n$ and in the half space $\mathbb{R}^n_+$. For more literatures on the blowing-up arguments and a priori estimates, please refer to the celebrated papers by Gidas and Spruck \cite{GS1} and Brezis and Merle \cite{BM}, see also \cite{CDQ,CL4,DD,DQ3,DQ0,MR,MP,PQS,RW} and the references therein. Besides the boundary H\"{o}lder estimates, the \emph{$C^{1}$-smoothness of the boundary} has to be \emph{imposed} to guarantee that, the limiting shape of the domain is \emph{a half space} after the blowing-up procedure and hence Liouville theorems in the half space could be applied.

\medskip

Since we have introduced the method of scaling spheres as \emph{a unified approach} to establishing Liouville type theorems on general MSS applicable domains (including cone-like domains) in $\mathbb{R}^{n}$, we are able to carry out blowing-up analysis on more general domains \emph{without $C^{1}$-smoothness} of the boundary and derive a priori estimates and existence of solutions. The limiting shape of the domain can \emph{at least} be allowed to be \emph{a cone-like domain} (half space is a cone) after the blowing-up procedure, so such domains include not only \emph{all $C^{1}$-smooth domains} but also \emph{quite extensive domains on which the boundary H\"{o}lder estimates hold} (which can be guaranteed by \emph{Lipschitz domains}, \emph{domains with uniform exterior cone property} and so on...). In fact, with the help of the Liouville results on MSS applicable domains (including cone-like domains) in subsections 1.2-1.3 (see Theorems \ref{pde-unbd}, \ref{pde-unbd-a}, \ref{pde-unbd-2}--\ref{pde-bd-2}, \ref{ie-unbd}--\ref{pG}, \ref{equi}--\ref{NPDE} and Corollaries \ref{HHcor}, \ref{cor-g} and \ref{cone-cor}), we can (at least) apply the blowing-up argument on domains with \emph{blowing-up cone boundary} (\emph{BCB domains} for short) and hence derive a priori estimates for Dirichlet or Navier problems of ($\leq n$-th order) elliptic equations on \emph{BCB domains with the boundary H\"{o}lder estimates}.

\medskip

To this end, we need the following definition on domains with blowing-up cone boundary (BCB domains for short).
\begin{defn}[BCB domains]\label{BCB}
A domain $\Omega$ is called \emph{a domain with blowing-up cone boundary} $\partial\Omega$ (\emph{BCB domain} for short) provided that for any $x(\rho): \, (0,+\infty)\to\partial\Omega$ such that $x(\rho)\to x_{0}\in\partial\Omega$, $\Omega_{\rho}:=\{x\mid \rho x+x(\rho)\in\Omega\}\to \mathcal{C}=\mathcal{C}_{P,\Sigma}$ with $P\in\mathbb{R}^{n}$ and $\Sigma\subseteq S^{n-1}$, as $\rho\rightarrow0$ along some subsequence of any given infinitesimal sequence $\{\rho_{k}\}$ (the cone $\mathcal{C}$ may depends on the subsequence).
\end{defn}

\smallskip

One should note that, there is \emph{no smoothness assumption} on the boundary $\partial\Omega$ in the Definition \ref{BCB} of BCB domains.

\begin{figure}[H]\label{F4}

 \begin{tikzpicture}
 \begin{scope}[scale=0.25]
  \draw  plot[smooth] coordinates{ (0,0)(0.1,0.2)(0.12,0.22)(0.13,0.24)(0.2,0.3)(0.3,0.4)(1,0.8)(2,1.8)(3,3)(3.1,3.5)(2.5,3.4)};
  \draw    (2.5,3.4)--(2,4)--(1.5,3)--(1.2,3.5)--(1,3)--(0.5,3.3) ;
  \draw  plot[smooth] coordinates{(0.5,3.3)(0,5)(-1,4.5)(-2,4.1)(-2.2,3)};
  \draw    (-2.2,3)--(-3,2.5)--(-2.5,2.2) ;
  \draw  plot[smooth] coordinates{ (0,0)(-0.1,0.2)(-0.12,0.22)(-0.13,0.24)(-0.2,0.3)(-0.3,0.4)(-1, 0.6)(-1.6, 1)(-2,1.5)(-2.5,2.2)};
  \filldraw (0,0) circle(3pt) node[below]  {$x_0$};
  \filldraw  (1,0.8) circle(3pt);
  \node[right]at(1,0.8) {$x(\rho)$};
  \draw[->](1.4,0.4)--(0.6,-0.2);
  \node[] at(-0.5,2){$\Omega$};
 \end{scope}
 \draw[->](1.5,1)--(2.5,1);
 \node[above]at(2,1){$\rho=0.5$};
 \begin{scope}[scale=0.5, xshift=250]
  \draw  plot[smooth] coordinates{ (0,0)(0.1,0.2)(0.12,0.22)(0.13,0.24)(0.2,0.3)(0.3,0.4)(1,0.8)(2,1.8)(3,3)(3.1,3.5)(2.5,3.4)};
  \draw    (2.5,3.4)--(2,4)--(1.5,3)--(1.2,3.5)--(1,3)--(0.5,3.3) ;
  \draw  plot[smooth] coordinates{(0.5,3.3)(0,5)(-1,4.5)(-2,4.1)(-2.2,3)};
  \draw    (-2.2,3)--(-3,2.5)--(-2.5,2.2) ;
  \draw  plot[smooth] coordinates{ (0,0)(-0.1,0.2)(-0.12,0.22)(-0.13,0.24)(-0.2,0.3)(-0.3,0.4)(-1, 0.6)(-1.6, 1)(-2,1.5)(-2.5,2.2)};
  \filldraw (0,0) circle(1.5pt);
  \node[below] at(0,0) {$ 2(x_0-x(0.5)) $};
  \filldraw  (0.2,0.3) circle(1.5pt) node[above]  {$0$};
  \node[] at(-0.5,2){$\Omega_{0.5}$};
 \end{scope}
 \draw[->](6.5,1)--(7.5,1);
 \node[above]at(7,1){$\rho=0.1$};
 \begin{scope}[scale=5, xshift=50,yshift=0]
  \filldraw[draw=red!40, fill=red!40]  (0.051,0.11) arc (65:241:0.03)--cycle  ;
  \filldraw[draw=blue!50, fill=blue!50]  (0,0)--(0.02,0.04) arc (60:120:0.04) (-0.02,0.04)--(0,0)--cycle  ;
  \draw  plot[smooth] coordinates{ (0,0)(0.1,0.2)(0.12,0.22)(0.13,0.24)(0.2,0.3)};
  \draw[dashed] (0.2,0.3)--(0.3,0.4);
  \draw  plot[smooth] coordinates{(0,0)(-0.1,0.2)(-0.12,0.22)(-0.13,0.24)(-0.2,0.3)};
  \draw[dashed] (-0.2,0.3)--(-0.3,0.4);
  \filldraw (0,0) circle(0.1pt) node[below]  {$ 10(x_0-x(0.1)) $};

  \filldraw  (0.04,0.08) circle(0.1pt) node[right]  {$0$};

  \node[] at(0,0.3){$\Omega_{0.1}$};
 \end{scope}
 \draw[->](10.8,2)--(12.5,3.2);
 \node[above,rotate=33]at(11.8,1.9){$\rho\to0$};
 \node[]at(11.1,3.5){if $\frac{x_0-x(\rho)}{\rho}\to P$};
 \begin{scope}[scale=0.4, xshift=1000,yshift=200]
  \filldraw[draw=blue!50, fill=blue!50]  (0,0)--(2,4) arc (60:120:4) (-2,4)--(0,0)--cycle  ;
  \draw   (0,0)--(2,4);
  \draw[dashed] (2,4)--(3,6);
  \draw   (0,0)--(-2,4);
  \draw[dashed] (-2,4)--(-3,6);
  \filldraw  (0,0) circle(2.5pt);
  \filldraw  (1,2) circle(2.5pt) node[right]  {$0$};
  \node[below] at(0,0){$P$};
  \node[] at(0,5){$\mathcal C_{P,\Sigma}$};
 \end{scope}
 \draw[->](10.8,-0.3)--(12.5,-1.3);
 \node[above,rotate=-33]at(11.6,-0.8){$\rho\to0$};
 \node[]at(11.1,-1.9){if $\frac{x_0-x(\rho)}{\rho}\to \infty$};
 \begin{scope}[scale=0.6, xshift=680,yshift=-90]
  \filldraw[draw=red!40, fill=red!40]  (-1,-2)--(-3,-1)--(-1,3)--(1,2)--(-1,0-2)--cycle  ;
  \draw   (-1,-2)--(1,2);
  \draw[dashed] (-1,-2)--(-2,-4);
  \draw[dashed] (1,2)--(2,4);
  \node[left] at(-0.5,0){$H$};
  \filldraw  (0,0) circle(1.7pt) node[right]  {$0$};
 \end{scope}
\end{tikzpicture}

 \caption{Blowing-up analysis on BCB domains}
\end{figure}
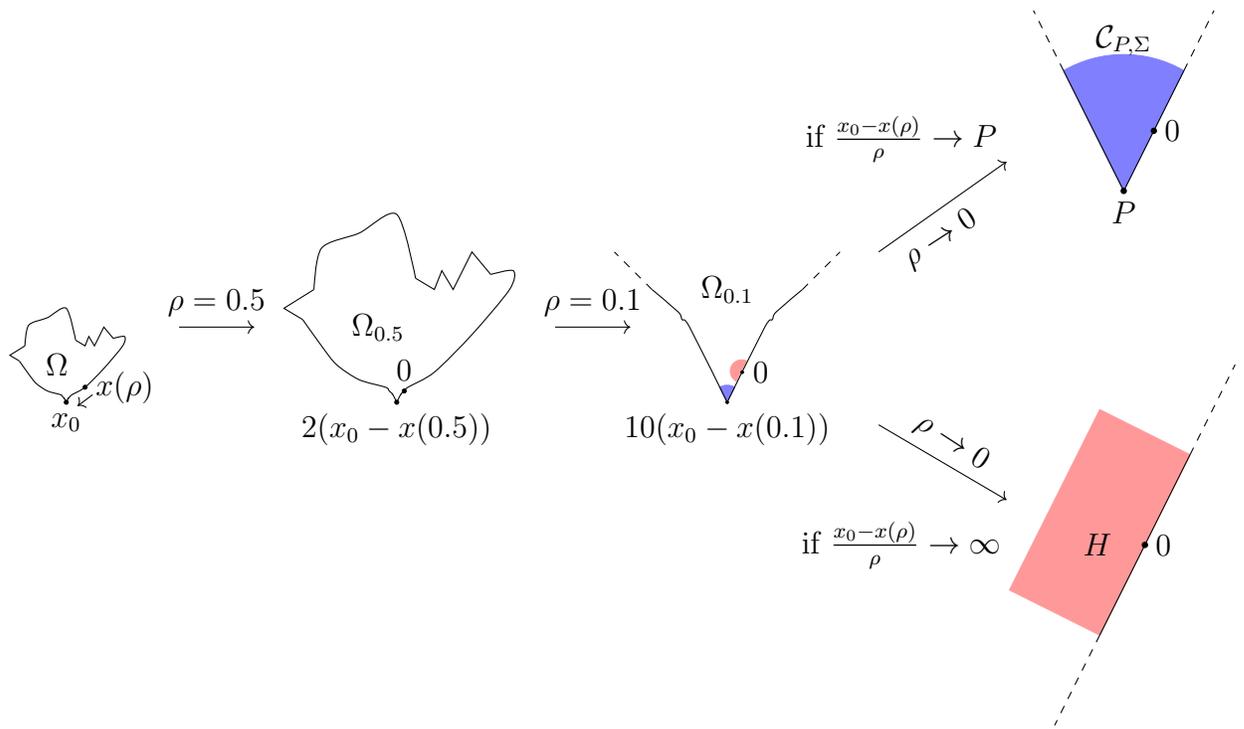

\begin{figure}[H]\label{F5}
	\centering	
	\subfigure[Polygon]{
		\begin{minipage}[b]{0.24\linewidth}
			\centering
			\begin{tikzpicture}[scale=0.8]
					\draw (0,0)--(2,1)--(3,4)--(0,5)--(-2,4)--(-1,3)--(-1.5,2)--(0,0);
			\end{tikzpicture}
		\end{minipage}
	}\hspace{5pt}
\subfigure[Dumb-bell-shaped domain]{
	\begin{minipage}[b]{0.25\linewidth}
		\centering
		\begin{tikzpicture}[scale=0.8]
				\draw (0,0)--(0.4, 0.8)--(1,1)--(1,1.5)--(0,2)--(-2,2)--(-1,0)--(-1.5,-1)--(0,-0.1)--(3,-0.1)--(3.5, -0.8)--(4.5,0)--(4.5,1)--(3.3, 1.5)--(3, 0)--(0,0);

		\end{tikzpicture}
	\end{minipage}
}\hspace{25pt}
\subfigure[Circular cone]{
	\begin{minipage}[b]{0.24\linewidth}
		\centering
		\begin{tikzpicture}[scale=1]
			\draw (-2,0)--(0,3)--(2,0)   arc (0:-180: 2 and 0.4);
			\draw[dashed] (2,0)   arc (0:180: 2 and 0.4);
		\end{tikzpicture}
	\end{minipage}
}\hspace{15pt}
	 \subfigure [Polyhedron] {
			\centering
				\includegraphics[width=1.5in,height=1.5in]{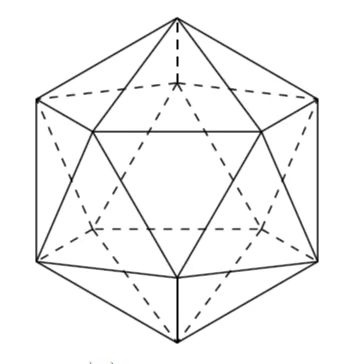}
			\label{F5-2}
	}\hspace{15pt}
  \subfigure[Cuboid with  a hole] {
  	\begin{minipage}[b]{0.24\linewidth}
  		\centering	\begin{tikzpicture}[scale=0.7]
  	\draw (-2,0)--(2,0)--(2,3)--(-2,3)--(-2, 0) (-2,3)--(-0.5,4.5)--(3.5,4.5)--(2,3) (2,0)--(3.5,1.5)--(3.5,4.5) (0.75,3.75) ellipse (0.6 and 0.3);
  	\draw[dashed] (-2,0)--(-0.5,1.5)--(-0.5,4.5) (-0.5,1.5)--(3.5,1.5)  (0.75,0.75) ellipse (1 and 0.5)  (-0.25,0.75)--(0.15, 3.75) (1.75, 0.75)--(1.35, 3.75);
  \end{tikzpicture} 	\end{minipage}
 }\hspace{15pt}
	\subfigure[Piecewise smooth domain] {
		\begin{minipage}[b]{0.24\linewidth}
			\centering
				\begin{tikzpicture}[scale=0.6]
				\draw  plot[smooth] coordinates{ (0,0)(0.1,0.7)(0.2,1)(0, 2)(-0.1,3)(-0.2,4)(-0.1,5)};
				\draw[dashed] plot[smooth] coordinates{ (0,0)(0.6,0.7)(1,1)(1.5, 1.2)  };
				\draw[dashed] plot[smooth] coordinates{ (1.5, 1.2)  (2,1)(3,0.9)(4,0.7)};
				\draw  plot[smooth] coordinates{ (0,0)(0.6,0.1)(1,0)(1.5, 0.1)(2,0.1)(3,0.5)(4,0.7)};
				\draw  plot[smooth] coordinates{(-0.1,5)(1,4.8)(1.5, 4)(3,3.7)(4,3.8)(5,3.4)(6,3.3)};
				\draw  plot[smooth] coordinates{(4,0.7)(5,2.8)(6,3.3)};
				\draw  plot[smooth] coordinates{(-0.1,5)(1,5.5)(2,5.7)(3,6.1)};
				\draw  plot[smooth] coordinates{(3,6.1)(4,5.5)(5,4)(6,3.3)};
				\draw[dashed] plot[smooth] coordinates{ (1.5, 1.2)(1.8,2)(1.85,2.1)(1.9,3)(2.5,5)(3,6.1)  };
				
			\end{tikzpicture}
			\label{F5-3}
		\end{minipage}
	}%
	\centering
	\caption{BCB domains (examples)}
\end{figure}
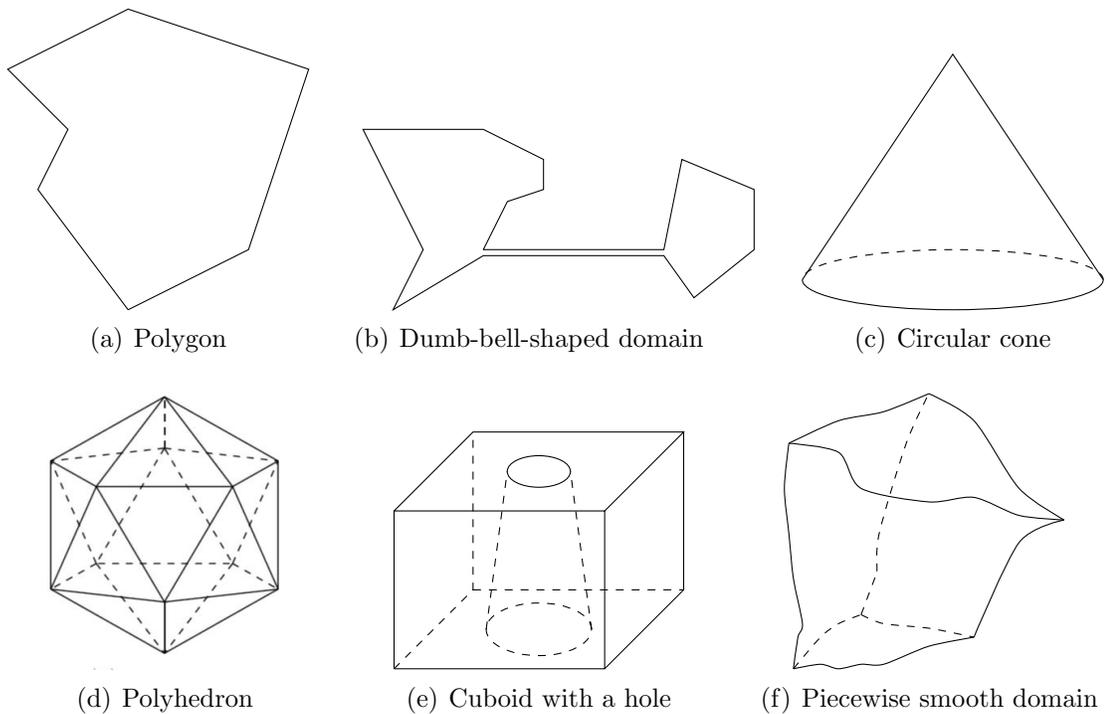

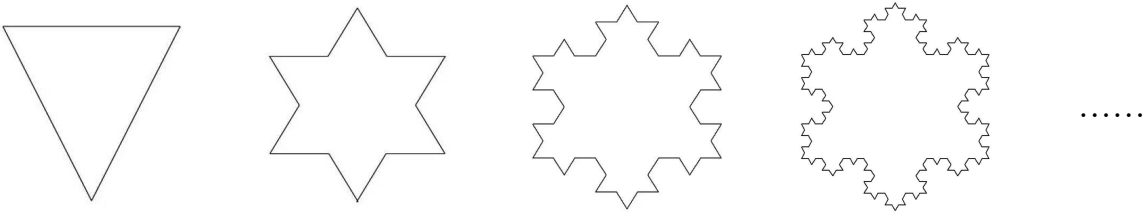
\begin{figure}[H]\label{F6}
	\centering	
	\subfigure {	
		\includegraphics[width=1.1in,height=1.1in]{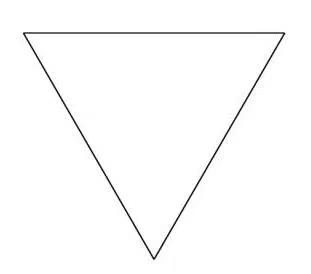}			
			\label{F6-1}
	}\hspace{10pt}
	\subfigure {
			\centering
				\includegraphics[width=1.1in,height=1.1in]{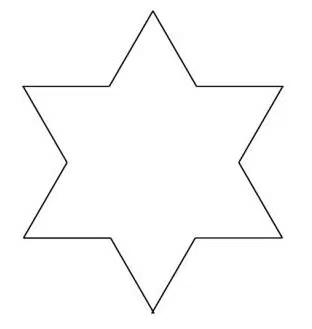}
			\label{F6-2}
	}\hspace{10pt}
	\subfigure {
			\centering
				\includegraphics[width=1.1in,height=1.1in]{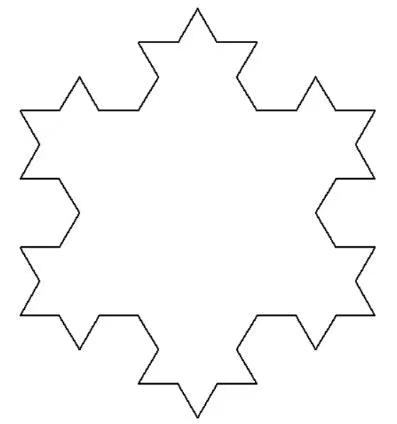}
			\label{F6-3}
	}\hspace{10pt}
	\subfigure {
		\centering
		\includegraphics[width=1.1in,height=1.1in]{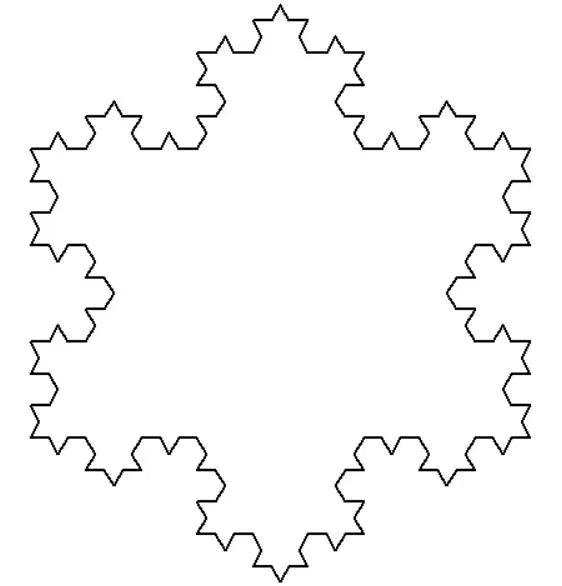}
		\label{F6-4}
}\hspace{-30pt}
\subfigure{
	\begin{minipage}[t]{0.24\linewidth}
		\centering
		\begin{tikzpicture}[scale=2]
			\node at(0,0.6) {\Large{......}};
			\filldraw[draw=white, fill=white] (0,0) circle(1pt);
		\end{tikzpicture}
		
		\label{F6-5}
	\end{minipage}
}
	\centering
	\caption{A sequence of BCB domains}
\end{figure}

We consider the following fractional order Dirichlet problem
\begin{equation}\label{f-1}
	\begin{cases}
		(-\Delta)^{s}u(x)=f(x,u(x)),  & x \in \Omega,\\
		u(x)\equiv0, &x \not\in \Omega,
	\end{cases}
\end{equation}
where $n\geq\max\{2s,1\}$, $0<s<1$, $\Omega\subset\mathbb{R}^{n}$ is a bounded BCB domain, the classical solution $u\in C^{[2s],\{2s\}+\eps}_{loc}(\Omega)\bigcap\mathcal{L}_{s}(\mathbb{R}^{n})\bigcap C(\overline{\Omega})$ with $\eps>0$ arbitrarily small.

\medskip

As an application of the Liouville theorems (Theorems \ref{pde-unbd}, \ref{pde-unbd-a}, \ref{pde-unbd-2}--\ref{pde-bd-2}) in subsection 1.2 and Theorems \ref{pG} and \ref{f2PDE}, we establish the following a priori estimates for all positive classical solutions $u$ to the fractional order Dirichlet problem \eqref{f-1} via the blowing-up arguments.
\begin{thm}\label{prid}
	Assume there exists a positive, continuous functions $h(x)$: $\overline{\Omega}\rightarrow(0,+\infty)$ such that
	\begin{equation}\label{f-2}
		\lim_{s\rightarrow+\infty}\frac{f(x,s)}{s^{p}}=h(x)
	\end{equation}
	for some $1<p<\frac{n+2s}{n-2s}$ if $n>2s$ and $1<p<+\infty$ if $n=2s$ uniformly with respect to $x\in\overline{\Omega}$. Moreover, assume there exist $\sigma_1,\sigma_2>0$ such that
	\begin{equation}\label{f-2-1}
		|f(x,r)-f(y,r)|\leq C(1+r)^{p+\frac{(p-1)\sigma_1}{2s}} |x-y|^{\sigma_1}, \ \ \ \forall\, x,y \in \Omega\,\,\text{and}\,\,r\in[0,+\infty),
	\end{equation}
	and
	\begin{equation}\label{f-2-2}
		|f(x,r)-f(x,t)|\leq C (1+ r + t )^{p-\sigma_2} |r-t|^{\sigma_2}, \ \ \ \forall\, x\in\Omega\,\,\text{and}\,\, r,t\in [0,+\infty).
	\end{equation}
Suppose that the boundary H\"{o}lder estimates for $(-\Delta)^{s}$ holds uniformly on the scaling domain $\Omega_{\rho}$ w.r.t. the scale $\rho\in(0,1]$, i.e., $[v]_{C^{0,\alpha}(\overline{B_{1}(x_{0})\bigcap\Omega_{\rho}})}$ ($\forall \, x_{0}\in\partial\Omega_{\rho}$) can be controlled in terms of $\|v\|_{L^{\infty}(\overline{\Omega_{\rho}})}$ and $\|(-\Delta)^{s}v\|_{L^{p}(B_{2}(x_{0})\bigcap\Omega_{\rho})}$ for some $p>\frac{n}{2s}$ if $v=0$ in $\Omega_{\rho}^{c}$ up to a constant independent of $\rho$. Then there exists a constant $C>0$ depending only on $\Omega$, $n$, $s$, $p$, $h(x)$ such that
	\begin{equation}\label{f-3}
		\|u\|_{L^\infty(\overline{\Omega})}\leq C
	\end{equation}
	for any nonnegative classical solution $u$ of \eqref{f-1}.
\end{thm}

\smallskip

For $s=1$, consider the following general second order elliptic Dirichlet problem:
\begin{equation}\label{PDE-D}\\\begin{cases}
Lu(x):=-\sum\limits_{i,j=1}^{n}a_{ij}(x)\frac{\partial^{2}u}{\partial x_{i}\partial x_{j}}(x)+\sum\limits_{k=1}^{n}b_{k}(x)\frac{\partial u}{\partial x_{k}}(x)+c(x)u(x)=f(x,u(x)), \,\,\,\, x\in\Omega, \\
u(x)=0, \,\,\,\,\,\,\,\, x\in\partial\Omega,
\end{cases}\end{equation}
where $n\geq2$, $\Omega\subset\mathbb{R}^{n}$ is a bounded BCB domain, the classical solution $u\in C^{2}(\Omega)\bigcap C(\overline{\Omega})$, the coefficients $b_{k},\,c\in L^{\infty}(\Omega)$ ($k=1,\cdots,n$) and $a_{ij}\in C(\overline{\Omega})$ ($i,j=1,\cdots,n$) such that there exists constant $\tau>0$ with
\begin{equation}\label{0-2}
  \tau|\xi|^{2}\leq\sum_{i,j=1}^{n}a_{ij}(x)\xi_{i}\xi_{j}\leq\tau^{-1}|\xi|^{2}, \,\,\,\,\,\,\,\,\,\, \forall \,\, \xi\in\mathbb{R}^{n}, \,\, x\in\Omega.
\end{equation}

\medskip

By virtue of the Liouville theorems (Theorems \ref{pde-unbd}, \ref{pde-unbd-a} and \ref{pde-bd}) in subsection 1.2 and Theorems \ref{pG} and \ref{f2PDE}, we establish the following a priori estimates for all nonnegative classical solutions $u$ to the second order Dirichlet problem \eqref{PDE-D} via the blowing-up arguments.
\begin{thm}\label{prid2}
Assume $n\geq2$, $1<p<\frac{n+2}{n-2}$ if $n>2$ and $1<p<+\infty$ if $n=2$, $\Omega$ is a bounded BCB domain such that the boundary H\"{o}lder estimates for the general 2nd order elliptic operator $L$ holds uniformly on the scaling domain $\Omega_{\rho}$ w.r.t. the scale $\rho\in(0,1]$, i.e., $[v]_{C^{0,\alpha}(\overline{\Omega_{\rho}\bigcap B_{1}(x_{0})})}$ ($\forall \, x_{0}\in\partial\Omega_{\rho}$) can be controlled in terms of $\|v\|_{L^{\infty}(\overline{\Omega_{\rho}})}$ and $\|Lv\|_{L^{q}(\Omega_{\rho}\bigcap B_{2}(x_{0}))}$ for some $q>\frac{n}{2}$ if $v=0$ on $\partial\Omega_{\rho}$ up to a constant depending on $\tau$, $\|b\|_{L^{\infty}(\Omega_{\rho})}$ and $\|c\|_{L^{\infty}(\Omega_{\rho})}$ but independent of $\rho$. Suppose there exists positive, continuous function $h(x)$: $\overline{\Omega}\rightarrow(0,+\infty)$ such that
\begin{equation}\label{0-3-2}
  \lim_{s\rightarrow+\infty}\frac{f(x,s)}{s^{p}}=h(x)
\end{equation}
uniformly with respect to $x\in\overline{\Omega}$, and
\begin{equation}\label{0-6-2}
  \sup_{x\in \Omega, u\in [0, M]} |f(x,u)|\leq C_M, \qquad \forall \, M>0.
\end{equation}
Then there exists a constant $C>0$ depending only on $\Omega$, $n$, $p$ and $h(x)$, such that
\begin{equation}\label{0-4-2}
  \|u\|_{L^{\infty}(\overline{\Omega})}\leq C
\end{equation}
for every nonnegative classical solution $u$ of problem \eqref{PDE-D}.
\end{thm}

\smallskip

As a corollary of Theorems \ref{prid} and \ref{prid2}, we can deduce the a priori estimates for fractional and 2nd order Dirichlet problem
\begin{equation}\label{tD}\\\begin{cases}
(-\Delta)^{s}u(x)=u^{p}(x)+t \,\,\,\,\,\,\,\,\,\, \text{in} \,\,\, \Omega, \\
u(x)=0 \,\,\,\,\,\,\,\, \text{in} \,\,\, \mathbb{R}^{n}\setminus\Omega,
\end{cases}\end{equation}
where $0<s\leq1$ and $t$ is any nonnegative real number.
\begin{cor}\label{cor-prid}
Assume $n\geq\max\{2s,1\}$, $0<s\leq1$, $t\geq0$, $\Omega\subset\mathbb{R}^{n}$ is a bounded BCB domain such that the boundary H\"{o}lder estimates for $(-\Delta)^{s}$ holds uniformly on the scaling domain $\Omega_{\rho}$ w.r.t. the scale $\rho\in(0,1]$, $1<p<\frac{n+2s}{n-2s}$ if $n>2s$ and $1<p<+\infty$ if $n=2s$. Then, for any nonnegative solution $u$ to the Dirichlet problem \eqref{tD}, we have
\begin{equation*}
  \|u\|_{L^{\infty}(\overline{\Omega})}\leq C(n,s,p,t,\Omega).
\end{equation*}
\end{cor}

\smallskip

In order to derive the boundary H\"{o}lder estimates, we need the following definition from \cite{LZLH}.
\begin{defn}\label{gc2}
Let $\Omega\subset\mathbb{R}^{n}$ be a bounded domain and $x_{0}\in\partial\Omega$. We say that $\Omega$ satisfies the geometric condition $\mathbf{\widehat{G}}$ at $x_{0}$ if there exist a constant $\nu\in(0,1)$ and a quasi-geometric sequence $\{r_{k}\}_{k=0}^{+\infty}$ such that
\begin{equation}\label{gc0-2}
  r_{k}^{-(n-1)}\mathcal{H}^{n-1}\left(\partial B_{r_{k}}(x_{0})\bigcap \Omega^c \right)\geq\nu, \qquad \forall \,\, k\geq0,
\end{equation}
where $\mathcal{H}^{n-1}$ denotes the $n-1$ dimensional Hausdorff measure.
\end{defn}

\medskip

Let $\Omega\subset \mathbb R^n$ be a domain and $P\in \mathbb R^n$ be a point. Moreover, we assume that, for any $R>0$, the geometric condition $\mathbf{\widehat{G}}$ holds uniformly on $\partial\Omega\bigcap B_R(P)$, that is,
\begin{eqnarray}\label{pde-boundary2}
	&& \forall \,\, R>0, \,\, \text{there exists} \,\, \nu_R>0 \,\, \text{s.t.}, \,\, \forall \,\, x\in\partial\Omega\bigcap B_R(P), \,\, \text{there are} \\ \nonumber && \text{quasi-geometric sequence} \,\, \{r_{k}\}_{k=0}^{+\infty} \,\,\text{and} \,\, \nu\geq\nu_{R} \,\, \text{s.t.}\,\, \eqref{gc0-2} \,\, \text{holds}.	
\end{eqnarray}

\medskip

One should note that the geometric conditions \eqref{pde-boundary} and \eqref{pde-boundary2}, the uniform exterior cone condition and Lipschitz continuity of $\partial\Omega$ are \emph{invariant under scaling} (with constants unchanged). The uniform boundary H\"{o}lder estimates for $(-\Delta)^{s}$ ($0<s\leq1$) w.r.t. $\rho$ on $\Omega_{\rho}$ can be guaranteed by \emph{the geometric conditions} \eqref{pde-boundary} if $s\in(0,1)$ and \eqref{pde-boundary2} if $s=1$ on $\partial\Omega$, which can also be deduced from \emph{the uniform exterior cone condition} or \emph{Lipschitz continuity} of $\partial\Omega$; the uniform boundary H\"{o}lder estimates for general 2nd ordr elliptic operator $L$ w.r.t. $\rho$ on $\Omega_{\rho}$ can be guaranteed by \emph{the uniform exterior cone condition} or \emph{Lipschitz continuity} of $\partial\Omega$ (c.f. Theorems 3.3 and 3.11 in \cite{LZLH} and Corollaries 9.28 and 9.29 in \cite{GT}). Thus we have the following corollary.
\begin{cor}\label{cor-prid0}
Assume $n\geq\max\{2s,1\}$, $0<s\leq1$, $1<p<\frac{n+2s}{n-2s}$ if $n>2s$ and $1<p<+\infty$ if $n=2s$, and $\Omega$ is a bounded BCB domain. Suppose $\partial\Omega$ satisfies \eqref{pde-boundary} if $s\in(0,1)$ and \eqref{pde-boundary2} if $s=1$, or the uniform exterior cone condition, or Lipschitz continuity. Then the a priori estimates in Theorem \ref{prid} and Corollary \ref{cor-prid} hold true. Suppose $\partial\Omega$ satisfies the uniform exterior cone condition or Lipschitz continuity. Then the a priori estimates in Theorem \ref{prid2} holds true.
\end{cor}

\smallskip

For $s\in\mathbb{Z}^{+}$, as an immediate application of the Liouville theorems for higher order Navier problem on $\frac{1}{2^{k}}$-space $\mathcal{C}_{P,S^{n-1}_{2^{-k}}}$ (Theorems \ref{nie} and \ref{NPDE}) in subsection 1.3, we can derive a priori estimates of positive solutions to higher order elliptic Navier problems in bounded BCB domains with the limiting cone $\mathcal{C}=\mathcal{C}_{P,S^{n-1}_{2^{-k}}}$ ($k=1,\cdots,n$) in Definition \ref{BCB} for all $1<p<\frac{n+2s}{n-2s}$ if $n>2s$ and $1<p<+\infty$ if $n=2s$.

\medskip

For $s\in\mathbb{Z}^{+}$, consider the following general higher order Navier boundary value problem:
\begin{equation}\label{PDE-N}\\\begin{cases}
(-\Delta)^{s}u(x)=f(x,u(x)), \,\,\,\,\,\,\,\,\,\,\,\,\,\,\,\, x\in\Omega, \\
u(x)=(-\Delta)u(x)=\cdots=(-\Delta)^{s-1}u(x)=0, \,\,\,\,\,\,\,\, x\in\partial\Omega,
\end{cases}\end{equation}
where $n\geq2$, $s\in\mathbb{Z}^{+}$, $2s\leq n$, nonnegative solution $u\in C^{2s}(\Omega)$ such that $(-\Delta)^{i}u\in C(\overline{\Omega})$ ($i=0,\cdots,s-1$), $f(x,u(x))\geq0$, and $\Omega$ is a bounded BCB domain such that the limiting cone in Definition \ref{BCB} is $\frac{1}{2^{k}}$-space $P+\frac{1}{2^{k}}\mathbb{R}^{n}:=\mathcal{C}_{P,S^{n-1}_{2^{-k}}}$ with $P\in\mathbb{R}^{n}$ and $k=1,\cdots,n$, i.e., the limiting shape of the domain can be an arbitrary $\frac{1}{2^{k}}$-space after the blowing-up procedure.

\medskip

By virtue of the Liouville theorems for higher order Navier problem on $\frac{1}{2^{k}}$-space $\mathcal{C}_{P,S^{n-1}_{2^{-k}}}$ (Theorems \ref{nie} and \ref{NPDE}) in subsection 1.3, using the blowing-up methods, we can derive the following a priori estimate for nonnegative classical solutions to the higher order Navier problem \eqref{PDE-N} in the full range $1<p<\frac{n+2s}{n-2s}$ if $n>2s$ and $1<p<+\infty$ if $n=2s$.
\begin{thm}\label{prin}
Assume $n\geq2$, $s\in\mathbb{Z}^{+}$, $2s\leq n$, $1<p<\frac{n+2s}{n-2s}$ if $n>2s$ and $1<p<+\infty$ if $n=2s$, $\Omega$ is a bounded BCB domain such that the limiting cone $\mathcal{C}$ is $\frac{1}{2^{k}}$-space ($k=1,\cdots,n$) and the boundary H\"{o}lder estimates for $-\Delta$ holds uniformly on the scaling domain $\Omega_{\rho}$ w.r.t. the scale $\rho\in(0,1]$, i.e., $[v]_{C^{0,\alpha}(x_{0})}$ ($\forall \, x_{0}\in\partial\Omega_{\rho}$) can be controlled in terms of $\|v\|_{L^{\infty}(\Omega_{\rho}\bigcap B_{1}(x_{0}))}$ and $\|\Delta v\|_{L^{q}(\Omega_{\rho}\bigcap B_{1}(x_{0}))}$ for some $q>\frac{n}{2}$ if $v=0$ in $\partial\Omega_{\rho}$ up to a constant independent of $\rho$. Suppose there exists positive, continuous function $h(x)$: $\overline{\Omega}\rightarrow(0,+\infty)$ such that
\begin{equation}\label{0-3-n}
  \lim_{s\rightarrow+\infty}\frac{f(x,s)}{s^{p}}=h(x)
\end{equation}
uniformly with respect to $x\in\overline{\Omega}$, and
\begin{equation}\label{0-6-n}
  \sup_{x\in \Omega, u\in [0, M]} |f(x,u)|\leq C_M, \qquad \forall \, M>0.
\end{equation}
Then there exists a constant $C>0$ depending only on $\Omega$, $n$, $s$, $p$ and $h(x)$, such that
\begin{equation}\label{0-4-n}
  \|u\|_{L^{\infty}(\overline{\Omega})}\leq C
\end{equation}
for every nonnegative classical solution $u$ of problem \eqref{PDE-N}.
\end{thm}

\begin{figure}[H]\label{F7}
	\centering	
	\subfigure [Planar domain with vertical corners]{
		\begin{minipage}[b]{0.24\linewidth}
			\centering
			\begin{tikzpicture} [scale=1.8]
				\draw (0,0.5)--(0,0)--(0.5,0) (2.2,-0.2)--(2.2,1) (0.5,1)--(0.2,1.2);
				\draw plot[smooth] coordinates{(0.5,0) (0.57,-0.4) (1,-0.5) (1.5,0) (2,-0.2) (2.2,-0.2)};
				\draw plot[smooth] coordinates{(2.2,1) (1.5, 1.07) (1,1.5) (0.5,1)};
				\draw plot[smooth] coordinates{(0.2,1.2) (0,0.9) (0,0.5)};
			\end{tikzpicture}			
		\end{minipage}
	}\hspace{15pt}
	\subfigure [Cuboid]{
		\begin{minipage}[b]{0.24\linewidth}
			\centering
			\begin{tikzpicture} [xscale=0.9, yscale=0.8]
				\draw (0,0)--(4,0)--(4,3)--(0,3)--(0, 0) (0,3)--(1.5,4.5)--(5.5,4.5)--(4,3) (4,0)--(5.5,1.5)--(5.5,4.5);
				\draw[dashed] (0,0)--(1.5,1.5)--(1.5,4.5) (1.5,1.5)--(5.5,1.5);
			\end{tikzpicture}
		\end{minipage}
	}\hspace{15pt}
\subfigure [Symmetric bounded cone]{
	\begin{minipage}[b]{0.24\linewidth}
		\centering
			\begin{tikzpicture}[scale=1.1, rotate=-90]
			\draw  (2,-2)--(0,0)--(2,2) arc (90:-90: 0.3 and 2) (2,2) arc (90:-90: 0.9 and 2);
			\draw[dashed]  (2,2) arc (90:270:  0.3 and 2) (0,0)--(2,0)--(2,2);
			\filldraw[ ] (2,0) circle(1pt);
			\node[below] at (1, -0.2) {1};
			\node[below] at (1.9, 0.8) {1};
		\end{tikzpicture}
	\end{minipage}
}\hspace{65pt}
	\subfigure [$\frac{1}{8}$-Ball] {
		\begin{minipage}[b]{0.24\linewidth}
			\centering
			\begin{tikzpicture}[yscale=0.9]
				\draw [dashed] (4,0)--(0,0);
				\draw (0.2,-0.2)--(0.6,-0.2)--(0.4,0)  (0,0)--(0.1,0) (0,0)--(0,4) (0,0)--(1.4,-1.4) arc (-71.5:-14.2: 4 and 2)   (1.4,-1.4) arc (-7:72.5: 2 and 5)   (4,0) arc (0:90:4);
			\end{tikzpicture}
		\end{minipage}
	} \hspace{25pt}
	\subfigure [ Cylinder] {
		\begin{minipage}[b]{0.24\linewidth}
			\centering
			\begin{tikzpicture}[yscale=0.7]
				\draw [dashed]  (2,0) arc(0: 180: 2 and 1) ;
				\draw   (0,4) ellipse (2 and 1) (2, 4)--(2,0) arc(0:-180: 2 and 1)--(-2,4);
			\end{tikzpicture}
			\label{F7-3}
		\end{minipage}
	}\hspace{15pt}  \subfigure[Cuboid with  a cylinder-shaped hole] {
	\begin{minipage}[b]{0.24\linewidth}
		\centering	\begin{tikzpicture}[scale=0.9]
			\draw (-2,0)--(2,0)--(2,3)--(-2,3)--(-2, 0) (-2,3)--(-0.5,4.5)--(3.5,4.5)--(2,3) (2,0)--(3.5,1.5)--(3.5,4.5) (0.75,3.75) ellipse (1 and 0.5);
			\draw[dashed] (-2,0)--(-0.5,1.5)--(-0.5,4.5) (-0.5,1.5)--(3.5,1.5)  (0.75,0.75) ellipse (1 and 0.5)  (-0.25,0.75)--(-0.25, 3.75) (1.75, 0.75)--(1.75, 3.75);
	\end{tikzpicture} 	\end{minipage}
}\hspace{15pt}
	
	\centering
	\caption{BCB domains such that the limiting cone $\mathcal{C}$ is $\frac{1}{2^{k}}$-space (examples)}
\end{figure}
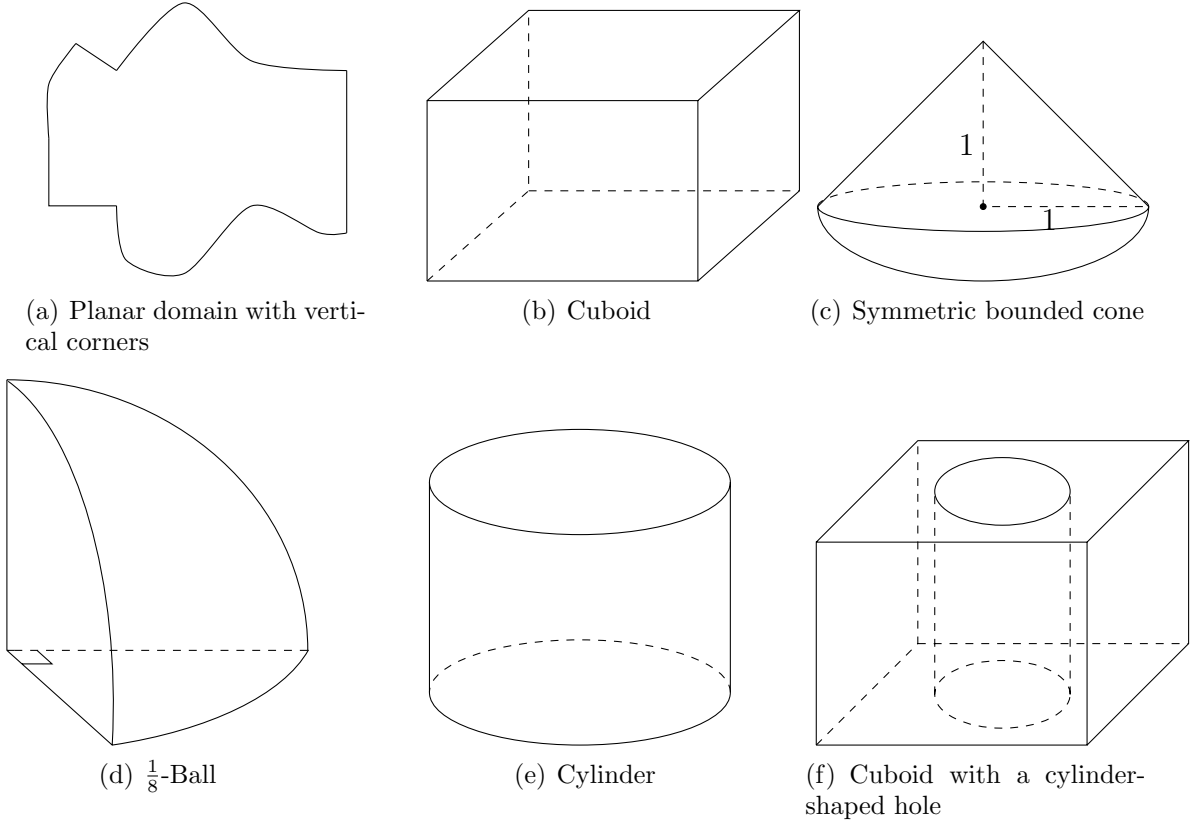

One can immediately apply Theorem \ref{prin} to the following higher order Navier problem
\begin{equation}\label{tNavier}\\\begin{cases}
(-\Delta)^{s}u(x)=u^{p}(x)+t \,\,\,\,\,\,\,\,\,\, \text{in} \,\,\, \Omega, \\
u(x)=(-\Delta) u(x)=\cdots=(-\Delta)^{s-1}u(x)=0 \,\,\,\,\,\,\,\, \text{on} \,\,\, \partial\Omega,
\end{cases}\end{equation}
where $n\geq2$, $s\in\mathbb{Z}^{+}$, $2s\leq n$, $1<p<\frac{n+2s}{n-2s}$ if $2s<n$ and $1<p<+\infty$ if $2s=n$, nonnegative solution $u\in C^{2s}(\Omega)$ such that $(-\Delta)^{i}u\in C(\overline{\Omega})$ ($i=0,\cdots,s-1$), $\Omega$ is a bounded BCB domain such that the limiting cone in Definition \ref{BCB} is $\frac{1}{2^{k}}$-space $\mathcal{C}_{P,S^{n-1}_{2^{-k}}}$ and $t$ is an arbitrary nonnegative real number.

\medskip

We can deduce the following corollary from Theorem \ref{prin}.
\begin{cor}\label{cor-n}
Assume $n\geq2$, $s\in\mathbb{Z}^{+}$, $2s\leq n$, $1<p<\frac{n+2s}{n-2s}$ if $n>2s$ and $1<p<+\infty$ if $n=2s$, $\Omega$ is a bounded BCB domain such that the limiting cone $\mathcal{C}$ is $\frac{1}{2^{k}}$-space ($k=1,\cdots,n$) and the boundary H\"{o}lder estimates for $-\Delta$ holds uniformly on the scaling domain $\Omega_{\rho}$ w.r.t. the scale $\rho\in(0,1]$, and $t\geq0$. Then, for any nonnegative solution $u$ to the higher order Navier problem \eqref{tNavier}, we have
\begin{equation}\label{0-5}
  \|u\|_{L^{\infty}(\overline{\Omega})}\leq C(n,s,p,t,\Omega).
\end{equation}
\end{cor}

\smallskip

Note that the geometric conditions \eqref{pde-boundary} and \eqref{pde-boundary2}, the uniform exterior cone condition and Lipschitz continuity of $\partial\Omega$ are invariant under scaling (with constants unchanged). The uniform boundary H\"{o}lder estimates for $-\Delta$ on the scaling domain $\Omega_{\rho}$ w.r.t. the scale $\rho$ can be deduced from the geometric condition \eqref{pde-boundary2} on $\partial\Omega$, which can also be guaranteed by the uniform exterior cone condition or Lipschitz continuity of $\partial\Omega$ (c.f. Theorem 3.3 in \cite{LZLH} and Corollaries 9.28 and 9.29 in \cite{GT}). Thus we have the following corollary.
\begin{cor}\label{cor-prin}
Assume $n\geq2$, $s\in\mathbb{Z}^{+}$, $2s\leq n$, $1<p<\frac{n+2s}{n-2s}$ if $n>2s$ and $1<p<+\infty$ if $n=2s$. Suppose $\Omega$ is a bounded BCB domain such that the limiting cone $\mathcal{C}$ is $\frac{1}{2^{k}}$-space ($k=1,\cdots,n$) and $\partial\Omega$ satisfies \eqref{pde-boundary2}, or the uniform exterior cone condition, or Lipschitz continuity. Then the a priori estimates in Theorem \ref{prin} and Corollary \ref{cor-n} hold true.
\end{cor}

\begin{rem}\label{rem11}
For literature on Dirichlet or Neumann problems involving higher order elliptic operators, poly-harmonic boundary value problems and higher order equations on compact Riemannian manifolds arising from conformal geometry, please refer to Barton, Hofmann and Mayboroda \cite{BHM1}, Hebey, Robert and Wen \cite{HRW}, Mayboroda and Maz$'$ya \cite{MM1,MM2} and the references therein.
\end{rem}

\subsection{Applcations: existence of solutions on bounded BCB domains}

As a consequence of the a priori estimates in subsection 1.4 (Corollaries \ref{cor-prid}, \ref{cor-prid0}, \ref{cor-n} and \ref{cor-prin}), by using \emph{the Leray-Schauder fixed point theorem}, we can derive existence of positive solutions to Dirichlet or Navier problems of $\leq n$-th order Lane-Emden equations in \emph{bounded BCB domains}.

\medskip

For $0<s\leq1$, consider the following Dirichlet problem for fractional and second order Lane-Emden equations
\begin{equation}\label{Dirichlet}\\\begin{cases}
(-\Delta)^{s}u(x)=u^{p}(x) \,\,\,\,\,\,\,\,\,\, \text{in} \,\,\, \Omega, \\
u(x)=0 \,\,\,\,\,\,\,\, \text{on} \,\,\, \mathbb{R}^{n}\setminus\Omega,
\end{cases}\end{equation}
where $n\geq\max\{2s,1\}$, $\Omega\subset\mathbb{R}^{n}$ is a bounded BCB domain, $1<p<\frac{n+2s}{n-2s}$ if $2s<n$ and $1<p<+\infty$ if $2s=n$.

\medskip

From the a priori estimates in Corollary \ref{cor-prid} and the Leray-Schauder fixed point theorem, we derive the existence of positive solutions to Dirichlet problem \eqref{Dirichlet}.
\begin{thm}\label{exisd}
Assume $0<s\leq1$, $n\geq\max\{2s,1\}$, $\Omega\subset\mathbb{R}^{n}$ is a bounded BCB domain such that the boundary H\"{o}lder estimates for $(-\Delta)^{s}$ holds on the scaling domain $\Omega_{\rho}$ uniformly w.r.t. the scale $\rho\in(0,1]$, $1<p<\frac{n+2s}{n-2s}$ if $2s<n$ and $1<p<+\infty$ if $2s=n$. Then, the Dirichlet problem \eqref{Dirichlet} possesses at least one positive solution $u\in C^{[2s],\{2s\}+\eps}_{loc}(\Omega)\bigcap\mathcal{L}_{s}(\mathbb{R}^{n})\bigcap C(\overline{\Omega})$ with $\eps>0$ arbitrarily small if $s\in(0,1)$ and $u\in C^{2}(\Omega)\bigcap C(\overline{\Omega})$ if $s=1$. Moreover, the positive solution $u$ satisfies
\begin{equation}\label{lower-bound-d}
  \|u\|_{L^{\infty}(\overline{\Omega})}\geq\left(\frac{1}{\sqrt{C_{n,s}}diam\,\Omega}\right)^{\frac{2s}{p-1}},
\end{equation}
where $C_{n,s}$ is a positive constant depending only on $n,s$ and $C_{n,1}=\frac{1}{2n}$.
\end{thm}

\smallskip

Now consider the following Navier problem for higher order Lane-Emden equations
\begin{equation}\label{Navier}\\\begin{cases}
(-\Delta)^{s}u(x)=u^{p}(x) \,\,\,\,\,\,\,\,\,\, \text{in} \,\,\, \Omega, \\
u(x)=(-\Delta) u(x)=\cdots=(-\Delta)^{s-1}u(x)=0 \,\,\,\,\,\,\,\, \text{on} \,\,\, \partial\Omega,
\end{cases}\end{equation}
where $n\geq2$, $s\in\mathbb{Z}^{+}$, $2s\leq n$, $1<p<\frac{n+2s}{n-2s}$ if $n>2s$ and $1<p<+\infty$ if $n=2s$, and $\Omega$ is a bounded BCB domain such that the limiting cone $\mathcal{C}$ in Definition \ref{BCB} is $\frac{1}{2^{k}}$-space $\mathcal{C}_{P,S^{n-1}_{2^{-k}}}$ ($k=1,\cdots,n$).

\medskip

By virtue of the a priori estimates in Corollary \ref{cor-n}, applying the Leray-Schauder fixed point theorem, we can derive existence result for positive solution to the Navier problem for higher order Lane-Emden equations \eqref{Navier} on BCB domains such that the limiting cone $\mathcal{C}$ is $\frac{1}{2^{k}}$-space.
\begin{thm}\label{exisn}
Assume $n\geq2$, $s\in\mathbb{Z}^{+}$, $2s\leq n$, $1<p<\frac{n+2s}{n-2s}$ if $n>2s$ and $1<p<+\infty$ if $n=2s$, and $\Omega$ is a bounded BCB domain such that the limiting cone $\mathcal{C}$ is $\frac{1}{2^{k}}$-space ($k=1,\cdots,n$) and the boundary H\"{o}lder estimates for $-\Delta$ holds on the scaling domain $\Omega_{\rho}$ uniformly w.r.t. the scale $\rho\in(0,1]$. Then, the higher order Navier problem \eqref{Navier} possesses at least one positive solution $u\in C^{2s}(\Omega)$ such that $(-\Delta)^{i}u\in C(\overline{\Omega})$ ($i=0,\cdots,s-1$). Moreover, the positive solution $u$ satisfies
\begin{equation}\label{lower-bound-n}
  \|u\|_{L^{\infty}(\overline{\Omega})}\geq\left(\frac{\sqrt{2n}}{diam\,\Omega}\right)^{\frac{2s}{p-1}}.
\end{equation}
\end{thm}

\begin{rem}\label{rem9}
The lower bounds \eqref{lower-bound-d} and \eqref{lower-bound-n} on the $L^{\infty}$ norm of positive solutions $u$ indicate that, if $diam\,\Omega<\frac{1}{\sqrt{C_{n,s}}}$ when $s\in(0,1)$ and $diam\,\Omega<\sqrt{2n}$ when $s=1$, then a uniform priori estimate does not exist and blow-up may occur when $p\rightarrow1+$.
\end{rem}

\smallskip

As consequence of Theorems \ref{exisd} and \ref{exisn}, we have the following corollary.
\begin{cor}\label{cor-exis}
Assume $n\geq\max\{2s,1\}$, $0<s\leq1$ or $s\in\mathbb{Z}^{+}$, $1<p<\frac{n+2s}{n-2s}$ if $n>2s$ and $1<p<+\infty$ if $n=2s$. Suppose $\Omega$ is a bounded BCB domain with $\partial\Omega$ satisfying \eqref{pde-boundary} if $s\in(0,1)$ and \eqref{pde-boundary2} if $s=1$ or $s\in\mathbb{Z}^{+}$, or the uniform exterior cone condition, or Lipschitz continuity. Then the existence of solutions for Dirichlet problem \eqref{Dirichlet} and Navier problem \eqref{Navier} in Theorems \ref{exisd} and \ref{exisn} hold true.
\end{cor}

\smallskip

\begin{rem}\label{rem15}
For other existence results and related properties of solutions to fractional, second and higher order H\'{e}non-Hardy equations on bounded smooth domains, refer to \cite{CPY,CDQ,CLM,DQ3,DQ0,GGN,N} and the references therein.
\end{rem}

\medskip

Intuitively, for a variety of bounded domains $\Omega$, the limiting domain of the scaling domain $\Omega_{\rho}$ could probably be a MSS applicable domain. We may generalize the definition of BCB domains and define \emph{domains with blowing-up MSS applicable domain boundary}.
\begin{defn}[Domains with blowing-up MSS applicable domain boundary]\label{BMSSAB}
A domain $\Omega$ is called \emph{a domain with blowing-up MSS applicable domain boundary} $\partial\Omega$ provided that for any $x(\rho): \, (0,+\infty)\to\partial\Omega$ such that $x(\rho)\to x_{0}\in\partial\Omega$, $\Omega_{\rho}:=\{x\mid \rho x+x(\rho)\in\Omega\}\to \mathcal{M}$, as $\rho\to0$ along some subsequence of any given infinitesimal sequence $\{\rho_{k}\}$, where $\mathcal{M}$ is an \emph{arbitrary MSS applicable domain} that may depends on the subsequence.
\end{defn}

For $0<s\leq1$, by applying blowing-up analysis on domains with blowing-up MSS applicable domain boundary instead of BCB domains and Liouville theorems for MSS applicable domains in subsections 1.2-1.3, the a priori estimates and existence of solutions in Theorems \ref{prid}, \ref{prid2} and \ref{exisd} and Corollaries \ref{cor-prid}, \ref{cor-prid0} and \ref{cor-exis} can also be established on domains with blowing-up MSS applicable domain boundary.
\begin{thm}\label{pri-exi}
Assume $0<s\leq1$, $n\geq\max\{2s,1\}$ and $\Omega$ is a domain with blowing-up MSS applicable domain boundary $\partial\Omega$ such that the limiting domain $\mathcal{M}$ in Definition \ref{BMSSAB} satisfies $\mathbb{R}^{n}\setminus\overline{\mathcal{M}}$ is a g-radially convex domain. Suppose the boundary H\"{o}lder estimates for $(-\Delta)^{s}$ holds on the scaling domain $\Omega_{\rho}$ uniformly w.r.t. the scale $\rho\in(0,1]$, then the a priori estimates and existence of solutions in Theorems \ref{prid} and \ref{exisd} and Corollary \ref{cor-prid} hold true. Suppose the boundary H\"{o}lder estimates for $L$ holds on $\Omega_{\rho}$ uniformly w.r.t. $\rho\in(0,1]$, then the a priori estimates of solutions in Theorem \ref{prid2} holds true. Consequently, suppose $\partial\Omega$ satisfies \eqref{pde-boundary} if $s\in(0,1)$ and \eqref{pde-boundary2} if $s=1$, or the uniform exterior cone condition, or Lipschitz continuity, then the a priori estimates and existence of solutions in Theorem \ref{prid} and Corollaries \ref{cor-prid} and \ref{cor-exis} hold true. Suppose $\partial\Omega$ satisfies the uniform exterior cone condition or Lipschitz continuity, then the a priori estimates in Theorem \ref{prid2} holds true.
\end{thm}

The results in our paper clearly reveal that \emph{the existence and nonexistence of solutions mainly rely on topology (not smoothness) of the domain}. We believe that \emph{the blowing-up analysis on BCB domains} (or even domains with blowing-up MSS applicable domain boundary) can be applied to prove the a priori estimates and hence existence of solutions for various problems (PDEs/IEs/Systems) including \emph{free boundary problems}.

\medskip

The rest of this paper is organized as follows. In Section 2, we introduce the \emph{direct method of scaling spheres} on MSS applicable domains and apply it to prove Liouville theorems (Theorems \ref{pde-unbd}, \ref{pde-unbd-a}, \ref{pde-unbd-2}--\ref{pde-bd-2}) for fractional and second order PDEs \emph{absolutely without any help} of the \emph{integral representation formulae} of solutions. In Section 3, we introduce the \emph{method of scaling spheres in integral forms} and \emph{the method of scaling spheres in a local way} on MSS applicable domains and apply it to prove Liouville theorems for $\leq n$-th order IEs, and then derive Liouville theorems for $\leq n$-th order PDEs by establishing the equivalence between PDEs and IEs (Theorems \ref{ie-unbd}--\ref{pG}, \ref{equi}--\ref{NPDE} and Corollaries \ref{HHcor}, \ref{cor-g} and \ref{cone-cor}). In Section 4, as applications of Liouville type results (Theorems \ref{pde-unbd}, \ref{pde-unbd-a}, \ref{pde-unbd-2}--\ref{pde-bd-2}, \ref{ie-unbd}--\ref{pG}, \ref{equi}--\ref{NPDE} and Corollaries \ref{HHcor}, \ref{cor-g} and \ref{cone-cor}), we derive the a priori estimates of nonnegative solutions for $\leq n$-th order Dirichlet or Navier problems in BCB domains. Section 5 is devoted to establishing the existence of positive solutions for Dirichlet or Navier problems of $\leq n$-th order Lane-Emden equations in BCB domains, by using the a priori estimates derived in Section 4 and the Leray-Schauder fixed point theorem.

\medskip

In the following, we will use $C$ to denote a general positive constant that may depend on $n$, $s$, $a$, $p$, $\Omega$ and $u$, and whose value may differ from line to line.

\section{The direct method of scaling spheres on MSS applicable domains and proofs of Theorems \ref{pde-unbd}, \ref{pde-unbd-a}, \ref{pde-unbd-2}--\ref{pde-bd-2}}

\begin{proof}[Proof of Theorems \ref{pde-unbd} and \ref{pde-unbd-a}]

We prove Theorems \ref{pde-unbd} and \ref{pde-unbd-a} together. Let $\Omega\subset \mathbb R^n$ be a MSS applicable domain such that $\mathbb R^n\setminus \overline\Omega$ is g-radially convex with $P$ as the radially convex center. For any $d_{\Omega}:=\min\limits_{x\in\Omega}|x-P|<r<+\infty$, the exterior fan-shaped domain $\mathcal{C}^{r,e}_{P,\Sigma^{r}_{\Omega}}:=\mathcal{C}_{P,\Sigma^{r}_{\Omega}}\setminus\overline{\mathcal{C}^{r}_{P,\Sigma^{r}_{\Omega}}}\subset\Omega$ (see Remark \ref{rem1}).

\medskip
	
If $f(x,u)\geq0$ on $\Omega\times[0,+\infty)$, $u(x)\geq0$ but $u\not\equiv0$ is a solution to the Dirichlet problem of PDEs \eqref{PDE} in $\Omega$, then it follows from maximum principle that $u>0$ in $\Omega$. \emph{Without any help} of the \emph{integral representation formulae} of solutions, we will apply the \emph{direct method of scaling spheres} to show the following lower bound estimates for positive solution $u$, which contradicts the integrability of $u$ provided that $f$ is subcritical in the sense of Definition \ref{defn1}.
\begin{thm}\label{lower0}
Assume $\mathbb{R}^{n}\setminus\overline{\Omega}$ is a g-radially convex domain with radially convex center $P$, $n\geq\max\{2s,1\}$, $0<s\leq1$, $f\geq0$ satisfies $(\mathbf{f_{1}})$ and $n\neq2$ if $s=1$ and $\Omega=\mathbb{R}^{n}$, $f$ is subcritical in the sense of Definition \ref{defn1} and satisfies $(\mathbf{f_{3}})$. Suppose $f$ satisfies $(\mathbf{f_{2}})$ with $q=+\infty$ when $s\in(0,1)$, and with $q=\frac{n}{2}$ if $n\geq3$ (with $q=1+\delta$ if $n=2$, where $\delta>0$ is arbitrarily small) provided that $P\in\mathbb{R}^{n}\setminus\overline{\Omega}$, $\partial\Omega$ is locally Lipschitz and $f$ satisfies $(\mathbf{A})$, or else with $q=n$ when $s=1$. If $u$ is a positive classical solution to the Dirichlet problem \eqref{PDE}, then it satisfies the following lower bound estimates: for $|x-P|\geq R_{0}+\frac{4\sigma_{2}}{\sigma_{2}-\sigma_{1}}r+1$ with $x\in\mathcal{C}_{Q,\Sigma_{2}}$,
\begin{equation}\label{lb2}
  u(x)\geq C_{\kappa}|x-P|^{\kappa}, \quad\quad \forall \, \kappa<+\infty,
\end{equation}
where $r>d_{\Omega}$ and $\Sigma\subset\Sigma^{r}_{\Omega}$ are given in $(\mathbf{f_{3}})$, $\Sigma_{2}$ is any cross-section such that $\overline{\Sigma_{2}}\subsetneq\Sigma_{1}$, $R_{0}>0$ and $0<\sigma_{1}<\sigma_{2}$ are given by assumption $(\mathbf{H_2})$ which, as proved in Theorem \ref{pG}, is satisfied by Green's function $G^{s}_{\mathcal{C}}$ for Dirichlet problem of $(-\Delta)^{s}$ on $\mathcal{C}_{Q,\Sigma_{1}}$, $\Sigma_{1}$ is any smooth cross-section such that $\overline{\Sigma_{1}}\subsetneq\Sigma$ and $Q$ is some point in $\mathcal{C}^{r,e}_{P,\Sigma}$ such that $\mathcal{C}_{Q,\Sigma_{1}}\subset\mathcal{C}_{P,\Sigma}\bigcap\Omega$.
\end{thm}
\begin{proof}
Without loss of generalities, we may assume that $P=0$. For $\lambda>0$ and $x\in \mathbb R^n\setminus\{0\}$, let $x^\lambda=\frac{\lambda^{2}x}{|x|^{2}}$ be the reflection of $x$ with respect to the sphere $S_\lambda:=\{x\mid\,|x|=\lambda\}$. From Definitions \ref{g-rcd} and \ref{mss-a}, one can infer that $\Omega\subset \mathbb R^n$ is a MSS applicable domain with center $0$ such that $\mathbb R^n\setminus \overline\Omega$ is g-radially convex implies that $x^\lambda\in \Omega$ for any $x\in \Omega\bigcap\left(B_\lambda(0)\setminus\{0\}\right)$, that is, $\left(\Omega\bigcap\left(B_\lambda(0)\setminus\{0\}\right)\right)^{\lambda}\subseteq\Omega\setminus\overline{B_{\lambda}(0)}$, where $\left(\Omega\bigcap\left(B_\lambda(0)\setminus\{0\}\right)\right)^{\lambda}$ denotes the reflection of $\Omega\bigcap\left(B_\lambda(0)\setminus\{0\}\right)$ w.r.t. the sphere $S_{\lambda}$.

\medskip

For any $\lambda>d_{\Omega}$, we define the Kelvin type transform of $u$ centered at $0$ by
	\begin{equation}\label{pde-2}
		u_{\lambda}(x)=\left(\frac{\lambda}{|x|}\right)^{n-2s}u\left(\frac{\lambda^{2}x}{|x|^{2}}\right).
	\end{equation}

\medskip
	
We will carry out the process of scaling spheres in $\Omega$ with respect to the center $P=0\in\mathbb{R}^{n}$.
	
\medskip

	To this end, let $\omega^{\lambda}(x):=u_{\lambda}(x)-u(x)$ for $\lambda>d_{\Omega}$. By the definition of $u_{\lambda}$ and $\omega^{\lambda}$, we have
	\begin{equation}\label{pde-3}
		\omega^{\lambda}(x)=u_{\lambda}(x)-u(x)=-\left(\frac{\lambda}{|x|}\right)^{n-2s}\omega^{\lambda}(x^{\lambda})=-\big(\omega^{\lambda}\big)_{\lambda}(x),
	\end{equation}
	which indicates that $\omega^{\lambda}$ is an anti-symmetric function with respect to the sphere $S_\lambda$.
	
\medskip

	We will first show that, for $\lambda>d_{\Omega}$ sufficiently close to $d_{\Omega}$,
	\begin{equation}\label{pde-4}
		\omega^{\lambda}(x)\geq0, \qquad \forall \,\, x\in\Omega\bigcap\left(B_\lambda(0)\setminus\{0\}\right).
	\end{equation}
	Then, we start dilating the sphere piece $S_{\lambda}\bigcap\Omega$ from near $\lambda=d_{\Omega}$ toward the interior of $\Omega$ as long as \eqref{pde-4} holds, until its limiting position $\lambda=+\infty$ and then derive lower bound estimates on asymptotic behavior of $u$ as $|x|\rightarrow+\infty$ in a smaller cone via Bootstrap technique. Therefore, the direct method of scaling spheres can be divided into three steps.
	
	\medskip
	
	\emph{Step 1. Start dilating the sphere piece $S_{\lambda}\bigcap\Omega$ from near $\lambda=d_{\Omega}$.} Define
	\begin{equation}\label{pde-5}
		\Omega_\lambda ^{-}:=\{x\in \Omega\bigcap\left(B_\lambda(0)\setminus\{0\}\right) \, | \, \omega^{\lambda}(x)<0\}.
	\end{equation}
	We will show that, there exists an $\eps>0$ sufficiently small such that, for any $\lambda\in(d_{\Omega},d_{\Omega}+\eps)$,
	\begin{equation}\label{pde-6}
			\Omega_\lambda ^{-}=\emptyset.
	\end{equation}

\medskip
	
	For arbitrary $x\in \Omega\bigcap\left(B_\lambda(0)\setminus\{0\}\right)$, we infer from \eqref{PDE} and subcritical condition on $f$ that
	\begin{eqnarray}\label{pde-7}
		&& (-\Delta)^{s}\omega^{\lambda}(x)=\left(\frac{\lambda}{|x|}\right)^{n+2s}f(x^\lambda, u(x^\lambda))-f(x,u(x)) \\
\nonumber &&\qquad\qquad\quad\quad >f(x, u_\lambda(x))-f(x, u(x)) = c_\lambda(x)\omega^{\lambda}(x),
	\end{eqnarray}
	where $c_\lambda(x):=\frac{f(x, u_\lambda(x))- f(x, u(x))}{u_\lambda(x)-u(x)}$. Since $0<u_\lambda(x)<u(x)\leq M_\lambda $ for any $x\in 	\Omega_\lambda ^{-}$ with $M_{\lambda}:=\sup\limits_{\Omega\cap B_{\lambda}(0)}u<+\infty$, we derive from the assumption $(\mathbf{f_{2}})$ that
	\begin{equation}\label{pde-8}
		\left|c_\lambda(x)\right|=\left|\frac{f(x, u_\lambda(x))- f(x, u(x))}{u_\lambda(x)-u(x)}\right|\leq \sup_{\alpha,\beta\in [0, M_\lambda] } \frac{|f(x, \alpha)- f(x, \beta)|}{|\alpha-\beta|}\leq h_\lambda(x),
	\end{equation}
where $h_\lambda$ can be chosen increasingly depending on $\lambda$ since $M_\lambda$ is increasing as $\lambda$ increases.
	
\medskip

Next we carry out our proof by discussing two different cases.

\medskip

\emph{Case (i).} $f$ satisfies $(\mathbf{f_{2}})$ with $q=+\infty$ when $s\in(0,1)$ and with $q=n$ when $s=1$. In such case, we have $h_\lambda \in L^q (\Omega\cap B_R(0))$ ($\forall\, R>0$) with $q=+\infty$ if $s\in(0,1)$ and with $q=n$ if $s=1$. We first establish the following \emph{small region principle} by using the ABP estimate. The smallness of the region and the upper bound of positive part of the coefficient $c(x)$ are characterized by an integral condition. When $s\in(0,1)$, the narrow region principle in Theorem 2.2 of \cite{CLZ} can be deduced from the following theorem as a corollary.
\begin{thm}\label{NP-1}(Small region principle)
	Assume  $0<s\leq1$, $\lambda>0$ and $D\subset B_\lambda(0)\setminus\{0\}$. Suppose that $\omega \in\mathcal{L}_{s}(\mathbb{R}^{n})\cap C^{[2s],\{2s\}+\eps}_{loc}(D)$ with arbitrarily small $\eps>0$ if $0<s<1$ and $\omega \in C^{2}(D)$ if $s=1$. There is a constant $C_{n,s}$ depending merely on $n,s$ such that, if $\omega$ is lower semi-continuous on $\overline{D}$ and satisfies
		\begin{equation}\label{pde-9}\\\begin{cases}
				(-\Delta)^{s}\omega (x)-c(x)\omega (x)\geq0, \,\,\,\,\, &x\in D,\\
				\omega \geq 0 &\text{a.e. in}\,  \, B_{\lambda}(0)\setminus \overline{D},\\
				\omega =-\omega_{\lambda} & \text{a.e. in}\,  \, B_{\lambda}(0)
		\end{cases}\end{equation}
		for some function $c$ satisfying $ \| c^+\|_{L^q(D)}\leq C_{n,s}(\text{diam}(D))^{-s}| D|^{\frac{s}{q}-\frac{s}{n}}$ with $q=+\infty$ if $s\in(0,1)$ and $q=n$ if $s=1$, and $\liminf\limits_{x\in D,\,x\rightarrow\partial D}\omega(x)\geq0$ when $s=1$, then $$\omega(x)\geq 0, \quad  \  \ \forall\, x\in D.$$
		
		In particular, suppose that $c$ is bounded from above by a constant $M\geq0$ in $D$, then
 $$|D|\leq \begin{cases}
				\left[\frac{C_{n,s}}{M (2\lambda)^s}\right]^{\frac{n}{s}},\,\,\,\,\, \text{if} \,\, M>0,\\
				+\infty, \quad \text{if} \,\, M=0
		\end{cases}  \  \   \ \text{implies}\  \ \   \omega(x)\geq 0, \  \  \ \forall\, x\in D.$$
	\end{thm}
	\begin{proof}
		Set $D^-:=\{x\in D\mid\, \omega<0\}$. For  $s=1$, the spherical anti-symmetry of $\omega$ (i.e. the third equation in \eqref{pde-9}) is  redundant. Indeed, by the first two equations in \eqref{pde-9} we find that $\omega$ solves
		\begin{equation*}\begin{cases}
				\Delta(- \omega )(x)\geq  c^+(x) \omega(x),\,\,\,\,\, x\in D^-,\\
				\omega(x)= 0,\quad x\in \partial D^-,
		\end{cases}\end{equation*}
		where $c^+(x):=\max\{0, c(x)\}$.

\medskip

The well-known 2nd order Alexandrov-Bakelman-Pucci maximum principle (see e.g. Theorem 9.1 in \cite{GT}) yields that 	
		\begin{equation}\label{pde-10}
			\sup_{D^-} (-\omega) \leq\hat C_n \text{diam}(D^{-})\|  c^+\omega \|_{L^{n}\left(D^-\right)}
			\leq  \hat C_n \left(\sup_{D^-} (-\omega)\right)\text{diam}(D)\|  c^+  \|_{L^{n}\left(D^-\right)},
		\end{equation}
		where $\hat C_n$ is a constant depending merely on $n$. Then we deduce from the above inequality that $|D^-|=0$ once $$\|  c^+  \|_{L^{n}\left(D^-\right)} \leq (2\hat C_n)^{-1}\text{diam}(D)^{-1},$$
which proves Theorem \ref{NP-1} in the case $s=1$.

\medskip
		
		Now we consider more delicate non-local cases $s\in(0,1)$, where the spherical anti-symmetry property of $\omega$ seems necessary. We split $\omega$ into two parts $\omega=\omega_1+\omega_2$, where
$$\omega_1(x)=\begin{cases} \omega(x),\ \ \ &\text{if}\, x\in D^-,\\ 0, &\text{if}\, x\in\mathbb R^n\setminus D^-\end{cases} \qquad  \text{and} \qquad  \omega_2(x)=\begin{cases} 0, \qquad &\text{if}\, x\in D^-,\\ \omega(x), \qquad &\text{if}\, x\in\mathbb R^n\setminus D^-\end{cases}.$$
The key step is to show that
\begin{equation}\label{pde-11}
			(-\Delta)^s \omega_2\leq 0  \qquad  \text{on}\,\, D^-.
\end{equation}
		Using the third equation in \eqref{pde-9}, by direct calculations, we get that, for any $x\in D^-$,
		\begin{align*}
			(-\Delta)^s \omega_2(x)&=C_{s,n}P.V. \int_{\mathbb R^n} \frac{0-\omega_2(y)}{|x-y|^{n+2s}}   \mathrm{d}y\\
			&=C_{s,n}P.V. \int_{B_\lambda(0)\setminus D^-} \frac{ -\omega(y)}{|x-y|^{n+2s}}   \mathrm{d}y+C_{s,n}P.V. \int_{\mathbb R^n\setminus  B_\lambda(0)} \frac{ -\omega(y)}{|x-y|^{n+2s}}   \mathrm{d}y\\
			&=C_{s,n}P.V. \int_{B_\lambda(0)\setminus D^-} \frac{ -\omega (y)}{|x-y|^{n+2s}}   \mathrm{d}y+C_{s,n}P.V. \int_{B_\lambda(0)} \frac{ -\omega (y^\lambda)}{|x-y^\lambda|^{n+2s}}   \left(\frac{\lambda}{|y|}\right)^{2n}  \mathrm{d}y\\
			&=C_{s,n}P.V. \int_{B_\lambda(0)\setminus D^-} \frac{ -\omega(y)}{|x-y|^{n+2s}}   \mathrm{d}y+C_{s,n}P.V. \int_{B_\lambda(0)} \frac{ \omega(y)}{|x-y^\lambda|^{n+2s}}   \left(\frac{\lambda}{|y|}\right)^{n+2s}  \mathrm{d}y	\\
			&=C_{s,n}P.V. \int_{B_\lambda(0)\setminus D^-} \left( \frac{ 1}{|x-y^\lambda|^{n+2s}}   \left(\frac{\lambda}{|y|}\right)^{n+2s} -\frac{ 1}{|x-y|^{n+2s}}  \right) \omega(y) \mathrm{d}y\\
			&\quad +C_{s,n}P.V. \int_{D^-} \frac{ \omega(y)}{|x-y^\lambda|^{n+2s}}   \left(\frac{\lambda}{|y|}\right)^{n+2s}  \mathrm{d}y \leq 0.
		\end{align*}
		Here we have used the point-wise inequality
$$|x-y^\lambda|\cdot|y|\geq \lambda|x-y|,\  \ \quad  \forall\, x,y \in B_\lambda(0).$$
		Hence we obtain \eqref{pde-11} and consequently we find that $\omega_1$ satisfies
		\begin{equation*}
			\begin{cases}
				-(-\Delta)^{s} \omega_1  (x)\leq -c^+(x) \omega_1  (x) \,\,\,\,\, \text{in} \,\,\, D^-,\\
				-\omega_1(x)=0,\quad x\in \mathbb R^n\setminus D^-.\\
			\end{cases}
		\end{equation*}

\medskip

Now we need the Alexandrov-Bakelman-Pucci type estimate for fractional Laplacians (see Corollary 2.1 in \cite{FW}).
\begin{thm}[Corollary 2.1 in \cite{FW}]\label{ABP}
Let $\Omega$ be a bounded open subset of $\mathbb{R}^{n}$. Suppose that $h:\,\Omega\to\mathbb{R}$ is in $L^{\infty}(\Omega)$ and $w\in L^{\infty}(\mathbb{R}^{n})$ is a classical solution of
\begin{equation*}
			\begin{cases}
				-(-\Delta)^{s}w(x)\leq h(x), \qquad\,\, x\in\Omega,\\
				w(x)\geq0,\qquad x\in \mathbb R^n\setminus\Omega.\\
			\end{cases}
		\end{equation*}
Then there exists a positive constant $C$, depending on $n$ and $s$, such that
\begin{equation*}
  -\inf_{\Omega}w\leq Cd^{s}\|h^{+}\|_{L^{\infty}(\Omega)}|\Omega|^{\frac{s}{n}},
\end{equation*}
where $d=diam(\Omega)$ is the diameter of $\Omega$ and $h^{+}(x):=\max\{h(x),0\}$.
\end{thm}

Theorem \ref{ABP} implies that
		\begin{equation}\label{pde-12}
			\sup_{D^-} (-\omega_1) \leq C_{n,s}\left(\sup_{ D} c^+\right)\left(\sup_{D^-}(-\omega_1)\right)\text{diam}(D^-)^s |D^-|^{\frac{s}{n}},
		\end{equation}
		where $C_{n,s}$ is a constant depending on $n,s$. Then one can deduce from \eqref{pde-12} that $|D^-|=0$ once $$\sup_{ D} c^+ \leq (2 C_{n,s})^{-1}\text{diam}(D)^{-s} |D|^{-\frac{s}{n}},$$
which proves Theorem \ref{NP-1} for the cases $s\in(0,1)$.	
	\end{proof}
	
\medskip

\emph{Case (ii).} $s=1$, $P\in\mathbb{R}^{n}\setminus\overline{\Omega}$ (i.e., $d_\Omega>0$), $\partial\Omega$ is locally Lipschitz, $f$ satisfies $(\mathbf{A})$ and $(\mathbf{f_{2}})$ with $q=\frac{n}{2}$ if $n\geq3$ and with $q=1+\delta$ if $n=2$, where $\delta>0$ is arbitrarily small. First, from $P=0\in\mathbb{R}^{n}\setminus\overline{\Omega}$, one gets $u_{\lambda}(x)\leq \left(\frac{\lambda}{d_{\Omega}}\right)^{n-2s}\sup\limits_{\Omega\cap B_{\lambda^{2}/d_{\Omega}}(0)}u=:\widetilde{M}_{\lambda}<+\infty$ for any $x\in\Omega\bigcap B_{\lambda}(0)$, and hence
\begin{equation}\label{pde-36}
  \left|c_\lambda(x)\right|=\left|\frac{f(x, u_\lambda(x))- f(x, u(x))}{u_\lambda(x)-u(x)}\right|\leq \sup_{\alpha,\beta\in [0, \widetilde{M}_{\lambda}] } \frac{|f(x, \alpha)- f(x, \beta)|}{|\alpha-\beta|}\leq \widetilde{h}_\lambda(x),
\end{equation}
where $\widetilde{h}_\lambda$ could be chosen larger than $h_{\lambda}$ and increasingly depending on $\lambda$. Moreover, one can infer that $\widetilde{h}_\lambda \in L^q (\Omega\cap B_R(0))$ ($\forall\, R>0$) with $q=\frac{n}{2}$ if $n\geq3$ and $q=1+\delta$ with arbitrarily small $\delta>0$ if $n=2$. Second, by \eqref{PDE} and assumption $\mathbf{(A)}$ on $f$, one can infer that $\omega^{\lambda}=u_{\lambda}-u\in H^{1}(\Omega\bigcap B_{\lambda}(0))$. Indeed, for any $\phi\in C_{c}^{\infty}(\mathbb{R}^{n})$, by multiplying equation \eqref{PDE} by test function $\phi^{2}u$ and integrating, one has
\begin{eqnarray*}
  && \int_{\Omega\cap supp(\phi)}\phi^{2}|\nabla u|^{2}+2(\phi\nabla u)\cdot(u\nabla\phi)\mathrm{d} x=\int_{\Omega\cap supp(\phi)}(-\Delta u)(\phi^{2}u)\mathrm{d} x \\
 \nonumber && \qquad\qquad\qquad\qquad\qquad\qquad\qquad\qquad\quad=\int_{\Omega\cap supp(\phi)}f(x,u(x))\phi^{2}(x)u(x)\mathrm{d} x,
\end{eqnarray*}
combining this with assumption $\mathbf{(A)}$ on $f$ and H\"{o}lder inequality implies that
\begin{equation*}
  \int_{\Omega\cap supp(\phi)}\phi^{2}|\nabla u|^{2}\mathrm{d}x\leq 2\left|\int_{\Omega\cap supp(\phi)}f(x,u(x))\phi^{2}(x)u(x)\mathrm{d} x\right|+16\int_{\Omega\cap supp(\phi)}u^{2}|\nabla\phi|^{2}\mathrm{d}x<+\infty.
\end{equation*}
Thus $u\in H^{1}(\Omega\bigcap B_R(0))$ ($\forall\, R>0$). It follows from $P=0\in\mathbb{R}^{n}\setminus\overline{\Omega}$ that $0<d_{\Omega}\leq|x|<\lambda$ for any $x\in\Omega\bigcap B_{\lambda}(0)$, and hence $u_{\lambda}\in H^{1}(\Omega\bigcap B_{\lambda}(0))$. Therefore, we arrive at $\omega^{\lambda}=u_{\lambda}-u\in H^{1}(\Omega\bigcap B_{\lambda}(0))$.

\medskip

We can use the following maximum principle instead of the small region principle in Theorem \ref{NP-1} to deal with Case (ii).
\begin{thm}[Theorem 1.2 in \cite{LL}]\label{mp}
Assume that $c\in L^{\frac{n}{2}}(\Omega)$ if $n\geq3$ and $c\in L^{1+\delta}(\Omega)$ with arbitrarily small $\delta>0$ if $n=2$, and $u\in H^{1}(\Omega)$ is a weak solution of
\begin{equation*}
			\begin{cases}
				-\Delta u(x)+c(x)u(x)\geq 0 \qquad\,\, \text{in} \,\, \Omega,\\
				u(x)\geq0 \qquad \text{on} \,\, \partial\Omega.\\
			\end{cases}
		\end{equation*}
There exists a positive constant $c_{n}$ such that, if $\|c^{-}\|_{L^{\frac{n}{2}}(\Omega)}\leq c_{n}$ ($\|c^{-}\|_{L^{1+\delta}(\Omega)}\leq c_{n}$), then $u\geq 0$ in $\Omega$, where $c^{-}(x):=-\min\{c(x),0\}$.
\end{thm}

\begin{rem}\label{rem12}
Although Theorem 1.2 in \cite{LL} was proved for the unit ball $B_1$ and $n\geq3$, it is clear from the proof that the conclusion also holds true for arbitrary domain $\Omega$ such that the formula of integration by parts holds (which can be guaranteed by locally Lipschitz property of $\partial\Omega$), and for $n=2$ if $c\in L^{1+\delta}(\Omega)$ with arbitrarily small $\delta>0$. When $n=2$, Theorem \ref{mp} can also be deduced from Theorem 8.16 in \cite{GT}.
\end{rem}

\medskip

Since $\omega^\lambda$ solves \eqref{pde-7} on $\Omega\bigcap\left(B_{\lambda}(0)\setminus\{0\}\right)$ with $c_\lambda$ satisfying \eqref{pde-8} on $\Omega_\lambda^-$ and \eqref{pde-36} on $\Omega\bigcap B_{\lambda}(0)$ if $P=0\in\mathbb{R}^{n}\setminus\overline{\Omega}$, by applying the small region principle in Theorem \ref{NP-1} on $\Omega_{\lambda}^{-}$ to $\omega^\lambda$ for Case (i) and Theorem \ref{mp} on $\Omega\bigcap B_{\lambda}(0)$ to $\omega^\lambda$ for Case (ii), we can choose an $\eps>0$ small enough and derive that
$$\Omega_\lambda^-=\emptyset, \qquad \forall \,\,\lambda\in(d_{\Omega},d_{\Omega}+\eps).$$
Note that, when $s=1$, $n\geq3$ and $\Omega=\mathbb{R}^{n}$, the condition $\liminf\limits_{x\in D,\,x\rightarrow\partial D}\omega(x)\geq0$ in Theorem \ref{NP-1} can be derived from $f\geq0$ and $\mathbf{(f_{3})}$ via entirely similar way as (2.24) in \cite{DQ0} by showing a lower bound on positive solution $u$ as fundamental solution (see (2.16) in \cite{DQ0}). Thus we obtain \eqref{pde-6} for any $\lambda\in(d_{\Omega},d_{\Omega}+\eps)$ and finish the proof of step 1.

\medskip	
	
\emph{Step 2. Dilate the sphere piece $S_{\lambda}\bigcap\Omega$ toward the interior of $\Omega$ until $\lambda=+\infty$.} Step 1 provides us a starting point to dilate the sphere piece $S_{\lambda}\bigcap\Omega$ from near $\lambda=d_\Omega$. Now we dilate the sphere piece $S_{\lambda}\bigcap\Omega$ toward the interior of $\Omega$ as long as \eqref{pde-6} holds. Let
	\begin{equation}\label{pde-13}
		\lambda_{0}:=\sup\{\lambda>d_{\Omega}\,|\, \omega^{\mu}\geq0 \,\, \text{in} \,\,    \Omega\bigcap\left(B_\mu(0)\setminus\{0\}\right), \,\, \forall \, d_{\Omega}<\mu\leq\lambda\}\in (d_{\Omega}, +\infty],
	\end{equation}
	it follows that
	\begin{equation}\label{pde-14}
		\omega^{\lambda_{0}}(x)\geq0, \quad\quad \forall \,\, x\in   \Omega\bigcap\left(B_{\lambda_0}(0)\setminus\{0\}\right).
	\end{equation}
	In what follows, we will prove $\lambda_{0}=+\infty$ by contradiction arguments.

\medskip
	
	Suppose on the contrary that $d_{\Omega}<\lambda_{0}<+\infty$. We will obtain a contradiction via showing that the sphere piece $S_{\lambda}\bigcap\Omega$ can be dilated toward the interior of $\Omega$ a little bit further, more precisely, there exists an $\varepsilon>0$ small enough such that $\omega^{\lambda}\geq0$ in $\Omega\bigcap\left(B_{\lambda}(0)\setminus\{0\}\right)$ for all $\lambda\in[\lambda_{0},\lambda_{0}+\varepsilon]$. By \eqref{pde-7}  and \eqref{pde-14}, we have
	\begin{equation}\label{pde-15}\\\begin{cases}
			(-\Delta)^{s}\omega^{\lambda_{0}}(x)>f(x,u_{\lambda_{0}}(x))-f(x,u(x))=c_{\lambda_{0}}(x)\omega^{\lambda_{0}}(x) \,\,\,\,\, \text{in} \,\,\, \Omega\bigcap\left(B_{\lambda_0}(0)\setminus\{0\}\right),\\
			\omega^{\lambda_{0}}(x)\geq 0,\quad x\in B_{\lambda_{0}}(0)\setminus\{0\},\\
			\omega^{\lambda_{0}}(x)=-(\omega^{\lambda_{0}})_{{\lambda_{0}}}(x),\quad   x\in B_{{\lambda_{0}}}(0)\setminus \{0\}.
	\end{cases}\end{equation}
We will show that
	\begin{equation}\label{pde-16}
		\omega^{\lambda_{0}}(x)>0, \,\qquad \forall \,\, x\in\Omega\bigcap\left(B_{\lambda_0}(0)\setminus\{0\}\right).
	\end{equation}
Suppose on the contrary that there exists $x_{0}\in\Omega\cap B_{\lambda_0}(0)$ such that $\omega^{\lambda_{0}}(x_{0})=\inf\limits_{B_{\lambda_0}(0)}\omega^{\lambda_{0}}(x)=0$, then from the first inequality in \eqref{pde-15} we obtain $(-\Delta)^{s}\omega^{\lambda_{0}}(x_{0})>0$, which will lead to a contradiction.

\medskip

When $s=1$, $\omega^{\lambda_{0}}(x_{0})=\inf\limits_{B_{\lambda_0}(0)}\omega^{\lambda_{0}}(x)=0$ implies $-\Delta\omega^{\lambda_{0}}(x_{0})\leq0$, which contradicts $-\Delta\omega^{\lambda_{0}}(x_{0})>0$.

\medskip

For $s\in(0,1)$, we can deduce from $\omega^{\lambda_{0}}(x_{0})=\inf\limits_{B_{\lambda_0}(0)}\omega^{\lambda_{0}}(x)=0$ and the spherical anti-symmetry of $\omega_{\lambda_{0}}$ that
$$0<(-\Delta)^{s}\omega^{\lambda_{0}}(x_{0})=C_{s,n}P.V. \int_{B_{\lambda_{0}}(0)} \left( \frac{1}{\left|\frac{|y|x_{0}}{\lambda_{0}}-\frac{\lambda_{0}y}{|y|}\right|^{n+2s}}-\frac{1}{|x_{0}-y|^{n+2s}}\right) \omega(y)\mathrm{d}y\leq0,$$
which is absurd. Therefore, \eqref{pde-16} holds for all $0<s\leq1$.

\medskip
	
Now we choose a constant $\delta_{0}$ small enough by discussing two different cases.

\medskip

\emph{Case (i).} $f$ satisfies $(\mathbf{f_{2}})$ with $q=+\infty$ when $s\in(0,1)$ and with $q=n$ when $s=1$. In such case, let $h_{2\lambda_0}\in L^q (\Omega\cap B_{2\lambda_0}(0))$ with $q=+\infty$ if $s\in(0,1)$ and $q=n$ if $s=1$ be the function such that $$h_{2\lambda_0}(x)\geq \sup_{\alpha,\beta\in [0, M_{2\lambda_0}] } \frac{|f(x, \alpha)- f(x, \beta)|}{|\alpha-\beta|}, \qquad  \forall\,\, x\in \Omega\bigcap\left(B_{2\lambda_0}(0)\setminus\{0\}\right) $$
	with $M_{2\lambda_0}=\sup\limits_{x\in \Omega\cap B_{2\lambda_0}(0)} u(x)$. We take a small constant $\delta_0>0$ such that $$\|h_{2\lambda_0}\|_{L^n(D)}\leq \frac{C_{n,s}}{4\lambda_0},  \qquad  \forall\, D\subset \Omega\cap B_{2\lambda_0}(0)\,\,\,\text{with}\,\,\, |D|\leq \delta_0, \quad    \  \text{if}\,\, s=1$$
and $$\delta_0=\left( \frac{C_{n,s}}{\left(4 \lambda_0\right)^s\|h_{2\lambda_0}\|_{L^\infty (\Omega\cap B_{2\lambda_0}(0))}} \right)^{\frac{n}{s}}\  \qquad    \  \text{if}\,\, s\in(0,1),$$
where $C_{n,s}$ is the same as in the small region principle in Theorem \ref{NP-1}.

\medskip

\emph{Case (ii).} $s=1$, $P\in\mathbb{R}^{n}\setminus\overline{\Omega}$ (i.e., $d_{\Omega}>0$), $\partial\Omega$ is locally Lipschitz, $f$ satisfies $(\mathbf{A})$ and $(\mathbf{f_{2}})$ with $q=\frac{n}{2}$ if $n\geq3$ and with $q=1+\delta$ if $n=2$, where $\delta>0$ is arbitrarily small. In such case, let $\widetilde{h}_{2\lambda_{0}} \in L^q (\Omega\cap B_{2\lambda_{0}}(0))$ with $q=\frac{n}{2}$ if $n\geq3$ and $q=1+\delta$ with arbitrarily small $\delta>0$ if $n=2$ such that
\begin{equation*}
  \widetilde{h}_{2\lambda_{0}}(x)\geq\sup_{\alpha,\beta\in [0, \widetilde{M}_{2\lambda_{0}}] } \frac{|f(x, \alpha)- f(x, \beta)|}{|\alpha-\beta|}, \qquad  \forall\,\, x\in \Omega\cap B_{2\lambda_0}(0)
\end{equation*}
with $\widetilde{M}_{2\lambda_0}=\left(\frac{2\lambda_{0}}{d_{\Omega}}\right)^{n-2s}\sup\limits_{x\in\Omega\cap B_{4\lambda_{0}^{2}/d_{\Omega}}(0)}u(x)$. We take a small constant $\delta_0>0$ such that $$\|\widetilde{h}_{2\lambda_0}\|_{L^q(D)}\leq c_{n},  \qquad  \forall\, D\subset \Omega\cap B_{2\lambda_0}(0)\,\,\,\text{with}\,\,\, |D|\leq \delta_0,$$
where $c_{n}$ is the same as in the maximum principle in Theorem \ref{mp}.

\medskip
	
	For such $\delta_0>0$,  we can choose a compact set $K\subset \Omega\bigcap\left(B_{\lambda_0}(0)\setminus\{0\}\right)$ (with regular boundary $\partial K$ such that the formula of
integration by parts holds in Case (ii)) such that
	\begin{equation}\label{pde-17}
		\left|\left(\Omega\bigcap\left(B_{\lambda_0}(0)\setminus\{0\}\right)\right)\setminus  K\right| \leq \frac{\delta_0}{2}.
	\end{equation}
	Since $K$ is a compact set, there is a positive constant $c_K>0$ such that
	$$\omega^{\lambda_0}\geq c_K>0 \quad \ \  \text{on}\,\, K.$$  Note that $K$ is far away from the point $0$. By the continuity of $\omega^\lambda$ with respect to $\lambda$, we can find a small constant $\varepsilon_0\in\left(0,\frac{\lambda_0}{2}\right)$ such that
	\begin{equation}\label{pde-19}
		\omega^\lambda\geq \frac{c_K}{2} \quad \  \ \text{on}\,\, K, \qquad  \   \forall\,\lambda\in [\lambda_0, \lambda_0+\varepsilon_0).
	\end{equation}
	
\medskip

	Next we take a small constant $\varepsilon_1\in (0, \varepsilon_0)$ such that
	\begin{equation}\label{pde-20}
		| B_{\lambda_0+\varepsilon_1}(0)\setminus  B_{\lambda_0}(0)|  \leq \frac{\delta_0}{2}.
	\end{equation}
	We will show that
	\begin{equation}\label{pde-22}
		\omega^\lambda\geq 0  \quad \text{in}\,\,  \Omega\bigcap\left(B_{\lambda }(0)\setminus\{0\}\right), \ \qquad \ \forall\, \lambda\in [\lambda_0, \lambda_0+\varepsilon_1),
	\end{equation}

\medskip

	Indeed, suppose on the contrary there is some $\lambda\in [\lambda_0, \lambda_0+\varepsilon_1)$ such that $\Omega_\lambda^-\not=\emptyset$. Then in view of \eqref{pde-19}, we  must have $$\Omega_\lambda^-\subset \left(\Omega\bigcap B_{\lambda }(0)\right)\setminus  K= \left((\Omega\bigcap B_{\lambda_0}(0))\setminus  K \right) \bigcup \left(\Omega\bigcap(B_\lambda(0)\setminus B_{\lambda_0}(0))\right),$$
	which implies that $\Omega_\lambda^-$ is an open set with  measure $|\Omega_\lambda^-|\leq |(\Omega\bigcap B_{\lambda }(0))\setminus  K|\leq\delta_0$.
	Noting that on $\Omega\bigcap\left(B_{\lambda}(0)\setminus\{0\}\right)$, $\omega^{\lambda}$ satisfies
\begin{eqnarray}\label{pde-21}
		&& (-\Delta)^{s}\omega^{\lambda}(x)=\left(\frac{\lambda}{|x|}\right)^{n+2s}f(x^\lambda, u(x^\lambda))-f(x,u(x)) \\
\nonumber &&\qquad\qquad\quad\,\,\, >f(x, u_\lambda(x))-f(x, u(x)) = c_\lambda(x)\omega^{\lambda}(x),
	\end{eqnarray}
	where $|c_\lambda(x)|=\left|\frac{f(x, u_\lambda(x))- f(x, u(x))}{u_\lambda(x)-u(x)}\right|\leq h_{2\lambda_0}(x)$ on $\Omega_{\lambda}^{-}$ or $|c_\lambda(x)|=\left|\frac{f(x, u_\lambda(x))- f(x, u(x))}{u_\lambda(x)-u(x)}\right|\leq \widetilde{h}_{2\lambda_0}(x)$ on $\Omega\bigcap B_{\lambda}(0)$ provided that $P\in\mathbb{R}^{n}\setminus\overline{\Omega}$.

\medskip	
	
By \eqref{pde-21} and the definition of $\delta_0$, we can apply Theorem \ref{NP-1} to $\omega^\lambda$ on $\Omega_{\lambda}^{-}$ in Case (i) and apply Theorem \ref{mp} to $\omega^\lambda$ on $\left(\Omega\bigcap B_{\lambda}(0)\right)\setminus K$ in Case (ii), and obtain $\omega^\lambda\geq 0$ on $\Omega_\lambda^-$ for any $\lambda\in[\lambda_{0},\lambda_{0}+\varepsilon_1)$, which is absurd and hence proves \eqref{pde-22}. Note that, when $s=1$, $n\geq3$ and $\Omega=\mathbb{R}^{n}$, the condition $\liminf\limits_{x\in D,\,x\rightarrow\partial D}\omega(x)\geq0$ in Theorem \ref{NP-1} can be proved in entirely similar way as (2.64) in \cite{DQ0}. However, \eqref{pde-22} contradicts the definition \eqref{pde-13} of $\lambda_{0}$ and hence we must have $\lambda_0=+\infty$.
	
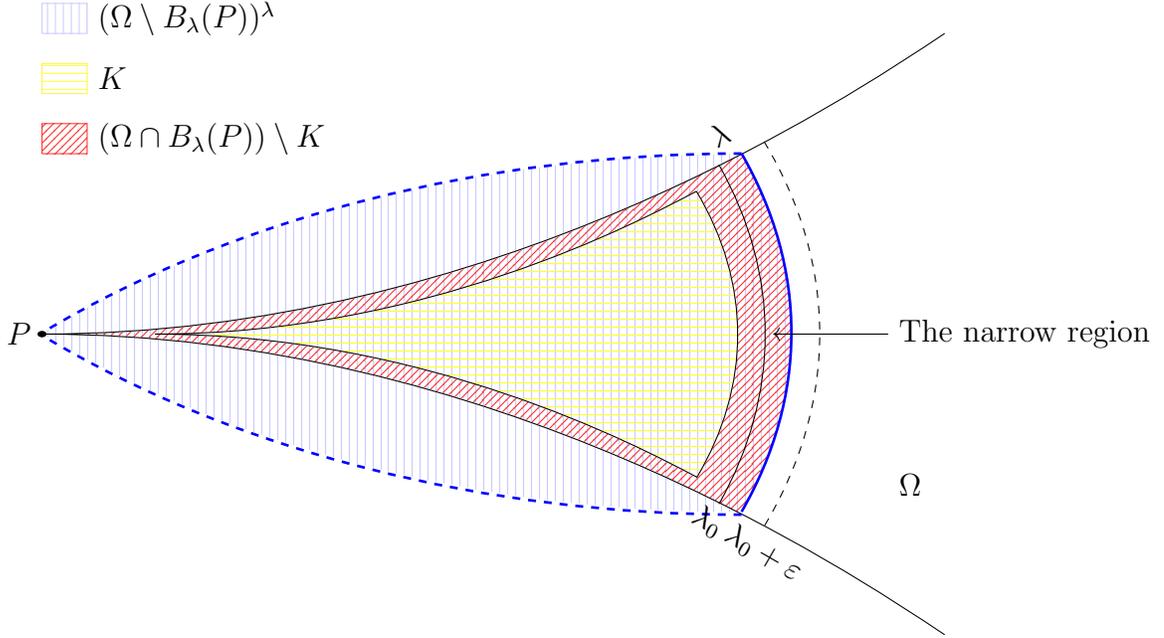
\begin{figure}[H]\label{F8}
    	\begin{tikzpicture}[ yscale=1, xscale=1.5]	
    		\filldraw[ draw=white, pattern color=blue!20,  pattern=vertical lines]  (6.2,  -2.4025) arc (-21:21:6.65)-- (6.2,  2.4025) arc (90:132:9.25)--(0,0)  arc ( 228:270:9.25)--(6.2,  -2.4025);
    		\filldraw[ draw=white, pattern color=yellow!80,  pattern=horizontal lines] (1,0) parabola bend (1,0) (5.8, 1.9) arc (21.7:-21.7:5.16)--(5.8,-1.9) parabola bend (1,0) (1,0);
    		\filldraw[ draw=white, pattern color=red!80,  pattern=north east lines, even odd rule]  (0,0) parabola bend (0,0)  (6.2,  2.4025) arc (21:-21:6.65)--(6.2,  -2.4025)    parabola bend (0,0)  (0,0)            (1,0) parabola bend (1,0) (5.8, 1.9) arc (21.7:-21.7:5.16)--(5.8,-1.9) parabola bend (1,0) (1,0);
    		\draw (0,0) parabola bend (0,0) (8,4)    (0,0) parabola bend (0,0) (8,-4)  (6,9/4) arc (20.6:-20.6:6.41 )   ;
    		\draw (1,0) parabola bend (1,0) (5.8, 1.9)    (1,0) parabola bend (1,0) (5.8,-1.9)  (5.8, 1.9) arc (21.7:-21.7:5.16);
    		\draw[bline] (6.2,  2.4025) arc (21:-21:6.65) ;
    		\draw[dashed, bline] (6.2,  2.4025) arc (90:132:9.25)  (6.2,  -2.4025) arc (-90:-132:9.25);
    		\draw[dashed]    (6.4,  2.56) arc (21.8:-21.8: 6.89);
    		 \draw[<-] (6.48,0)--(7.5,0) node[right]{The narrow region};    		
    		\filldraw[ ] (0:0) circle(1pt) node[left] {$P$};
    		\node[below, rotate=-28] at (6, -2.2) {$\lambda_0$};
    		\node[below, rotate=-30] at  (6.5,  -2.63) {$\lambda_0+\varepsilon$};
    		\node[above, rotate=29] at (6.1, 2.4) {$\lambda$};
    		\node[right] at (7.5, -2) {$\Omega$};
    		
    		\filldraw[ draw=blue!20, pattern color=blue!20,  pattern=vertical lines] (0, 4)--(0.4,4)--(0.4,4.4)--(0, 4.4)--(0, 4);
    		\node[right] at (0.4, 4.2) {$(\Omega\setminus B_\lambda(P))^\lambda$};
    		\filldraw[ draw=yellow!80, pattern color=yellow!80,  pattern=horizontal lines] (0, 3.2)--(0.4,3.2)--(0.4,3.6)--(0, 3.6)--(0, 3.2);
    		\node[right] at (0.4, 3.4) {$K$};
    		\filldraw[ draw=red!80, pattern color=red!80,  pattern=north east lines] (0, 2.4)--(0.4,2.4)--(0.4,2.8)--(0, 2.8)--(0, 2.4);
    		\node[right] at (0.4, 2.6) {$(\Omega\cap B_\lambda(P))\setminus K$ };

    	\end{tikzpicture}
    	\caption{Narrow regions}
    \end{figure}

	\emph{Step 3. Derive lower bound estimates on the asymptotic behavior of u via Bootstrap technique in a smaller cone.} Let $\mathcal{C}_{0,\Sigma} $ with $\mathcal{C}^{r,e}_{0,\Sigma}=(\mathcal{C}_{0,\Sigma} \setminus \overline{B_r(0)})\subset\Omega$ for some $r>d_{\Omega}$ be the cone in assumption $(\mathbf{f_{3}})$. Now $\lambda_{0}=+\infty$ implies that
	\begin{equation*}
		u(x)\geq\left(\frac{\lambda}{|x|}\right)^{n-2s}u\left(\frac{\lambda^{2}x}{|x|^{2}}\right), \quad\quad \forall \,\, x\in \Omega\setminus  B_{\lambda }(0), \quad \forall \,\, \lambda>d_{\Omega}.
	\end{equation*}
	In particular, we have
	\begin{equation}\label{pde-23}
		u(x)\geq\left(\frac{\lambda}{|x|}\right)^{n-2s}u\left(\frac{\lambda^{2}x}{|x|^{2}}\right), \quad\quad \forall \,\, x\in \mathcal{C}^{r,e}_{0, \Sigma} \, \,\, \text{with}\,\, |x|\geq\lambda, \quad \forall \,\,  \lambda>d_{\Omega}.
	\end{equation}
	For arbitrary $ x\in \mathcal C_{0, \Sigma}$ with $|x|\geq r+1$, take $d_{\Omega}<\lambda:=\sqrt{(r+1) |x| }$, then \eqref{pde-23} yields that
	\begin{equation}\label{pde-24}
		u(x)\geq\left(\frac{r+1}{|x|}\right)^{\frac{n-2s}{2}}u\left(\frac{(r+1) x}{|x|}\right).
	\end{equation}
	Then we arrive at the following lower bound estimate on asymptotic behavior of $u$ as $|x|\rightarrow+\infty$:
	\begin{equation}\label{pde-25}
		u(x)\geq\left(\min_{|x|=r+1 , \frac{x}{|x|}\in \Sigma}u(x)\right)\left(\frac{r+1}{ |x|}\right)^{\frac{n-2s}{2}}:=\frac{C_{0}}{|x|^{\frac{n-2s}{2}}}, \quad\quad \forall \,\,x\in  \mathcal C_{0, \Sigma} \  \   \text{with}\,\, |x|\geq r+1.
	\end{equation}

\medskip
	
	The lower bound estimate \eqref{pde-25} can be improved remarkably by using assumption $(\mathbf{f_{3}})$ and the ``Bootstrap" iteration technique. For any smooth cross-section $\Sigma_{1}$ such that $\overline{\Sigma_{1}}\subsetneq\Sigma$, there exists a point $Q\in\mathcal{C}^{r,e}_{0,\Sigma} \subset \Omega$ such that the cone $ \mathcal C_{Q, \Sigma_{1}}\subset  \mathcal C_{0, \Sigma}\bigcap \Omega$.  Let $$v_\delta(x):=C\int_{\mathcal C_{Q, \Sigma_{1}} \bigcap B_{\delta^{-1}}(0)} G^{s}_{\mathcal{C}_{\delta}}(x,y) |y|^a u^p(y) \mathrm{d}y,$$
	where $G^{s}_{\mathcal{C}_{\delta}}(x,y)$ stands for the Green function of $(-\Delta)^s$ in $\mathcal C_{Q, \Sigma_{1}}\bigcap B_{\delta^{-1}}(0)$ with zero Dirichlet boundary condition and $C$ is the constant in $(\mathbf{f_{3}})$. By using assumption $(\mathbf{f_{3}})$ and noticing that $u\geq 0 $ in $\Omega$, we have
	\begin{equation*}
		\begin{cases}
			(-\Delta)^s(u-v_\delta)\geq 0 \ \  \ &\text{in}\,\, \mathcal C_{Q, \Sigma_{1}} \cap B_{\delta^{-1}}(0),\\
			u-v_\delta\geq 0 &\text{on}\,\, \mathbb R^n\setminus (\mathcal C_{Q, \Sigma_{1}} \cap B_{\delta^{-1}}(0)).
		\end{cases}
	\end{equation*}
	Then applying the maximum principle, we get $u\geq v_\delta$ in $\mathcal C_{Q, \Sigma_{1}} \bigcap B_{\delta^{-1}}(0)$ for arbitrary $\delta>0$. Letting $\delta\to0$, by monotone convergence theorem we obtain the following integral inequality:
	\begin{equation}\label{pde-26}
		u(x)\geq C\int_{\mathcal C_{Q, \Sigma_{1}}} G^{s}_{\mathcal{C}}(x,y) |y|^a u^p(y) \mathrm{d}y, \quad\quad \forall \,\,x\in \mathcal C_{Q, \Sigma_{1}},
	\end{equation}
where $G^{s}_{\mathcal{C}}$ is the Green function for Dirichlet problem of $(-\Delta)^s$ with $0<s\leq1$ on $\mathcal C_{Q, \Sigma_{1}}$.

\medskip
	
	Set $\mu_{0}:=\frac{n-2s}{2}$, we infer from the integral inequality \eqref{pde-26} and property $\mathbf{(H_{2})}$ of the Green function $G^{s}_{\mathcal{C}}$ for Dirichlet problem of $(-\Delta)^s$ with $0<s\leq1$ on $\mathcal C_{Q, \Sigma_{1}}$ proved in Theorem \ref{pG} that, for any given cross-section $\Sigma_{2}$ such that $\overline{\Sigma_{2}}\subsetneq\Sigma_{1}$, there exist $R_{0}>0$ and $0<\sigma_{1}<\sigma_{2}$ such that, for any $x\in\mathcal C_{Q, \Sigma_2}$ with $|x|\geq R_{0}+\frac{4\sigma_{2}}{\sigma_{2}-\sigma_{1}}r+1$,
	\begin{eqnarray*}
		u(x)&\geq&C\int_{y\in \mathcal C_{Q, \Sigma_2}, \, \sigma_{1}|x-Q|\leq|x-y|\leq \sigma_2|x-Q|}G^{s}_{\mathcal{C}}(x,y)|y|^a u^p(y)\mathrm{d}y \\
		\nonumber &\geq&C\int_{y\in \mathcal C_{Q,\Sigma_2}, \, \frac{\sigma_{1}(5\sigma_{2}-\sigma_{1})}{4\sigma_{2}}|x|\leq|x-y|\leq \frac{\sigma_{1}+3\sigma_2}{4}|x|, \, |y|\geq R_{0}+\frac{4\sigma_{2}}{\sigma_{2}-\sigma_{1}}r+1}\frac{1}{|x-y|^{n-2s}}\cdot\frac{1}{|y|^{p\mu_{0}-a}}\mathrm{d}y \\
		\nonumber  &\geq&\frac{C_{1}}{|x|^{p\mu_{0}-(2s+a)}}.
	\end{eqnarray*}
	Now, let $\mu_{1}:=p\mu_{0}-(2s+a)$. Due to $1\leq p<p_{c}(a):=\frac{n+2s+2a}{n-2s}$ ($:=+\infty$ if $n=2s$), our important observation is
	\begin{equation*}
		\mu_{1}:=p\mu_{0}-(2s+a)>\mu_{0}.
	\end{equation*}
	Thus we have obtained a better lower bound estimate than \eqref{pde-25} after one iteration, that is,
	\begin{equation}\label{pde-27}
		u(x)\geq\frac{C_{1}}{|x|^{\mu_{1}}}, \quad\quad \forall \,\, x\in \mathcal C_{Q, \Sigma_2}  \  \  \text{ with} \ \ |x|>R_{0}+\frac{4\sigma_{2}}{\sigma_{2}-\sigma_{1}}r+1.
	\end{equation}

\medskip
	
	For $k=0,1,2,\cdots$, define
	\begin{equation}\label{pde-28}
		\mu_{k+1}:=p\mu_{k}-(2s+a).
	\end{equation}
	Since $1\leq p<p_{c}(a):=\frac{n+2s+2a}{n-2s}$ ($:=+\infty$ if $n=2s$) and $a>-2s$, it is easy to see that the sequence $\{\mu_{k}\}$ is monotone decreasing with respect to $k$. Continuing the above iteration process involving the integral inequality \eqref{pde-26}, we have the following lower bound estimates for every $k=0,1,2,\cdots$,
	\begin{equation}\label{pde-29}
		u(x)\geq\frac{C_{k}}{|x|^{\mu_{k}}}, \quad\quad \forall \,\, x\in  \mathcal C_{Q, \Sigma_2}  \  \  \text{ with} \ \ |x|>R_{0}+\frac{4\sigma_{2}}{\sigma_{2}-\sigma_{1}}r+1.
	\end{equation}
	Using the definition \eqref{pde-28} and the assumption $a>-2s$, one can check that as  $k\rightarrow+\infty$,
	\begin{equation}\label{pde-30}
		\mu_{k}\rightarrow-\infty, \,\,\,\quad \text{if} \,\,\, 1\leq p<p_{c}(a):=\frac{n+2s+2a}{n-2s} \, (:=+\infty \,\,\, \text{if} \,\, n=2s).
	\end{equation}
This concludes our proof of Theorem \ref{lower0}.
\end{proof}

	However, property $\mathbf{(H_{2})}$ of the Green function $G^{s}_{\mathcal{C}}$ for Dirichlet problem of $(-\Delta)^s$ with $0<s\leq1$ on $\mathcal C_{Q, \Sigma_{1}}$ proved in Theorem \ref{pG} yields that, there exist $\theta\in\mathbb{R}$ and a point $x_0\in\mathcal C_{Q, \Sigma_1}$ such that
$$\liminf\limits_{y\in \mathcal C_{Q, \Sigma_2},  |y-Q|\to+\infty} |y-Q|^{\theta}G^{s}_{\mathcal{C}}(x_0,y)>0,$$
where the point $Q\in\mathcal{C}^{r,e}_{P,\Sigma}$ comes from Theorem \ref{lower0}. As a consequence, we infer from \eqref{pde-26} and the lower bound estimates in Theorem \ref{lower0} that
	\begin{eqnarray*}
		&& +\infty>u(x_0)\geq  C\int_{\mathcal C_{Q, \Sigma_1}} G^{s}_{\mathcal{C}}(x_0,y) |y-P|^a u^p(y) \mathrm{d}y \\
&& \qquad\qquad\quad\,\,\,\, \geq C\int_{\mathcal C_{Q, \Sigma_2},\,|y-P|\geq R_{0}+\frac{4\sigma_{2}}{\sigma_{2}-\sigma_{1}}r+1}  |y-P|^{a+p\kappa-\theta} \mathrm{d}y=+\infty,
	\end{eqnarray*}
	which is absurd. Hence we must have $u\equiv0$ and $f(\cdot,0)\equiv0$. This concludes our proof of Theorems \ref{pde-unbd} and \ref{pde-unbd-a}.
\end{proof}

\begin{proof}[Proof of Theorem \ref{pde-unbd-2}]
Without loss of generalities, we may assume the center $P=0$. Instead of the small region principle in Theorem \ref{NP-1}, we need the following narrow region principle.		
\begin{thm}\label{NP-2}
		Assume  $0<s<1$, $\lambda>0$ and $D\subset (\Omega\cap B_\lambda(0))\setminus\{0\}$. Suppose that $\omega \in\mathcal{L}_{s}(\mathbb{R}^{n})\cap C^{[2s],\{2s\}+\eps}_{loc}(D)$ with arbitrarily small $\eps>0$ is lower semi-continuous on $\bar{D}$ and satisfies
		\begin{equation}\label{pde-31}\\\begin{cases}
				(-\Delta)^{s}\omega (x)-c_{\lambda}(x)\omega (x)\geq0, \,\,\,\,\, &x\in D,\\
				\omega \geq 0 &\text{a.e. in}\,  \, B_{\lambda}(0)\setminus D,\\
				\omega =-\omega_{\lambda} & \text{a.e. in}\,  \, B_{\lambda}(0)
		\end{cases}\end{equation}
		for some function $c$ satisfying
		\begin{equation}\label{pde-32}
			\limsup_{x\in D, \text{dist}(x, \partial(\Omega\cap B_\lambda(0)))\to0}  \text{dist}(x, \partial(\Omega\cap B_\lambda(0)))^{2s} c_{\lambda}^+(x)=0 \   \  \ \text{uniformly  for } \,\, \lambda\in I_0,
		\end{equation} where $I_0\subset (0, +\infty)$ is an interval. Then there is a constant $\delta_0>0$ depending only on $n, s, I_0$ such that, for any $\lambda\in I_0$,
		  $$ D \subset \{x\in\Omega\cap B_\lambda(0)\mid\,  \text{dist}(x, \partial(\Omega\cap B_\lambda(0)))\leq\delta_0\} \  \   \ \text{implies}\  \ \   \omega(x)\geq 0, \  \  \ \forall\, x\in D.$$
	\end{thm}
\begin{proof}
	Suppose on the contrary that there are $D_k\subset (\Omega\cap B_{\lambda_k}(0))\setminus\{0\}$ with $\lambda_k\in I_0$, $x_k\in D_k$ and $\omega_k$ such that $\omega_k$ solves \eqref{pde-31} in $B_{\lambda_{k}}(0)$ with $c_{\lambda_{k}}(x)$ satisfying \eqref{pde-32}, $\omega_k(x_k)=\min\limits_{D_k} \omega_k<0$ and $\text{dist}(x_k, \partial(\Omega\cap B_{\lambda_k}(0)))\to 0$ as $k\to+\infty$.
	By similar calculations as (2.16) and (2.17) in \cite{CLZ}, we use the third equation in \eqref{pde-31} to compute $(-\Delta)^s \omega_k(x_k)$ as follows:
	\begin{align*}
	&\quad	(-\Delta)^s \omega_k(x_k)=C_{s,n} P.V. \int_{\mathbb R^n} \frac{\omega_k(x_k)-\omega_k(y)}{|x_k-y|^{n+2s}}  \mathrm{d}y\\
		&=C_{s,n} \left(  P.V. \int_{\Omega\cap B_{\lambda_k}(0)}+\int_{(\Omega\cap B_{\lambda_k}(0))^{\lambda_{k}}}+\int_{(\Omega\setminus B_{\lambda_k}(0))\setminus ((\Omega\cap B_{\lambda_k}(0))^{\lambda_{k}})}+\int_{\mathbb R^n \setminus\overline{\Omega} } \right)\frac{\omega_k(x_k)-\omega_k(y)}{|x_k-y|^{n+2s}}  \mathrm{d}y\\
		&= C_{s,n}P.V. \int_{\Omega\cap B_{\lambda_k}(0)}\frac{\omega_k(x_k)-\omega_k(y)}{|x_k-y|^{n+2s}}\mathrm{d}y-C_{s,n}\int_{(\Omega\cap B_{\lambda_k}(0))^{\lambda_{k}}}\frac{\omega_k(y)}{|x_k-y|^{n+2s}} \mathrm{d}y\\
		&\quad -C_{s,n}\left(\int_{ B_{\lambda_k}(0)\setminus\Omega}+\int_{\left(B_{\lambda_k}(0)\setminus\Omega\right)^{\lambda_{k}}}\right)\frac{\omega_k(y)}{|x_k-y|^{n+2s}} \mathrm{d}y+  C_{s,n}\int_{\mathbb R^n \setminus( \Omega \cap B_{\lambda_k}(0))}  \frac{\omega_k(x_k) }{|x_k-y|^{n+2s}}  \mathrm{d}y\\
		&=C_{s,n}P.V. \int_{\Omega\cap B_{\lambda_k}(0)}\frac{\omega_k(x_k)-\omega_k(y)}{|x_k-y|^{n+2s}}  \mathrm{d}y +C_{s,n}\int_{\Omega\cap B_{\lambda_k}(0)}\frac{\omega_k(y)}{\left|\frac{|y|x_k}{\lambda_k}-\frac{{\lambda_k} y}{|y|}\right|^{n+2s}}\mathrm{d}y\\
		&\quad + C_{s,n}P.V. \int_{B_{\lambda_k}(0)\setminus\Omega}\left(   \frac{1}{\left|\frac{|y|x_k}{\lambda_k}-\frac{{\lambda_k} y}{|y|}\right|^{n+2s}}  -\frac{1}{|x_k-y|^{n+2s}}\right)\omega_k(y) \mathrm{d}y \\
&\quad + C_{s,n} \int_{\mathbb R^n \setminus( \Omega \cap B_{\lambda_k}(0))}  \frac{\omega_k(x_k) }{|x_k-y|^{n+2s}}  \mathrm{d}y\\
		&\leq C_{s,n}P.V. \int_{\Omega\cap B_{\lambda_k}(0)}\left(\frac{1}{\left|\frac{|y|x_k}{\lambda_k}-\frac{{\lambda_k} y}{|y|}\right|^{n+2s}}  -\frac{1}{|x_k-y|^{n+2s}}\right)\left(\omega_k(y)-\omega_k(x_k)\right)\mathrm{d}y\\
		&\quad  +C_{s,n}P.V. \int_{\Omega\cap B_{\lambda_k}(0)}  \frac{\omega_k(x_k) }{\left|\frac{|y|x_k}{\lambda_k}-\frac{{\lambda_k} y}{|y|}\right|^{n+2s}}  \mathrm{d}y+ C_{s,n} \int_{\mathbb R^n \setminus( \Omega \cap B_{\lambda_k}(0))}  \frac{\omega_k(x_k) }{|x_k-y|^{n+2s}}  \mathrm{d}y\\
		&\leq C_{s,n}\int_{\mathbb R^n \setminus( \Omega \cap B_{\lambda_k}(0) )}  \frac{\omega_k(x_k) }{|x_k-y|^{n+2s}}  \mathrm{d}y.
	\end{align*}
     Let $d_k:=\text{dist}(x_k, \partial(\Omega\cap B_{\lambda_k}(0)))$, then $d_k\to0$ as $k\to+\infty$. There is a point $\tilde x_k\in \partial(\Omega\cap B_{\lambda_k}(0))$ such that $d_k=|x_k-\tilde x_k|$. By the assumption \eqref{pde-boundary} on $\partial \Omega$, we have
     $$\left|\left(\mathbb R^n \setminus(\Omega \cap B_{\lambda_k}(0))\right)\cap B_{d_k}(\tilde x_k)\right|\geq Cd_k^n$$
     for some $C$ depending on $I_0$ but independent of $k$. Thus we reach that
     \begin{eqnarray}\label{pde-33}
     		&& (-\Delta)^s \omega_k(x_k)\leq C_{s,n}\int_{\mathbb R^n \setminus( \Omega \cap B_{\lambda_k}(0) )}  \frac{\omega_k(x_k) }{|x_k-y|^{n+2s}}  \mathrm{d}y \\
     \nonumber &&\qquad\qquad\quad\,\,\,\,\, \leq C_{s,n}\int_{\left(\mathbb R^n \setminus(\Omega \cap B_{\lambda_k}(0))\right)\cap B_{d_k}(\tilde x_k)}  \frac{\omega_k(x_k) }{|x_k-y|^{n+2s}}  \mathrm{d}y\leq C\frac{\omega_k(x_k)}{d_k^{2s}}.
     \end{eqnarray}
     On the other hand, we infer from the first equation in \eqref{pde-31} that
     \begin{equation}\label{pde-34}
     (-\Delta)^s \omega_k(x_k)\geq  c_{\lambda_k}^+(x_k)\omega_k(x_k).
 \end{equation}
 Therefore, we deduce from \eqref{pde-33} and \eqref{pde-34} that
 \begin{equation}\label{pde-35}
 	 d_k^{2s}c_{\lambda_k}^+(x_k)\geq C>0, \ \  \  \forall\,\, k\geq1
 \end{equation}
   for some $C$ depending on $n,s, I_0$.  Since $d_k\to0$ as $k\to+\infty$, \eqref{pde-35} contradicts \eqref{pde-32} when $k$ is sufficiently large. This completes our proof of Theorem \ref{NP-2}.
\end{proof}

\medskip

From \eqref{PDE} and \eqref{pde-7}, we know that $\omega^{\lambda}=u_{\lambda}-u$ (defined above \eqref{pde-3}) satisfies
\begin{equation}\label{pde-31a}\\\begin{cases}
				(-\Delta)^{s}\omega^{\lambda} (x)-c_{\lambda}(x)\omega^{\lambda} (x)\geq0, \,\,\,\,\, &x\in\Omega^{-}_{\lambda},\\
				\omega \geq 0 &\text{in}\,  \, B_{\lambda}(0)\setminus\Omega^{-}_{\lambda},\\
				\omega =-\omega_{\lambda} & \text{in}\,  \, B_{\lambda}(0)\setminus\{0\},
		\end{cases}\end{equation}
where $c_{\lambda}(x)$ satisfies \eqref{pde-8} on $\Omega^{-}_{\lambda}$. If follows from assumption $(\mathbf{f'_{2}})$ that, there exists a function $g:\,[0,+\infty)$ with $\lim\limits_{t\rightarrow0^{+}}g(t)=+\infty$ such that $g(\text{dist}(x, \partial\Omega))\text{dist}(x, \partial\Omega)^{2s}h_{\lambda}(x)\in L^{\infty}(\Omega\bigcap B_{\lambda}(0))$. Thus $\text{dist}(x, \partial\Omega)^{2s}c^{+}_{\lambda}(x)\leq\text{dist}(x, \partial\Omega)^{2s}|c_{\lambda}(x)|\leq\text{dist}(x, \partial\Omega)^{2s}h_{\lambda}(x)$ implies that $c_{\lambda}(x)$ satisfies \eqref{pde-32} with $D=\Omega^{-}_{\lambda}$ in Theorem \ref{NP-2}. As a consequence, we can prove Theorem \ref{pde-unbd-2} via the \emph{direct method of scaling spheres} in similar way as Theorem \ref{pde-unbd} by using Theorem \ref{NP-2} instead of Theorem \ref{NP-1}. This finishes our proof of Theorem \ref{pde-unbd-2}.
\end{proof}

\begin{proof}[Proof of Theorems \ref{pde-bd} and \ref{pde-bd-2}]
Without loss of generalities, we may assume that $P=0$ and $\rho_{\Omega}:=\max\limits_{x\in\Omega}|x-P|\in(0,+\infty]$. We define the Kelvin transform of the solution $u$ (possibly singular at $P$) to the supercritical Dirichlet problem \eqref{pdeb-1} by
$$\hat u(x):= \frac{1}{|x|^{n-2s}}  u\left(\frac{x}{|x|^2}\right)$$
for any $x\in \widehat \Omega:=\left\{x\in\mathbb R^n\setminus\{0\} \bigg|\,\, \frac{x}{|x|^2}\in \Omega\right\}$. If $\Omega$ is unbounded, we assume that $u=o\left(\frac{1}{|x|^{n-2s}}\right)$ as $|x|\rightarrow+\infty$ with $x\in\Omega$ so that $\hat{u}=0$ at $P=0\in\partial\widehat{\Omega}$. Since $\Omega$ is a g-radially convex domain with radially convex center $0$, one has $\widehat \Omega$ is a MSS applicable domain such that $\mathbb{R}^{n}\setminus\overline{\widehat{\Omega}}$ is a g-radially convex domain with radially convex center $0$. Now, one can verify that $\hat u$ satisfies
\begin{equation}\label{pdeb-2}
	\begin{cases}
		(-\Delta)^s \hat u(x)=\hat f(x,\hat u(x)),\ \ \ & x\in \widehat \Omega,\\
		u(x)=0, \ \ \  & x\in\mathbb R^n\setminus \widehat\Omega,
	\end{cases}
\end{equation}
where
\begin{equation}\label{pdeb-3}
	\hat f(x,\hat u(x))=\frac{1}{|x|^{n+2s}}f\left(\frac{x}{|x|^2},  |x|^{n-2s} \hat u(x)\right),\  \qquad \  \forall\,\, x\in \widehat \Omega
\end{equation}
 and $\hat{u}\in \mathcal{L}_{s}(\mathbb{R}^{n})\bigcap C^{[2s],\{2s\}+\eps}_{loc}(\widehat \Omega)\bigcap C(\overline{\widehat\Omega})$ with arbitrarily small $\eps>0$ if $0<s<1$, $\hat{u}\in C^{2}(\widehat\Omega)\bigcap C(\overline{\widehat\Omega})$ if $s=1$.

\medskip

By \eqref{pdeb-3}, one can verify that if $f$ is supercritical in the sense of Definition \ref{defn1} and satisfies $(\mathbf{\widehat{f}_{2}})$ and $(\mathbf{f_{3}})$, then the nonlinear term $\hat f(x,\hat u)$ is subcritical in the sense of Definition \ref{defn1} and satisfies the assumptions $(\mathbf{f_{2}})$ and $(\mathbf{f_{3}})$. Note also that $\Omega$ is bounded is equivalent to $P\in\mathbb{R}^{n}\setminus\overline{\widehat{\Omega}}$. Thus we deduce immediately from Theorems \ref{pde-unbd} and \ref{pde-unbd-a} that $\hat u\equiv0$ in $\widehat \Omega$, and hence the solution $u$ (possibly singular at $P$) to the supercritical Dirichlet problem \eqref{pdeb-1} must satisfies $u\equiv0$ in $\Omega$. This finishes our proof of Theorem \ref{pde-bd}.

\medskip

By \eqref{pdeb-3}, one can verify that if $f$ is supercritical in the sense of Definition \ref{defn1} and satisfies the assumptions $(\mathbf{\widehat{f}'_{2}})$ and $(\mathbf{f_{3}})$, then the nonlinear term $\hat f(x,\hat u)$ is subcritical in the sense of Definition \ref{defn1} and satisfies the assumptions $(\mathbf{f'_{2}})$ and $(\mathbf{f_{3}})$. Note also that $\Omega$ satisfies the geometric condition \eqref{pdeb-boundary} is equivalent to $\widehat\Omega$ satisfies the geometric condition \eqref{pde-boundary}. Thus we deduce immediately from Theorem \ref{pde-unbd-2} that $\hat u\equiv0$ in $\widehat \Omega$, and hence the solution $u$ (possibly singular at $P$) to the supercritical Dirichlet problem \eqref{pdeb-1} (with $f(x,u)$ itself is not locally $L^\infty$-Lipschitz) must satisfies $u\equiv0$ in $\Omega$. This completes our proof of Theorem \ref{pde-bd-2}.
\end{proof}

\section{The method of scaling spheres in integral forms and in local way on MSS applicable domains}

In this section, we will introduce the method of scaling spheres in integral forms and in local way on general MSS applicable domains. To illustrate our ideas, we first apply the method of scaling spheres in integral forms to the integral equation \eqref{IE} on MSS applicable domain $\Omega$ such that $\mathbb{R}^{n}\setminus\overline{\Omega}$ is a g-radially convex domain and prove Theorem  \ref{ie-unbd}. Then, we will apply the Kelvin transform or the method of scaling spheres in local way to the integral equation \eqref{IE} on MSS applicable domain $\Omega$ which itself is a g-radially convex domain and prove Theorem \ref{ie-bd}.

\subsection{MSS in integral forms and Liouville theorems for general $\leq n$-th order IEs: proof of Theorem \ref{ie-unbd}}
In this subsection, we will carry out the proof of Theorem \ref{ie-unbd} by applying the method of scaling spheres in integral forms.

\medskip

We will consider the  integral equations \eqref{IE}. That is,
\begin{equation}\label{I-1}
	u(x)=\int_{\Omega}K^{s}_{\Omega}(x,y)f(y,u(y)) \mathrm{d}y,
\end{equation}
where $\Omega$ is an arbitrary MSS applicable domain in $\mathbb{R}^{n}$ such that $\mathbb{R}^{n}\setminus\overline{\Omega}$ is a g-radially convex domain  with center $P$, $u\in C(\overline{\Omega})$, $s>0$, $n\geq\max\{2s,1\}$ and the integral kernel $K^{s}_{\Omega}(x,y)$ satisfies assumptions $(\mathbf{H_{1}})$, $(\mathbf{H_{2}})$ and $(\mathbf{H_{3}})$ given at the beginning of subsection 1.3. Moreover, the integral kernel $K^{s}_{\Omega}(x,y)=0$ if $x$ or $y\notin\Omega$. We assume that the nonlinear term $f\geq0$ is subcritical and satisfies  $(\mathbf{f_{1}})$, $(\mathbf{f_{2}})$ with $q=\frac{n}{2s}$ if $n>2s$ and $q=1+\delta$ ($\delta>0$ arbitrarily small) if $n=2s$ and $(\mathbf{f_{3}})$ with the same cross-section $\Sigma$ as $(\mathbf{H_{2}})$.

\medskip

We aim to prove that $u\equiv0$ in $\Omega$. Suppose on the contrary that $u$ is a nonnegative continuous solution of IE \eqref{I-1} but $u\not\equiv0$, then there is a point $x_0\in \Omega $ such that $0<u(x_0)= \int_{\Omega}K^{s}_{\Omega}(x_0,y)f(y,u(y)) \mathrm{d}y$. So we find that $f(\cdot,u(\cdot)) >0$ on a subset of $\Omega$ with positive Lebesgue measure and hence we get
$$u(x)>0, \qquad  \ \forall\,\, x\in\Omega$$
by the IEs \eqref{I-1}. We will apply the method of scaling spheres in integral forms to show the following lower bound estimates for positive
solution $u$, which contradicts the integrability of $u$ provided that $f$ is subcritical in the sense of Definition \ref{defn1}.
\begin{thm}\label{lowerIE-1}
   	Suppose that $u\in C({\overline {\Omega}})$ is a positive solution to \eqref{I-1}. Then it satisfies the following lower bound estimates:
	\begin{equation}\label{I-2}
		u(x)\geq C_{\kappa}|x-P|^{\kappa}, \quad\quad \forall \, \kappa<+\infty, \quad\quad \forall\,\, x\in \mathcal{C}^{R_{\ast},e}_{P,\Sigma},
	\end{equation}
	where $R_{\ast}>d_\Omega$ is a large constant and the cross-section $\Sigma\subseteq\Sigma^{R_{\ast}}_{\Omega}$ is the same as in $(\mathbf{f_{3}})$ and $(\mathbf{H_{2}})$.
\end{thm}
\begin{proof}
In order to apply the method of scaling spheres in integral forms, we first give some notations. Recall that $d_\Omega=\inf\limits_{x\in\Omega} |x-P|\in [0, +\infty)$. Let $\lambda>0$ be an arbitrary positive real number and let the scaling sphere piece be
\begin{equation}\label{I-3}
	S_{\lambda}(P)\bigcap\Omega=\{x\in\Omega\mid\, |x-P|=\lambda\}.
\end{equation}
We denote the reflection of $x\in\mathbb{R}^{n}$ about the sphere $S_{\lambda}(P)$ by $x^{\lambda}:=\frac{\lambda^{2}(x-P)}{|x-P|^{2}}+P$. For any subset $D\subset \mathbb R^n$, we denote $$D^\lambda:=\{x\in\mathbb R^n\mid\, x^\lambda\in D\}$$ as the reflection of $D$ with respect to the sphere $S_\lambda(P)$. In what follows, we assume $P=0$ and divide the rest of proof into two different cases: $2s<n$ and $2s=n$.

\medskip

\textbf{Proof of Theorem \ref{lowerIE-1} in the cases $2s<n$.}

\medskip

		Define the Kelvin transform of $u$ of order $2s$ centered at $0$ by
		\begin{equation}\label{I-4}
			u_{\lambda}(x)=\left(\frac{\lambda}{|x|}\right)^{n-2s}u\left(\frac{\lambda^{2}x}{|x|^{2}}\right).
		\end{equation}		
		It's obvious that the Kelvin transform $u_{\lambda}$ may have singularity at $0$ and $\lim\limits_{|x|\rightarrow+\infty}u_{\lambda}(x)=0$. Moreover, one can verify that $\mathbb R^n\setminus \overline \Omega$ is g-radially convex implies $$\Omega\cap B_\lambda(0)\subset (\Omega\setminus B_\lambda(0))^\lambda,\qquad \  \forall \,\, \lambda>d_\Omega:=\inf_{x\in\Omega}|x-P|.$$

\smallskip		
		
		We will compare the value of $u$ and $u_\lambda$ on $\Omega\cap B_\lambda(0)$ for every $\lambda>d_\Omega$. Let $\omega^{\lambda}(x):=u_{\lambda}(x)-u(x)$. Since $u$ satisfies \eqref{I-1}, by changing variables, we have
		\begin{align}\label{I-5}
			&\quad u(x)=\int_{\Omega}K^{s}_{\Omega}(x,y)f(y, u(y)) \mathrm dy\\
			&=\int_{\Omega\cap B_\lambda(0) }K^{s}_{\Omega}(x,y) f(y, u(y))\mathrm{d}y+\int_{\Omega\setminus B_\lambda(0)}K^{s}_{\Omega}(x,y)f(y, u(y)) \mathrm dy\nonumber\\
			&=\int_{\Omega\cap B_\lambda(0)  } K^{s}_{\Omega}(x,y) f(y, u(y))\mathrm dy  +\int_{\Omega\cap B_\lambda(0) } {\left(\frac{\lambda}{|y|}\right)}^{2n}K^{s}_{\Omega}(x,y^{\lambda})f\left[y^\lambda, {\left(\frac{\lambda}{|y|}\right)}^{2s-n}u_\lambda(y)\right] \mathrm dy\nonumber\\
			&\quad+\int_{(\Omega\setminus B_\lambda(0))^\lambda\setminus (\Omega\cap B_\lambda(0))} {\left(\frac{\lambda}{|y|}\right)}^{2n}K^{s}_{\Omega}(x,y^{\lambda})f\left[y^\lambda, {\left(\frac{\lambda}{|y|}\right)}^{2s-n}u_\lambda(y)\right] \mathrm dy. \nonumber
		\end{align}
		Similarly, by  \eqref{I-1} and \eqref{I-4}, we obtain
		\begin{align}\label{I-6}
			u_\lambda(x)&=\left(\frac{\lambda}{|x|}\right)^{n-2s}\int_{\Omega}K^{s}_{\Omega}(x^\lambda,y)f(y, u(y))\mathrm dy\\
			&=\left(\frac{\lambda}{|x|}\right)^{n-2s}\int_{\Omega\cap B_\lambda(0)} K^{s}_{\Omega}(x^\lambda,y)f(y, u(y))  \mathrm dy \nonumber\\
			&\quad+\left(\frac{\lambda}{|x|}\right)^{n-2s}\int_{\Omega\cap B_\lambda(0)}  {\left(\frac{\lambda}{|y|}\right)}^{2n} K^{s}_{\Omega}(x^\lambda,y^\lambda)f\left(y^\lambda, {\left(\frac{\lambda}{|y|}\right)}^{2s-n}u_\lambda(y)\right)\mathrm dy\nonumber\\
			&\quad+ \left(\frac{\lambda}{|x|}\right)^{n-2s}\int_{(\Omega\setminus B_\lambda(0))^\lambda\setminus (\Omega\cap B_\lambda(0))}{\left(\frac{\lambda}{|y|}\right)}^{2n} K^{s}_{\Omega}(x^\lambda,y^\lambda)f\left(y^\lambda, {\left(\frac{\lambda}{|y|}\right)}^{2s-n}u_\lambda(y)\right) \mathrm dy. \nonumber
		\end{align}
		Then, using    \eqref{I-5}, \eqref{I-6},  $(\mathbf{H_{3}})$ and  the assumption that $f$ is subcritical in the sense of Definition \ref{defn1}, we get, for $x\in \Omega\cap B_\lambda(0)$,
		
			\begin{align}\label{I-7}
				&\quad \omega^{\lambda}(x)=u_{\lambda}(x)-u(x)\\
				&=\int_{\Omega\cap B_\lambda(0)} \left[\left(\frac{\lambda^2}{|x||y|}\right)^{n-2s}K^{s}_{\Omega}(x^\lambda,y^\lambda)-{\left(\frac{\lambda}{|y|}\right)}^{n-2s}K^{s}_{\Omega}(x,y^\lambda)\right] \nonumber \\
&\qquad \times {\left(\frac{\lambda}{|y|}\right)}^{n+2s}f\left[y^\lambda, {\left(\frac{\lambda}{|y|}\right)}^{2s-n}u_\lambda(y)\right]  \mathrm dy\nonumber \\
				&\quad+\int_{\Omega\cap B_\lambda(0)}  \left(\left(\frac{\lambda}{|x|}\right)^{n-2s}K^{s}_{\Omega}(x^\lambda,y)- K^{s}_{\Omega}(x,y)\right) f(y, u(y))  \mathrm dy  \nonumber\\
				&\quad+ \int_{(\Omega\setminus B_\lambda(0))^\lambda\setminus (\Omega\cap B_\lambda(0))} \left[\left(\frac{\lambda^2}{|x||y|}\right)^{n-2s}K^{s}_{\Omega}(x^\lambda,y^\lambda)-{\left(\frac{\lambda}{|y|}\right)}^{n-2s}K^{s}_{\Omega}(x,y^\lambda)\right] \nonumber \\
&\qquad \times {\left(\frac{\lambda}{|y|}\right)}^{n+2s}f\left[y^\lambda, {\left(\frac{\lambda}{|y|}\right)}^{2s-n}u_\lambda(y)\right] \mathrm dy\nonumber \\
				&\geq \int_{\Omega\cap B_\lambda(0)} \left[\left(\frac{\lambda^2}{|x||y|}\right)^{n-2s}K^{s}_{\Omega}(x^\lambda,y^\lambda)-{\left(\frac{\lambda}{|y|}\right)}^{n-2s}K^{s}_{\Omega}(x,y^\lambda)\right] \nonumber \\
&\qquad \times \left[{\left(\frac{\lambda}{|y|}\right)}^{n+2s}f\left(y^\lambda, {\left(\frac{\lambda}{|y|}\right)}^{2s-n}u_\lambda(y)\right)-f(y, u(y))  \right] \mathrm dy\nonumber\\
				&>\int_{\Omega\cap B_\lambda(0)}\left[\left(\frac{\lambda^2}{|x||y|}\right)^{n-2s}K^{s}_{\Omega}(x^\lambda,y^\lambda)-{\left(\frac{\lambda}{|y|}\right)}^{n-2s}K^{s}_{\Omega}(x,y^\lambda)\right] \nonumber\\
&\qquad \times \left( f(y, u_\lambda(y))-     f(y, u(y)) \right) \mathrm dy.\nonumber
		\end{align}

\medskip		
		
		Now we carry out the method of scaling spheres in integral forms in three steps.
	
\medskip
	
		\emph{Step 1. Start dilating the sphere piece $S_{\lambda}(P)\bigcap\Omega$ from near $\lambda=d_\Omega$.}
		We will first show that, for $\lambda>d_\Omega$ sufficiently close to $d_{\Omega}$,
		\begin{equation}\label{I-8}
			\omega^\lambda(x)\geq 0,\,\,\,\quad \forall \,\, x\in  \Omega\cap B_\lambda(0).
		\end{equation}
		Define
		\begin{equation}\label{I-9}
			\Omega_\lambda^-:=\{x\in   \Omega\cap B_\lambda(0) \,|\,\omega_\lambda(x)<0\}.
		\end{equation}
		Then the proof of \eqref{I-8} is reduced to showing that $	\Omega_\lambda^-=\emptyset$ for $\lambda-d_\Omega>0$ sufficiently small.
		
\medskip
		
		Since $0<u_\lambda(x)<u(x)\leq M_\lambda $ for any $x\in 	\Omega_\lambda^-$ with $M_{\lambda}=\sup\limits_{\Omega\cap B_{\lambda}(0)}u<+\infty$,
		we derive from the assumptions $(\mathbf{f_{1}})$ and $(\mathbf{f_{2}})$ that
		\begin{equation}\label{I-10}
			0\leq c_\lambda(x):=\frac{f(x, u_\lambda(x))- f(x, u(x))}{u_\lambda(x)-u(x)}\leq \sup_{\alpha,\beta\in [0, M_\lambda] } \frac{|f(x, \alpha)- f(x, \beta)|}{|\alpha-\beta|}\leq h_\lambda(x),
		\end{equation}
		for some $h_\lambda \in L^{\frac{n}{2s}} (\Omega\cap B_R(0)) \   (\forall\, R>0)$. Moreover, it can be seen that $h_\lambda$ can be chosen increasingly depending on $\lambda$ since $M_\lambda$ is increasing as $\lambda$ increases.
		
\medskip		
		
		By using the assumption $\mathbf{(f_{1})}$ on $f$, $(\mathbf{H_{1}})$, $(\mathbf{H_{3}})$, inequalities \eqref{I-7} and \eqref{I-10}, we have, for any $x\in \Omega_\lambda^-$,
		\begin{eqnarray}\label{I-11}
			&&\quad\omega^{\lambda}(x) \\
			&&>\int_{\Omega\cap B_\lambda(0)}\left(\left(\frac{\lambda^2}{|x||y|}\right)^{n-2s}K^{s}_{\Omega}(x^\lambda,y^\lambda)-{\left(\frac{\lambda}{|y|}\right)}^{n-2s}K^{s}_{\Omega}(x,y^\lambda)\right) \left( f(y, u_\lambda(y))-     f(y, u(y)) \right) \mathrm dy  \nonumber\\
			&&\geq \int_{\Omega_\lambda^-}\left(\left(\frac{\lambda^2}{|x||y|}\right)^{n-2s}K^{s}_{\Omega}(x^\lambda,y^\lambda)-{\left(\frac{\lambda}{|y|}\right)}^{n-2s}K^{s}_{\Omega}(x,y^\lambda)\right) \left( f(y, u_\lambda(y))-     f(y, u(y)) \right) \mathrm dy  \nonumber\\
			&&\geq \int_{\Omega_\lambda^-} \left(\frac{\lambda^2}{|x||y|}\right)^{n-2s}K^{s}_{\Omega}(x^\lambda,y^\lambda)  \left(    f(y , u_\lambda(y))-  f(y, u(y))\right) \mathrm dy \nonumber\\
			&&\geq \int_{\Omega_\lambda^-} \frac{C_{2}}{|x-y|^{n-2s}}  h_\lambda(y)\omega^{\lambda}(y) \mathrm dy.\nonumber
		\end{eqnarray}
	
\medskip
	
Now we need the following Hardy-Littlewood-Sobolev inequality.
\begin{lem}\label{HL}(Hardy-Littlewood-Sobolev inequality)
Let $n\geq1$, $0<s<n$ and $1<p<q<\infty$ be such that $\frac{n}{q}=\frac{n}{p}-s$. Then we have
\begin{equation}\label{HLS}
  \Big\|\int_{\mathbb{R}^{n}}\frac{f(y)}{|x-y|^{n-s}}\mathrm{d}y\Big\|_{L^{q}(\mathbb{R}^{n})}\leq C_{n,s,p,q}\|f\|_{L^{p}(\mathbb{R}^{n})}
\end{equation}
for all $f\in L^{p}(\mathbb{R}^{n})$.
\end{lem}

		By \eqref{I-11}, Hardy-Littlewood-Sobolev inequality and H\"{o}lder inequality, for arbitrary $\frac{n}{n-2s}<q<\infty$, we obtain
		\begin{equation}\label{I-12}
			{\| \omega^{\lambda}\|}_{L^q(\Omega_\lambda^-)}
			\leq  C  {\left\|  h_\lambda(y) \omega^{\lambda}\right\|}_{L^{\frac{nq}{n+2sq}}(\Omega_\lambda^- )}
			\leq  C  {\left\| h_\lambda\right\|}_{L^{\frac{n}{2s}}(\Omega_\lambda^-)}{\left\| \omega^{\lambda}\right\|}_{L^{q}(\Omega_\lambda^-)}.
		\end{equation}
		Since it follows from the assumption on $f$ that $0\leq h_\lambda\leq h_{d_\Omega+1}$, $h_{d_\Omega+1}\in L^{\frac{n}{2s}} (\Omega\cap B_R(0)) \   (\forall\, R>0)$ and the measure of $\Omega_\lambda^-\subset \Omega\cap B_\lambda(0)$ is small as soon as $\lambda-d_\Omega>0$ is small, there exists $\varepsilon_0\in (0,1)$ sufficiently small, such that
		\begin{equation}\label{I-13}
			C  {\left\| h_\lambda\right\|}_{L^{\frac{n}{2s}}(\Omega_\lambda^-)}\leq C  {\left\| h_{d_{\Omega}+1}\right\|}_{L^{\frac{n}{2s}}(\Omega_\lambda^-)} <\frac{1}{2}
		\end{equation}
		for any $d_\Omega<\lambda<d_\Omega+\varepsilon_0$. Thus, by \eqref{I-12}, we must have
		\begin{equation}\label{I-14}
			{\| \omega^{\lambda}\|}_{L^q(\Omega_\lambda^-)}=0
		\end{equation}
		for any $d_\Omega<\lambda<d_\Omega+\varepsilon_0$. From the continuity of $\omega^{\lambda}$ in $\overline{\Omega}\setminus\{0\}$ and the definition of $\Omega_\lambda^-$, we immediately get that $\Omega_\lambda^-=\emptyset$ and hence \eqref{I-8} holds true for any $d_\Omega<\lambda<d_\Omega+\varepsilon_0$. This completes Step 1.
		
\medskip

		\emph{Step 2. Dilate the sphere piece $S_{\lambda}(P)\bigcap\Omega$ toward the interior of $\Omega$ until $\lambda=+\infty$.} Step 1 provides us a starting point to dilate the sphere piece $S_{\lambda}(P)\bigcap\Omega$ from near $\lambda=d_\Omega$. Now we dilate the sphere piece $S_{\lambda}(P)\bigcap\Omega$ toward the interior of $\Omega$ as long as \eqref{I-8} holds. Define the limiting radius
		\begin{equation}\label{I-15}
			\lambda_{0}:=\sup\{\lambda>d_{\Omega}\,|\, \omega^{\mu}\geq0 \,\, \text{in} \,\, \Omega\cap B_{\mu}(0), \,\, \forall \, d_{\Omega}<\mu\leq\lambda\}\in(d_\Omega,+\infty].
		\end{equation}
		By the continuity, one has
		\begin{equation}\label{I-16}
			\omega^{\lambda_{0}}(x)\geq0, \quad\quad \forall \,\, x\in \Omega\cap B_{\lambda_{0}}(0).
		\end{equation}

\medskip
		
		In what follows, we will prove $\lambda_{0}=+\infty$ by driving a contradiction under the assumption that $\lambda_0<+\infty$.

\medskip
		
		In fact, suppose on the contrary that $\lambda_0<+\infty$. In this case,  we must have
		\begin{equation}\label{I-17}
			\omega^{\lambda_{0}}(x)\equiv 0, \quad\quad \forall \,\, x\in \Omega\cap B_{\lambda_{0}}(0).
		\end{equation}
		Suppose that \eqref{I-17} does not hold, that is, there exists an $x_0 \in \Omega\cap B_{\lambda_{0}}(0)$ such that $\omega^{\lambda_{0}}(x_0)>0$. Then, by \eqref{I-11} and \eqref{I-16}, we have
		\begin{equation}\label{I-18}
			\omega^{\lambda_{0}}(x)>0, \quad\quad \forall \,\, x\in \Omega\cap B_{\lambda_{0}}(0).
		\end{equation}
	
\medskip
	
		Let $h_{ \lambda_0 +1}$ be the function such that
		\begin{equation*}
			\sup_{\alpha,\beta\in [0, M_{ \lambda_0+1}] } \frac{|f(x, \alpha)- f(x, \beta)|}{|\alpha-\beta|}\leq h_{\lambda_0+1}(x)
		\end{equation*}
		with $M_{\lambda_0+1}= \sup\limits_{\Omega\cap B_{\lambda_0+1}(0)} u<+\infty$ and  $h_{\lambda_0+1} \in L^{\frac{n}{2s}} (\Omega\cap B_R(0)) \   (\forall\, R>0)$. In particular, $h_{\lambda_0+1} \in L^{\frac{n}{2s}} (\Omega\cap B_{\lambda_0+1}(0))$. So there is a constant $\delta_0>0$ such that, for any subset $D\subset (\Omega\cap B_{\lambda_0+1}(0))$,
		\begin{equation}\label{I-18-1}
			|D|\leq \delta_0 \ \  \ \text{implies }\  \	C \left\|h_{\lambda_0+1} \right\|_{L^{\frac{n}{2s}}(D)}\leq\frac{1}{2},
		\end{equation}
		where $C$ is the constant in \eqref{I-12}.

\medskip		
		
		Choose $\varepsilon_1\in (0,1)$ sufficiently small such that  the narrow region $A_{\varepsilon_1}\subset  \Omega\cap B_{\lambda_{0}}(0)$ satisfies $|A_{\varepsilon_1}|<\frac{ \delta_0 }{2}$, where $A_{\varepsilon_1}$ is defined by
		\begin{equation}\label{I-19}
			A_{\varepsilon_1}:=\{x\in \Omega\cap B_{\lambda_{0}}(0)\,|\, \text{dist}(x,\partial (\Omega\cap B_{\lambda_{0}}(0)))<\varepsilon_1\}.
		\end{equation}

\medskip
		
		Since $\omega^{\lambda_{0}}$ is continuous in $\overline{\Omega}\setminus\{0\}$ and $ K_{\varepsilon_1}:=(\Omega\cap B_{\lambda_{0}}(0))\setminus A_{\varepsilon_1} $ is a compact subset, there exists a positive constant $C_{\varepsilon_1}$ such that
		\begin{equation}\label{I-20}
			\omega^{\lambda_{0}}(x)>C_{\varepsilon_1}, \quad\quad \forall \,\, x\in K_{\varepsilon_1}.
		\end{equation}
		By continuity, we can choose $0<\varepsilon_2<\varepsilon_1$ sufficiently small such that, for any $\lambda\in [\lambda_0,\,\lambda_0+\varepsilon_2]$,
		\begin{equation}\label{I-21}
			\omega^{\lambda}(x)>\frac{C_{\varepsilon_1}}{2}, \quad\quad \forall \,\, x\in K_{\varepsilon_1}.
		\end{equation}
		Taking $\varepsilon_2$ smaller if necessary, we can assume that the narrow annular domain has measure
		\begin{equation}\label{I-21-1}
			|B_{\lambda}(0)\setminus B_{\lambda_{0}}(0)|\leq \frac{ \delta_0 }{2},\ \quad\,  \forall \,\, \lambda\in [\lambda_0,\,\lambda_0+\varepsilon_2].
		\end{equation}
		By \eqref{I-21}, we must have the set $\Omega_{\lambda}^{-}$ where $\omega^\lambda$ is negative satisfies
		\begin{equation}\label{I-22}
			\Omega_\lambda^-\subset( \Omega\cap B_{\lambda}(0))\setminus K_{\varepsilon_1}=\left((B_{\lambda}(0)\setminus B_{\lambda_{0}}(0))\cap \Omega\right)\cup A_{\varepsilon_{1}}
		\end{equation}
		for any $\lambda\in [\lambda_0,\,\lambda_0+\varepsilon_2]$, which implies that $	\Omega_\lambda^-$ is contained in a sufficiently narrow region. Using \eqref{I-18-1}, \eqref{I-21-1} and \eqref{I-22}, we find that for any $\lambda\in [\lambda_0,\,\lambda_0+\varepsilon_2]$,
		\begin{equation}\label{I-23}
			C  {\left\| h_\lambda\right\|}_{L^{\frac{n}{2s}}(\Omega_\lambda^-)}\leq C  {\left\| h_{\lambda_0+1}\right\|}_{L^{\frac{n}{2s}}(\Omega_\lambda^-)} <\frac{1}{2},
		\end{equation}
		where $h_\lambda$ is the function given in \eqref{I-10} and $C$ is the constant in \eqref{I-12}.

\medskip
		
		Then, by the same argument as in step 1, we deduce from \eqref{I-12} and \eqref{I-23} that for any $\lambda\in[\lambda_0,\,\lambda_0+\varepsilon_2]$, it holds
		\begin{equation}\label{I-24}
			\omega^{\lambda}(x)\geq 0, \quad\quad \forall \,\, x\in  \Omega\cap B_\lambda(0).
		\end{equation}
		This contradicts the definition of $\lambda_0$ in \eqref{I-15} and hence \eqref{I-17} must hold true.
		
\medskip

		However, by \eqref{I-11} and  \eqref{I-17}, we obtain
		\begin{eqnarray*}
			&&	0=\omega^{\lambda_{0}}(x)=u_{\lambda_0}(x)-u(x)\\
			&& \,\,\,\, >\int_{\Omega\cap B_{\lambda_0}(0)}\left(\left(\frac{{\lambda_0}^2}{|x||y|}\right)^{n-2s}K^{s}_{\Omega}(x^{\lambda_0},y^{\lambda_0})-{\left(\frac{{\lambda_0}}{|y|}\right)}^{n-2s}K^{s}_{\Omega}(x,y^{\lambda_0})\right) \\
			&&\qquad\times \left( f(y, u_{\lambda_0}(y))-     f(y, u(y)) \right) \mathrm dy    =0\nonumber
		\end{eqnarray*}
		for any $x\in\, \Omega\cap B_{\lambda_{0}}(0) $, which is absurd. Thus we must have $\lambda_0=+\infty$, that is,
		\begin{equation}\label{I-27}
			u(x)\geq \left(\frac{\lambda}{|x|}\right)^{n-2s}u\left(\frac{\lambda^{2}x}{|x|^{2}}\right), \quad\quad \forall \,\, x\in\Omega \,\,\,\text{s.t.}\,\,|x|\geq\lambda,\,\,\, \quad \forall \,\, d_\Omega<\lambda<+\infty.
		\end{equation}

\medskip
		
		\emph{Step 3. Derive lower bound estimates on the asymptotic behavior of $u$ via Bootstrap technique in a smaller cone.}
		Let $\Sigma$ be the cross-section in $(\mathbf{f_{3}})$ and $\mathbf{(H_{2})}$, then $ \mathcal{C}^{R,e}_{0,\Sigma} \subset\Omega$ for some $R>d_{\Omega}$. We infer from \eqref{I-27} that
		\begin{equation}\label{I-28}
			u(x)\geq \left(\frac{\lambda}{|x|}\right)^{n-2s}u\left(\frac{\lambda^{2}x}{|x|^{2}}\right), \quad\quad \forall \,\, x\in \mathcal{C}^{R,e}_{0,\Sigma} \,\, \text{with} \,\, |x|\geq\lambda, \quad \forall \,\,  \lambda>R.
		\end{equation}
		For arbitrary $ x\in  \mathcal{C}^{R+1,e}_{0,\Sigma}$, take $0<\lambda:=\sqrt{(R+1) |x| }<|x|$, then \eqref{I-28} yields that
		\begin{equation}\label{I-29}
			u(x)\geq \left(\frac{R+1}{|x|}\right)^{\frac{n-2s}{2}}u\left(\frac{(R+1) x}{|x|}\right).
		\end{equation}
		Then we arrive at the following lower bound estimate on asymptotic behavior of $u$ as $|x|\rightarrow+\infty$:
		\begin{equation}\label{I-30}
			u(x)\geq \left(\frac{R+1}{|x|}\right)^{\frac{n-2s}{2}}\left(\min_{|x|=R+1 , \frac{x}{|x|}\in \Sigma}u(x)\right) :=\ \frac{C_0}{|x|^{\frac{n-2s}{2}}}, \quad\quad \forall \,\,x\in  \mathcal C^{R+1,e}_{0, \Sigma}.
		\end{equation}
	
\medskip
	
		The lower bound estimate \eqref{I-30} can be improved remarkably when $f$ satisfies $(\mathbf{f_{3}})$ with $p<p_{c}(a)$ using a ``Bootstrap" iteration technique. Let $R_\ast:=R_0+R+1$ with $R_0$ the constant in $(\mathbf{H_{2}})$. Set $\mu_{0}:=-\frac{n-2s}{2}$, we infer from  the integral equation \eqref{I-1},   assumptions $(\mathbf{H_{2}})$ and $(\mathbf{f_{3}})$ that, for any $x\in \mathcal C^{R_\ast,e}_{0, \Sigma}$,
		\begin{eqnarray*}
			u(x)&\geq& C\int_{y\in   \mathcal C^{R_\ast,e}_{0, \Sigma}, \sigma_1|x|\leq |x-y|\leq \sigma_2|x|}K^s_{\Omega} (x,y)|y|^a u^p(y)\mathrm dy \\
			\nonumber &\geq&C\int_{y\in   \mathcal C^{R_\ast,e}_{0, \Sigma}, \sigma_1|x|\leq |x-y|\leq \sigma_2|x|} \frac{|y|^a u^p(y)}{|x-y|^{n-2s}} \mathrm dy  \\
			\nonumber  &\geq& C_{1} |x|^{p\mu_{0}+(n+a)}.
		\end{eqnarray*}
		Now, let $\mu_{1}:=p\mu_{0}+(2s+a)$. Due to $1\leq p<p_c(a)=\frac{n+2s+2a}{n-2s}$ and $a>-2s$, our important observation is
		\begin{equation*}
			\mu_{1}=p\mu_{0}+(2s+a)>\mu_{0}.
		\end{equation*}
		Thus we have obtained a better lower bound estimate than \eqref{I-30} after one iteration, that is,
		\begin{equation}\label{I-31}
			u(x)\geq C_{1} |x|^{\mu_{1}}, \quad\quad \forall \,\, x\in \mathcal C^{R_\ast,e}_{0, \Sigma}.
		\end{equation}

\medskip
		
		For $k=0,1,2,\cdots$, define
		\begin{equation}\label{I-32}
			\mu_{k+1}:=p\mu_{k}+(2s+a).
		\end{equation}
		Since $1\leq p<p_c(a)$ and $a>-2s$, it is easy to see that the sequence $\{\mu_{k}\}$ is monotone increasing with respect to $k$. Repeating the above iteration process involving the integral equation \eqref{I-1}, we have the following lower bound estimates for every $k=0,1,2,\cdots$,
		\begin{equation}\label{I-33}
			u(x)\geq C_{k} |x|^{\mu_{k}}, \quad\quad \forall \,\, x\in  \mathcal C^{R_\ast,e}_{0, \Sigma}.
		\end{equation}
		Using the definition \eqref{I-32} and the assumption $a>-2s$, one can check that as  $k\rightarrow+\infty$,
		\begin{equation}\label{I-34}
			\mu_{k}\rightarrow+\infty \quad \text{provided that } \,\,\, 1\leq p<p_c(a),
		\end{equation}
		which combined with \eqref{I-33} imply the desired estimates \eqref{I-2} in Theorem \ref{lowerIE-1} in the case $2s<n$.

\bigskip

Next, we will prove Theorem \ref{lowerIE-1} in critical order case (i.e. $2s=n$) with more technical difficulties.

\medskip

\textbf{Proof of Theorem \ref{lowerIE-1} in the critical order case $2s=n$.}

\medskip

	Recall that we have assumed $P=0$. In this case, we define the Kelvin transform of $u$ of order $\frac{n}{2}$ centered at $P=0$ by
	\begin{equation}\label{c-5}
		u_{\lambda}(x):=u(x^\lambda)=u\left(\frac{\lambda^{2}x}{|x|^{2}}\right).
	\end{equation}
    It's obvious that the Kelvin transform $u_{\lambda}$ may have singularity at $0$ and $\lim\limits_{|x|\rightarrow+\infty}u_{\lambda}(x)=u(0)=0$.

\medskip

    We compare the value of $u_\lambda$ and $u$ on $\Omega\cap B_{\lambda}(0)$ and let $\omega^{\lambda}(x):=u_{\lambda}(x)-u(x)$.
    Using \eqref{I-1}, \eqref{c-5}, $(\mathbf{H_{3}})$ and  the assumption that $f$ is sub-critical w.r.t. singularities in the sense of Definition \ref{defn1},   through similar calculations as the proof of \eqref{I-7}, we get, for $x\in \Omega\cap B_\lambda(0)$,
    \begin{align}\label{c-9}
    	&\omega^{\lambda}(x)=u_{\lambda}(x)-u(x)\\
    	&\qquad\,\,\, >\int_{\Omega\cap B_\lambda(0)}\left(K^{\frac{n}{2}}_{\Omega}(x^\lambda,y^\lambda)-K^{\frac{n}{2}}_{\Omega}(x,y^\lambda)\right)\left( f(y, u_\lambda(y))-     f(y, u(y)) \right) \mathrm dy.\nonumber
    \end{align}

\medskip

    Now we carry out the method of scaling spheres in integral forms in three steps.

\medskip

    \emph{Step 1. Start dilating the sphere piece $S_{\lambda}(P)\bigcap\Omega$ from near $\lambda=d_{\Omega}\geq0$.}
    We will first show that, for $\lambda>d_\Omega$ such that $\lambda-d_{\Omega}$ sufficiently small,
    \begin{equation}\label{c-10}
    	\omega^\lambda(x)\geq 0,\,\,\,\quad \forall \,\, x\in  \Omega\bigcap B_\lambda(0).
    \end{equation}
    Define
    \begin{equation}\label{c-11}
    	\Omega_\lambda^-:=\{x\in \Omega\bigcap B_\lambda(0)\,|\,\omega^\lambda(x)<0\}.
    \end{equation}

\medskip

    Since $0<u_\lambda(x)<u(x)\leq M_\lambda $ for any $x\in 	\Omega_\lambda^-$ with $M_{\lambda}=\sup\limits_{\Omega\cap B_{\lambda}(0)}u<+\infty$,
    we derive from the assumptions $(\mathbf{f_{1}})$ and $(\mathbf{f_{2}})$ that
    \begin{equation}\label{c-12}
    	0\leq c_\lambda(x):=\frac{f(x, u_\lambda(x))- f(x, u(x))}{u_\lambda(x)-u(x)}\leq \sup_{\alpha,\beta\in [0, M_\lambda] } \frac{|f(x, \alpha)- f(x, \beta)|}{|\alpha-\beta|}\leq h_\lambda(x)
    \end{equation}
    for some $h_\lambda \in L^{1+\delta} (\Omega\cap B_R(0)) \  (\forall\, R>0)$ with $\delta>0$ arbitrarily small. Moreover, it can be seen that $h_\lambda$ can be chosen increasingly depending on $\lambda$ since $M_\lambda$ is increasing as $\lambda$ increases.

\medskip

    One can easily verify that
    \begin{equation}\label{c-13}
    	\ln{\frac{1}{t}}\leq  \frac{C_\varepsilon}{t^\varepsilon}, \quad \,\quad \forall \  t\in(0,\frac{1}{2}),
    \end{equation}
    where $\varepsilon$ is an arbitrary positive real number and $C_\varepsilon $ is a constant depending on $\varepsilon $. Using $(\mathbf{H_{1}})$, we find that, for any $x,y \in  \Omega\cap B_\lambda(0)$,
    \begin{eqnarray}\label{c-13-1}
    &&	K^{\frac{n}{2}}_{\Omega} (x^\lambda, y^\lambda)\leq C_1\ln \left[\frac{C_0(1+ |x^\lambda|) (1+|y^\lambda|)}{|x^\lambda-y^\lambda|}\right]
    =C_1\ln \left[\frac{C_0  (\lambda^2+|y|+|x|+1)}{|x -y |}\right] \\
    \nonumber && \qquad\qquad\quad \leq C_1\ln \left[\frac{C_0  (\lambda+1 )^2}{|x -y |}\right].
    \end{eqnarray}

\medskip

    By the assumption $\mathbf{(f_{1})}$ on $f$, inequalities \eqref{c-9}, \eqref{c-12}, \eqref{c-13} and \eqref{c-13-1}, we have, for any $x\in \Omega_\lambda^-$,
    \begin{eqnarray}\label{c-14}
    	\omega^{\lambda}(x)
    	&>& \int_{ \Omega\cap B_\lambda(0) }\left(K^{\frac{n}{2}}_{\Omega}(x^\lambda,y^\lambda)-K^{\frac{n}{2}}_{\Omega}(x,y^\lambda)\right) \left(    f(y , u_\lambda(y))-  f(y, u(y))\right)\mathrm dy\nonumber\\
    	&\geq& \int_{\Omega_\lambda^-}\left(K^{\frac{n}{2}}_{\Omega}(x^\lambda,y^\lambda)-K^{\frac{n}{2}}_{\Omega}(x,y^\lambda)\right) \left(    f(y , u_\lambda(y))-  f(y, u(y))\right)\mathrm dy\nonumber\\
    	&\geq& \int_{\Omega_\lambda^-} K^{\frac{n}{2}}_{\Omega}(x^\lambda,y^\lambda) \left(    f(y , u_\lambda(y))-  f(y, u(y))\right)\mathrm dy \\
    	&\geq&  \int_{\Omega_\lambda^-} K^{\frac{n}{2}}_{\Omega}(x^\lambda,y^\lambda)h_\lambda(y)\omega^{\lambda}(y) \mathrm dy\nonumber\\
    	&\geq& C \int_{\Omega_\lambda^-\,\cap\,B_{\frac{C_0 (\lambda+1)^2}{2}}(x)}\frac{(1+\lambda)^{2\varepsilon}}{|x-y|^{\varepsilon}} h_\lambda(y) \omega^{\lambda}(y)\mathrm dy  +C \int_{\Omega_\lambda^-\,\setminus\,B_{\frac{C_0(\lambda+1)^2}{2}}(x)}  h_\lambda(y) \omega^{\lambda}(y)\mathrm dy.\nonumber
    \end{eqnarray}

\medskip

    By \eqref{c-14}, Hardy-Littlewood-Sobolev inequality and H\"{o}lder inequality, we obtain, for arbitrary $\frac{n}{\varepsilon}<q<\infty$,
    \begin{eqnarray}\label{c-15}
    	{\| \omega^{\lambda}\|}_{L^q(\Omega_\lambda^-)}
    	&\leq& C (1+\lambda)^{2\varepsilon}{\left\| {\int_{\Omega_\lambda^-}\frac{1}{|x-y|^{\varepsilon}} h_\lambda \omega^{\lambda}(y)\mathrm dy}\right\|}_{L^q(\Omega_\lambda^-)}  +C |\Omega_\lambda^- |^{\frac{1}{q}}\int_{\Omega_\lambda^- } h_\lambda(y) |\omega^{\lambda}(y)|\mathrm dy \nonumber \\
    	&\leq& C (1+\lambda)^{2\varepsilon}{\left\|  h_\lambda  \omega^{\lambda}\right\|}_{L^{\frac{nq}{n+(n-\varepsilon)q}}(\Omega_\lambda^- )}    +C |\Omega_\lambda^- |^{\frac{1}{q}}\int_{\Omega_\lambda^- } h_\lambda(y) |\omega^{\lambda}(y)|\mathrm dy\\
    	&\leq& C (1+\lambda)^{2\varepsilon} {\left\| h_\lambda\right\|}_{L^{\frac{n}{n- \varepsilon}}(\Omega_\lambda^-)}{\left\| \omega^{\lambda}\right\|}_{L^{q}(\Omega_\lambda^-)}  +C |\Omega_\lambda^-|^{\frac{1}{q}}{\left\| h_\lambda \right\|}_{L^{\frac{q}{q-1}}(\Omega_\lambda^-)}{\left\| \omega^{\lambda}\right\|}_{L^{q}(\Omega_\lambda^-)}.\nonumber
    \end{eqnarray}
     Since $h_{d_{\Omega}+1} \in L^{1+\delta} (\Omega\cap B_R(0)) \  (\forall\, R>0)$ with arbitrarily small $\delta>0$. We first choose $\varepsilon>0$ sufficiently small such that $\frac{n }{n- \varepsilon}\leq 1+\delta$, then choose $q>\frac{n}{ \varepsilon}$ sufficiently large such that $\frac{q }{q-1}\leq 1+\delta$.   Therefore, there exists $\varepsilon_0\in(0,1)$ sufficiently small such that $|\Omega\bigcap B_{\lambda}(P)|$ is small enough and hence
    \begin{equation}\label{c-16}
    	\begin{split}
    	&\quad C(1+\lambda)^{2\varepsilon}{\left\| h_\lambda\right\|}_{L^{\frac{n}{n- \varepsilon}}(\Omega_\lambda^-)}  +C |\Omega_\lambda^-|^{\frac{1}{q}}{\left\| h_\lambda \right\|}_{L^{\frac{q}{q-1}}(\Omega_\lambda^-)} \\
    	&\leq C(1+\lambda)^{2\varepsilon}{\left\| h_{d_{\Omega}+1}\right\|}_{L^{\frac{n}{n- \varepsilon}}(\Omega_\lambda^-)}  +C |\Omega_\lambda^-|^{\frac{1}{q}}{\left\| h_{d_{\Omega}+1} \right\|}_{L^{\frac{q}{q-1}}(\Omega_\lambda^-)} <\frac{1}{2}\end{split}
    \end{equation}
    for any $d_{\Omega}<\lambda<d_{\Omega}+\varepsilon_0$. Thus, by \eqref{c-15}, we must have
    \begin{equation}\label{c-17}
    	{\| \omega^{\lambda}\|}_{L^q(\Omega_\lambda^-)}=0
    \end{equation}
    for any $d_{\Omega}<\lambda<d_{\Omega}+\varepsilon_0$. From the continuity of $\omega^{\lambda}$ in $\overline{\Omega}\setminus\{0\}$ and the definition of $\Omega_\lambda^-$, we immediately get   $\Omega_\lambda^-=\emptyset$ and hence \eqref{c-10} holds true for any $d_{\Omega}<\lambda<d_{\Omega}+\varepsilon_0$. This completes Step 1.

\medskip

    \emph{Step 2. Dilate the sphere piece $S_{\lambda}(P)\bigcap\Omega$ toward the interior of $\Omega$ until $\lambda=+\infty$.} Step 1 provides us a starting point to dilate the sphere piece $S_{\lambda}(P)\bigcap\Omega$ from near $\lambda=d_{\Omega}$. Now we dilate the sphere piece $S_{\lambda}(P)\bigcap\Omega$ toward the interior of $\Omega$ as long as \eqref{c-10} holds. Let
    \begin{equation}\label{c-18}
    	\lambda_{0}:=\sup\{\lambda>d_{\Omega}\,|\, \omega^{\mu}\geq0 \,\, \text{in} \,\, \Omega\cap B_{\mu}(0), \,\, \forall \, d_{\Omega}<\mu\leq\lambda\}\in(d_{\Omega},+\infty],
    \end{equation}
    and hence, one has
    \begin{equation}\label{c-19}
    	\omega^{\lambda_{0}}(x)\geq0, \quad\quad \forall \,\, x\in \Omega\cap B_{\lambda_{0}}(0).
    \end{equation}

\medskip

    In what follows, we will prove $\lambda_{0}=+\infty$ by driving a contradiction under the assumption that $\lambda_0<+\infty$.

\medskip

    In fact, suppose $\lambda_0<+\infty$, we must have
    \begin{equation}\label{c-20}
    	\omega^{\lambda_{0}}(x)\equiv 0, \quad\quad \forall \,\, x\in \Omega\cap B_{\lambda_{0}}(0).
    \end{equation}
    Suppose on the contrary that \eqref{c-20} does not hold, that is, there exists an $x_0 \in \Omega\bigcap B_{\lambda_{0}}(0)$ such that $\omega^{\lambda_{0}}(x_0)>0$. Then, by \eqref{c-14} and \eqref{c-19}, we have
    \begin{equation}\label{c-21}
    	\omega^{\lambda_{0}}(x)>0, \quad\quad \forall \,\, x\in \Omega\cap B_{\lambda_{0}}(0).
    \end{equation}

\medskip

    Let $h_{ \lambda_0 +1}$ be the function such that
    \begin{equation*}
    	\sup_{\alpha,\beta\in [0, M_{ \lambda_0+1}] } \frac{|f(x, \alpha)- f(x, \beta)|}{|\alpha-\beta|}\leq h_{\lambda_0+1}(x)
    \end{equation*}
    with $M_{\lambda_0+1}= \sup\limits_{\Omega\bigcap B_{\lambda_0+1}(0)} u<+\infty$ and  $h_{\lambda_0+1} \in L^{\frac{n}{2s}} (\Omega\bigcap B_R(0)) \   (\forall\, R>0)$.  In particular, $h_{\lambda_0+1} \in L^{\frac{n}{2s}} (\Omega\bigcap B_{\lambda_0+1}(0))$. So there is a constant $\delta_0>0$ such that, for any subset $D\subset \Omega\bigcap B_{\lambda_0+1}(0)$,
    \begin{equation}\label{c-21-1}
    	|D|\leq \delta_0 \ \  \ \text{implies }\  \	C (2+\lambda_{0})^{2\varepsilon}{\left\| h_{\lambda_{0}+1}\right\|}_{L^{\frac{n}{n- \varepsilon}}(D)}  +C |B_{\lambda_0+1}|^{\frac{1}{q}}{\left\| h_{\lambda_{0}+1} \right\|}_{L^{\frac{q}{q-1}}(D)}\leq\frac{1}{2},
    \end{equation}
    where $C$ is the constant in \eqref{c-15}.

\medskip

    Choose $\varepsilon_1\in (0,1)$ sufficiently small such that  the narrow region $A_{\varepsilon_1}\subset  \Omega\bigcap B_{\lambda_{0}}(0)$ satisfies $|A_{\varepsilon_1}|<\frac{ \delta_0 }{2}$, where $A_{\varepsilon_1}$ is defined by
    \begin{equation}\label{c-22}
    	A_{\varepsilon_1}:=\{x\in \Omega\cap B_{\lambda_{0}}(0)\,|\, \text{dist}(x,\partial (\Omega\cap B_{\lambda_{0}}(0)))<\varepsilon_1\}.
    \end{equation}

\medskip

   Since that $\omega^{\lambda_{0}}$ is continuous in $\overline{\Omega}\setminus\{0\}$ and $ K_{\varepsilon_1}:=(\Omega\bigcap B_{\lambda_{0}}(0))\setminus A_{\varepsilon_1} $ is a compact subset, there exists a positive constant $C_{\varepsilon_1}$ such that
    \begin{equation}\label{c-23}
    		\omega^{\lambda_{0}}(x)>C_{\varepsilon_1}, \quad\quad \forall \,\, x\in K_{\varepsilon_1}.
    \end{equation}
    By continuity, we can choose $0<\varepsilon_2<\varepsilon_1$ sufficiently small such that, for any $\lambda\in [\lambda_0,\,\lambda_0+\varepsilon_2]$,
    \begin{equation}\label{c-24}
    	\omega^{\lambda}(x)>\frac{C_{\varepsilon_1}}{2}, \quad\quad \forall \,\, x\in K_{\varepsilon_1},
    \end{equation}
as well as
\begin{equation}\label{c-24-1}
	|B_{\lambda}(0)\setminus B_{\lambda_{0}}(0)|\leq \frac{ \delta_0 }{2}.
\end{equation}
    Hence the set where $\omega^\lambda$ is negative must be contained in a sufficiently narrow region, that is,
    \begin{equation}\label{c-25}
    	\Omega_\lambda^-\subset( \Omega\cap B_{\lambda}(0))\setminus K_{\varepsilon_1}=\left((B_{\lambda}(0)\setminus B_{\lambda_{0}}(0))\cap \Omega\right)\cup A_{\varepsilon_{1}}
    \end{equation}
    for any $\lambda\in [\lambda_0,\,\lambda_0+\varepsilon_2]$. By   \eqref{c-21-1}, \eqref{c-24-1} and $|A_{\varepsilon_1}|<\frac{ \delta_0 }{2}$, we get, for any $\lambda\in [\lambda_0,\,\lambda_0+\varepsilon_2]$,
    \begin{equation}\label{c-26}\begin{split}
    	C &(1+\lambda)^{2\varepsilon}{\left\| h_\lambda\right\|}_{L^{\frac{n}{n- \varepsilon}}(\Omega_\lambda^-)}  +C |\Omega_\lambda^-|^{\frac{1}{q}}{\left\| h_\lambda \right\|}_{L^{\frac{q}{q-1}}(\Omega_\lambda^-)}\\
    	&\leq C (2+\lambda_{0})^{2\varepsilon}{\left\| h_{\lambda_0+1}\right\|}_{L^{\frac{n}{n- \varepsilon}}(\Omega_\lambda^-)}  +C |B_{\lambda_0+1}|^{\frac{1}{q}}{\left\| h_{\lambda_0+1} \right\|}_{L^{\frac{q}{q-1}}(\Omega_\lambda^-)} <\frac{1}{2}.\end{split}
    \end{equation}
    Then, by the same argument as in deriving \eqref{c-17} in step 1, we obtain that, for any $\lambda\in[\lambda_0,\,\lambda_0+\varepsilon_2]$,
    \begin{equation}\label{c-27}
    	\omega^{\lambda}(x)\geq 0, \quad\quad \forall \,\, x\in  \Omega\cap B_\lambda(0).
    \end{equation}
    This contradicts the definition of $\lambda_0$ and hence \eqref{c-20} must hold true.

\medskip

    However, by \eqref{c-9} and  \eqref{c-20}, we obtain
    \begin{eqnarray*}
    	0&=&\omega^{\lambda_{0}}(x)=u_{\lambda_0}(x)-u(x)\\
    	&>& \int_{\Omega\cap B_{\lambda_{0}}(0)}\left(K^{\frac{n}{2}}_{\Omega}(x^{\lambda_{0}},y^{\lambda_{0}})-K^{\frac{n}{2}}_{\Omega}(x,y^{\lambda_{0}})\right) \left(f(y , u_{\lambda_{0}}(y))-  f(y, u(y))\right)\mathrm dy\nonumber=0\nonumber
    \end{eqnarray*}
    for any $x\in\, \Omega\cap B_{\lambda_{0}}(0) $, which is absurd. Thus we must have $\lambda_0=+\infty$, that is,
    \begin{equation}\label{c-28}
    	u(x)\geq u\left(\frac{\lambda^{2}x}{|x|^{2}}\right), \quad\quad \forall \,\, x\in\Omega \,\,\,\text{s.t.}\,\,|x|\geq\lambda, \,\,\, \quad \forall \,\, d_{\Omega}<\lambda<+\infty.
    \end{equation}

\medskip

     \emph{Step 3. Derive lower bound estimates on asymptotic behavior of $u$ via Bootstrap technique in a smaller cone.}
     Let $\Sigma$ be the cross-section in $(\mathbf{f_{3}})$ and $\mathbf{(H_{2})}$ such that $ \mathcal{C}^{R,e}_{0,\Sigma} \subset\Omega$ for some $R>d_{\Omega}\geq0$. We infer from \eqref{c-28} that
     \begin{equation}\label{c-29}
     	u(x)\geq  u\left(\frac{\lambda^{2}x}{|x|^{2}}\right), \quad\quad \forall \,\, x\in \mathcal{C}^{R,e}_{0,\Sigma} \,\,\text{with}\,\, |x|\geq\lambda, \quad \forall \,\,  \lambda>R.
     \end{equation}
     For arbitrary $ x\in  \mathcal{C}^{R+1,e}_{0,\Sigma}$, take $0<\lambda:=\sqrt{(R+1) |x| }<|x|$, then \eqref{c-29} yields that
     \begin{equation}\label{c-30}
     	u(x)\geq  u\left(\frac{(R+1) x}{|x|}\right).
     \end{equation}
     Then we arrive at the following lower bound estimate on asymptotic behavior of $u$ as $|x|\rightarrow+\infty$:
     \begin{equation}\label{c-31}
     	u(x)\geq \min_{|x|=R+1 , \frac{x}{|x|}\in \Sigma}u(x)=:\  C_{0}>0, \quad\quad \forall \,\,x\in  \mathcal C^{R+1,e}_{0, \Sigma}.
     \end{equation}

\medskip

     By using a ``Bootstrap" iteration technique, the lower bound estimate \eqref{c-31} can be improved remarkably when $f$ satisfies $(\mathbf{f_{3}})$ with $1\leq p<+\infty$. Let $R_\ast:=R_0+R+1$, where $R_0$ is the constant in $(\mathbf{H_{2}})$. Set $\mu_{0}:=0$, we infer from  the integral equation \eqref{I-1}, assumptions $(\mathbf{H_{2}})$ and $(\mathbf{f_{3}})$ that, for any $x\in \mathcal C^{R_\ast,e}_{0, \Sigma}$,
     \begin{eqnarray*}
     	u(x)&\geq& C\int_{y\in   \mathcal C^{R_\ast,e}_{0, \Sigma}, \, \sigma_1|x|\leq |x-y|\leq \sigma_2|x|}K^s_{\Omega} (x,y)|y|^a u^p(y)\mathrm dy \\
     	\nonumber &\geq&C\int_{y\in   \mathcal C^{R_{\ast},e}_{0, \Sigma}, \, \sigma_1|x|\leq |x-y|\leq \sigma_2|x|}|y|^a u^p(y) \mathrm  dy  \\
     	\nonumber  &\geq& C_{1} |x|^{p\mu_{0}+(n+a)}.
     \end{eqnarray*}
     Now, let $\mu_{1}:=p\mu_{0}+(n+a)$. Due to $1\leq p<+\infty$ and $a>-n$, our important observation is
     \begin{equation*}
     	\mu_{1}=p\mu_{0}+(n+a)>\mu_{0}.
     \end{equation*}
     Thus we have obtained a better lower bound estimate than \eqref{c-31} after an iteration, that is,
     \begin{equation}\label{c-32}
     	u(x)\geq C_{1} |x|^{\mu_{1}}, \quad\quad \forall \,\, x\in \mathcal C^{R_{\ast},e}_{0, \Sigma}.
     \end{equation}

\medskip

     For $k=0,1,2,\cdots$, define
     \begin{equation}\label{c-33}
     	\mu_{k+1}:=p\mu_{k}+(n+a).
     \end{equation}
     Since $1\leq p+\infty$ and $a>-n$, it is easy to see that the sequence $\{\mu_{k}\}$ is monotone increasing with respect to $k$. Repeating the above iteration process involving the integral equation \eqref{I-1}, we have the following lower bound estimates for every $k=0,1,2,\cdots$,
     \begin{equation}\label{c-34}
     	u(x)\geq C_{k} |x|^{\mu_{k}}, \quad\quad \forall \,\, x\in  \mathcal C^{R_\ast,e}_{0, \Sigma}.
     \end{equation}
     Using the definition \eqref{c-33} and the assumption $a>-n$, one can verify that as  $k\rightarrow+\infty$,
     \begin{equation}\label{c-35}
     	\mu_{k}\rightarrow+\infty \quad \text{provided that } \,\,\, 1\leq p<+\infty,
     \end{equation}
     which combined with \eqref{c-34} imply the desired estimates \eqref{I-2} in Theorem \ref{lowerIE-1} for the critical order case $2s=n$.

\medskip

The proof of Theorem \ref{lowerIE-1} is thus completed.
\end{proof}	

\medskip

	Now with the lower bound estimates \eqref{I-2} in Theorem \ref{lowerIE-1} in hand, we are ready to prove Theorem \ref{ie-unbd}.
	
\begin{proof}[\textbf{Proof of Theorem \ref{ie-unbd}}]	
	The assumption $(\mathbf{H_{2}})$ provides a fixed point $x_0\in\Omega$ such that, for some $\theta\in\mathbb{R}$ and constant $\hat R>d_{\Omega}$, it holds
	\begin{equation}\label{I-35}
		K^{s}_{\Omega}(x_0,y)\geq \frac{C}{|y-P|^{\theta}},\ \qquad  \forall \,\, y\in \mathcal C^{\hat R,e}_{P, \Sigma}.
	\end{equation}
	Then we infer from $\mathbf{(f_{3})}$, \eqref{I-1}, \eqref{I-2} with $\kappa=\frac{\theta-a}{p}$ and \eqref{I-35} that
	\begin{equation*}
		+\infty>u(x_0)\geq   \int_{\mathcal C^{\max\{R_\ast, \hat R\},e}_{P, \Sigma}} K^{s}_{\Omega}(x_0,y) f(y, u(y)) \mathrm dy\geq \int_{\mathcal C^{\max\{R_\ast, \hat R\},e}_{P, \Sigma}} C \mathrm dy=+\infty,
	\end{equation*}
	which is absurd. Hence we must have $u\equiv0$ and $f(\cdot,0)\equiv0$. This concludes our proof of Theorem \ref{ie-unbd}.
\end{proof}

\medskip

\subsection{MSS in local way and Proof of Theorem \ref{ie-bd}}

Similar as the proof of Theorems \ref{pde-bd} and \ref{pde-bd-2}, Theorem \ref{ie-bd} $(i)$ follows easily from Theorem \ref{ie-unbd} and Kelvin transforms, so we omit the details.

\medskip

Next, we shall give the proof of Theorem  \ref{ie-bd} $(ii)$ via the method of scaling spheres in local way.

\medskip

Let $\Omega\subset \mathbb R^n$ is a bounded g-radially convex domain with $P\in \overline{\Omega}$ as the center. Recall that we consider the integral equation \eqref{IE} in  $\Omega \subset\mathbb{R}^n$, that is,
\begin{equation}\label{b-1}
	u(x)=\int_{\Omega}K^{s}_{\Omega}(x,y)f(y,u(y))\mathrm dy,
\end{equation}
where $u\in C(\overline{\Omega}\setminus\{P\})$, $n\geq1$ and $s\in\left(0,\frac{n}{2}\right)$, the integral kernel $K^{s}_{\Omega}$ satisfies the assumptions $(\mathbf{H_{1}})$, $(\mathbf{\widetilde{H}_{2}})$ and $(\mathbf{\widetilde{H}_{3}})$, $f\geq0$ is critical or supercritical in the sense that $\mu^{\frac{n+2s}{n-2s}}f(\mu^{\frac{2}{n-2s}}(x-P)+P,\mu^{-1}u)$ is non-increasing w.r.t. $\mu\geq1$ or $\mu\leq1$ for all $(x,u)\in(\Omega\setminus\{P\})\times\mathbb{R}_{+}$ as well as  satisfies the assumptions $(\mathbf{f_{1}})$, $(\mathbf{\widetilde{f}_{2}})$ and $(\mathbf{f_{3}})$ with the same cross-section $\Sigma$ as $(\mathbf{\widetilde{H}_{2}})$. Moreover, the integral kernel $K^{s}_{\Omega}(x,y)=0$ if $x$ or $y\notin\Omega$.

\medskip

We aim to prove that $u\equiv0$ in $\Omega$. Suppose on the contrary that $u$ is a nonnegative continuous solution of IE \eqref{b-1} but $u\not\equiv0$, then there is a point $x_0\in \Omega $ such that $0<u(x_0)= \int_{\Omega}K^{s}_{\Omega}(x_0,y)f(y,u(y)) \mathrm{d}y$. So we find that $f(\cdot,u(\cdot)) >0$ on a subset of $\Omega$ with positive Lebesgue measure and hence we get $$u(x)>0, \qquad  \ \forall\,\, x\in\Omega$$ by using \eqref{b-1}. We will derive a contradiction via the method of scaling spheres in local way. The key step is to establish the following \emph{lower bound estimates on the asymptotic behavior} of $u$.
\begin{thm}\label{lowerIE-2}
	Suppose that $u\in C({\overline {\Omega}}\setminus\{P\})$ is a positive solution to IEs \eqref{b-1}. Then it satisfies the following lower bound estimates:
	\begin{itemize}
		\item [(i).] If $f$ satisfies $(\mathbf{ {f}_{3}})$ with only $p=p_{c}(a):=\frac{n+2s+2a}{n-2s}$, then
		\begin{equation}\label{lb-1}
			u(x)\geq C |x-P|^{-\frac{n-2s}{2}},  \quad\quad \forall\,\, x\in \mathcal{C}^{\tilde r}_{P,\Sigma}.
		\end{equation}
		\item [(ii).]  If $f$ satisfies $(\mathbf{ {f}_{3}})$ with  $p_{c}(a):=\frac{n+2s+2a}{n-2s}<p<+\infty$, then 	\begin{equation}\label{lb-2}
			u(x)\geq C_{\kappa}|x-P|^{-\kappa}, \quad\quad \forall \, \kappa<+\infty, \quad\quad \forall\,\, x\in \mathcal{C}^{\tilde r}_{P,\Sigma}.
		\end{equation}
	\end{itemize}
Here $0<\tilde r<\rho_\Omega:=\max\limits_{x\in\overline{\Omega}}|x-P|$ is a small radius and $\Sigma\subseteq\Sigma^{\tilde r}_{\Omega}$ is the cross-section in assumptions  $(\mathbf{\widetilde{H}_{2}})$ and $(\mathbf{ {f}_{3}})$. 	
\end{thm}
\begin{proof}
Without loss of generalities, we may assume that $P=0$ and $\rho_{\Omega}:=\max\limits_{x\in\overline{\Omega}}|x-P|=1$. Given any $0<\lambda<1$, we define $x^{\lambda}:=\frac{\lambda^{2}x}{|x|^{2}}$ and the Kelvin transform of $u$ centered at $0$ by
\begin{equation}\label{b-2}
	u_{\lambda}(x)=\left(\frac{\lambda}{|x|}\right)^{n-2s}u\left(\frac{\lambda^{2}x}{|x|^{2}}\right).
\end{equation}

\medskip

Now, we will carry out the process of scaling spheres in $\Omega$ with respect to the center $P=0\in\mathbb{R}^{n}$.

\medskip

Let $0<\lambda<1$ be an arbitrary positive real number and let $\omega^{\lambda}(x):=u_{\lambda}(x)-u(x)$. We aim to show that
\begin{equation}\label{b-4}
	\omega^{\lambda}(x)\leq0, \quad\,\,\,\,\, \forall \,\, x\in ( \Omega\setminus  B_{\lambda}(0) )^\lambda:=\{x\in\Omega |\, x^\lambda\in \Omega\setminus  B_{\lambda}(0) \}.
\end{equation}

\medskip

Note that the g-radial convexity of $\Omega$ implies $$( \Omega\setminus  B_{\lambda}(0) )^\lambda\subset \Omega\cap B_{\lambda}(0).$$
We first prove \eqref{b-4} for $0<\lambda<\rho_{\Omega}=1$ sufficiently close to $1$. Then, we start shrinking the sphere piece $S_{\lambda}\bigcap\Omega$ from near the unit sphere $S_{1}$ toward the interior of $\Omega$ as long as \eqref{b-4} holds, until its limiting position $\lambda=0$ and derive lower bound estimates on asymptotic behavior of $u$ as $x\rightarrow P=0$. Therefore, the method of scaling spheres in local way will be carried out by three steps.

\medskip

\emph{Step 1. Start shrinking the sphere piece $S_{\lambda}\bigcap\Omega$ from near $\lambda=\rho_{\Omega}=1$.} Define
\begin{equation}\label{b-5}
	\Omega_\lambda ^{+}:=\{x\in ( \Omega\setminus  B_{\lambda}(0) )^\lambda \, | \, \omega^{\lambda}(x)>0\}.
\end{equation}
We will show that, for $0<\lambda<1$ sufficiently close to $1$,
\begin{equation}\label{b-6}
	\Omega_\lambda^{+}=\emptyset.
\end{equation}

\medskip

Since the positive solution $u$ solves the integral equation \eqref{b-1}, through direct calculations, we get, for any $0<\lambda<1$,
\begin{equation}\label{b-7}\begin{split}
		u(x)&=\int_{\Omega\cap B_{\lambda}(0) }K^{s}_{\Omega}(x,y) f(y, u(y))\mathrm dy\\
		&+\int_{( \Omega\setminus  B_{\lambda}(0) )^\lambda}\left(\frac{\lambda}{|y|}\right)^{n-2s}K^{s}_{\Omega}(x,y^{\lambda})
		\left(\frac{\lambda}{|y|}\right)^{n+2s} f\left(y^\lambda, \left(\frac{\lambda}{|y|}\right)^{ 2s-n} u_{\lambda}(y) \right)\mathrm dy.\end{split}
\end{equation} By similar calculations, one can also verify that $u_{\lambda}$ satisfies
\begin{align}\label{b-8}
	u_{\lambda}(x)&=\left(\frac{\lambda}{|x|}\right)^{n-2s} \int_{\Omega}K^{s}_{\Omega}(x^{\lambda},y)f(y, u(y))\mathrm  dy\nonumber\\
	&=\int_{\Omega\cap B_{\lambda}(0) }K^{s}_{\Omega} (x^{\lambda},y)\left(\frac{\lambda}{|x|}\right)^{n-2s}f(y, u(y)) \mathrm  dy \\
	\nonumber   &+\int_{( \Omega\setminus  B_{\lambda}(0) )^\lambda}\left(\frac{\lambda^{2}}{|x| |y|}\right)^{n-2s}K^{s}_{\Omega}(x^{\lambda},y^{\lambda})\left(\frac{\lambda}{|y|}\right)^{n+2s}
	f\left(y^\lambda, \left(\frac{\lambda}{|y|}\right)^{ 2s-n} u_{\lambda}(y) \right)\mathrm  dy.
\end{align}
Therefore, we have,  for any $x\in ( \Omega\setminus  B_{\lambda}(0))^\lambda$,
\begin{eqnarray}\label{b-9}
	&& \quad \omega^{\lambda}(x)=u_{\lambda}(x)-u(x) \\
	\nonumber &&= \int_{( \Omega\setminus  B_{\lambda}(0) )^\lambda}\left[\left(\frac{\lambda^{2}}{|x| |y|}\right)^{n-2s}K^{s}_{\Omega}(x^{\lambda},y^{\lambda})
	-\left(\frac{\lambda}{|y|}\right)^{n-2s}K^{s}_{\Omega}(x,y^{\lambda})\right]\\
	&&\quad \times   \left(\frac{\lambda}{|y|}\right)^{n+2s}
	f\left(y^\lambda, \left(\frac{\lambda}{|y|}\right)^{2s-n} u_{\lambda}(y) \right) \mathrm dy	\nonumber \\
	&& \quad -\int_{( \Omega\setminus  B_{\lambda}(0) )^\lambda}\left[K^{s}_{\Omega}(x,y)-\left(\frac{\lambda}{|x|}\right)^{n-2s}K^{s}_{\Omega}(x^{\lambda},y)\right] f(y, u(y))\mathrm dy \nonumber\\
	\nonumber && \quad + \int_{(\Omega\cap B_\lambda(0))\setminus ( \Omega\setminus B_{\lambda}(0))^\lambda}\left[\left(\frac{\lambda}{|x|}\right)^{n-2s}K^{s}_{\Omega}(x^{\lambda},y)-K^{s}_{\Omega}(x,y)\right] f(y, u(y))\mathrm dy.
\end{eqnarray}

Since $0<u(x)<u_\lambda(x)\leq M_\lambda $ for any $x\in \Omega_\lambda^+$ with $\sup\limits_{( \Omega\setminus  B_{\lambda}(0) )^\lambda} u_\lambda \leq M_{\lambda}:=\lambda^{2s-n}\sup\limits_{\Omega\setminus B_{\lambda}(0)} u<+\infty$, we derive from the assumptions $(\mathbf{f_{1}})$ and $(\mathbf{f_{2}})$ that
\begin{equation}\label{b-10}
	0\leq c_\lambda(x):=\frac{f(x, u_\lambda(x))- f(x, u(x))}{u_\lambda(x)-u(x)}\leq \sup_{\alpha,\beta\in [0, M_\lambda] } \frac{|f(x, \alpha)- f(x, \beta)|}{|\alpha-\beta|}\leq h_\lambda(x)
\end{equation}
for some $h_\lambda \in L^{\frac{n}{2s}} (\Omega\setminus B_\varepsilon(0))\  (\forall\, \varepsilon>0)$. Moreover, it can be seen that $h_\lambda$ can be chosen decreasingly depending on $\lambda$ since $M_\lambda$ is decreasing as $\lambda$ increases.

\medskip

From assumptions $\mathbf{(H_1)}$, $\mathbf{(\widetilde{H}_3)}$, the integral equations \eqref{b-9}, \eqref{b-10}, $\mathbf{(f_{1})}$ and the non-subcritical assumption on $f$, one can derive that, for any $x\in(\Omega\setminus  B_{\lambda}(0))^\lambda$,
\begin{eqnarray}\label{b-11}
	&&\quad \omega^{\lambda}(x)=u_{\lambda}(x)-u(x)  \\
	\nonumber &\leq&\int_{( \Omega\setminus  B_{\lambda}(0) )^\lambda}\Bigg[K^{s}_{\Omega}(x,y)-\left(\frac{\lambda}{|x|}\right)^{n-2s}K^{s}_{\Omega}(x^{\lambda},y)\Bigg]  \\
\nonumber &&\times\left[ \left(\frac{\lambda}{|y|}\right)^{n+2s}
	f\left(y^\lambda, \left(\frac{\lambda}{|y|}\right)^{2s-n} u_{\lambda}(y) \right)	- f(y, u(y))\right]  \mathrm dy\\
\nonumber && + \int_{(\Omega\cap B_\lambda(0))\setminus ( \Omega\setminus B_{\lambda}(0))^\lambda}\left[\left(\frac{\lambda}{|x|}\right)^{n-2s}K^{s}_{\Omega}(x^{\lambda},y)-K^{s}_{\Omega}(x,y)\right] f(y, u(y))\mathrm dy \\
	\nonumber &<&\int_{( \Omega\setminus  B_{\lambda}(0) )^\lambda}\Bigg(K^{s}_{\Omega}(x,y)-\left(\frac{\lambda}{|x|}\right)^{n-2s}K^{s}_{\Omega}(x^{\lambda},y)\Bigg) (f(y, u_\lambda(y))- f(y, u(y)))\mathrm dy\\
\nonumber &\leq&    \int_{\Omega_\lambda^{+}}K^{s}_{\Omega}(x,y)(f(y, u_\lambda(y))- f(y, u(y)))\mathrm dy \\
	\nonumber &\leq&    \int_{\Omega_\lambda^{+}}\frac{C }{|x-y|^{n-2s}}h_\lambda(y)\omega^{\lambda}(y)\mathrm dy.
\end{eqnarray}
The second strict inequality in \eqref{b-11} comes from the assumption that $f$ is supercritical in the sense of Definition \ref{defn1} or else $f(x,u)>0$ provided that $u>0$ for a.e. $x\in\Omega\setminus\{P\}$.

\medskip

By Hardy-Littlewood-Sobolev inequality and \eqref{b-11}, we have, for any $\frac{n}{n-2s}<q<\infty$,
\begin{eqnarray}\label{b-12}
	\|\omega^{\lambda}\|_{L^{q}(\Omega_\lambda^{+})} \leq  C  \left\|h_\lambda \omega^{\lambda}\right\|_{L^{\frac{nq}{n+2sq}}(\Omega_\lambda^{+})}
	\leq  C \left\|h_\lambda \right\|_{L^{\frac{n}{2s}}(\Omega_\lambda^{+})}
	\cdot\|\omega^{\lambda}\|_{L^{q}(\Omega_\lambda^{+})}.
\end{eqnarray}
Since $h_\lambda \in L^{\frac{n}{2s}} (\Omega\setminus B_\varepsilon(0)) \ (\forall\, \varepsilon>0)$ and $0\leq h_\lambda\leq h_{\frac{1}{2}}$ if $\lambda\geq\frac{1}{2}$, there exists an $\epsilon_{0}\in\left(0,\frac{1}{2}\right)$ small enough, such that $|( \Omega\setminus  B_{\lambda}(0) )^\lambda|$ is sufficiently small and hence
\begin{equation}\label{b-13}
	C \left\|h_\lambda \right\|_{L^{\frac{n}{2s}}(\Omega_\lambda^+)}\leq\frac{1}{2}
\end{equation}
for all $1-\epsilon_{0}\leq\lambda<1$, so \eqref{b-12} implies
\begin{equation}\label{b-14}
	\|\omega^{\lambda}\|_{L^{q}(\Omega_\lambda^+)}=0,
\end{equation}
which means $\Omega_\lambda^+=\emptyset$. Therefore, we have proved that $\Omega_\lambda^+=\emptyset$ for all $1-\epsilon_{0}\leq\lambda<1$, that is,
\begin{equation}\label{b-15}
	\omega^{\lambda}(x)\leq0, \,\,\,\,\,\,\, \forall \, x\in ( \Omega\setminus  B_{\lambda}(0) )^\lambda.
\end{equation}
This completes Step 1.

\medskip

\emph{Step 2. Shrink the sphere piece $S_{\lambda}\bigcap\Omega$ toward the interior of $\Omega$ until $\lambda=0$.} Step 1 provides us a starting point to shrink the sphere piece $S_{\lambda}\bigcap\Omega$ from near $\lambda=\rho_{\Omega}=1$. Now we shrink the sphere piece $S_{\lambda}\bigcap\Omega$ toward the interior of $\Omega$ as long as \eqref{b-4} holds. Let
\begin{equation}\label{b-16}
	\lambda_{0}:=\inf\{0<\lambda<1\,|\, \omega^{\mu}\leq0 \,\, \text{in} \,\, ( \Omega\setminus  B_{\mu}(0) )^\mu, \,\, \forall \, \lambda\leq\mu<1\}\in[0,1),
\end{equation}
and hence, one has
\begin{equation}\label{b-17}
	\omega^{\lambda_{0}}(x)\leq0, \quad\quad \forall \,\, x\in  ( \Omega\setminus  B_{\lambda_0}(0) )^{\lambda_0 }.
\end{equation}
In what follows, we will prove $\lambda_{0}=0$ by contradiction arguments.

\medskip

Indeed, suppose on the contrary that $\lambda_0>0$. It follows from \eqref{b-11}, \eqref{b-17} and $\mathbf{(f_{1})}$ that
\begin{equation}\label{b-18}
	\omega^{\lambda_{0}}(x)<0, \,\,\,\,\quad\, \forall \, x\in  ( \Omega\setminus  B_{\lambda_0}(0) )^{\lambda_0 }.
\end{equation}
Let $h_{\frac{\lambda_0}{2}}$ be the function such that
\begin{equation*}
	\sup_{\alpha,\beta\in \left[0, M_{\frac{\lambda_0}{2}}\right] } \frac{|f(x, \alpha)- f(x, \beta)|}{|\alpha-\beta|}\leq h_{\frac{\lambda_0}{2}}(x)
\end{equation*}
with $M_{\frac{\lambda_0}{2}}:= \left({\frac{\lambda_0}{2}}\right)^{2s-n}\sup\limits_{\Omega\setminus B_{\frac{\lambda_0}{2}}(0)} u<+\infty$ and  $h_{\frac{\lambda_0}{2}} \in L^{\frac{n}{2s}} (\Omega\setminus B_\varepsilon(0)) \  (\forall\, \varepsilon>0)$. In particular, $h_{\frac{\lambda_0}{2}} \in L^{\frac{n}{2s}}\left(\Omega\setminus B_{\frac{\lambda_0^2}{4}}(0)\right)$. So there is a constant $\varepsilon_0>0$ such that, for any subset $D\subset\Omega\setminus B_{\frac{\lambda_0^2}{4}}(0)$,
\begin{equation}\label{b-19}
	|D|\leq \varepsilon_0 \quad\,\, \text{implies} \,\,\quad\,\, C \left\|h_{\frac{\lambda_0}{2}} \right\|_{L^{\frac{n}{2s}}(D)}\leq\frac{1}{2},
\end{equation}
where $C$ is the constant in \eqref{b-12}.

\medskip

We choose a small constant $\varepsilon_1\in (0, \frac{\lambda_0}{2})$ and  compact set $K_0\subset ( \Omega\setminus  B_{\lambda_0}(0) )^{\lambda_0 }$ such that
\begin{equation}\label{b-20}
	\text{dist}(K_0, S_{\lambda_0})=\varepsilon_1>0 \quad \  \text{and}\quad\, | ( \Omega\setminus  B_{\lambda_0}(0) )^{\lambda_0 } \setminus K_0|\leq \frac{\varepsilon_0}{2}.
\end{equation}
Now, for any $\lambda\in[\lambda_{0}-\varepsilon_1,\lambda_{0}]$,  we have $K_0\subset ( \Omega\setminus  B_{\lambda}(0) )^{\lambda}$. So there exists a constant $\varepsilon_2\in (0, \varepsilon_1)$ such that, for any $\lambda\in[\lambda_{0}-\varepsilon_2,\lambda_{0}]$,
\begin{equation}\label{b-21}
	|( \Omega\setminus  B_{\lambda}(0) )^{\lambda }\setminus K_0|\leq  \varepsilon_0
\end{equation}
due to $| ( \Omega\setminus  B_{\lambda}(0) )^{\lambda }| \to |( \Omega\setminus  B_{\lambda_0}(0) )^{\lambda_0 }|$ as $\lambda\to\lambda_0$.
\begin{figure}[H]\label{F9}
	\begin{tikzpicture}[scale=15]
		\draw (0:0) -- (-15:1) arc (-15:15:1) (15:1) -- (0:0);	
		\filldraw[fill=gray!20] (-14:0.4)--(-14:0.77) arc (-14:14:0.77) (14:0.77)--(14:0.4) arc (14:-14:0.4);
		\filldraw[ draw=red!60, pattern color=red!60,  pattern=north east lines] (-15:0.36) arc (-15:15:0.36) --(15:0.4) arc (15:-15:0.4)--(-15:0.36);
		\filldraw[ draw=red!60, pattern color=red!60,  pattern=north east lines] (-15:0.77) arc (-15:15:0.77) --(15:0.785) arc (15:-15:0.785)--(-15:0.77);
		\filldraw[ draw=red!60, pattern color=red!60,  pattern=north east lines] (-15:0.4) arc (-15:-14:0.4) --(-14:0.77) arc (-14:-15:0.77)--(-15:0.4);
		\filldraw[ draw=red!60, pattern color=red!60,  pattern=north east lines] (15:0.4) arc (15:14:0.4) --(14:0.77) arc (14:15:0.77)--(15:0.4);
		\draw (-15:0.8) arc (-15:15:0.8)  ;
		\draw (-15:0.38) arc (-15:15:0.38)  ;
		\draw[dashed] (-15:0.35) arc (-15:15:0.35) (-15:0.77) arc (-15:15:0.77);
		\draw[bline] (-15:0.36) arc (-15:15:0.36)  (-15:0.785) arc (-15:15:0.785);
		\filldraw (0:0) circle (0.07pt) node[left]{$P$};
		\node[below] at(-15:1){1};
		\node[rotate=-15, below, font=\small] at(-15:0.81){$\lambda_0$};
		\node[rotate=15, above, font=\small] at(15:0.785){$\lambda$};
		\node[rotate=15, above, font=\small] at(15:0.36){$\lambda^2$};
		\node[rotate=-15,below, font=\small] at(-15:0.745){$\lambda_0-\varepsilon_1$};
		\node[rotate=-15, below, font=\small] at(-15:0.385){$\lambda_0^2$};
		\node[rotate=-15,below, font=\small] at(-15:0.32){$(\lambda_0 -\varepsilon_1)^2$};
		\node[right] at (0:1.02) {$\mathcal C_{P,\Sigma}^1$};
		\node[right] at (0:0.55) {$K_0$};
		\node[right] at (90:0.25) {$\Omega=\mathcal C_{P,\Sigma}^1, \  \ ( \Omega\setminus  B_{\lambda}(P) )^{\lambda}=(B_\lambda(P)\setminus B_{\lambda^2}(P))\cap \mathcal C_{P,\Sigma}^1$};
		\filldraw[ draw=red!60, pattern color=red!60,  pattern=north east lines] (0.01,0.192)--(0.04,0.192)--(0.04,0.212)--(0.01,0.212)--(0.01,0.192);
		\node[right] at (0.04, 0.2) {$( \Omega\setminus  B_{\lambda}(P) )^{\lambda}\setminus K_0$};
		\draw[->] (20:0.2)--(0.375,0) ;\node[above] at (30:0.12) {The narrow region};
	\end{tikzpicture}
	\caption{Narrow regions when $\Omega$ is a bounded cone}
\end{figure}
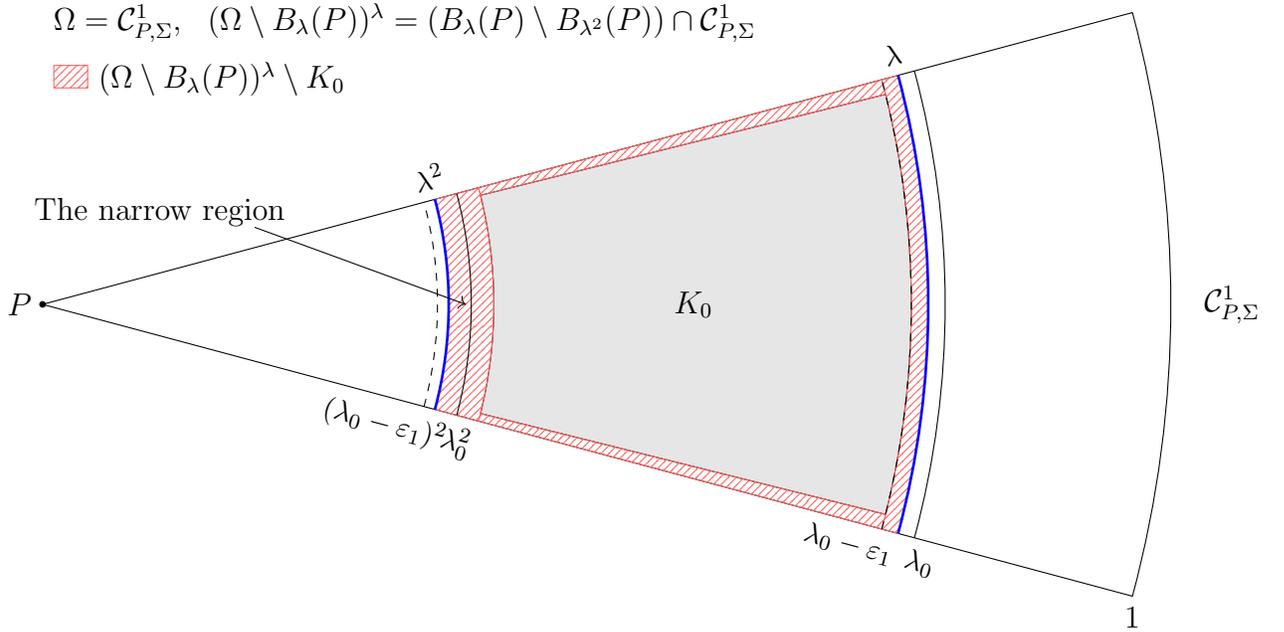

\medskip

By \eqref{b-18}, the continuity of $u$ and the compactness of $K_0$, there is a constant $c_0<0$ such that
\begin{equation*}
	\omega^{\lambda_{0}}(x)\leq c_0<0, \,\,\,\,\,\quad \forall \, x\in K_0.
\end{equation*}
Then by the continuity of $\omega^\lambda$ with respect to $\lambda$, there exists a constant  $\varepsilon_3\in (0, \varepsilon_2)$ such that
\begin{equation*}
	\omega^{\lambda}(x)\leq \frac{c_0}{2}<0, \,\quad \,\,\,\, \forall \, x\in K_0,\,\, \forall \,\,\lambda\in [\lambda_{0}-\varepsilon_3,\lambda_{0}],
\end{equation*}
which implies that $\Omega_\lambda^+\subset ( \Omega\setminus  B_{\lambda}(0) )^{\lambda }\setminus K_0$ and hence
\begin{equation}\label{b-22}
	|\Omega_\lambda^+|\leq \varepsilon_0,\,\quad\,\, \forall \,\,\lambda\in [\lambda_{0}-\varepsilon_3,\lambda_{0}]
\end{equation}
due to \eqref{b-21}. By \eqref{b-11}, one can easily verify that inequality as \eqref{b-12} (with the same constant $C$) also holds for any $\lambda\in[\lambda_{0}-\varepsilon_3,\lambda_{0}]$. On the other hand, by \eqref{b-19} and \eqref{b-22}, we find
\begin{equation}\label{b-23}
	C \left\|h_\lambda \right\|_{L^{\frac{n}{2s}}(\Omega_\lambda^+)}\leq C \left\|h_{\frac{\lambda_0}{2}} \right\|_{L^{\frac{n}{2s}}(\Omega_\lambda^+)} \leq\frac{1}{2},\,\,\quad\, \forall \,\,\lambda\in [\lambda_{0}-\varepsilon_3,\lambda_{0}],
\end{equation}
which together with \eqref{b-12} show that $\Omega_\lambda^+=\emptyset$. That is,
\begin{equation}\label{b-24}
	\omega^{\lambda}(x)\leq0, \,\quad\,\,\,\,\, \forall \,\, x\in  ( \Omega\setminus  B_{\lambda}(0) )^{\lambda}, \,\, \forall\,\, \lambda\in [\lambda_{0}-\varepsilon_3,\lambda_{0}],
\end{equation}
which contradicts the definition \eqref{b-16} of $\lambda_{0}$. As a consequence,  we must have $\lambda_{0}=0$, that is,
\begin{equation}\label{b-25}
	u(x)\geq\left(\frac{\lambda}{|x|}\right)^{n-2s}u\left(\frac{\lambda^{2}x}{|x|^{2}}\right), \quad\quad \forall \,\, x\in (\Omega\setminus B_{\lambda}(0))^{\lambda}, \,\quad \forall \,\, 0<\lambda<1.
\end{equation}

\medskip

\emph{Step 3. Derive lower bound estimates on asymptotic behavior of $u$ as $x\rightarrow P=0$ in a smaller cone via Bootstrap technique.}
Let the bounded cone $\mathcal{C}^{r_0}_{P,\Sigma}\subset\Omega$ with $r_{0}<\rho_\Omega=1$ be given in $(\mathbf{\widetilde{H}_{2}})$. For arbitrary $ x\in \mathcal{C}^{r_0}_{P,\Sigma}$ with $0<|x|<\frac{r_0}{2}$, let $\lambda:=\sqrt{ \frac{r_0|x|}{2} }$, then one has $x\in \Omega\cap B_\lambda(0)$ and $x^\lambda\in \mathcal{C}^{r_0}_{0,\Sigma}\subset \Omega$. So \eqref{b-25} yields that
\begin{equation}\label{b-26}
	u(x)\geq\left(\frac{r_0}{2|x|}\right)^{\frac{n-2s}{2}}u\left(\frac{r_0 x}{2|x|}\right),\ \quad  \ \forall \,\, x\in \mathcal{C}^{\frac{r_0}{2}}_{0,\Sigma}.
\end{equation}
Therefore, we arrive at the following lower bound estimate on asymptotic behavior of $u$ as $x\rightarrow P=0$:
\begin{equation}\label{b-27}
	u(x)\geq\left(\min_{|x|=\frac{r_0}{2},\,\, \frac{x}{|x|}\in \Sigma}u(x)\right)\left(\frac{r_0}{2|x|}\right)^{\frac{n-2s}{2}}=:C_{0} |x|^{-\frac{n-2s}{2}}, \quad\quad \forall \,\,x\in \mathcal{C}^{\frac{r_0}{2}}_{0,\Sigma}.
\end{equation}
This proves Theorem \ref{lowerIE-2} in the case that $f$ satisfies $(\mathbf{ {f}_{3}})$ with only  $p=p_{c}(a):=\frac{n+2s+2a}{n-2s}$.

\medskip

The lower bound estimate \eqref{b-27} can be improved remarkably when $f$ satisfies $(\mathbf{ {f}_{3}})$ with some $p_{c}(a):=\frac{n+2s+2a}{n-2s}<p<+\infty$ via the Bootstrap iteration technique and the integral equation \eqref{b-1}.

\medskip

In fact, let $\mu_{0}:=-\frac{n-2s}{2}$ and $\tilde r:=\min\left\{\frac{r_0}{2}, r\right\}$, where $r$ is the radius in $(\mathbf{ {f}_{3}})$, we infer from the integral equation \eqref{b-1}, $(\mathbf{\widetilde{H}_{2}})$, $(\mathbf{ {f}_{3}})$  and \eqref{b-27} that, for any $x\in \mathcal C_{0,\Sigma}^{\tilde r}$,
\begin{eqnarray}\label{b-28}
	u(x)&\geq&C\int_{\mathcal C_{0,\Sigma}^{\tilde r}, \, \sigma_1|y|\leq |x-y|\leq \sigma_2|y|}K^{s}_{\Omega}(x,y)f(y, u(y))\mathrm dy \\
\nonumber &\geq&C\int_{\mathcal C_{0,\Sigma}^{\tilde r}, \, \sigma_1|y|\leq |x-y|\leq \sigma_2|y|,\,\frac{|x|}{1+\sigma_{2}}\leq |y|\leq \frac{|x|}{1+\sigma_{1}}}K^{s}_{\Omega}(x,y)f(y, u(y))\mathrm dy \\
	\nonumber &\geq&C\int_{ \mathcal C_{0,\Sigma}^{\tilde r}, \, \frac{\sigma_1}{1+\sigma_{1}}|x|\leq |x-y|\leq \frac{\sigma_{2}}{1+\sigma_{2}}|x|, \, \frac{|x|}{1+\sigma_{2}}\leq |y|\leq \frac{|x|}{1+\sigma_{1}}}\frac{ |y|^{p\mu_{0}+a}}{|x-y|^{n-2s}} \mathrm  dy \\
	\nonumber  &\geq& C_{1} |x|^{p\mu_{0}+(2s+a)}.
\end{eqnarray}
Now, let $\mu_{1}:=p\mu_{0}+(2s+a)$. Due to $p_{c}(a):=\frac{n+2s+2a}{n-2s}<p<+\infty$, our important observation is
\begin{equation}\label{b-29}
	\mu_{1}:=p\mu_{0}+(2s+a)<\mu_{0}.
\end{equation}
Thus we have obtained a better lower bound estimate than \eqref{b-27} after one iteration, that is,
\begin{equation}\label{b-30}
	u(x)\geq C_{1} |x|^{\mu_{1}}, \quad\quad \forall \,\,x\in \mathcal C_{0,\Sigma}^{\tilde r}.
\end{equation}

\medskip

For $k=0,1,2,\cdots$, define
\begin{equation}\label{b-31}
	\mu_{k+1}:=p\mu_{k}+(2s+a).
\end{equation}
Since $p_{c}(a)=\frac{n+2s+2a}{n-2s}<p<+\infty$, it is easy to see that the sequence $\{\mu_{k}\}$ is monotone decreasing with respect to $k$. Continuing the above iteration process involving the integral equation \eqref{b-1}, we have the following lower bound estimates for every $k=0,1,2,\cdots$:
\begin{equation}\label{b-32}
	u(x)\geq C_{k}|x|^{\mu_{k}}, \quad\quad \forall \,\, x \in \mathcal C_{0,\Sigma}^{\tilde r}.
\end{equation}
Then Theorem \ref{lowerIE-2} follows easily from \eqref{b-32} and the obvious property that as $k\rightarrow+\infty$,
\begin{equation}\label{b-33}
	\mu_{k}\rightarrow-\infty, \qquad \text{if} \,\,\, \frac{n+2s+2a}{n-2s}<p<+\infty.
\end{equation}
This concludes our proof of Theorem \ref{lowerIE-2}.
\end{proof}

\medskip

	Thanks to the lower bound estimates \eqref{lb-1} and \eqref{lb-2} in Theorem \ref{lowerIE-2}, we are ready to prove the Liouville results in Theorem \ref{ie-bd}.

\medskip

\begin{proof}[\textbf{Proof of Theorem \ref{ie-bd}}]	
	
	 In the case that $(\mathbf{f_{3}})$ is only fulfilled by $f$  for $p=p_{c}(a)$, we derive a contradiction from the assumption $u(x)=o\left(\frac{1}{|x-P|^{\frac{n-2s}{2}}}\right)$ as $x\to P$ and the estimate \eqref{lb-1}.
	
\medskip

	In the cases that $(\mathbf{f_{3}})$ is fulfilled by $f$ for some $p>p_{c}$. The assumption $(\mathbf{\widetilde{H}_{2}})$ provides a fixed point $x_0\in\Omega$ such that, for some $\theta\in\mathbb{R}$ and small radius $\hat r>0$, one has
	\begin{equation}\label{b-34}
		K^{s}_{\Omega}(x_0,y)\geq C |y-P|^{ \theta},\ \qquad  \forall \,\, y\in \mathcal C^{\hat r}_{P, \Sigma}.
	\end{equation}
	Then we infer from \eqref{b-1}, \eqref{lb-2} with $\kappa=\frac{n+\theta+a}{p}$ and \eqref{b-34} that
	\begin{equation*}
		+\infty>u(x_0)\geq   \int_{\mathcal C^{\min\{\tilde r, \hat r\},e}_{P, \Sigma}} K^{s}_{\Omega}(x_0,y) f(y, u(y)) \mathrm dy\geq \int_{\mathcal C^{\min\{\tilde r, \hat r\},e}_{P, \Sigma}} \frac{C}{|y-P|^n} \mathrm dy=+\infty,
	\end{equation*}
	which is absurd. Hence we must have $u\equiv0$ and $f(\cdot,0)\equiv0$.
	This finishes our proof of Theorem \ref{ie-bd}.
\end{proof}

\medskip

\subsection{Properties of Green's function for Dirichlet problem of $(-\Delta)^{s}$ with $0<s\leq1$ on MSS applicable domains and proof of Theorem \ref{pG}}

To show Theorem \ref{pG}, we first prove some  properties of the Green function  in a unbounded cone $\mathcal{C}_{P,\Sigma}$ such that $\overline{\mathcal{C}_{P,\Sigma}}\neq\mathbb{R}^{n}$, the cross-section $\Sigma$ is $C^{1,1}$ if $s\in(0,1)$ and is Dirichlet-regular if $s=1$.

\medskip

In what follows, we will assume $P=0$ and  abbreviate $\mathcal{C}_{0,\Sigma}$ to $\mathcal{C}$ for simplicity when there is no confusions.

\medskip

We need the following classification result for $s$-harmonic functions in cones $\mathcal{C}$, which was proved in \cite{A, BaBo}, see also \cite{TTV}.
\begin{thm}[\cite{A, BaBo}]\label{RMI-h}
	For $s\in(0, 1]$, all non-negative $s$-harmonic functions in $\mathcal C$ with zero Dirichlet boundary condition are the functions $h$ of the form
	\begin{equation*}
		h(x)=C |x|^{\gamma_s( \mathcal C)}\varphi_s\left(\frac{x}{|x|}\right),
	\end{equation*}
where $\gamma_s( \mathcal C)$ is a nonnegative constant depending on $n$, $s$, the cross-section $\Sigma$ of $\mathcal C$ ($\gamma_{s}(\mathcal{C})=0$ if $\Sigma=S^{n-1}$, $\gamma_{1}(\mathcal{C})=k$ if $\Sigma=S^{n-1}_{2^{-k}}$ with $k=1,\cdots,n$), independent of the vertex of the cone $\mathcal{C}$ and decreasing w.r.t. the inclusion relationship of the cross-section $\Sigma$ (i.e., $\gamma_s( \mathcal C_{P,\Sigma})\leq\gamma_s( \mathcal C_{Q,\Sigma_{1}})$ provided that $\Sigma_{1}\subseteq\Sigma$), and $\varphi_s$ is a positive function on $\Sigma$ vanishing on $S^{n-1}\setminus\Sigma$. Moreover, when $s=1$, $\gamma_1(\mathcal C)=\frac{2-n+\sqrt{(n-2)^2+\lambda_1}}{2}$, where $\lambda_1$ is the first eigenvalue of the spherical Laplacian $-\Delta_{S^{n-1}}$ on $\Sigma$ and  $\varphi_1$ is the first eigenfunction of $-\Delta_{S^{n-1}}$ on $\Sigma$ with zero Dirichlet boundary condition.
\end{thm}

Recall that the Green function $G^{s}_{\mathcal{C}}(x,y)$ for Dirichlet problem of $(-\Delta)^{s}$ on the cone $\mathcal{C}$ solves
\begin{equation}\label{CG-1}
	\begin{cases}(-\Delta)^{s} G^{s}_{\mathcal{C}}(x, y)=\delta(x-y), & y \in\mathcal{C} \\ G^{s}_{\mathcal{C}}(x, y)=0, & y \in \mathbb{R}^{n}\setminus \mathcal{C},\\
		\lim\limits_{y\in\mathcal{C}, |y|\to+\infty} G^{s}_{\mathcal{C}}(x,y)=0
	\end{cases}
\end{equation}
for any given $x\in\mathcal{C}$. We establish the following lemma on properties of  $G^{s}_{\mathcal{C}}(x,y)$.
\begin{lem}\label{PCG}
	Assume $n\geq1$, $0<s\leq1$, $2s\leq n$ and $G^{s}_{\mathcal{C}}(x,y)$ is the Green function for $(-\Delta)^{s}$ on the cone $\mathcal{C}$ defined by \eqref{CG-1}. Then, for  any smaller cone $\mathcal{C}_1:=\mathcal{C}_{0,\Sigma_1}\subset \mathcal{C}$ with   $\overline{\Sigma_1} \subsetneq \Sigma$, there are constant $0<C_1\leq C_1'<+\infty$ depending on $\mathcal{C}_1$ such that
	\begin{equation}\label{fG-5}
		\frac{C_1 |x|^{\gamma_s(\mathcal C)}} {|y|^{n-2s+\gamma_s( \mathcal C)}}\leq   G^{s}_{\mathcal{C}}(x,y)\leq \frac{C_1' |x|^{\gamma_s( \mathcal C)}} {|y|^{n-2s+\gamma_s( \mathcal C)}},\ \quad \ \forall\, x, y\in\mathcal C_1\,\,\text{with} \,\, \frac{|x|}{|y|}\leq \frac{1}{2},
	\end{equation}
	where $\gamma_s(\mathcal C)$ is the constant in Theorem \ref{RMI-h}.
	Furthermore, there are constants $C_{0}>0$,  $0<\sigma_1<\sigma_2$ depending on $\mathcal{C}_1$ such that
	\begin{equation}\label{fG-6}
		G^{s}_{\mathcal{C}}(x, y)\geq \frac{C_0}{|x-y|^{n-2s}}, \ \quad  \ \forall \,\,  x  \in  \mathcal{C}_1, \, y\in \mathcal{C} \,\, \text{with}\,\,    \sigma_1|x| \leq |x-y|\leq \sigma_2 |x|.
	\end{equation}
\end{lem}
\begin{proof}
	
	We first prove \eqref{fG-5}. For the fractional cases $s\in(0,1)$, \eqref{fG-5} follows immediately from Theorem 3.11 in \cite{Mich} if the cone $\mathcal C$ is symmetry. We note that   estimates in  Theorem 3.11 in \cite{Mich} can be easily generalized to any cones $\mathcal C_{P, \Sigma}$
	with $C^{1,1}$ cross-section $\Sigma\subset\partial B_1(0)$, since the reasoning and proof in  \cite{Mich} rely on a boundary Harnack principle and on sharp estimates for the Green function for bounded $C^{1,1}$ domains. As a consequence, the proof of \eqref{fG-5} in the cases $s\in(0,1)$ is finished.
	
	\medskip
	
	Next, we consider the classical case $s=1$. The proof of \eqref{fG-5} for $s=1$ relies on the asymptotic behavior of $G^{1}_{\mathcal C}(x,y)$ as $|x|\to 0$. Let $u_h(x)=|x|^{\gamma_1(  \mathcal C)}\varphi_1\left(\frac{x}{|x|}\right)$ be the harmonic function in $\mathcal C$ obtained in Theorem \ref{RMI-h}. For $\delta>0$ small enough, define the smaller cone $\mathcal C^\delta=\mathcal C_{0, \Sigma^\delta}$ with $\Sigma^\delta:=\{x\in \Sigma\, \mid\, \text{dist}(x , \Sigma)>\delta\}$. We will show that, for $\delta$ sufficiently small,
	\begin{equation}\label{fG-9}
		C^{-1} u_h(x)	\leq G^{1}_{\mathcal C}(x,y)\leq C u_h(x), \quad  \  \ \forall\, x\in B_{\frac{1}{2}}(0)\cap\mathcal{C}, \, y\in \Sigma^\delta,
	\end{equation}
	where $C>0$ is a positive constant depending on $\delta$.

\medskip

	Indeed, since $ G^{1}_{\mathcal C} (x,y)$ is positive for $x,y\in\mathcal C$ and continuous as long as $x\not=y$, there is a positive constant $C>0$ such that
	\begin{equation}\label{fG-10}
		C^{-1} u_h(x)\leq  G^{1}_{\mathcal C}(x,y)\leq C u_h(x), \quad  \  \ \forall\, x\in   \frac{1}{2}\Sigma^\delta, \, y\in  \Sigma^\delta,
	\end{equation}
	For any $y\in   \Sigma^\delta$ fixed, $G^{1}_{\mathcal C}(x,y)$ is  harmonic in $x\in B_{\frac{2}{3}}(0)\cap \mathcal{C}$. Therefore, the boundary Harnack inequality (see e.g. \cite{DS} and Theorem B.1 in \cite{FR}) provides a positive constant $C>0$ such that, for $\delta>0$ sufficiently small,
	\begin{equation}\label{fG-11}
		C^{-1} u_h(x)\leq G^{1}_{\mathcal C}(x,y)\leq C u_h(x), \  \quad  \ \forall\, x\in   \frac{1}{2}(\Sigma\setminus\Sigma^\delta), \, y\in  \Sigma^\delta.
	\end{equation}
For any $y\in  \Sigma^\delta$ fixed, note that both  $G^{1}_{\mathcal C}(\cdot,y)$ and $u_h$ are  harmonic in $ B_{\frac{1}{2}}(0)\cap \mathcal{C}$ and vanish on $\partial C\cap B_{\frac{1}{2}}(0)$.  So using \eqref{fG-10} and \eqref{fG-11}, we infer from the maximum principle that \eqref{fG-9} is true.

\medskip
	
	Now we turn to prove \eqref{fG-5} for $s=1$. For any cone $\mathcal{C}_1=\mathcal{C}_{0, \Sigma_1}\subset \mathcal{C}$ with   $\overline{\Sigma_1 } \subsetneq \Sigma$, there is a $\delta>0$ small enough such that $\mathcal{C}_1\subset \mathcal{C}^\delta$. Then \eqref{fG-5} is a consequence of \eqref{fG-9}, the fact $0<\inf\limits_{\mathcal C^\delta} \varphi_1\leq \sup\limits_{\mathcal C^\delta} \varphi_1<+\infty$ and the following identity
	\begin{equation*}
		G^{1}_{\mathcal C}(x,y)=|y|^{2-n} G^{1}_{\mathcal C}\left(\frac{x}{|y|}, \frac{y}{|y|}\right).
	\end{equation*}

\medskip
	
	Next we prove  \eqref{fG-6}. Since $\overline{\Sigma_1 } \subsetneq \Sigma$, one has $d_1:=\text{dist}(\Sigma_1 ,   \partial \Sigma)>0$. So we can take a constant $\sigma_2>0$ sufficiently small such that, for any $x\in \Sigma_1 $, if $|x-y|\leq \sigma_2$, then $y\in \Sigma$ and $\text{dist}(y,   \partial \Sigma)>\frac{d_1}{2}$.  Taking $\sigma_1=\frac{1}{2}\sigma_2$, by the continuity and positivity of $ G^{s}_{\mathcal C}$, we obtain a constant $C_{2}>0$ such that
	$$ G^{s}_{\mathcal C}(x,y)>C_2,\ \quad \  \forall\, x\in \Sigma_1 , \  \sigma_1\leq |x-y|\leq \sigma_2,$$ which together with the identity $G^{s}_{\mathcal C}(x,y)=|x|^{2s-n}G^{s}_{\mathcal C}\left(\frac{x}{|x|}, \frac{y}{|x|}\right) $ prove \eqref{fG-6} immediately. This finishes our proof of Lemma \ref{PCG}.	
\end{proof}	

\medskip

Now, we are in position to prove Theorem \ref{pG}.

\begin{proof}[\textbf{Proof of Theorem \ref{pG}}]	
	
	For $s\in (0, 1]$, we first consider the case that $\Omega\subset \mathbb R^n$ is a domain such that $\mathbb R^n\setminus \overline\Omega$ is g-radially convex with $P$ as the center. We aim to show that  the Green function $G^{s}_{\Omega}(x,y)$ satisfies  all the assumptions $(\mathbf{H_{1}})$, $(\mathbf{H_{2}})$ and $(\mathbf{H_{3}})$.

\medskip		

	For the sake of simplicity, we may assume $P=0$ from now on.

\bigskip

	\noindent{\bf   Proof of  $(\mathbf{H_{1}})$.} The Green function can be written as $G^{s}_{\Omega}(x,y)=\Gamma_{s}(x,y)-H^s_{\Omega}(x,y)$, where the fundamental solution $\Gamma_{s}(x,y)=\begin{cases} \frac{C_{n,s}}{|x-y|^{n-2s}},\ \  \ &\text{if}\,\, 0<s<\frac{n}{2},\\ c_{n}\ln\frac{1}{|x-y|}, &\text{if}\,\,  s=\frac{n}{2},\end{cases}$ $H^s_\Omega$ is the harmonic part and $c_{n}=\frac{1}{2\pi}$ if $n=2$. In the subcritical order cases $0<s<\frac{n}{2}$, $(\mathbf{H_{1}})$ is an easy consequence of the maximum principle.

\medskip
	
We only need to take the critical order cases $s=\frac{n}{2}$ with $n=1$ or $2$ into account. In the critical order cases, we require $\mathbb R^n\setminus \overline\Omega$ is g-radially convex and nonempty.

\medskip

If $n=2s=1$, then $\Omega$ can only take these forms: $(-\infty,P)$ or $(P,+\infty)$. By assuming $P=0$, we may assume that $\Omega=(-\infty,0)$ or $(0,+\infty)$. Theorem 3.1 in \cite{Buc} gives the expression of Green's function for Dirichlet problem of $(-\Delta)^{\frac{1}{2}}$ on $(-r,r)$:
\begin{equation*}
G^{\frac{1}{2}}_{(-r,r)}(x, y)=c\ln \left(\frac{r^2-xy+\sqrt{\left(r^2-x^2\right)\left(r^2-y^2\right)}}{r|x-y|}\right).
\end{equation*}
Consequently, the Green function for $(-\Delta)^{\frac{1}{2}}$ on $(0,2r)$ is $G^{\frac{1}{2}}_{(0,2r)}(x, y)=G^{\frac{1}{2}}_{(-r,r)}(x-r, y-r)$, taking the limit $r\rightarrow+\infty$, we derive the Green function for $(-\Delta)^{\frac{1}{2}}$ on $(0,+\infty)$ is $G^{\frac{1}{2}}_{(0,+\infty)}(x, y)=c\ln\left[\frac{x+y+2\sqrt{xy}}{|x-y|}\right]$. Thus we have
\begin{equation*}
  G^{\frac{1}{2}}_{\Omega}(x,y)=c\ln\left[\frac{|x|+|y|+2\sqrt{|xy|}}{|x-y|}\right]\leq c\ln\left[\frac{(1+|x|)(1+|y|)}{|x-y|}\right],
\end{equation*}
that is, $(\mathbf{H_{1}})$ holds true for $n=2s=1$.

\medskip

If $n=2s=2$, we take a point $x_0\in \mathbb R^n\setminus \overline \Omega$ such that $0<\text{dist}(x_0, \Omega) <1$. Let $\lambda_0:=\frac{\text{dist}(x_0, \Omega)}{2} <\frac{1}{2}$. Define $x^{\lambda_0}:=\frac{{\lambda_0}^2}{|x-x_0|^2}(x-x_{0})+x_{0}$ and $\Omega_{\lambda_0}:=\{x\in\mathbb R^n |\, x^{\lambda_0}\in \Omega\}$. Then one can see that $\Omega_{\lambda_0}\subset B_{{\lambda_0}}(x_0)$ and hence $\text{diam}(\Omega_{\lambda_0})\leq 1$. We consider the function $\tilde G(x,y):=G^{1}_{\Omega}(x^{\lambda_0}, y^{\lambda_0})$.

\medskip

Now we discuss two different cases on the position of the possibly singular point $x_{0}$ to show $\tilde G(x,y)$ is the Green function of $-\Delta$ on $\Omega_{\lambda_0}$. If $\partial \Omega$ is unbounded, we have $x_0\in \partial \Omega_{\lambda_0}$ and $$\lim\limits_{x\in \Omega_{\lambda_0},\, x\to x_0} \tilde G(x,y)=\lim\limits_{x\in \Omega,\, |x|\to +\infty} G^{1}_{\Omega}(x, y^{\lambda_0})=0.$$
Thus we obtain
$$\tilde G(x,y)\equiv0,\ \qquad   \forall \,\, x\in \partial \Omega_{\lambda_0}, \,\, y\in  \Omega_{\lambda_0},$$
that is, the zero Dirichlet boundary condition holds. If $\partial \Omega$ is bounded, then $x_0\in \Omega_{\lambda_0}$ and
$$\lim\limits_{x\in \Omega_{\lambda_0},\, x\to x_0} \frac{\tilde G(x,y)}{-\ln |x-x_0|}=\lim\limits_{x\in \Omega,\, |x|\to +\infty} \frac{G^1_{\Omega}(x, y^{\lambda_0})}{\ln( |x-x_0|/{\lambda_0}^2)}=0.$$
Hence $x_0$ is a removable singularity of $\tilde G(x,y)$ in $\Omega_{\lambda_{0}}$. Therefore, by the properties of Kelvin transforms, we have $\tilde G(x,y)$ is the Green function of $-\Delta$ on $\Omega_{\lambda_0}$ and hence
$$\tilde G(x,y)=\frac{1}{2\pi} \ln\frac{1}{|x-y|}-H^1_{\Omega_{\lambda_0}}(x,y),$$
where $H^1_{\Omega_{\lambda_0}}(x,y)$ is the harmonic part. Since $\text{diam}(\Omega_{\lambda_0})\leq 1$, one has $\ln \frac{1}{|x-y|}\geq 0$ for any $x\in \partial \Omega_{\lambda_0}, \, y\in \Omega_{\lambda_0}$, and hence $ H^1_{\Omega_{\lambda_0}}\geq 0$ in $\Omega_{\lambda_0}$ by the maximum principle.

\medskip

On the other hand, by the definition of $\tilde G(x,y)$, we have
	\begin{equation}\label{g-13}
		\begin{split}
			\tilde G(x,y)&=G^1_{\Omega}(x^{\lambda_0}, y^{\lambda_0})=\frac{1}{2\pi} \ln\frac{1}{|x^{\lambda_0}- y^{\lambda_0}|}-H^1_{\Omega}(x^{\lambda_0}, y^{\lambda_0})\\
			&=\frac{1}{2\pi} \ln\frac{1}{|x - y |}+\frac{1}{2\pi} \ln\frac{|x-x_0||y-x_0|}{{\lambda_0}^2}-H^1_{\Omega}(x^{\lambda_0}, y^{\lambda_0}).
		\end{split}
	\end{equation}
	Thus we find
	\begin{equation*}
		\begin{split}
			\frac{1}{2\pi} \ln\frac{|x-x_0||y-x_0|}{{\lambda_0}^2}-H^1_{\Omega}(x^{\lambda_0}, y^\lambda_{0})=-H^1_{\Omega_{\lambda_0}}(x,y)\leq 0,\  \  \ \forall \,\, x, y\in \Omega_{\lambda_0}.
		\end{split}
	\end{equation*}
	Then, it follows from the above inequality that
	\begin{equation*}
		H^1_{\Omega}(x^{\lambda_0}, y^{\lambda_0})\geq  \frac{1}{2\pi} \ln\left[\frac{|x-x_0||y-x_0|}{{\lambda_0}^2}\right],\  \quad  \ \forall \,\, x, y\in \Omega_{\lambda_{0}},
	\end{equation*}
and hence, for any $x, \, y\in \Omega$,
	\begin{eqnarray*}
		&& G^{1}_{\Omega}(x,y)=\Gamma_{1}(x,y)-H^1_{\Omega}(x, y)\leq \frac{1}{2\pi}\ln\frac{1}{|x-y|}+\frac{1}{2\pi} \ln\left[\frac{{\lambda_0}^2}{|x^{\lambda_{0}}-x_0||y^{\lambda_{0}}-x_0|}\right] \\
\nonumber &&\qquad\qquad \leq \frac{1}{2\pi}\ln\left[\frac{(|x|+|x_{0}|)(|y|+|y_{0}|)}{\lambda_{0}^{2}|x-y|}\right]
\leq \frac{1}{2\pi}\ln\left[\frac{(1+|x_{0}|)(1+|y_{0}|)(1+|x|)(1+|y|)}{\lambda_{0}^{2}|x-y|}\right],
	\end{eqnarray*}
which proves $(\mathbf{H_{1}})$ when $n=2s=2$.
	
	\bigskip
	
	\noindent{\bf   Proof of $(\mathbf{H_{2}})$.} For any $d_{\Omega}<r<+\infty$, the g-radially convexity of $\mathbb R^n\setminus \overline\Omega$ implies  $\mathcal C^{r,e}_{0, \Sigma^{r}_{\Omega}} \subset\Omega$. Choose arbitrarily a smooth simply connected cross-section $\widehat{\Sigma}$ such that $\overline{\widehat{\Sigma}}\subsetneq\Sigma^{r}_{\Omega}$. Then one can infer that there exists $Q\in\mathcal C^{r,e}_{0, \Sigma^{r}_{\Omega}}$ such that the cone $Q+\mathcal C_{0, \widehat\Sigma}\subset \mathcal C^{r,e}_{0, \Sigma^{r}_{\Omega}}\subset\Omega$. The maximum principle implies that
\begin{equation}\label{a+1}
  G^s_{\Omega}(x,y)\geq G^{s}_{Q+\mathcal C_{0, \widehat\Sigma}}(x,y), \ \qquad  \forall\,\, x,y \in Q+\mathcal C_{0, \widehat\Sigma},
\end{equation}
where $G^{s}_{Q+\mathcal C_{0, \widehat\Sigma}}(x,y)$ denotes the Green function of $(-\Delta)^s$ on the cone $Q+\mathcal C_{0, \widehat\Sigma}$ with zero  Dirichlet boundary condition. It follows from estimates of the Green function on cones in Lemma \ref{PCG} that, for any smaller cone $Q+\mathcal{C}_{0,\Sigma_{1}}$ with $\overline{\Sigma_{1}}\subsetneq\widehat\Sigma$, there exists $C_{0},\,C_{1}>0$ and $0<\sigma_{1}<\sigma_{2}$ such that
\begin{equation}\label{a+2}
  G^{s}_{Q+\mathcal C_{0, \widehat\Sigma}}(x,y)\geq\frac{C_{1}|x-Q|^{\gamma_{s}(\mathcal{C}_{0,\widehat\Sigma})}}{|y-Q|^{n-2s+\gamma_{s}(\mathcal{C}_{0,\widehat\Sigma})}}, \qquad \forall \,\, x,y\in Q+\mathcal{C}_{0,\Sigma_{1}} \,\,\, \text{with} \,\, \frac{|x-Q|}{|y-Q|}\leq\frac{1}{2},
\end{equation}
and
\begin{equation}\label{a+3}
  G^{s}_{Q+\mathcal C_{0, \widehat\Sigma}}(x,y)\geq\frac{C_{0}}{|x-y|^{n-2s}}, \qquad \forall \,\, x\in\mathcal{C}_{Q,\Sigma_{1}}, \, y\in \mathcal{C}_{Q,\widehat\Sigma} \,\,\, \text{with} \,\, \sigma_1\leq\frac{|x-y|}{|x-Q|}\leq\sigma_{2}.
\end{equation}
For any cross-section $\Sigma_{2}$ such that $\overline{\Sigma_{2}}\subsetneq\Sigma_{1}$, we can take $R_{0}>\frac{4\sigma_{2}}{\sigma_{2}-\sigma_{1}}|Q|>0$ large enough such that $\mathcal{C}^{R_{0},e}_{0,\Sigma_{2}}\subset\mathcal{C}_{Q,\Sigma_{1}}$. Define $\theta:=n-2s+\gamma_{s}(\mathcal{C}_{0,\Sigma_{2}})$. Combining \eqref{a+1} with \eqref{a+2} and \eqref{a+3} yields that, for arbitrarily chosen $x_{0}\in Q+\mathcal{C}_{0,\Sigma_{1}}\subset\Omega$,
\[ G^s_{\Omega}(x_{0},y)\geq\frac{C}{|y|^{\theta}}, \qquad \forall \,\, y\in \mathcal{C}_{0,\Sigma_{2}} \,\,\, \text{with} \,\, |y|>R_{0} \,\, \text{sufficiently large},\]
\[G^s_{\Omega}(x,y)\geq\frac{C_{0}}{|x-y|^{n-2s}}, \qquad \forall \,\, x,y\in\mathcal{C}_{0,\Sigma_{2}} \,\,\, \text{with} \,\, |x|,\,|y|>R_{0}, \,\, \frac{\sigma_1(5\sigma_{2}-\sigma_{1})}{4\sigma_{2}}\leq\frac{|x-y|}{|x|}\leq\frac{3\sigma_{2}+\sigma_{1}}{4}.\]
Recall that we have assumed $P=0$, so property $\mathbf{(H_{2})}$ has been derived.

\medskip

If $\mathbb{R}^{n}\setminus\overline{\Omega}$ is a nonempty bounded g-radially convex domain, then we pick arbitrarily a point $Q\in\mathbb{R}^{n}\setminus\overline{\Omega}$ and define the Kelvin transform $\widetilde{x}:=\frac{x-Q}{|x-Q|^{2}}+Q$. Define $\widetilde{\Omega}:=\left\{x\in\mathbb{R}^{n}\mid\,\frac{x-Q}{|x-Q|^{2}}+Q\in\Omega\right\}$, then $\widetilde{\Omega}$ is a bounded domain with $Q\in\widetilde{\Omega}$. One can easily verify that $\left(\frac{1}{|x-Q|\cdot|y-Q|}\right)^{n-2s}G^{s}_{\Omega}(\widetilde{x},\widetilde{y})$ is the Green function for Dirichlet problem of $(-\Delta)^{s}$ on $\widetilde{\Omega}$. Thus we can deduce from the uniqueness of the Green function (guaranteed by maximum principle) that $G^{s}_{\widetilde{\Omega}}(x,y)=\left(\frac{1}{|x-Q|\cdot|y-Q|}\right)^{n-2s}G^{s}_{\Omega}(\widetilde{x},\widetilde{y})$. For any $x_{0}\in\Omega$, we have
\[\lim_{|y-P|\rightarrow+\infty}|y-P|^{n-2s}G^{s}_{\Omega}(x_{0},y)
=\lim_{|y-P|\rightarrow+\infty}\frac{1}{|x_{0}-Q|^{n-2s}}G^{s}_{\widetilde{\Omega}}\left(\widetilde{x_{0}},\widetilde{y}\right)
=\frac{G^{s}_{\widetilde{\Omega}}\left(\widetilde{x_{0}},Q\right)}{|x_{0}-Q|^{n-2s}}>0.\]
This proved property $\mathbf{(H_{2})}$ with $\Sigma=S^{n-1}$ and $\theta=n-2s$ if $\mathbb{R}^{n}\setminus\overline{\Omega}$ is nonempty and bounded.

	\bigskip
	
	\noindent{\bf   Proof of  $(\mathbf{H_{3}})$.}
	For any fixed $y\in (\Omega\setminus B_{\lambda}(0))^\lambda$, we define $$w(x):=\left(\frac{\lambda }{|x| }\right)^{n-2s}G^{s}_{\Omega}(x^\lambda, y^\lambda) - G^{s}_{\Omega}(x, y^\lambda), \  \qquad  \forall\,\, x\in  \Omega\bigcap B_\lambda(0),$$
where $x^\lambda:=\frac{\lambda^2}{|x|^2}x$. For any set $D$, let $D^\lambda:=\{x\in \mathbb R^n\,|\, x^\lambda\in D\}$ denote the reflection of $D$ with respect to the sphere $S_\lambda$. Note that the g-radially convexity of $\mathbb R^n\setminus \overline\Omega$ with respect to the center $P=0$ implies
	$$\Omega\cap B_\lambda(0)\subset (\Omega\setminus B_{\lambda}(0))^\lambda.$$

\medskip

For $s\in(0,1)$, we need the following maximum principle for spherical anti-symmetric functions from \cite{CLZ}.
	\begin{thm}[Proposition 2.1 in \cite{CLZ}]\label{MPfa}
		Let $D$ be an open subset in $B_{{\lambda_{0}}}(0)$. Assume that $w\in \mathcal{L}_{s}(\mathbb{R}^{n})\bigcap C_{loc}^{[2s],\{2s\}+\eps}(D)$ (with arbitrarily small $\eps>0$) is lower semi-continuous on $\overline{D}$. If
		\begin{equation*}
			\begin{cases}
				(-\Delta)^{s}w(x)\geq 0,\quad &x\in D,\\
				w \geq 0 &\text{a.e.  in }\,\,   B_{\lambda_{0}}(0)\setminus D,\\
				w =-w_{\lambda_{0}}=-\left(\frac{\lambda_{0}}{|x|}\right)^{n-2s}w(x^{\lambda_{0}}) &\text{a.e.  in }\,\,  B_{\lambda_{0}}(0),
			\end{cases}
		\end{equation*}
		then $w(x)\geq 0$ for every $x\in D$. Moreover, if $w(x)=0$ for some $x\in D$, then $w(x)\equiv0$ for almost every $x\in \mathbb R^n$.
	\end{thm}	
	
	On the one hand, if $y\in  (\Omega\setminus B_{\lambda}(0))^\lambda \setminus \left(\Omega\bigcap B_\lambda(0)\right)$, we have $w(x)$ is smooth in $\Omega\bigcap B_\lambda(0)$. Using the definition of Green's function and the properties of Kelvin transform, we find that $w$ satisfies
	\begin{equation*}
		\begin{cases}
			(-\Delta)^{s} w(x)=0,\ \ \ \ &x\in \Omega\bigcap B_\lambda(0),\\
			w(x)=\left(\frac{\lambda }{|x| }\right)^{n-2s}G^{s}_{\Omega}(x^\lambda, y^\lambda)\geq 0, &x\in (\Omega\setminus B_{\lambda}(0))^\lambda \setminus  (\Omega\bigcap B_\lambda(0)),\\
			w(x)\equiv0, &x\in B_\lambda(0)\setminus (\Omega\setminus B_{\lambda}(0))^\lambda,\\
			w(x)=-\left(\frac{\lambda}{|x|}\right)^{n-2s}w(x^\lambda), &x\in  B_\lambda(0)\setminus\{0\}.
		\end{cases}
	\end{equation*}
When $s=1$, we deduce from the maximum principle for $-\Delta$ that $w>0$ in $ \Omega\bigcap B_\lambda(0)$. When $s\in(0,1)$, we apply the maximum principle for anti-symmetric functions in Theorem \ref{MPfa} with $D=\Omega\bigcap B_\lambda(0)$ to obtain that $w>0$ in $\Omega\bigcap B_\lambda(0)$.

\medskip
	
	On the other hand, if  $y\in  \Omega\bigcap B_\lambda(0)$, then for $\delta>0$ sufficiently small, one has $w>0$ in $B_{\delta}(y)$ and $w(x)$ is smooth in $\left(\Omega\bigcap B_\lambda(0)\right)\setminus B_{\delta}(y)$. Thus we obtain
	\begin{equation*}
		\begin{cases}
			(-\Delta)^{s} w(x)=0,\ \ \ \ &x\in \left(\Omega\bigcap B_\lambda(0)\right)\setminus B_{\delta}(y),\\
			w(x)>0, &x\in  B_{\delta}(y),\\
			w(x)=\left(\frac{\lambda }{|x| }\right)^{n-2s}G^{s}_{\Omega}(x^\lambda, y^\lambda)\geq 0, &x\in (\Omega\setminus B_{\lambda}(0))^\lambda \setminus  (\Omega\bigcap B_\lambda(0)),\\
			w(x)\equiv0, &x\in B_\lambda(0)\setminus (\Omega\setminus B_{\lambda}(0))^\lambda,\\
			w(x)=-\left(\frac{\lambda}{|x|}\right)^{n-2s}w(x^\lambda), &x\in  B_\lambda(0)\setminus\{0\}.
		\end{cases}
	\end{equation*}
Similar as above, we can get $w>0$ in $\Omega\bigcap B_\lambda(0)\setminus B_{\delta}(y)$ via the maximum principle for $-\Delta$ if $s=1$, and the maximum principle for anti-symmetric functions in Theorem \ref{MPfa} with $D=\Omega\bigcap B_\lambda(0)\setminus B_{\delta}(y)$ if $s\in(0,1)$. Then, we take the limit by letting $\delta\to0$ and get $w>0$ in $\Omega\bigcap B_\lambda(0)$. Now the first inequality in $(\mathbf{H_{3}})$ follows easily from $w>0$ in $\Omega\bigcap B_\lambda(0)$ by changing $y$ to $y^{\lambda}$ and taking Kelvin transforms.

\medskip
	
	In order to prove the second inequality in $(\mathbf{H_{3}})$, for arbitrarily fixed $y\in \Omega\bigcap B_\lambda(0)$, let us define
	\begin{equation*}
		\begin{split}\tilde w(x)&=\left(\frac{\lambda^{2}}{|x|\cdot|y|}\right)^{n-2s}G^{s}_{\Omega}(x^\lambda, y^\lambda)  -\left(\frac{\lambda}{|y|}\right)^{n-2s}G^{s}_{\Omega}(x, y^\lambda)\\&\quad -\left[G^{s}_{\Omega}(x, y)-\left(\frac{\lambda}{|x|}\right)^{n-2s}G^{s}_{\Omega}(x^\lambda, y) \right].\end{split}
	\end{equation*}
	By direct calculations, we get
	\begin{equation*}
		\begin{split}\tilde w(x)&=-\left(\frac{\lambda^{2}}{|x|\cdot|y|}\right)^{n-2s}H^{s}_{\Omega}(x^\lambda, y^\lambda)  +\left(\frac{\lambda}{|y|}\right)^{n-2s}H^{s}_{\Omega}(x, y^\lambda)\\&\quad+\left[H^{s}_{\Omega}(x, y)-\left(\frac{\lambda}{|x|}\right)^{n-2s}H^{s}_{\Omega}(x^\lambda, y) \right],\end{split}
	\end{equation*}
	where $H^{s}_{\Omega}$ denotes the harmonic part in the Green function $G^{s}_{\Omega}$, which implies that $\tilde w$ is smooth in $\Omega\bigcap B_\lambda(0)$.  Using the definition of Green's function and the properties of Kelvin transform, one can verify that $\tilde w$ satisfies
	\begin{equation*}
		\begin{cases}
			(-\Delta)^{s} \tilde w(x)=0,\ \ \ \ &x\in \Omega\bigcap B_\lambda(0),\\
			\tilde w(x)=\left(\frac{\lambda^{2}}{|x|\cdot|y|}\right)^{n-2s}G^{s}_{\Omega}(x^\lambda, y^\lambda) +\left(\frac{\lambda}{|x|}\right)^{n-2s}G^{s}_{\Omega}(x^\lambda, y)\geq 0, &x\in (\Omega\setminus B_{\lambda}(0))^\lambda \setminus  (\Omega\bigcap B_\lambda(0)),\\
			\tilde w(x)\equiv0, &x\in B_\lambda(0)\setminus (\Omega\setminus B_{\lambda}(0))^\lambda,\\
			\tilde w(x)=-\left(\frac{\lambda}{|x|}\right)^{n-2s}\tilde w(x^\lambda), &x\in  B_\lambda(0)\setminus\{0\}.
		\end{cases}
	\end{equation*}
	Then, by applying the maximum principle for $-\Delta$ if $s=1$ and the maximum principle for anti-symmetric functions in Theorem \ref{MPfa} with $D=\Omega\bigcap B_\lambda(0)$ if $s\in(0,1)$, we obtain  $$\tilde w\geq 0\  \qquad  \text{in}\,\, \Omega\bigcap B_\lambda(0),$$
	which proves the second inequality in $(\mathbf{H_{3}})$ by applying Kelvin transforms.
		
		\bigskip
	
	Now, it remains to show that the Green function $G^s_{\Omega}$ satisfies $(\mathbf{H_{1}})$, $(\mathbf{\widetilde{H}_{2}})$ and $(\mathbf{\widetilde{H}_{3}})$ when the MSS applicable domain $\Omega$ is a g-radially convex domain with the center $P\in \overline{\Omega}$. We take the Kelvin transform $x\mapsto \frac{x-P}{|x-P|^2}+P$ and define
	$$\widetilde G^s(x,y):=\left(\frac{1}{|x-P|\cdot|y-P|}\right)^{n-2s} G^s_{\Omega}\left(\frac{x-P}{|x-P|^2}+P, \frac{y-P}{|y-P|^2}+P \right),$$
	and $$\Omega_1:=\left\{x\in \mathbb R^n \,\bigg |\,\frac{x-P}{|x-P|^2}+P\in \Omega\right\}.$$
	One can verify that $\Omega_1$ is a MSS applicable domain such that $\mathbb R^n\setminus \overline{\Omega_1}$ is g-radially convex. Moreover, by the properties of Kelvin transform, we have $\widetilde G^s(x,y)$ is the Green function for $(-\Delta)^s$ with zero Dirichlet boundary condition on $\Omega_1$. Thus  we conclude that $\widetilde G^s(x,y)$  satisfies $(\mathbf{H_{1}})$, $(\mathbf{ {H}_{2}})$ and $(\mathbf{ {H}_{3}})$ on $\Omega_{1}$, which implies immediately that $G^s_{\Omega}$ satisfies $(\mathbf{H_{1}})$, $(\mathbf{\widetilde{H}_{2}})$ and $(\mathbf{\widetilde{H}_{3}})$ on $\Omega$. The proof of Theorem \ref{pG} is therefore completed.
\end{proof}

\subsection{The equivalence between IEs and PDEs and proofs of Theorems \ref{equi} and \ref{equin}}

\begin{proof}[Proof of Theorem \ref{equi}]

We prove Theorem \ref{equi} by using ideas from \cite{ZCCY} (see Remark 1 in \cite{ZCCY}, c.f. also \cite{CDQ0,CDQ,DQ3,DQ0}). If $\Omega$ is a bounded g-radially convex domain, then the equivalence between the Dirichlet problem of PDEs \eqref{PDE} and IEs \eqref{IE1} in Theorem \ref{equi} is obvious, so we omit the details. We only consider the cases that $\Omega=\mathcal{C}_{P,\Sigma}$ such that $\overline{\mathcal{C}_{P,\Sigma}}\neq\mathbb{R}^{n}$ or $\mathbb{R}^{n}\setminus\overline{\Omega}\neq\emptyset$ is a bounded g-radially convex domain with convex center $P$ (when $s=1$).

\medskip

Suppose $u$ is a nonnegative classical solution to Dirichlet problem of PDEs \eqref{PDE} and $f$ satisfies $(i)$ in $(\mathbf{f_{3}})$ with the same cross-section as $\mathbf{(H_{2})}$, under assumptions in Theorem \ref{equi}, we will show that $u$ also solves the IEs \eqref{IE1}. To this end, for arbitrary $R>0$, let
\begin{equation}\label{2-1-29}
v_R(x):=\int_{\Omega^{R}}G^{s}_{\Omega^{R}}(x,y)f(y,u(y))\mathrm{d}y,
\end{equation}
where $G^{s}_{\Omega^{R}}(x,y)$ is the Green function for $(-\Delta)^{s}$ with $0<s\leq1$ on $\Omega\bigcap B_{R}(P)$. Then, one can derive that $v_{R}\in C^{[2s],\{2s\}+\eps}_{loc}\left(\Omega^{R}\right)\bigcap C(\mathbb{R}^{n})\bigcap\mathcal{L}_{s}(\mathbb{R}^{n})$ with arbitrarily small $\eps>0$ if $s\in(0,1)$, $v_{R}\in C^{2}\left(\Omega^{R}\right)\bigcap C(\mathbb{R}^{n})$ if $s=1$ and satisfies
\begin{equation}\label{2-1-31}\\\begin{cases}
(-\Delta)^{s}v_R(x)=f(x,u(x)), \qquad x\in \Omega^{R},\\
v_R(x)=0,\ \ \ \ \ \ \ x\in \mathbb{R}^{n}\setminus \Omega^{R}.
\end{cases}\end{equation}
Let $w_R(x):=u(x)-v_R(x)$. By \eqref{PDE} and \eqref{2-1-31}, we have $w_R\in C^{[2s],\{2s\}+\eps}_{loc}\left(\Omega^{R}\right)\bigcap C(\mathbb{R}^{n})\bigcap\mathcal{L}_{s}(\mathbb{R}^{n})$ with arbitrarily small $\eps>0$ if $s\in(0,1)$ and $w_R\in C^{2}\left(\Omega^{R}\right)\cap C(\mathbb{R}^{n})$ if $s=1$ and satisfies
\begin{equation}\label{2-1-32}\\\begin{cases}
(-\Delta)^{s}w_R(x)=0, \qquad x\in \Omega^{R},\\
w_R(x)\geq0, \qquad x\in \mathbb{R}^{n}\setminus \Omega^{R}.
\end{cases}\end{equation}

\medskip

Now we need the following maximum principle for fractional Laplacians.
\begin{lem}\label{max}(Maximum principle, \cite{CLL,Si})
Let $\Omega$ be a bounded domain in $\mathbb{R}^{n}$ and $0<s<1$. Assume that $u\in\mathcal{L}_{s}(\mathbb{R}^{n})\bigcap C^{[2s],\{2s\}+\eps}_{loc}(\Omega)$ with arbitrarily small $\eps>0$ and is l.s.c. on $\overline{\Omega}$. If $(-\Delta)^{s}u\geq 0$ in $\Omega$ and $u\geq 0$ in $\mathbb{R}^n\setminus\Omega$, then $u\geq 0$ in $\mathbb{R}^n$. Moreover, if $u=0$ at some point in $\Omega$, then $u=0$ a.e. in $\mathbb{R}^{n}$. These conclusions also hold for an unbounded domain $\Omega$ if we assume further that
\[\liminf_{|x|\rightarrow+\infty}u(x)\geq0.\]
\end{lem}

\begin{rem}\label{rem13}
Lemma \ref{max} has been established first by Silvestre \cite{Si} without the assumption $u\in C^{[2s],\{2s\}+\eps}_{loc}(\Omega)$. In \cite{CLL}, Chen, Li and Li provided a much more elementary and simpler proof for Lemma \ref{max} under the assumption $u\in C^{[2s],\{2s\}+\eps}_{loc}(\Omega)$.
\end{rem}

By Lemma \ref{max} and maximal principles for $-\Delta$, we deduce that for any $R>0$,
\begin{equation}\label{2-1-33}
  w_R(x)=u(x)-v_{R}(x)\geq0, \,\,\,\, \forall \,\, x\in\mathbb{R}^{n}.
\end{equation}
Now, for each fixed $x\in\mathbb{R}^{n}$, letting $R\rightarrow\infty$ in \eqref{2-1-33}, we have
\begin{equation}\label{2-1-34}
u(x)\geq \int_{\Omega}G^{s}_{\Omega}(x,y)f(y,u(y))\mathrm{d}y=:v(x)\geq0,
\end{equation}
where $G^{s}_{\Omega}(x,y)$ denotes the Green function for $(-\Delta)^{s}$ with $0<s\leq1$ on $\Omega$. One can observe that $v\in C^{[2s],\{2s\}+\eps}_{loc}(\Omega)\bigcap C(\mathbb{R}^{n})\bigcap\mathcal{L}_{s}(\mathbb{R}^{n})$ with arbitrarily small $\eps>0$ if $s\in(0,1)$ and $v\in C^{2}(\Omega)\cap C(\mathbb{R}^{n})$ if $s=1$ is a solution of
\begin{equation}\label{2-1-36}\\\begin{cases}
(-\Delta)^{s}v(x)=f(x,u(x)),  \,\,\,\,\,\forall \, x\in \Omega,\\
v(x)=0, \qquad x\in \mathbb{R}^{n}\setminus \Omega.
\end{cases}\end{equation}
Define $w(x)=u(x)-v(x)\geq0$, then by \eqref{PDE} and \eqref{2-1-36}, we have $w\in C^{[2s],\{2s\}+\eps}_{loc}(\Omega)\bigcap C(\mathbb{R}^{n})\bigcap\mathcal{L}_{s}(\mathbb{R}^{n})$ with arbitrarily small $\eps>0$ if $s\in(0,1)$, $v\in C^{2}(\Omega)\cap C(\mathbb{R}^{n})$ if $s=1$ and satisfies
\begin{equation}\label{2-1-37}\\\begin{cases}
(-\Delta)^{s}w(x)=0, \,\,\,\,\,  x\in\Omega,\\
w(x)=0, \,\,\,\,\,\,  x\in \mathbb{R}^n\setminus\Omega.
\end{cases}\end{equation}

\medskip

Next we will discuss two different cases and show that equality in \eqref{2-1-34} must hold, i.e., $u=v$.

\medskip

\emph{Case (i).} $\Omega=\mathcal{C}_{P,\Sigma}$ such that $\overline{\mathcal{C}_{P,\Sigma}}\neq\mathbb{R}^{n}$ and $0<s\leq1$. The classification result for $s$-harmonic functions in cone from \cite{A, BaBo} (see also Theorem \ref{RMI-h}) implies that $w$ must take the following form
	\begin{equation*}
		w(x)=u(x)-v(x)=C |x-P|^{\gamma_s( \mathcal {C}_{P,\Sigma})}\varphi_s\left(\frac{x-P}{|x-P|}\right),
	\end{equation*}
where $C\geq0$, $\gamma_s( \mathcal C)$ is a nonnegative constant depending on $n$, $s$, the cross-section $\Sigma$ of $\mathcal C=\mathcal{C}_{P,\Sigma}$, independent of the vertex of the cone $\mathcal{C}$ and decreasing w.r.t. the inclusion relationship of the cross-section $\Sigma$ (i.e., $\gamma_s( \mathcal C_{P,\Sigma})\leq\gamma_s( \mathcal C_{Q,\Sigma_{1}})$ provided that $\Sigma_{1}\subseteq\Sigma$), and $\varphi_s$ is a positive function on $\Sigma$ vanishing on $S^{n-1}\setminus\Sigma$. Thus, we reach that
\begin{eqnarray}\label{2-1-38}
  && u(x)=\int_{\Omega}G^{s}_{\Omega}(x,y)f(y,u(y))\mathrm{d}y+C |x-P|^{\gamma_s( \mathcal {C}_{P,\Sigma})}\varphi_s\left(\frac{x-P}{|x-P|}\right) \\
 \nonumber && \qquad\, \geq C |x-P|^{\gamma_s( \mathcal {C}_{P,\Sigma})}\varphi_s\left(\frac{x-P}{|x-P|}\right)\geq0.
\end{eqnarray}
Since Theorem \ref{pG} implies that the Green function $G^{s}_{\Omega}$ satisfy $\mathbf{(H_{2})}$ with arbitrary $r\in(d_{\Omega},+\infty)$, any cross-section $\Sigma_{1}$ such that $\overline{\Sigma_{1}}\subsetneq\Sigma^{r}_{\Omega}=\Sigma$ and arbitrary $\theta>n-2s+\gamma_{s}(\mathcal{C}_{P,\Sigma})$, and $f$ satisfies $(i)$ in $(\mathbf{f_{3}})$, so there exist $r>0$, cross-section $\Sigma_{1}$ such that $\overline{\Sigma_{1}}\subsetneq\Sigma$, $C>0$, $a>-2s$, $p\geq1$, $x_{0}\in\Omega$ and $\theta$ arbitrarily close to $n-2s+\gamma_{s}(\mathcal {C}_{P,\Sigma})$ such that $f(x,u)\geq C|x-P|^{a}u^{p}$ for any $x\in\mathcal{C}^{r,e}_{P,\Sigma_{1}}$, and $G^{s}_{\Omega}(x_{0},y)\geq\frac{C}{|y-P|^{\theta}}$ for any $y\in\mathcal{C}^{r,e}_{P,\Sigma_{1}}$. Therefore, \eqref{2-1-34} implies that $u$ satisfies the following integrability
\begin{equation}\label{2-1-35}
+\infty>u(x_{0})\geq C\int_{\mathcal{C}^{r,e}_{P,\Sigma_{1}}}\frac{u^{p}(y)}{|y-P|^{\theta-a}}\mathrm{d}y.
\end{equation}
Now, by combining \eqref{2-1-35} with \eqref{2-1-38} and noting that $\inf\limits_{x\in\mathcal{C}_{P,\Sigma_{1}}}\varphi_s\left(\frac{x-P}{|x-P|}\right)>c_{0}>0$, we get
\begin{equation}\label{2-1-39}
  C^{p}\int_{\mathcal{C}^{r,e}_{P,\Sigma_{1}}}\frac{1}{|y-P|^{\theta-a-p\gamma_s(\mathcal {C}_{P,\Sigma})}}\mathrm{d}y
 \leq \int_{\mathcal{C}^{r,e}_{P,\Sigma_{1}}}\frac{u^{p}(y)}{|y-P|^{\theta-a}}\mathrm{d}y<+\infty,
\end{equation}
from which and $\theta$ is larger than but arbitrarily close to $n-2s+\gamma_{s}(\mathcal {C}_{P,\Sigma})$, we can infer immediately that $C=0$, and hence $w(x)=u(x)-v(x)=0$. Therefore, we arrive at
\begin{equation}\label{2-1-40}
  u(x)=\int_{\Omega}G^{s}_{\Omega}(x,y)f(y,u(y))\mathrm{d}y,
\end{equation}
that is, $u$ satisfies the integral equation \eqref{IE1}.

\medskip

\emph{Case (ii).} $s=1$, $u\in C^{0,1}(\overline{\Omega})$, $\partial\Omega$ is Lipschitz and $\mathbb{R}^{n}\setminus\overline{\Omega}\neq\emptyset$ is a bounded g-radially convex domain with convex center $P$. Pick $Q\in\mathbb{R}^{n}\setminus\overline{\Omega}$. Theorem \ref{pG} implies that the Green function $G^{1}_{\Omega}$ satisfies $\mathbf{(H_{2})}$ with $\Sigma=S^{n-1}$ and $\theta=n-2$, hence there exist $x_{0}\in\Omega$ and $R_{0}>diam(\mathbb{R}^{n}\setminus\overline{\Omega})$ large enough such that $G^{s}_{\Omega}(x_{0},y)\geq\frac{C}{|y-Q|^{\theta}}$ for any $|y-Q|\geq R_{0}$. Since $f$ satisfies $(i)$ in $(\mathbf{f_{3}})$ with cross-section equals to $S^{n-1}$, there exist $C>0$, $a>-2$, $p\geq1$ such that $f(x,u)\geq C|x-Q|^{a}u^{p}$ for any $|x-Q|\geq R_{0}$ (choose $R_{0}$ larger if necessary). Therefore, \eqref{2-1-34} implies that $u$ satisfies the following integrability
\begin{equation}\label{2-1-41}
+\infty>u(x_{0})\geq C\int_{|y-Q|\geq R_{0}}\frac{u^{p}(y)}{|y-Q|^{n-2-a}}\mathrm{d}y.
\end{equation}
Note that $0\leq w(x)\leq u(x)$, we also have $w\in C^{0,1}(\overline{\Omega})$ and
\begin{equation}\label{2-1-41'}
\int_{|y-Q|\geq R_{0}}\frac{w^{p}(y)}{|y-Q|^{n-2-a}}\mathrm{d}y<+\infty.
\end{equation}
Now define the spherical average of function $g$ over spheres $S_{r}(Q)$ centered at $Q$ by
$$\overline{g}(r):=\frac{1}{\sigma_{n-1}r^{n-1}}\int_{|x-Q|=r}g(x)\dd\sigma,$$
where $\sigma_{n-1}$ denotes the area of $n-1$-dimensional sphere $S^{n-1}\subset\mathbb{R}^{n}$. For any $r>R_{0}$, it follows from \eqref{2-1-37} and integrating by parts that
\[\int_{R_{0}\leq |x-Q|\leq r}\Delta w(x)\mathrm{d}x=\int_{|x-Q|=r}\frac{\partial w}{\partial r}\dd\sigma-\int_{|x-Q|=R_{0}}\frac{\partial w}{\partial r}\dd\sigma=0\]
and
\[\int_{B_{R_{0}}(Q)\setminus(\mathbb{R}^{n}\setminus\Omega)}\Delta w(x)\mathrm{d}x=\int_{|x-Q|=R_{0}}\frac{\partial w}{\partial r}\dd\sigma+\int_{\partial\Omega}\frac{\partial w}{\partial \vec{n}}\dd\sigma=0,\]
where $\frac{\partial w}{\partial \vec{n}}$ denotes the outward normal derivative of $w$ on $\partial\Omega$. As a consequence, one has
\[\int_{|x-Q|=r}\frac{\partial w}{\partial r}\dd\sigma=\sigma_{n-1}r^{n-1}\frac{\mathrm{d} \overline{w}(r)}{\mathrm{d} r}=-\int_{\partial\Omega}\frac{\partial w}{\partial \vec{n}}\dd\sigma=C_{\Omega}\geq0, \qquad \forall \, r>R_{0}.\]
If $C_{\Omega}>0$, one has $\frac{\mathrm{d} \overline{w}(r)}{\mathrm{d} r}\geq\frac{C_{\Omega}}{\sigma_{n-1}r^{n-1}}$. Integrating on $[R_{0},r]$ yields that
\begin{equation*}
  \overline{w}(r)-\overline{w}(R_{0})\geq\int_{R_{0}}^{r}\frac{C_{\Omega}}{\sigma_{n-1}s^{n-1}}\mathrm{d}s=
  \begin{cases}
\frac{C_{\Omega}}{(n-2)\sigma_{n-1}}\left(\frac{1}{R_{0}^{n-2}}-\frac{1}{r^{n-2}}\right), \,\,\,\,\, \text{if} \,\, n\geq2,\\
\frac{C_{\Omega}}{\sigma_{n-1}}\ln\left[\frac{r}{R_{0}}\right], \,\,\,\,\,\,  \text{if} \,\, n=2,
\end{cases}
\end{equation*}
that is, for any $r>R_{0}$,
\begin{equation*}
  \int_{|x-Q|=r}w^{p}(x)\dd\sigma\geq\frac{1}{(\sigma_{n-1}r^{n-1})^{p-1}}\left(\int_{|x-Q|=r}w(x)\dd\sigma\right)^{p}\geq
  \begin{cases}
C_{n,p,\Omega}r^{n-1}, \,\,\,\,\, \text{if} \,\, n\geq2,\\
C_{p,\Omega}r\left(\ln\left[\frac{r}{R_{0}}\right]\right)^{p}, \,\,\,\,\,\,  \text{if} \,\, n=2.
\end{cases}
\end{equation*}
It follows that
\begin{equation*}
  \int_{|y-Q|\geq R_{0}}\frac{w^{p}(y)}{|y-Q|^{n-2-a}}\mathrm{d}y\geq
  \begin{cases}
C_{n,p,\Omega}\int_{R_{0}}^{+\infty}r^{1+a}\dd r=+\infty, \,\,\,\,\, \text{if} \,\, n\geq2,\\
C_{p,\Omega}\int_{R_{0}}^{+\infty}r^{1+a}\left(\ln\left[\frac{r}{R_{0}}\right]\right)^{p}\dd r=+\infty, \,\,\,\,\,\,  \text{if} \,\, n=2,
\end{cases}
\end{equation*}
which contradicts \eqref{2-1-41'}. Thus we must have $C_{\Omega}=0$, then $\overline{w}(r)\equiv C$, that is, $\int_{|x-Q|=r}w(x)\dd\sigma=C\sigma_{n-1}r^{n-1}$ ($\forall \, r>R_{0}$). If $C>0$, then H\"{o}lder inequality yields that $\int_{|x-Q|=r}w^{p}(x)\dd\sigma\geq C^{p}\sigma_{n-1}r^{n-1}$, hence
\[\int_{|y-Q|\geq R_{0}}\frac{w^{p}(y)}{|y-Q|^{n-2-a}}\mathrm{d}y\geq C^{p}\sigma_{n-1}\int_{R_{0}}^{+\infty}r^{1+a}\dd r=+\infty,\]
which contradicts \eqref{2-1-41'}. Thus we must have $C=0$, that is, $w\equiv0$ in $\mathbb{R}^{n}\setminus B_{R_{0}}(Q)$. Then by \eqref{2-1-37} and applying maximum principle on $B_{R_{0}}(Q)\setminus(\mathbb{R}^{n}\setminus\Omega)$, we deduce that $w(x)=u(x)-v(x)\equiv0$ in $\mathbb{R}^{n}$. Therefore, we have reached that
\begin{equation}\label{2-1-40-2}
  u(x)=\int_{\Omega}G^{1}_{\Omega}(x,y)f(y,u(y))\mathrm{d}y,
\end{equation}
that is, $u$ satisfies the integral equation \eqref{IE1}.

\medskip

Conversely, assume that $u$ is a nonnegative solution of IEs \eqref{IE1}, then it is obvious that $u$ also solves the Dirichlet problem of PDEs \eqref{PDE}. This completes the proof of equivalence between PDEs \eqref{PDE} and IEs \eqref{IE1}.
\end{proof}

\begin{proof}[Proof of Theorem \ref{equin}]

Suppose $u$ is a nonnegative classical solution to Navier problem of PDEs \eqref{PDE} with $f$ satisfying $(\mathbf{f_{3}})$ if $\Omega=\mathcal{C}_{P,S^{n-1}_{2^{-k}}}$, and $(-\Delta)^{i}u\geq0$ in $\Omega$ ($i=1,\cdots,s-1$), under assumptions in Theorem \ref{equin}, we will show that $u$ also solves the IEs \eqref{IEN}. Recall that the exact formulae for Green's functions $G^{s}_{\Omega}$ have been given in Remark \ref{rem8}.

\medskip

If $\Omega=\mathcal{C}^{R}_{P,S^{n-1}_{2^{-k}}}$ is bounded, then the equivalence between the Navier problem of PDEs \eqref{PDE} and IEs \eqref{IEN} in Theorem \ref{equin} is obvious, so we omit the details. We only consider the cases that $\Omega=\mathcal{C}_{P,S^{n-1}_{2^{-k}}}$ with $k=1,\cdots,n$.

\medskip

Since $v_{i}:=(-\Delta)^{i}u\geq0$ in $\Omega$ ($i=1,\cdots,s-1$), then the Navier problem of \eqref{PDE} is equivalent to the following system
\begin{equation}\label{PDES}
\left\{{\begin{array}{l} {-\Delta v_{s-1}(x)=f(x,u(x))},  \\
	{-\Delta v_{s-2}(x)= v_{s-1}(x)}, \\ \cdots\cdots \\ {-\Delta u(x)= v_1(x)}. \\ \end{array}}\right.
\end{equation}

\medskip

We need the following Lemma on the Green function.
\begin{lem}\label{lem3-0}
	Assume that integer $k\geq2$ such that $2k\leq n$, then
	\begin{equation}\label{3-6-1}
	G^{k}_{\Omega}(x,y)=\int_{\Omega} G^{1}_{\Omega}(x,z)G^{k-1}_{\Omega}(z,y)\mathrm{d}z.
	\end{equation}
\end{lem}
One can easily verify that $\int_{\Omega} G^{1}_{\Omega}(x,z)G^{k-1}_{\Omega}(z,y)\mathrm{d}z$ is the Green function for Navier problem of $(-\Delta)^{k}$ on $\Omega$. Thus we can deduce from the uniqueness of the Green function (guaranteed by maximum principle) that $G^{k}_{\Omega}(x,y)=\int_{\Omega} G^{1}_{\Omega}(x,z)G^{k-1}_{\Omega}(z,y)\mathrm{d}z$.

\medskip

Now we abbreviate $\mathcal{C}_{P,S^{n-1}_{2^{-k}}}$ to $\mathcal{C}_{k}$ with $k=1,\cdots,n$ for simplicity, i.e., $\Omega=\mathcal{C}_{k}$. By \eqref{PDES}, similar to the proof of \eqref{2-1-38}, we arrive at
\begin{equation}\label{3-14}
\left\{{\begin{array}{l} {v_{s-1}(x)=\int_{\mathcal{C}_{k}}G^{1}_{\mathcal{C}_{k}}(x,y)f(y,u(y))\mathrm{d}y+C_{s-1} |x-P|^{\gamma_1(\mathcal{C}_{k})}\varphi_1\left(\frac{x-P}{|x-P|}\right),}  \\{}\\
	{v_{s-2}(x)=\int_{\mathcal{C}_{k}}G^{1}_{\mathcal{C}_{k}}(x,y) v_{s-1}(y)\mathrm{d}y+C_{s-2}|x-P|^{\gamma_1( \mathcal{C}_{k})}\varphi_1\left(\frac{x-P}{|x-P|}\right), }\\{}\\
	\cdots\cdots \\ {}\\
	{u(x)=\int_{\mathcal{C}_{k}}G^{1}_{\mathcal{C}_{k}}(x,y)v_1(y)\mathrm{d}y+C_{0}|x-P|^{\gamma_1(\mathcal{C}_{k})}\varphi_1\left(\frac{x-P}{|x-P|}\right),} \\ \end{array}}\right.
\end{equation}
where the constants $C_{j}\geq0$ for $0\leq j\leq s-1$. For the definition of $\gamma_1$ and $\varphi_1$, see Theorem \ref{RMI-h}. Note that the nonnegative harmonic function in $\frac{1}{2^{k}}$-space $\{x=(x_{1},\cdots,x_{n})\mid\,x_{j_{1}}>0,x_{j_{2}}>0,\cdots,x_{j_{k}}>0\}$ with $1\leq j_{1}<j_{2}<\cdots<j_{k}\leq n$ is $\prod\limits_{1\leq j_{1}<j_{2}<\cdots<j_{k}\leq n}x_{j_{1}}\cdots x_{j_{k}}$, it follows from Theorem \ref{RMI-h} that $\gamma_1(\mathcal{C}_{k})=k$. Next, we aim to show that $C_{j}=0$ for $0\leq j\leq s-1$.

\medskip

Since Remark \ref{rem8} implies that, for any integer $s\geq1$ such that $2s\leq n$, the Green function $G^{s}_{\mathcal{C}_{k}}$ satisfy $\mathbf{(H_{2})}$ with arbitrary $r\in(d_{\Omega},+\infty)$, any cross-section $\Sigma$ such that $\overline{\Sigma}\subsetneq\Sigma^{r}_{\Omega}=S^{n-1}_{2^{-k}}$ and $\theta=n-2s+k$, so there exist $C>0$, $x^{(s)}_{0}\in\Omega=\mathcal{C}_{k}$ and $R_{0}>0$ large enough such that
\begin{equation}\label{a0}
  G^{s}_{\mathcal{C}_{k}}(x^{(s)}_{0},y)\geq\frac{C}{|y-P|^{n-2s+k}}, \qquad \forall \,\, y\in\mathcal{C}_{P,\Sigma},  \,\, |y-P|\geq R_{0},
\end{equation}
where $\Sigma$ is any cross-section such that $\overline{\Sigma}\subsetneq S^{n-1}_{2^{-k}}$. Note that $f$ satisfies $(i)$ in $(\mathbf{f_{3}})$, there exist smaller cross-section $\Sigma_{1}\subseteq\Sigma\subsetneq S^{n-1}_{2^{-k}}$, $C>0$, $a>-2s$, $p\geq1$ such that $f(x,u)\geq C|x-P|^{a}u^{p}$ for any $x\in\mathcal{C}_{P,\Sigma_{1}}$ and $|x-P|\geq R_{0}$ (choose $R_{0}$ larger if necessary).

\medskip

Now choose arbitrarily a cross-section $\Sigma$ such that $\overline{\Sigma}\subsetneq S^{n-1}_{2^{-k}}$ and $\Sigma_{1}\subset\Sigma$, where $\Sigma_{1}$ comes from $(i)$ in $(\mathbf{f_{3}})$. Suppose on the contrary that $C_{s-1}>0$, then by the 1st equation in system \eqref{3-14} and noting that $\inf\limits_{x\in\mathcal{C}_{P,\Sigma}}\varphi_1\left(\frac{x-P}{|x-P|}\right)>c_{\Sigma}>0$, we get
\[v_{s-1}(x)\geq c_{\Sigma}C_{s-1}|x-P|^{k}, \qquad \forall \, x\in\mathcal{C}_{P,\Sigma},\]
and hence the 2nd equation and property $\mathbf{(H_{2})}$ for $G^{1}_{\mathcal{C}_{k}}$ yields that, for some $x^{(1)}_{0}\in \mathcal{C}_{k}$,
\[v_{s-2}\left(x^{(1)}_{0}\right)\geq \int_{\mathcal{C}_{k}}G^{1}_{\mathcal{C}_{k}}\left(x^{(1)}_{0},y\right) v_{s-1}(y)\mathrm{d}y\geq C\int_{\mathcal{C}_{P,\Sigma}, \, |y-P|\geq R_{0}} \frac{C_{s-1}}{|y-P|^{n-2}}\mathrm{d}y=+\infty, \]
which is absurd. Thus we must have $C_{s-1}=0$. Continuing in this way, we derive (from the $j-1$-th and $j$-th equations in system \eqref{3-14}) that $C_{j}=0$ for all $1\leq j\leq s-1$. The main task is to show $C_{0}=0$.

\medskip

Suppose on the contrary that $C_{0}>0$, then the last equation in system \eqref{3-14} yields that
\begin{equation}\label{3-16}
  u(x)\geq c_{\Sigma_{1}}C_{0}|x-P|^{k}, \qquad \forall \, x\in\mathcal{C}_{P,\Sigma_{1}}.
\end{equation}
By \eqref{3-14} with $C_{j}=0$ ($j=1,\cdots,s-1$) and Lemma \ref{lem3-0}, it is easy to see that
\begin{eqnarray}\label{3-15+}
u(x)&=&\int_{\mathcal{C}_{k}}G^{1}_{\mathcal{C}_{k}}(x,y)v_{1}(y)\mathrm{d}y+C_{0}|x-P|^{k}\varphi_1\left(\frac{x-P}{|x-P|}\right) \\
\nonumber &=&\int_{\mathcal{C}_{k}}G^{1}_{\mathcal{C}_{k}}(x,y)\int_{\mathcal{C}_{k}}G^{1}_{\mathcal{C}_{k}}(y,z)v_2(z)\mathrm{d}z\mathrm{d}y
+C_{0}|x-P|^{k}\varphi_1\left(\frac{x-P}{|x-P|}\right)  \\
\nonumber &=&\int_{\mathcal{C}_{k}}G^{2}_{\mathcal{C}_{k}}(x,y)v_2(y)\mathrm{d}y
+C_{0}|x-P|^{k}\varphi_1\left(\frac{x-P}{|x-P|}\right)\\
\nonumber &=&\cdots\\
\nonumber &=&\int_{\mathcal{C}_{k}} G^{s}_{\mathcal{C}_{k}}(x,y)f(y,u(y))\mathrm{d}y+C_{0}|x-P|^{k}\varphi_1\left(\frac{x-P}{|x-P|}\right).
\end{eqnarray}
By substituting \eqref{3-16} into \eqref{3-15+}, we get from $(i)$ in $(\mathbf{f_{3}})$ that
\[u(x^{(s)}_{0})>\int_{\mathcal{C}_{k}}G^{s}_{\mathcal{C}_{k}}(x^{(s)}_{0},y)f(y,u(y))\mathrm{d}y\geq C\int_{\mathcal{C}_{P,\Sigma_{1}}, \, |y-P|\geq R_{0}} \frac{c_{\Sigma_{1}}^{p}C_{0}^{p}}{|y-P|^{n-2s+k-a-kp}}\mathrm{d}y=+\infty,\]
which is absurd. Therefore, we must have $C_{0}=0$, and hence \eqref{3-15+} immediately implies
\begin{equation}\label{3-15}
u(x)=\int_{\mathcal{C}_{k}} G^{s}_{\mathcal{C}_{k}}(x,y)f(y,u(y))\mathrm{d}y,
\end{equation}
that is, $u$ satisfies the integral equation \eqref{IEN}.

\medskip

Conversely, assume that $u$ is a nonnegative solution of IEs \eqref{IEN}, then it is obvious that $u$ also solves the Navier problem of PDEs \eqref{PDE} on $\frac{1}{2^{k}}$-space $\mathcal{C}_{P,S^{n-1}_{2^{-k}}}$ with super poly-harmonic property $(-\Delta)^{i}u\geq0$ in $\mathcal{C}_{P,S^{n-1}_{2^{-k}}}$ ($i=1,\cdots,s-1$). This completes the proof of equivalence between PDEs \eqref{PDE} and IEs \eqref{IEN}.
\end{proof}

\section{Blowing-up analysis and the a priori estimates of solutions on BCB domains}

\subsection{Proof of Theorem \ref{prid}}
\begin{proof}[Proof of Theorem \ref{prid}]

Suppose $\Omega$ is a BCB domain such that the boundary H\"{o}lder estimates for $(-\Delta)^{s}$ holds uniformly on the scaling domain $\Omega_{\rho}$ w.r.t. the scale $\rho\in(0,1]$. We will prove the a priori estimates in Theorem \ref{prid} by contradiction arguments and the Liouville theorems (Theorems \ref{pde-unbd}, \ref{pde-unbd-a}, \ref{pde-unbd-2}, \ref{pde-bd}, \ref{pde-bd-2}, \ref{pG}, \ref{f2PDE} and Corollary \ref{cone-cor}). Suppose on the contrary that \eqref{f-3} does not hold, then there exist  a sequence
	of solutions $\{u_j\}_{j=1}^\infty$ to the Dirichlet problem \eqref{f-1} and a sequence of points $\{x_j\}_{j=1}^\infty \subset \Omega$ such that
	\begin{equation}\label{f-4}
		u_j(x_j)=\max_{x\in\Omega} \,u_j(x)=:m_j \rightarrow +\infty, \  \ \ \text{as}\, j\to+\infty.
	\end{equation}

\smallskip	
	
	Let the scale sequence $\{\lambda_{j}\}$ be defined by $\lambda_{j}:=m_{j}^{\frac{1-p}{2s}}\rightarrow0$, as $j\rightarrow+\infty$. For $x\in\Omega_{j}:=\Omega_{\lambda_{j}}=\{x\in\mathbb{R}^{n}\,|\,\lambda_{j}x+x_{j}\in\Omega\}$, we define
	\begin{equation}\label{f-5}
		v_{j}(x):=\frac{1}{m_{j}}u_{j}(\lambda_{j}x+x_j).
	\end{equation}
	Then $v_{j}(x)$ satisfies $\|v_{j}\|_{L^{\infty}(\overline{\Omega_{j}})}=v_{j}(0)=1$ and
	\begin{equation}\label{f-6}
		\begin{cases}
			(-\Delta)^{s}v_j(x)=f_j(x, v_j(x)),  & x \in \Omega_j,\\
			v_j(x)\equiv0, &x \not\in \Omega_j,
		\end{cases}
	\end{equation}
	where $f_j(x, v_j(x)):=\frac{f(\lambda_{j}x+x_j, m_j v_j(x))}{m_j^p}$.

\medskip
	
	Let $d_j=dist(x_j, \partial\Omega)$. We will carry out the proof using the contradiction argument while exhausting all three possibilities.
	
\vspace{12pt}
	
	\emph{Case (i).} $\lim\limits_{j \rightarrow +\infty} \frac{d_j}{\lambda_j} =+\infty$.

\medskip

In this case, it can be seen easily that
$$\Omega_j \rightarrow \mathbb{R}^n, \,\  \  \   \text{ as } j \to +\infty.$$
	In order to obtain a limiting equation, we need the following interior H\"older estimate for fractional Laplacians (see e.g. Proposition 2.3 and Corollary 2.4 in \cite{RS} or Propositions 2.8 and 2.9 in \cite{Si}).
	\begin{prop}[Interior regularity]\label{fr-in-re}
		The following two interior regularity results hold.
		\begin{itemize}
			\item[(a)]If   $w\in L^\infty(\mathbb R^n)$ and $(-\Delta)^s w\in L^\infty(B_2(0))$, then for every $\beta\in (0, 2s)$, one has $w\in C^\beta(\overline{B_1(0)})$ and for some constant $C>0$,
 $$\|w\|_{C^\beta(\overline{B_1(0)})}\leq C\left(\|w\|_{L^\infty(\mathbb R^n)}+\|(-\Delta)^s w\|_{L^\infty(B_2(0))}\right).$$
			\item[(b)]	If $(1+|x|)^{-n-2s}w\in L^1({\mathbb R^n})$ and  $(-\Delta)^s w\in C^\beta(\overline{B_2(0)})$ for some $\beta$ such that neither $\beta$ nor $\beta+2s$ is an integer. Then, there is a constant $C$ such that
			$$\|w\|_{C^{\beta+2s}(\overline{B_1(0)})}\leq C\left(\|(1+|x|)^{-n-2s}w\|_{L^1(\mathbb R^n)}+\|w\|_{C^{\beta}(\overline{B_2(0)})}+\|(-\Delta)^s w\|_{C^\beta(\overline{B_2(0)})}\right).$$
		\end{itemize}
	\end{prop}
	
	Since $\lim\limits_{j \rightarrow +\infty} \frac{d_j}{\lambda_j} =+\infty$, for any $R>0$, one has  $B_{2R}(0) \subset \Omega_j$  when $j$ is sufficiently large. For $j$ large, by using assumptions \eqref{f-2} and \eqref{f-2-2}, we find
\begin{equation}\label{fj}
  \|f_j\|_{L^\infty(B_{2R}(0))}\leq 2(1+\|h\|_{L^\infty(\Omega)}).
\end{equation}
	Therefore, we infer from Proposition \ref{fr-in-re} (a) that
	\begin{equation}\label{f-9}
		\|v_j\|_{C^\beta(\overline{B_R(0)})}\leq C_R(1+\|h\|_{L^\infty(\Omega)})
	\end{equation}
	for every $\beta\in (0, 2s)$ and some constant $C_R$ depending on $R$ but independent of $j$.  Now we verify that $f_j\in C^\sigma( \overline{B_R(0)})$ for some $\sigma>0$. Indeed, by assumptions \eqref{f-2-1} and \eqref{f-2-2} and the regularity of $v_j$ in \eqref{f-9}, we obtain, for any $x,y\in B_R(0)$,
	\begin{align}\label{fj1}
		&|f_j(x, v_j(x))-f_j(y, v_j(y))|\\
		\nonumber =&m_j^{-p}|f(\lambda_{j} x+x_j, m_j v_j(x))-f(\lambda_{j}y+x_j, m_j v_j(y))|\\
		\nonumber \leq& m_j^{-p}|f(\lambda_{j}x+x_j, m_j v_j(x))-f(\lambda_{j}x+x_j, m_j v_j(y))|\\
		\nonumber &+m_j^{-p}|f(\lambda_{j}x+x_j, m_j v_j(y))-f(\lambda_{j}y+x_j, m_j v_j(y))|\\
		\nonumber \leq &C m_j^{-p+\sigma_2}(1+m_j|v_j(x)|+m_j|v_j(y)|)^{p-\sigma_2} |v_j(x)-v_j(y)|^{\sigma_2}\\
		\nonumber &+Cm_j^{-p}(1+m_j|v_j(y)|)^{p+\frac{(p-1)\sigma_1}{2s}}|(\lambda_j x+x_j)-(\lambda_j y+x_j)|^{\sigma_1}\\
		\nonumber \leq& C|x-y|^{\min\{\beta\sigma_2, \sigma_1\}}.
	\end{align}
	Proposition \ref{fr-in-re} (b) implies that, there exists a small constant $\epsilon>0$ such that
	\begin{equation}\label{f-10}
		\|v_j\|_{C^{[2s],\{2s\}+2\epsilon}\left(\overline{B_{\frac{R}{2}}(0)}\right)}\leq C_R
	\end{equation}
	for some constant $C_{R}$ depending on $R$ but independent of $j$. Then Arzel$\grave{a}$-Ascoli theorem and a diagonal argument give a subsequence (still denoted by $\{v_j\}$) that converges to some nonnegative function $v$ in $C^{[2s],\{2s\}+ \epsilon}(K)$ for any compact subset $K\subset\mathbb R^n$. Moreover, by the pointwise convergence, one has $v(0)=\lim\limits_{j\to+\infty}v_j(0)=1$.
	
\medskip

	Next, we show that
	\begin{equation}\label{f-11}
		(-\Delta)^s v_j(x)\rightarrow (-\Delta)^s v(x),\ \qquad \ \forall\, x\in \mathbb{R}^n, \ \ \ \text{as}\, j\to+\infty.
	\end{equation}
	For any fixed $x\in \mathbb{R}^n$, taking $R=2(|x|+1)$ in \eqref{f-10}, we get, for $y\in B_1(0)$,
$$\frac{|2v_j(x)-v_j(x-y)-v_j(x+y)|}{|y|^{n+2s}}\leq C\frac{ |y|^{2s+\epsilon}}{|y|^{n+2s}}=\frac{C}{|y|^{n-\epsilon}}.$$
	On the other hand, we have, for $y\in \mathbb{R}^n\setminus B_1(0)$,
	$$\frac{|2v_j(x)-v_j(x-y)-v_j(x+y)|}{|y|^{n+2s}}\leq  \frac{4}{|y|^{n+2s}}\leq \frac{8}{1+|y|^{n+2s}}.$$ Hence, we obtain $$\frac{|2v_j(x)-v_j(x-y)-v_j(x+y)|}{|y|^{n+2s}} \leq  \frac{C}{1+|y|^{n+2s}}\left(1+\frac{1}{|y|^{n-\epsilon}}\right),\ \quad \ \forall\, y\in \mathbb{R}^n.$$
	Since $\frac{1}{1+|y|^{n+2s}}\left(1+\frac{1}{|y|^{n-\epsilon}}\right)\in L^{1}(\mathbb R^n)$, by the definition of $(-\Delta)^{s}$ and applying the Lebesgue's dominated convergence theorem, we derive \eqref{f-11}.
	
\medskip

	Since we have proved in \eqref{fj} and \eqref{fj1} that $\|f_j\|_{C^\sigma(B_R(0))}\leq C_R$ for any $R>0$ and some $\sigma>0$, by Arzel$\grave{a}$-Ascoli theorem, we can select a subsequence (still denoted by $\{f_{j}\}$) such that $f_j\to f_0$ uniformly in any compact set $K$ in $\mathbb R^n$ as $j\to+\infty$ for some $f_0\in L^\infty(\mathbb R^n)\cap C^{\frac{\sigma}{2}}(\mathbb R^n)$.  Then we infer from \eqref{f-11} that $v$ satisfies
	$$(-\Delta)^s v(x)=f_0(x), \qquad v(x)\geq 0,\,\qquad\, x\in \mathbb{R}^n.$$	
Now we determine the expression of $f_0$. Without loss of generalities, we may assume that $x_j\to x_0$ for some $x_0\in \overline \Omega$. For any $x\in \mathbb R^n$ such that $v(x)>0$,  one has $ v_{j}(x)\geq \frac{v(x)}{2}>0$   for $j$ sufficiently large. Then the assumption \eqref{f-2} on $f$ implies
	$$  f_j(x,  v_{j }(x))= \frac{f(\lambda_{j}x+x_j, m_j v_j(x))}{m_j^p}\to h(x_0)  v^p(x),  $$
	as $j\to +\infty$. That is, \begin{equation}\label{f-12}
		 f_0(x)=h(x_0)  v^p(x), \ \qquad  \  \forall\,\, x\in \mathbb R^n \ \ \text{with}\,\,  v(x)>0.
	\end{equation}
     We claim that
	\begin{equation}\label{f-13}
		v(x)>0,\ \qquad \ \forall\,\, x\in \mathbb R^n.
	\end{equation}
Indeed, suppose that the set $\{x\mid\,v(x)=0\}\not=\emptyset$, then there exist a point $\tilde x\in \partial \{x\mid\,v(x)>0\}$ due to $v(0)=1>0$. Using \eqref{f-12}, the continuity of $f_0$ and $v$, we have $$f_0(\tilde x)=0.$$ However, direct calculations yields that
$$ f_0(\tilde x)= (-\Delta)^s v(\tilde x)=C_{s,n} \, P.V.\int_{\mathbb{R}^n}\frac{v(\tilde x)-v(y)}{|\tilde x-y|^{n+2s}}\mathrm{d}y=C_{s,n} \, P.V.\int_{\mathbb{R}^n}\frac{ -v(y)}{|\tilde x-y|^{n+2s}}\mathrm{d}y<0,$$
which is absurd. So the claim \eqref{f-13} must hold and we infer from \eqref{f-12} that
$v$ satisfies
$$(-\Delta)^s v(x)=h(x_0)  v^p(x), \qquad v(x)\geq 0,\,\,\,\quad \forall \,\, x\in \mathbb{R}^n.$$	
Then the Liouville theorem in $\mathbb{R}^{n}$ in \cite{CLL} (also contained in our Theorems \ref{pde-unbd}, \ref{pde-unbd-a}, \ref{pde-unbd-2}, \ref{pde-bd}, \ref{pde-bd-2} and \ref{f2PDE}) implies $v\equiv0$, which contradicts $v(0)=1$.

\vspace{12pt}

	\emph{Case (ii).} $\lim\limits_{j \rightarrow+\infty} \frac{d_j}{\lambda_j} =C>0$.

\medskip	
	
	For each $j$, we can find a point $\bar x_j\in \partial\Omega$ such that $|\bar x_j-x_j|=d_j$. Since  $\lim\limits_{j \rightarrow \infty} \frac{d_j}{\lambda_j} =C>0$, we may assume that $\lim\limits_{j \rightarrow \infty} \frac{\bar x_j-x_j}{\lambda_j} =\hat{x}_{0}$ for some point $\hat{x}_{0}\in \mathbb R^n$. Note that
$$\Omega_j=\{x\in\mathbb R^n\  \mid \lambda_j x+x_j\in \Omega\}=\left\{x\in\mathbb R^n\  \mid \lambda_j x+\bar x_j\in \Omega\right\}+\frac{\bar x_j-x_j}{\lambda_j}$$
and $\Omega$ is a BCB domain, one immediately has (up to a subsequence)
	$$\Omega_j\rightarrow \mathcal C, \ \ \quad \text{as}\,\, j\to+\infty$$
for some cone $\mathcal C$ containing $0$.  	
	
\medskip

Since the boundary H\"{o}lder estimates for $(-\Delta)^{s}$ holds uniformly on the scaling domain $\Omega_{\rho}$ w.r.t. the scale $\rho\in(0,1]$, applying the boundary H\"older estimates to equation \eqref{f-6}, we obtain a constant $C_R$ independent of $j$ such that, for $j$ large,
	\begin{equation}\label{f-15}
		\| v_{j }\|_{C^{\alpha}(\Omega_j\cap B_{ R}(0))} \leq C_R
	\end{equation}
	for some $\alpha>0$ and any $R>0$. In the cone $\mathcal C$, through arguments similar as \emph{Case (i)}, one can prove that there is a subsequence (still denoted by $\{v_j\}_{j=1}^\infty$) that converges to some bounded function $v$ pointwisely in $\mathbb R^n$. Estimate \eqref{f-15} ensures that passing to a subsequence, the limiting functions $ v \in C(\overline{\mathcal{C}})$ and $  v =0 $ on $\mathbb R^n\setminus \mathcal C$. In addition, $ v_j\rightarrow v$ in $C^{[2s],\{2s\}+ \epsilon}(K)$ and hence $(-\Delta)^s v_j\rightarrow (-\Delta)^s v$ pointwisely for any compact subset $K\subset\mathcal C$. Moreover, one has $v(0)=\lim\limits_{j\to+\infty}v_j(0)=1$ and $v>0$ in $\mathcal C$. Therefore, we obtain a function $v$ satisfying
	\begin{equation*}
		\begin{cases}
			(-\Delta)^sv(x)=h(x_0)v^p(x), \qquad v(x)\geq 0,\,\,\, &x\in \mathcal C,\\
			v(x)=0,&x\in \mathbb R^n\setminus \mathcal C,
		\end{cases}
	\end{equation*}
	which implies $v\equiv0$ by our Liouville theorems in cones in Theorems \ref{pde-unbd}, \ref{pde-unbd-a}, \ref{pde-unbd-2}, \ref{pde-bd}, \ref{pde-bd-2}, \ref{f2PDE} and Corollary \ref{cone-cor}. This contradicts $v(0)=1$.
	
	\vspace{12pt}

	\emph{Case (iii).} $\lim\limits_{j \rightarrow+\infty} \frac{d_j}{\lambda_j} = 0$.

\medskip
	
	Now we only need to rule out \emph{Case (iii)}. In this case, there is a point $x_0\in \partial \Omega$ such that $$x_j\to x_0,\ \qquad \ \text{as}\,\, j\to+\infty.$$ We can find a point $\bar x_j\in \partial \Omega$ with $|\bar x_j-x_j| =d_j$. Denote $\tilde x_j:= \frac{\bar x_j-x_j}{\lambda_j}$, it can be seen that  $\tilde x_j\in \partial \Omega_j$ and $|\tilde x_j|\to 0$ as $j\to+\infty$. On the one hand, since the boundary H\"{o}lder estimates for $(-\Delta)^{s}$ holds uniformly on the scaling domain $\Omega_{\rho}$ w.r.t. the scale $\rho\in(0,1]$, using the boundary H\"older estimate \eqref{f-15}, there exist $\alpha>0$ and $C>0$ such that $\|v_j\|_{C^\alpha(B_1(0)\bigcap\Omega_{j})}\leq C$ for sufficiently large $j$. This implies
$$|v_j(\tilde x_j)-v_j(0)|\leq  C|\tilde x_j|^\alpha=o_j(1).$$
On the other hand, one easily has $$v_j(0)-v_j(\tilde x_j)=1,$$
which leads to a contradiction. The proof of Theorem \ref{prid} is therefore finished.
\end{proof}

\subsection{Proof of Theorem \ref{prid2}}

\begin{proof}[Proof of Theorem \ref{prid2}]

Suppose $\Omega$ is a bounded BCB domain such that the boundary H\"{o}lder estimates for the general 2nd order elliptic operator $L$ holds uniformly on the scaling domain $\Omega_{\rho}$ w.r.t. the scale $\rho\in(0,1]$. We will prove the a priori estimates in Theorem \ref{prid2} by contradiction arguments and the Liouville theorems (Theorems \ref{pde-unbd}, \ref{pde-unbd-a}, \ref{pde-bd}, \ref{pG}, \ref{f2PDE} and Corollary \ref{cone-cor}). Suppose on the contrary that \eqref{0-4-2} does not hold, then there exists a sequence of solutions $\{u_j\}_{j=1}^\infty$ to \eqref{PDE-D} and a sequence of points $\{ x_j\}_{j=1}^\infty \subset \Omega$ such that
		\begin{equation}\label{N-4}
			u_j(  x_j)=\max_{x\in\Omega} \,u_j(x)=:m_j \rightarrow +\infty, \quad  \ \ \text{as}\, j\to+\infty.
		\end{equation}

\medskip
		
	 	Let the scale sequence $\{\lambda_{j}\}$ be defined by $\lambda_{j}:=m_{j}^{\frac{1-p}{2}}\rightarrow0$, as $j\rightarrow+\infty$. For $x\in\Omega_{j}:=\Omega_{\lambda_{j}}=\{x\in\mathbb{R}^{n}\,|\,\lambda_{j}x+x_j\in\Omega\}$, we define
	 \begin{equation}\label{L-5}
	 	v_{j}(x):=\frac{1}{m_{j}}u_{j}(\lambda_{j}x+x_j).
	 \end{equation}
	 	Then $v_{j}(x)$ satisfies $\|v_{j}\|_{L^{\infty}( \Omega_{j})}=v_{j}(0)=1$ and
	 \begin{equation}\label{L-6}
	 	\begin{cases}
	 		A_j (x)v_j(x)+\sum\limits_{m=1}^{n}b_{j,m}\partial_{x_m} v_j(x)+ c_j(x) v_j(x)=f_j(x, v_j(x)),  & x \in \Omega_j,\\
	 		v_j(x)\equiv0, &x  \in \partial \Omega_j,
	 	\end{cases}
	 \end{equation}
	 where $A_j (x)= -\sum\limits_{i,m=1}^{n}  a_{j,im}(x)\frac{\partial^{2}}{\partial x_{i}\partial x_{m}}$ with $ a_{j,im}(x):=a_{im}(\lambda_j x+x_j)$, $  b_{j,m}(x)=\lambda_j b_m(\lambda_jx+x_j)$, $ c_j(x)=\lambda_j^2 c(\lambda_jx+x_j)$  and $f_j(x, v_j(x))=\frac{f(\lambda_{j}x+x_j, m_j v_j(x))}{m_j^p}$.

\medskip
	
	 Let $d_j=dist(x_j, \partial\Omega)$. We will carry out the proof using the contradiction argument while exhausting all three possibilities.

\vspace{12pt}
		
			\emph{Case (i).} $\lim\limits_{j \rightarrow+\infty} \frac{d_j}{\lambda_j} =+\infty$.

\medskip

In this case, it can be easily seen that $$\Omega_j  \rightarrow \mathbb{R}^{n},\,\  \  \   \text{ as } j \to +\infty.$$
		Then, for any $R>0$, one has  $B_{2 R}(0) \subset \Omega_j$  when $j$ is sufficiently large.

\medskip
		
			We assume $x_j\to \bar x\in\overline{\Omega}$ in the sense of subsequence. Note that   the ellipticity constant and the $L^\infty$ -norm of the coefficients of 2nd operators $A_j$ are the same as $A$, but the modulus of continuity of $a_{j,im}(x) $ is arbitrarily smaller than that of $a_{im}$ for $j$ large enough, and $\|b_{j,m}\|_{L^\infty}$ and $\|c_j\|_{L^\infty}$ can be as small as we wish for sufficiently large $j$.
		By assumptions \eqref{0-3-2} and \eqref{0-4-2} on $f$, we find that, for $j$ large, 	
\begin{equation}\label{L-8}
				\|f_j\|_{L^\infty(\Omega_j)}\leq C(1+\|h\|_{L^\infty(\Omega)}).
			\end{equation}
		Therefore, using  the interior $W^{2,p}$ estimate for elliptic equations (see e.g. Theorem 9.11 in \cite{GT}), we find, for any $1<q<+\infty$,
		\begin{equation}\label{L-9}
			\| v_{j }\|_{W^{2, q}(B_{  R}(0))}\leq C_R
		\end{equation}
		for some constant $C_R$ depending only on $R$.  Then the Sobolev embedding Theorem  gives
		\begin{equation}\label{L-10}
			\| v_{j }\|_{C^{1,\frac{1}{2}}(B_{  R}(0))}\leq C_R.
		\end{equation}
	    Since $R>0$ is arbitrary, the   Arzel$\grave{a}$-Ascoli theorem  and a diagonal argument give a subsequence (still denoted by $\{  v_{j }\}$) that converges to some function $ v $   in $W^{1,q}_{loc}(\mathbb R^n)$ and $C^{1,\frac{1}{3}}_{loc}(\mathbb R^n)$ as $j\to+\infty$. Moreover, one can verify that
		\begin{equation}\label{L-12}
			   v(0)=\lim_{j\to+\infty}    v_{j } (0)=1
		\end{equation}  and $0\leq   v \leq 1$ in $\mathbb R^n$  by the pointwise convergence. Since $ \|\nabla^2  v_{j}\|_{L^{q}(B_{  R}(0))}\leq C_R$, we may assume that $\nabla^2 v_{j }\to \nabla^2   v $ weakly in $L^q$. Now using the uniform boundedness of $f_j$ in \eqref{L-8}, we may assume that $f_j\to f_0$ weakly star in $L^\infty$ for some $f_0\in L^\infty(\mathbb R^n)$. Thus we obtain that $v \in W^{1,q}_{loc}(\mathbb R^n)$ satisfies the following equation in weak sense (c.f. e.g. Chapter 8 in \cite{GT} for weak solutions to elliptic equations):
	    \begin{equation}\label{L-13}
	    	-\sum\limits_{i,m=1}^{n}a_{im}(\bar x) \frac{\partial ^2}{\partial x_i \partial x_m} v = f_0,\ \qquad  0\leq v \leq 1 \ \qquad \, \text{in}\,\, \mathbb{R}^n.
	    \end{equation}
		Since $f_0\in  L^\infty(\mathbb R^n)$, we infer from the  $W^{2,p}$ estimate for elliptic equations that $v\in  W^{2,q}_{loc}(\mathbb R^n)$ for any $1<q<+\infty$, which means that $v$ is a strong solution to \eqref{L-13}.

\medskip
		
		Next, we determine the value of $f_0$ in \eqref{L-13}. For any $x\in \mathbb R^n$ such that $v(x)>0$,  one has $ v_{j}(x)\geq \frac{v(x)}{2}>0$   for $j$ sufficiently large. Then the assumption \eqref{0-3-2} on $f$ implies
		$$  f_j(x,  v_{j }(x))= \frac{f(\lambda_{j}x+x_j, m_j v_j(x))}{m_j^p}\to h(\bar x)v^p(x),$$
		as $j\to +\infty$.   On the other hand, we have ${\partial x_i \partial x_m} v =0\,\, \text{a.e. on }\, \{x\mid \, v(x)=0\}$ due to $v\in  W^{2,q}_{loc}(\mathbb R^n)$. Hence we conclude that $v$ is a strong solution of
		 \begin{equation}\label{L-14}
			-\sum\limits_{i,m=1}^{n}a_{im}(\bar x) \frac{\partial ^2}{\partial x_i \partial x_m} v = h(\bar x)v^p,\quad\,\, \ 0\leq v \leq 1 \ \ \, \text{in}\,\, \mathbb{R}^n.
		\end{equation}
		Since $h(\bar x)v^p\geq 0$, $v\geq 0$ and $v\not\equiv0$ in $\mathbb{R}^n$, we deduce from the strong maximum principle (see Theorem 9.6 in \cite{GT}) that $v>0$ in $\mathbb{R}^n$.
	
\medskip
	
Standard regularity theory and a bootstrap argument yield  that $ v \in C^\infty(\mathbb{R}^n)$. Through a suitable variable transformation (still denote the resulting function by $v$), we obtain the standard form of equation for $v$:
		\begin{equation*}
			 -\Delta   v =  v ^p,\ \quad\,  v>0  \ \ \, \text{in}\,\, \mathbb{R}^n,
		\end{equation*}
		which is a contradiction to the Liouville theorem in $\mathbb{R}^{n}$ proved by Gidas and Spruck in their celebrated article \cite{GS} (also contained in our Theorems \ref{pde-unbd}, \ref{pde-unbd-a}, \ref{pde-bd} and \ref{f2PDE}).

\vspace{12pt}
		
		\emph{Case (ii).} $\lim\limits_{j \rightarrow+\infty} \frac{d_j}{\lambda_j} =C>0$.
		
\medskip
		
		For each $j$, we can find a point $\bar x_j\in \partial\Omega$ such that $|\bar x_j-x_j|=d_j$. Since  $\lim\limits_{j \rightarrow +\infty} \frac{d_j}{\lambda_j} =C>0$, we may assume that $\lim\limits_{j \rightarrow \infty} \frac{\bar x_j-x_j}{\lambda_j} =\hat x_0 $ for some point $\hat x_0\in \mathbb R^n$. Note that
$$\Omega_j=\{x\in\mathbb R^n\  \mid \lambda_j x+x_j\in \Omega\}=\left\{x\in\mathbb R^n\  \mid \lambda_j x+\bar x_j\in \Omega\right\}+\frac{\bar x_j-x_j}{\lambda_j}$$
and $\Omega$ is a BCB domain, one immediately has  (up to a subsequence)
		$$\Omega_j\rightarrow \mathcal C, \ \quad \ \text{as}\, j\to+\infty$$
for some cone $\mathcal C$ containing $0$.

\medskip
		
		Since the boundary H\"{o}lder estimates for the general 2nd order elliptic operator $L$ holds uniformly on the scaling domain $\Omega_{\rho}$ w.r.t. the scale $\rho\in(0,1]$, applying the boundary H\"older estimates to equation \eqref{L-6}, we obtain a constant $C_R$ independent of $j$ such that, for $j$ large,
		\begin{equation}\label{L-15}
		   \| v_{j }\|_{C^{\alpha}(\Omega_j\cap B_{ R}(0))} \leq C_R
		\end{equation}
		for some $\alpha>0$ and any $R>0$. Estimate \eqref{L-15} ensures that passing to a subsequence, the limiting functions $ v \in C(\overline {\mathcal C})$ and $  v =0 $ on $\partial \mathcal C$. In the cone $\mathcal C$,  using the interior $W^{2,p}$ estimate  for elliptic equations and the Sobolev embedding theorem, through arguments similar as \emph{Case (i)}, one can prove that the limiting functions $v \in C^\infty(\mathcal C)\cap  C(\overline {\mathcal C})$ solve the same equation as \eqref{L-14}. That is, we arrive at
		\begin{equation*}
			\begin{cases}
				-\sum\limits_{i,m=1}^{n}a_{im}(\bar x ) \frac{\partial ^2}{\partial x_i \partial x_m}   v = h(\bar x ) v^p &\text{in}\,\,  \mathcal C,\\
				0\leq  v \leq 1 &\text{in}\,\,  \mathcal C,\\
				 v=0 &\text{on}\,\, \partial  \mathcal C,\\
				 v(0)=1.
			\end{cases}
		\end{equation*}
		Through a suitable transformation (still denote the resulting function by $v$), we obtain the standard form of equation for $v$:
		\begin{equation*}
			\begin{cases}
				 -\Delta   v =  v ^p,\ \    v \geq0, \,\, v\not\equiv0\,\,&\text{in}\,\,  \widetilde{\mathcal C},\\
				 v =0 &\text{on}\,\, \partial \widetilde {\mathcal C},
			\end{cases}
		\end{equation*}
		where the cone $\widetilde {\mathcal C}$ is the transformation of the cone $\mathcal  C$. This contradicts the Liouville theorem on cones in our Theorems \ref{pde-unbd}, \ref{pde-unbd-a}, \ref{pde-bd}, \ref{f2PDE} and Corollary \ref{cone-cor}.

\vspace{12pt}		

		\emph{Case (iii).} $\lim\limits_{j \rightarrow +\infty} \frac{d_j}{\lambda_j} =0$.

\medskip
		
Now we only need to rule out \emph{Case (iii)}. In this case, there is a point $x_0\in \partial \Omega$ such that
$$x_j\to x_0,\ \quad \ \text{as}\, j\to+\infty.$$
We can find a point $\bar x_j\in \partial \Omega$ with $|\bar x_j-x_j| =d_j$. Denote $\tilde x_j:= \frac{\bar x_j-x_j}{\lambda_j}$, it can be seen that  $\tilde x_j\in \partial \Omega_j$ and $|\tilde x_j|\to 0$ as $j\to+\infty$. On the one hand, the boundary H\"older estimate \eqref{L-15} implies
$$ |  v_{j }(\tilde x_j)- v_{j}(0)|\leq  C|\tilde x_j|^\alpha=o_j(1).$$
On the other hand, one easily gets $$ | v_{j }(\tilde x_j)-  v_{j }(0)| =   v_{j } (0)=1,$$
which leads to a contradiction. The proof of Theorem \ref{prid2} is therefore finished.
\end{proof}

\subsection{Proof of Theorem \ref{prin}}

\begin{proof}[Proof of Theorem \ref{prin}]

Suppose $\Omega$ is a bounded BCB domain such that the limiting cone $\mathcal{C}$ is $\frac{1}{2^{k}}$-space ($k=1,\cdots,n$) and the boundary H\"{o}lder estimates for $-\Delta$ holds uniformly on the scaling domain $\Omega_{\rho}$ w.r.t. the scale $\rho\in(0,1]$. We will prove the a priori estimates in Theorem \ref{prin} by contradiction arguments and the Liouville theorems for higher order Navier problem on $\frac{1}{2^{k}}$-space $\mathcal{C}_{P,S^{n-1}_{2^{-k}}}$ (Theorems \ref{nie} and \ref{NPDE} in subsection 1.3). Suppose on the contrary that \eqref{0-4-n} does not hold, then there exists a sequence of solutions $\{u_j\}_{j=1}^\infty$ to \eqref{PDE-N} and a sequence of points $\{\hat x^j\}_{j=1}^\infty \subset \Omega$ such that
		\begin{equation}\label{N-4}
			u_j(\hat x^j)=\max_{x\in\Omega} \,u_j(x) \rightarrow +\infty, \quad  \ \ \text{as}\,\, j\to+\infty.
		\end{equation}
		
\medskip

		 We define $$v_{j,l}=(-\Delta)^{l}  u_j, \,\qquad\, l=0, \cdots, s-1,\quad\, j=1,2,\cdots,$$
where we assume $v_{j,0}=(-\Delta)^{0}  u_j=u_j$. Then, we infer from \eqref{PDE-N} that $(v_{j,0}, \cdots, v_{j,s-1})$ satisfies the following system:
		\begin{equation}\label{N-S}
			\begin{cases}
				-\Delta v_{j,0}=v_{j,1} \  \ \  \ \  \ \  \ \ \ &\text{in}\,\, \Omega,\\
				\cdots,\\
				-\Delta v_{j,s-2}=v_{j,s-1} &\text{in}\,\, \Omega,\\
				-\Delta v_{j,s-1}=f(x, v_{j,0}) &\text{in}\,\, \Omega,\\
				v_{j,0}=\cdots=v_{j,s-1}=0 &\text{on}\,\, \partial \Omega.
			\end{cases}
		\end{equation}
		 The regularity assumption on $u_j$ implies $v_{j,l}\in C^2(\Omega)\cap C(\overline {\Omega})$ for $l=0, \cdots, s-1$. Moreover, since $f\geq 0$ in $\Omega$, we deduce from the maximum principle that $v_{j,s-1}\geq 0$ in $\Omega$. If $v_{j,s-1}(x_0)=0$ for some $x_0\in \Omega$, then the strong maximum principle yields that $ v_{j,s-1}\equiv0$ in $\Omega$, which further implies $ v_{j,s-2}\equiv0$ in $\Omega$ by the equation
$$\begin{cases} 	-\Delta v_{j,s-2}=v_{j,s-1} &\text{in}\,\, \Omega,\\ v_{j,s-1}=0 &\text{on}\,\, \partial \Omega \end{cases}$$
and the strong maximum principle. Repeating the argument, we will obtain $u_{j}=v_{j,0}\equiv0$ in $\Omega$, which contradicts \eqref{N-4}. Thus we must have $ v_{j,s-1}>0$ in $\Omega$. Similarly, one can prove $v_{j,l}> 0$ in $\Omega$ for $l=s-2,\cdots, 0$ by the strong maximum principle.
		
\medskip
		
		Set  $$\beta_l=2l+\frac{2s}{p-1},\,\qquad\, l=0, \cdots, s-1.$$
		In view of \eqref{N-4}, there exists $x_j\in \Omega$ such that
		\begin{equation}\label{N-4-1}
			 \sum_{l=0}^{s-1} v_{j,l}^{\frac{1}{\beta_l}}(x_j)=\max_{x\in\Omega} \,\left(\sum_{l=0}^{s-1} v_{j,l}^{\frac{1}{\beta_l}}(x)\right) \rightarrow +\infty, \  \qquad \ \text{as}\, j\to+\infty.
		\end{equation}
        Let $\lambda_j:=\left( \sum\limits_{l=0}^{s-1} v_{j,l}^{\frac{1}{\beta_l}}(x_j)\right)^{-1}$. One has $\lambda_j\to0$, as $j\to+\infty$.
		For $x\in\Omega_{j}:=\Omega_{\lambda_{j}}=\{x\in\mathbb{R}^{n}\,|\,\lambda_{j}x+x^{j}\in\Omega\}$, we define
		\begin{equation}\label{N-5}
			\tilde v_{j,l}(x):=\lambda_j^{\beta_l} v_{j,l}(\lambda_{j}x+x^{j}), \,\qquad\, l=0, \cdots, s-1.
		\end{equation}
		Then $\tilde v_{j,l}(x)$ satisfies $\sum\limits_{l=0}^{s-1} {\tilde v_{j,l}}^{\frac{1}{\beta_l}}(0)=1$ and for any $l=0, \cdots, s-1,\,\, j=1,2,\cdots,$
		\begin{equation}\label{N-6}
			\max_{\Omega_j} \tilde v_{j,l}=\left(\max_{\Omega_j} \tilde v_{j,l}^{\frac{1}{\beta_l}}\right)^{\beta_l}\leq \left(\max_{\Omega_j}\left(\sum_{m=0}^{s-1} {\tilde v_{j,m}}^{\frac{1}{\beta_m}}\right)\right)^{\beta_l}=\left(\lambda_j\max_{\Omega }\left(\sum_{m=0}^{s-1}v_{j,m}^{\frac{1}{\beta_m}}\right)\right)^{\beta_l}=1.
		\end{equation}

\medskip		
		
		By direct calculations, we infer from \eqref{N-S} and \eqref{N-6} that $(\tilde v_{j,0}, \ldots, \tilde v_{j,s-1})$ satisfies the following system:
	\begin{equation}\label{N-7}
		\begin{cases}
			-\Delta \tilde v_{j,0}=\tilde v_{j,1} \  \ \  \ \  \ \  \ \ \ &\text{in}\,\, \Omega_j,\\
			\cdots,\\
			-\Delta  \tilde v_{j,s-2}=\tilde v_{j,s-1} &\text{in}\,\, \Omega_j,\\
			-\Delta \tilde v_{j,s-1}=\tilde f_j(x, \tilde v_{j,0}) &\text{in}\,\, \Omega_j,\\
			0<\tilde v_{j,0}, \ldots, \tilde v_{j,s-1}\leq 1 &\text{in}\,\, \Omega_j,\\
			\tilde v_{j,0}=\cdots=\tilde v_{j,s-1}=0 &\text{on}\,\, \partial \Omega_j,
		\end{cases}
	\end{equation}
		where   $\tilde f_j(x, \tilde v_{j,0})=\lambda_j^{2+\beta_{s-1}}f(\lambda_jx+x_j, \lambda_j^{-\beta_0}  \tilde v_{j,0})= \frac{f(\lambda_jx+x_j, \lambda_j^{-\beta_0}  \tilde v_{j,0})}{\lambda_j^{-p\beta_0}}$.
		
\medskip

		Let $d_{j}:=dist(x_j, \partial\Omega)$. We will carry out the proof using the contradiction argument while exhausting all three possibilities.
		
\vspace{12pt}
		
		\emph{Case (i).} $\lim\limits_{j \rightarrow +\infty} \frac{d_j}{\lambda_j} =+\infty$.

\medskip

In this case, it can be seen easily that $$\Omega_j  \rightarrow \mathbb{R}^n,\,\  \  \   \text{ as } j \to +\infty.$$
Then, for any $R>0$, one has  $B_{2^{s+1}R}(0) \subset \Omega_j$, when $j$ is sufficiently large.

\medskip

	Using assumption \eqref{0-3-n} and \eqref{0-6-n}, we find, for $j$ large,
	\begin{equation}\label{N-8}
		 \| \tilde f_j\|_{L^\infty(\Omega_j)}\leq C(1+\|h\|_{L^\infty(\Omega)}).
	\end{equation}
		Therefore, using  the interior $W^{2,p}$ estimate for elliptic equations, we get, for any $1<q<+\infty$,
		\begin{equation}\label{N-9}
			\|\tilde v_{j,s-1}\|_{W^{2, q}(B_{2^s R}(0))}\leq C_R
		\end{equation}
		for some constant $C_R$ depending only on $R$.  Then the Sobolev embedding theorem  gives
			\begin{equation}\label{N-10}
			\|\tilde v_{j,s-1}\|_{C^{1,\frac{1}{2}}(B_{2^s R}(0))}\leq C_R,
		\end{equation}
	which in particular implies that $\|\tilde v_{j,s-1}\|_{L^\infty(B_{2^s R}(0))}\leq C_R.$ Then by the equation $ -\Delta \tilde v_{j,s-2}=\tilde v_{j,s-1} \,\, \text{in}\,\, \Omega_j$, the interior $W^{2,p}$ estimate  for elliptic equations and the Sobolev embedding theorem, we get $$\|\tilde v_{j,s-2}\|_{W^{2, q}(B_{2^{s-1} R}(0))}+\|\tilde v_{j,s-2}\|_{C^{1,\frac{1}{2}}(B_{2^{s-1} R}(0))}\leq C_R. $$
Repeating this procedure, we can show that
	\begin{equation}\label{N-11}
		\sum_{l=0}^{s-1} \left(\|\tilde v_{j,l}\|_{W^{2, q}(B_{2  R}(0))} + \|\tilde v_{j,l}\|_{C^{1,\frac{1}{2}}(B_{2  R}(0))}\right)\leq C_R
	\end{equation}
for some constant $C_R>0$ independent of $j$. Since $R>0$ is arbitrary, the Arzel$\grave{a}$-Ascoli theorem and a diagonal argument give a subsequence (still denoted by $\{\tilde v_{j,l}\}$) that converges to some function $\tilde v_l$ in $W^{1,q}_{loc}(\mathbb R^n)$ and $C^{1,\frac{1}{3}}_{loc}(\mathbb R^n)$ as $j\to+\infty$ for  $l=0, \ldots, s-1$. Moreover, one can verify that
\begin{equation}\label{N-12}
	\sum\limits_{l=0}^{s-1} {\tilde v_{l}}^{\frac{1}{\beta_l}}(0)=\lim_{j\to+\infty}  \sum_{l=0}^{s-1} {\tilde v_{j,l}}^{\frac{1}{\beta_l}}(0)=1
\end{equation}  and $0\leq \tilde v_l\leq 1$ in $\mathbb R^n$  for $l=0, \ldots, s-1$ by the pointwise convergence. Since
$$\sum_{l=0}^{s-1}  \|\nabla^2\tilde v_{j,l}\|_{L^{q}(B_{2  R}(0))}\leq C_R,$$
we may assume that $\nabla^2\tilde v_{j,l}\to \nabla^2 \tilde v_{l}$ weakly in $L^q$.

\medskip

Assume $x_j\to \bar x\in\overline{\Omega}$ in the sense of subsequence. It is easy to see that $\tilde v_0\in  W^{1,q}_{loc}(\mathbb R^n)$  satisfies (at least in the weak sense) the equation (and then in the classical sense by the regularity theory since $\tilde v_1\in C^{1,\frac{1}{3}}_{loc}(\mathbb R^n)$):
$$-\Delta  \tilde v_0=\tilde v_1\geq 0,\quad \ 0\leq \tilde v_0\leq1\,\,\qquad \text{in}\,\,  D,$$
and hence the strong maximum principle shows that either $\tilde v_0>0$ or $\tilde v_0\equiv0$ in $D$, where $D$ is an arbitrary bounded open set containing $0$ in $\mathbb{R}^{n}$. Note that $\tilde v_0\equiv0$ gives $\tilde v_1\equiv0$ by the equation of $ \tilde v_0$. Using the limiting equation $-\Delta \tilde v_l=\tilde v_{l+1}\geq 0 \ \,  \text{in}\,\,  D$ for $l=1,\cdots, s-2$,  we will obtain $\tilde v_l\equiv0$ for all $l=0,1,\cdots,s-1$ through a similar way, which leads to a contradiction to \eqref{N-12}. Therefore, $\tilde v_0>0$ in $D$, and by the arbitrariness of $D$, we must have $$\tilde v_0>0  \qquad  \   \text{  in }\,\, \mathbb R^n.$$

\medskip

For any bounded open set $D\subset \mathbb R^n$ containing $0$, by the continuity of $\tilde v_0$, one gets $\tilde v_0$ has positive lower bound on $\overline{D}$. So $\tilde v_{j,0}$ is bounded from below by a positive constant independent of $j$ on $\overline{D}$ for $j$ sufficiently large. Then the assumption \eqref{0-3-n} on $f$ implies
$$\tilde f_j(x, \tilde v_{j,0}(x))= \frac{f(\lambda_jx+x_j, \lambda_j^{-\beta_0}  \tilde v_{j,0})}{\lambda_j^{-p\beta_0}}\to h(\bar x)\tilde v_0^p(x), $$
as $j\to +\infty$ uniformly w.r.t. $x\in D$.  Then we obtain the equation for $\tilde v_{s-1}\in W^{1,q}_{loc}(\mathbb R^n)$ in  weak sense
\begin{equation}\label{N-13}
	-\Delta \tilde v_{s-1}= h(\bar x)\tilde v_0^p \qquad \  \text{in}\,\,  D.
\end{equation}
Combining this with $-\Delta \tilde v_l=\tilde v_{l+1}\geq 0 \ \,  \text{in}\,\,  D$ for $l=0,\cdots, s-2$, standard regularity theory and a bootstrap argument yield that $\tilde v_{0}\in C^\infty(D)$. Through a suitable scaling, by the arbitrariness of $D$, we obtain the equation for $\tilde v_{0}$:
		\begin{equation}\label{N-14}
			(-\Delta)^{s} \tilde v_{0}= \tilde v_0^p, \qquad \ 0<\tilde v_0\leq 1 \ \ \, \text{in}\,\, \mathbb{R}^n,
		\end{equation}
	which is a contradiction to the Liouville theorem for higher order Lane-Emden equation in \cite{Lin,WX} (c.f. also \cite{CDQ0,CDQ,DQ0} and Corollary \ref{cor-g}).

\vspace{12pt}
		
		\emph{Case (ii).} $\lim\limits_{j \rightarrow +\infty} \frac{d_j}{\lambda_j} =C>0$.
		
\medskip
		
		For each $j$, we can find a point $\bar x_j\in \partial\Omega$ such that $|\bar x_j-x_j|=d_j$. Since  $\lim\limits_{j \rightarrow \infty} \frac{d_j}{\lambda_j} =C>0$, we may assume that $\lim\limits_{j \rightarrow+\infty} \frac{\bar x_j-x_j}{\lambda_j} =\bar x_0 $ for some point $\bar x_0\in \mathbb R^n$. Note that $$\Omega_j=\{x\in\mathbb R^n\  \mid \lambda_j x+x_j\in \Omega\}=\left\{x\in\mathbb R^n\  \mid \lambda_j x+\bar x_j\in \Omega\right\}+\frac{\bar x_j-x_j}{\lambda_j}.$$
Since $\Omega$ is a bounded BCB domain such that the limiting cone $\mathcal{C}$ is $\frac{1}{2^{k}}$-space ($k=1,\cdots,n$), one has (up to a subsequence)
		$$\Omega_j\rightarrow \mathcal{C}_{k}:=\mathcal C_{P,S^{n-1}_{2^{-k}}}, \ \quad \ \text{as}\, j\to+\infty$$
		for some $\frac{1}{2^{k}}$-space $\mathcal C_{k}$ containing $0$.
	
\medskip
	
Since the boundary H\"{o}lder estimates for $-\Delta$ holds uniformly on the scaling domain $\Omega_{\rho}$ w.r.t. the scale $\rho\in(0,1]$, we deduce that there is a constant $C_R$ independent of $j$ such that, for $j$ large enough,
		\begin{equation}\label{N-15}
			\sum_{l=0}^{s-1}   \|\tilde v_{j,l}\|_{C^{\alpha}(\Omega_j\cap B_{ R}(0))} \leq C_R
		\end{equation}
		for some $\alpha>0$ and any $R>0$.
		Estimate \eqref{N-15} ensures that passing to a subsequence, the limiting functions $\tilde v_l\in C(\overline {\mathcal C_{k}})$ and $\tilde v_l=0 $ on $\partial \mathcal C_{k}$ for $l=0,\ldots,s-1$. In the $\frac{1}{2^{k}}$-space $ \mathcal C_{k}$,  using the   interior $W^{2,p}$ estimate  for elliptic equations and the Sobolev embedding theorem, through arguments similar as \emph{Case (i)}, one can prove that the limiting functions $\tilde v_l\in C^\infty(\mathcal C_{k})\cap  C(\overline {\mathcal C_{k}})$ solve the following system
			\begin{equation*}
			\begin{cases}
				-\Delta  \tilde v_{0}=\tilde v_{1} \  \ \  \ \  \ \  \ \ \ &\text{in}\,\,  \mathcal C_{k},\\
				\cdots,\\
				-\Delta  \tilde v_{ s-2}=\tilde v_{ s-1} &\text{in}\,\,  \mathcal C_{k},\\
				-\Delta \tilde v_{ s-1}= h(\bar x )\tilde v_0^p &\text{in}\,\,  \mathcal C_{k},\\
				0<\tilde v_{ 0}, \ldots, \tilde v_{ s-1}\leq 1 &\text{in}\,\,  \mathcal C_{k},\\
				\tilde v_{ 0}=\cdots=\tilde v_{ s-1}=0 &\text{on}\,\, \partial  \mathcal C_{k}.
			\end{cases}
		\end{equation*}
Through a suitable scaling, the above system gives the following equation
 		\begin{equation*}
 	\begin{cases}
 	(-\Delta)^{s} \tilde v_{0}=  \tilde v_0^p(x),\ \ 0<\tilde v_{ 0},  -\Delta\tilde v_{0}, \ldots, (-\Delta)^{s-1} \tilde v_{0} \leq 1 \,\,&\text{in}\,\,   \mathcal C_{k},\\
 		\tilde v_{ 0}=-\Delta \tilde v_{0}=\cdots=(-\Delta)^{s-1}\tilde v_{0}=0 &\text{on}\,\, \partial  \mathcal C_{k},
 	\end{cases}
 \end{equation*}
which contradicts the Liouville theorem for higher order Navier problem on $\frac{1}{2^{k}}$-space $\mathcal{C}_{P,S^{n-1}_{2^{-k}}}$ in Theorems \ref{nie} and \ref{NPDE} in subsection 1.3.

\bigskip	
	
		\emph{Case (iii).} $\lim\limits_{j \rightarrow+\infty} \frac{d_j}{\lambda_j} =0$.

\medskip
		
Now we only need to rule out Case (iii). In this case, there is a point $x_0\in \partial \Omega$ such that $$x_j\to x_0,\ \quad \ \text{as}\,\, j\to+\infty.$$ We can find a point $\bar x_j\in \partial \Omega$ with $|\bar x_j-x_j| =d_j$. Denote $\tilde x_j:= \frac{\bar x_j-x_j}{\lambda_j}$, it can be seen that  $\tilde x_j\in \partial \Omega_j$ and $|\tilde x_j|\to 0$ as $j\to+\infty$. On the one hand, the boundary H\"{o}lder estimate \eqref{N-15} implies $$\sum_{l=0}^{s-1} |\tilde v_{j,l}(\tilde x_j)-\tilde v_{j,l}(0)|\leq  C|\tilde x_j|^\alpha=o_j(1).$$ On the other hand, one easily gets $$\sum_{l=0}^{s-1} |\tilde v_{j,l}(\tilde x_j)-\tilde v_{j,l}(0)|^{\frac{1}{\beta_l}}=\sum_{l=0}^{s-1} {\tilde v_{j,l}}^{\frac{1}{\beta_l}}(0)=1,$$ which leads to a contradiction. The proof of Theorem \ref{prin} is therefore finished.
\end{proof}

\section{Existence of solutions on bounded BCB domains}

\subsection{Proof of Theorem \ref{exisd}}
\begin{proof}[Proof of Theorem \ref{exisd}]
In this subsection, by applying the a priori estimates in Corollary \ref{cor-prid} and the following Leray-Schauder fixed point theorem (see e.g. \cite{CLM}), we will prove the existence of positive solutions to the fractional and second order Lane-Emden equations \eqref{Dirichlet} on bounded BCB domain $\Omega$ with Dirichlet boundary conditions.

\begin{thm}\label{L-S}
	Suppose that $X$ is a real Banach space with a closed positive cone $P$, $U\subset P$ is bounded open and contains $0$. Assume that there exists $\rho>0$ such that $B_{\rho}(0)\cap P\subset U$ and that $K:\,\overline{U}\rightarrow P$ is compact and satisfies
	\vskip 5pt
	\noindent (i) For any $x\in P$ with $|x|=\rho$ and any $\lambda\in[0,1)$, $x\neq\lambda Kx$;
	\vskip 3pt
	\noindent (ii) There exists some $y\in P\setminus\{0\}$ such that $x-Kx\neq ty$ for any $t\geq0$ and $x\in\partial U$.
	\vskip 5pt
	\noindent Then, $K$ possesses a fixed point in $\overline{U_{\rho}}$, where $U_{\rho}:=U\setminus B_{\rho}(0)$.
\end{thm}

Now we let
\begin{equation}\label{fe-1}
	X:=C^{0}(\overline{\Omega}) \quad\quad \text{and} \quad\quad P:=\{u\in X \,|\, u\geq0\}.
\end{equation}
For $0<s\leq1$, define
\begin{equation}\label{fe-2}
	K_s(u)(x) =\int_{\Omega}G_{\Omega}^s(x,y)  u^{p}(y )
	\mathrm dy,
\end{equation}
where $G_{\Omega}^s(x,y)$ is the Green  function for $(-\Delta)^s$ with Dirichlet boundary condition in $\Omega$. Suppose $u\in C^{0}(\overline{\Omega})$ is a fixed point of $K_s$, i.e., $u=K_{s}u$, then it is easy to see that  $u$ is a solution to the Dirichlet problem \eqref{Dirichlet}.

\medskip

Our goal is to show the existence of a fixed point for $K_{s}$ in $P\setminus B_{\rho}(0)$ for some $\rho>0$ (to be determined later) by using Theorem \ref{L-S}. To this end, we need to verify the two conditions (i) and (ii) in Theorem \ref{L-S} separately.

\medskip

\emph{(i)} First, we show that there exists $\rho>0$ such that for any $u\in\partial B_{\rho}(0)\cap P$ and $0\leq\lambda<1$,
\begin{equation}\label{fe-4}
	u-\lambda K_s(u)\neq0.
\end{equation}
For any $x\in\overline{\Omega}$, it holds that
\begin{equation}\label{fe-5}
	|K_s(u)(x)| = \left|\int_{\Omega}G_{\Omega}^s(x,y)  u^{p}(y)
	\mathrm dy\right|
	 \leq  \rho^{p-1}\left\|\int_{\Omega}G_{\Omega}^s(x,y)	\mathrm{d}y \right\|_{C^{0}(\overline{\Omega})}\cdot\|u\|_{C^{0}(\overline{\Omega})}.
\end{equation}
Let $h(x):=\int_{\Omega}G_{\Omega}^s(x,y)	\mathrm{d}y $, then it solves
\begin{equation}\label{fe-6}\\\begin{cases}
		(-\Delta)^s h(x)=1, \,\,\,\,\,\,\,\,\,\,   \,\,\, &x\in\Omega, \\
		h(x)=0, &x\in \mathbb R^n\setminus\Omega.
\end{cases}\end{equation}
For a fixed point $x^{0}\in\Omega$, it is well-known that the function
\begin{equation}\label{fe-7}
	\zeta_s(x):=C_{n,s} (diam\,\Omega)^{2s} \left(1-\frac{|x-x^{0}|^{2}}{(diam\,\Omega)^{2}}\right)^s_{+}
\end{equation}
 satisfies
\begin{equation}\label{fe-8}\\\begin{cases}
		(-\Delta )^s\zeta_s(x)=1, \,\,\,\,\,\,\,\,\,\, &x\in\Omega, \\
		\zeta_s(x)>0, \,\,\,\,\,\,\,\,  & x\in B_{diam\,\Omega}(x^{0})\setminus\Omega, \\
        \zeta_s(x)=0, \,\,\,\,\,\,\,\,  & x\in \mathbb R^n\setminus B_{diam\,\Omega}(x^{0}),
\end{cases}\end{equation}
where $C_{n,s}$ is a constant depending only on $n,s$ and $C_{n,1}=\frac{1}{2n}$. By \eqref{fe-6}, \eqref{fe-7}, \eqref{fe-8} and maximum principle, we get
\begin{equation}\label{fe-9}
	0\leq h(x)<\zeta_s(x)\leq C_{n,s} (diam\,\Omega)^{2s}, \quad\quad \forall \,\, x\in\overline{\Omega}.
\end{equation}
Therefore, we infer from \eqref{fe-5} and \eqref{fe-9} that
\begin{equation}\label{fe-10}
	\|K_s(u)\|_{C^{0}(\overline{\Omega})}<C_{n,s}\rho^{p-1} (diam\,\Omega)^{2s} \|u\|_{C^{0}(\overline{\Omega})}=\|u\|_{C^{0}(\overline{\Omega})}
\end{equation}
if we take
\begin{equation}\label{fe-11}
	\rho=\left(\frac{1}{C_{n,s}^{\frac{1}{2s}}\,diam\,\Omega}\right)^{\frac{2s}{p-1}}>0.
\end{equation}
This implies that $u\neq\lambda K_s(u)$ for any $u\in\partial B_{\rho}(0)\cap P$ and $0\leq\lambda<1$.

\medskip

\emph{(ii)} Now let $\eta_s$ be the unique positive solution of
\begin{equation}\label{fe-12}\\\begin{cases}
		(-\Delta)^{s}\eta_s(x)=1, \,\,\,\,\,\,\,\,\,\, \,\,\,& x\in\Omega, \\
		\eta_s(x)=0, \,\,\,\,\,\,\,\,\,\,\,\, &x\in\mathbb R^n\setminus\Omega.
\end{cases}\end{equation}
We will show that
\begin{equation}\label{fe-13}
	u-K_s(u)\neq t\eta_s, \quad\quad \forall \,\, t\geq0, \quad \forall u\in\partial U,
\end{equation}
where $U:=B_{R}(0)\cap P$ with sufficiently large $R>\rho$ (to be determined later). First, observe that, for any $u\in\overline{U}$,
\begin{equation}\label{fe-14}
	\left\|(-\Delta)^{s}K_s(u)\right\|_{C^{0}(\overline{\Omega})}=\|u\|^{p}_{C^{0}(\overline{\Omega})}\leq R^{p},
\end{equation}
and hence
\begin{equation}\label{fe-15}
	\|K_s(u)\|_{C^{0,\gamma}(\Omega)}\leq CR^{p}, \quad\quad \forall \,\, 0<\gamma<\min\{1,2s\},
\end{equation}
thus $K_s:\, \overline{U}\rightarrow P$ is compact.

\medskip

We use contradiction arguments to prove \eqref{fe-13}. Suppose on the contrary that, there exists some $u\in\partial U$ and $t\geq0$ such that
\begin{equation}\label{fe-16}
	u-K_s(u)=t\eta_s,
\end{equation}
then one has $\|u\|_{C^{0}(\overline{\Omega})}=R>\rho>0$ and   $u  $ satisfies the Dirichlet problem
\begin{equation}\label{fe-17}\\\begin{cases}
		(-\Delta)^{s}u(x)=u^{p}(x)+t, \,\,\,\,\,\,\,\,\, u(x)>0, \,\,\,\,\,\,\,\,\,\, &x\in\Omega, \\
		u(x)=0, \,\,\,\,\,\,\,\,\,\,\,\, &x\in\mathbb R^n\setminus\Omega.
\end{cases}\end{equation}
Let $\lambda_1$ be the first eigenvalue of the corresponding Dirichlet problem, i.e.,
$$\lambda_1:=\sup\left\{\lambda\in\mathbb R\, |\, (-\Delta)^s u\geq \lambda u,\, u>0\,\, \text{in}\,\,\Omega, \,\, u=0\,\, \text{in}\,\, \mathbb R^n\setminus\Omega\right\}.$$
Choose a constant $C_{1}>\lambda_{1}$. Since $u(x)>0$ in $\Omega$ and $p>1$, it is easy to see that, there exists another constant $C_{2}>0$ (say, take $C_{2}=C_{1}^{\frac{p}{p-1}}$) such that
\begin{equation}\label{fe-18}
	u^{p}(x)\geq C_{1}u(x)-C_{2}.
\end{equation}
It follows from \eqref{fe-17} and \eqref{fe-18} that
\begin{equation}\label{fe-19}
	(-\Delta)^{s}u(x)=u^{p}(x)+t\geq C_{1}u(x)-C_{2}+t  \quad\quad \text{in} \,\, \Omega.
\end{equation}
If $t\geq C_2$, then we get $(-\Delta)^{s}u(x)\geq C_{1}u(x)$, which contradicts the definition of $\lambda_1$. Thus, we must have $0\leq t<C_{2}$. By the a priori estimates in Corollary \ref{cor-prid}, we know that
\begin{equation}\label{fe-22}
  \|u\|_{L^{\infty}(\overline{\Omega})}\leq C(n,s,p,t,\Omega).
\end{equation}
We will show that the above a priori estimates are uniform with respect to $0\leq t<C_{2}$, i.e., for $0\leq t<C_{2}$,
\begin{equation}\label{fe-23}
  \|u\|_{L^{\infty}(\overline{\Omega})}\leq C(n,s,p,C_{2},\Omega)=:C_{0}.
\end{equation}

\medskip

Indeed, if \eqref{fe-23} does not hold, there exist sequences $\{t_{k}\}\subset[0,C_{2})$, $\{x^{k}\}\subset\Omega$ and $\{u_{k}\}$ satisfying
\begin{equation}\label{fe-24}\\\begin{cases}
		(-\Delta)^{s}u_{k}(x)=u_{k}^{p}(x)+t_{k}, \,\,\,\,\,\,\,\,\, u_{k}(x)>0, \,\,\,\,\,\,\,\,\,\, &x\in\Omega, \\
		u_{k}(x)=0, \,\,\,\,\,\,\,\,\,\,\,\, &x\in\mathbb R^n\setminus\Omega,
\end{cases}\end{equation}
but $m_{k}:=u_{k}(x^{k})=\|u_{k}\|_{L^{\infty}(\overline{\Omega})}\rightarrow+\infty$ as $k\rightarrow+\infty$. For $x\in\Omega_{k}:=\{x\in\mathbb{R}^{n}\,|\,\lambda_{k}x+x^{k}\in\Omega\}$, we define $v_{k}(x):=\frac{1}{m_{k}}u_{k}(\lambda_{k}x+x^{k})$ with $\lambda_{k}:=m^{\frac{1-p}{2s}}_{k}\rightarrow0$ as $k\rightarrow+\infty$. Then $v_{k}(x)$ satisfies $\|v_{k}\|_{L^{\infty}(\overline{\Omega_{k}})}=v_{k}(0)=1$ and
\begin{equation}\label{fe-25}
(-\Delta)^{s}v_{k}(x)=v^{p}_{k}(x)+\frac{t_{k}}{m^{p}_{k}}
\end{equation}
for any $x\in\Omega_{k}$. Since $0\leq t<C_{2}$ and $m_{k}\rightarrow+\infty$, note that $\Omega$ is a BCB domain such that the boundary H\"{o}lder estimates for $(-\Delta)^{s}$ ($0<s\leq1$) holds uniformly on the scaling domain $\Omega_{\rho}$ w.r.t. the scale $\rho\in(0,1]$, by completely similar blowing-up methods as in the proofs of Theorems \ref{prid} and \ref{prid2} in subsections 4.1 and 4.2, we can also derive a contradiction from the Liouville theorem for fractional and second order Lane-Emden equation $(-\Delta)^{s}v(x)=v^{p}(x)$ in $\mathbb{R}^{n}$ (c.f. e.g. \cite{CLL,GS}, also contained in our Theorems \ref{pde-unbd}, \ref{pde-unbd-a}, \ref{pde-unbd-2}--\ref{pde-bd-2} and \ref{f2PDE}) and for Dirichlet problem of Lane-Emden equation $(-\Delta)^{s}v(x)=v^{p}(x)$ in cones in our Theorems \ref{pde-unbd}, \ref{pde-unbd-a}, \ref{pde-unbd-2}--\ref{pde-bd-2}, \ref{f2PDE} and Corollary \ref{cone-cor}. Therefore, the uniform estimates \eqref{fe-23} must hold.

\medskip

Now we let $R:=C_{0}+\rho$ and $U:=B_{C_{0}+\rho}(0)\cap P$, then \eqref{fe-23} implies
\begin{equation}\label{fe-28}
	\|u\|_{L^{\infty}(\overline{\Omega})}\leq C_{0}<C_{0}+\rho,
\end{equation}
which contradicts $u\in\partial U$. This implies that
\begin{equation}\label{4-29}
	u-K_s(u)\neq t\eta_s
\end{equation}
for any $t\geq0$ and $u\in\partial U$ with $U=B_{C_{0}+\rho}(0)\cap P$.

\medskip

From Theorem \ref{L-S}, we deduce that there exists a $u\in\overline{\big(B_{C_{0}+\rho}(0)\cap P\big)\setminus B_{\rho}(0)}$ satisfies
\begin{equation}\label{4-30}
	u=K_s(u),
\end{equation}
and hence $\rho\leq\|u\|_{L^{\infty}(\overline{\Omega})}\leq C_{0}+\rho$ solves the Dirichlet problem
\begin{equation}\label{fe-31}\\\begin{cases}
		(-\Delta)^{s}u(x)=u^{p}(x), \,\,\,\,\,\,\, u(x)>0, \,\,\,\,\,\,\,\,\, &x\in\Omega, \\
		u(x)=0, \,\,\,\,\,\,\,\,\,\,\,\, &x\in\mathbb R^n\setminus\Omega.
\end{cases}\end{equation}
By regularity theory, we can see that $u\in C^{\infty}(\Omega)\cap C^{0}(\overline{\Omega})$. This concludes our proof of Theorem \ref{exisd}.
\end{proof}

\subsection{Proof of Theorem \ref{exisn}}
\begin{proof}[Proof of Theorem \ref{exisn}]

In this subsection, by applying the a priori estimates in Corollary \ref{cor-n} and the Leray-Schauder fixed point theorem in Theorem \ref{L-S}, we will prove the existence of positive solutions to the higher order Lane-Emden equations \eqref{Navier} on bounded BCB domain $\Omega$ such that the limiting cone is $\frac{1}{2^{k}}$-space $\mathcal{C}_{P,S^{n-1}_{2^{-k}}}$ with Navier boundary conditions.

\medskip

Now we let
\begin{equation}\label{4-1}
  X:=C^{0}(\overline{\Omega}) \quad\quad \text{and} \quad\quad P:=\{u\in X \,|\, u\geq0\}.
\end{equation}
For $s\in\mathbb{Z}^{+}$ such that $2s\leq n$, define
\begin{equation}\label{4-2}
  K_{s}(u)(x):=\int_{\Omega}G^{1}_{\Omega}(x,y^{s})\int_{\Omega}G^{1}_{\Omega}(y^{s},y^{s-1})\int_{\Omega}\cdots\int_{\Omega}G^{1}_{\Omega}(y^{2},y^{1})u^{p}(y^{1})
  \mathrm{d}y^{1}\mathrm{d}y^{2}\cdots \mathrm{d}y^{s},
\end{equation}
where $G^{1}_{\Omega}(x,y)$ is the Green's function for $-\Delta$ with Dirichlet boundary condition in $\Omega$. Suppose $u\in C^{0}(\overline{\Omega})$ is a fixed point of $K_{s}$, i.e., $u=K_{s}u$, then it is easy to see that $u\in C^{2s}(\Omega)$ satisfies $(-\Delta)^{i}u\in C(\overline{\Omega})$ ($\forall \, i=0,\cdots,s-1$) and solves the Navier problem
\begin{equation}\label{4-3}\\\begin{cases}
(-\Delta)^{s}u(x)=u^{p}(x) \,\,\,\,\,\,\,\,\,\, \text{in} \,\,\, \Omega, \\
u(x)=-\Delta u(x)=\cdots=(-\Delta)^{s-1}u(x)=0 \,\,\,\,\,\,\,\, \text{on} \,\,\, \partial\Omega.
\end{cases}\end{equation}

\medskip

Our goal is to show the existence of a fixed point for $K_{s}$ in $P\setminus B_{\rho}(0)$ for some $\rho>0$ (to be determined later) by using Theorem \ref{L-S}. To this end, we need to verify the two conditions (i) and (ii) in Theorem \ref{L-S} separately.

\medskip

\emph{(i)} First, we show that there exists $\rho>0$ such that, for any $u\in\partial B_{\rho}(0)\cap P$ and $0\leq\lambda<1$,
\begin{equation}\label{4-4}
  u-\lambda K_{s}(u)\neq0.
\end{equation}
For any $x\in\overline{\Omega}$, it holds that
\begin{eqnarray}\label{4-5}
 \nonumber |K_{s}(u)(x)|&=&\left|\int_{\Omega}G^{1}_{\Omega}(x,y^{s})\int_{\Omega}G^{1}_{\Omega}(y^{s},y^{s-1})\cdots\int_{\Omega}G^{1}_{\Omega}(y^{2},y^{1})u^{p}(y^{1})
  \mathrm{d}y^{1}\cdots \mathrm{d}y^{s}\right| \\
  &\leq& \int_{\Omega}G^{1}_{\Omega}(x,y^{s})\int_{\Omega}G^{1}_{\Omega}(y^{s},y^{s-1})\cdots\int_{\Omega}G^{1}_{\Omega}(y^{2},y^{1})\mathrm{d}y^{1}\cdots \mathrm{d}y^{s}\cdot\|u\|^{p}_{C^{0}(\overline{\Omega})} \\
 \nonumber &\leq& \rho^{p-1}\left\|\int_{\Omega}G^{1}_{\Omega}(x,y)\mathrm{d}y\right\|^{s}_{C^{0}(\overline{\Omega})}\cdot\|u\|_{C^{0}(\overline{\Omega})}.
\end{eqnarray}
Let $h(x):=\int_{\Omega}G^{1}_{\Omega}(x,y)\mathrm{d}y$, then it solves
\begin{equation}\label{4-6}\\\begin{cases}
-\Delta_{x}h(x)=1 \,\,\,\,\,\,\,\,\,\, \text{in} \,\,\, \Omega, \\
h(x)=0 \,\,\,\,\,\,\,\, \text{on} \,\,\, \partial\Omega.
\end{cases}\end{equation}
For a fixed point $x^{0}\in\Omega$, we define the function
\begin{equation}\label{4-7}
  \zeta(x):=\frac{(diam\,\Omega)^{2}}{2n}\left(1-\frac{|x-x^{0}|^{2}}{(diam\,\Omega)^{2}}\right)_{+},
\end{equation}
then it satisfies
\begin{equation}\label{4-8}\\\begin{cases}
-\Delta_{x}\zeta(x)=1 \,\,\,\,\,\,\,\,\,\, \text{in} \,\,\, \Omega, \\
\zeta(x)>0 \,\,\,\,\,\,\,\, \text{on} \,\,\, \partial\Omega.
\end{cases}\end{equation}
By maximum principle, we get
\begin{equation}\label{4-9}
  0\leq h(x)<\zeta(x)\leq\frac{(diam\,\Omega)^{2}}{2n}, \quad\quad \forall \,\, x\in\overline{\Omega}.
\end{equation}
Therefore, we infer from \eqref{4-5} and \eqref{4-9} that
\begin{equation}\label{4-10}
  \|K_{s}(u)\|_{C^{0}(\overline{\Omega})}<\rho^{p-1}\frac{(diam\,\Omega)^{2s}}{(2n)^{s}}\|u\|_{C^{0}(\overline{\Omega})}=\|u\|_{C^{0}(\overline{\Omega})}
\end{equation}
if we take
\begin{equation}\label{4-11}
  \rho=\left(\frac{\sqrt{2n}}{diam\,\Omega}\right)^{\frac{2s}{p-1}}>0.
\end{equation}
This implies that $u\neq\lambda K_{s}(u)$ for any $u\in\partial B_{\rho}(0)\cap P$ and $0\leq\lambda<1$.

\medskip

\emph{(ii)} Now let $\eta_{s}\in C^{2s}(\Omega)$ such that $(-\Delta)^{i}\eta_{s}\in C(\overline{\Omega})$ ($\forall\, i=0,\cdots,s-1$) be the unique positive solution of
\begin{equation}\label{4-12}\\\begin{cases}
(-\Delta)^{s}\eta_{s}(x)=1, \,\,\,\,\,\,\,\,\,\, \,\,\, x\in\Omega, \\
\eta_{s}(x)=-\Delta\eta_{s}(x)=\cdots=(-\Delta)^{s-1}\eta_{s}(x)=0, \,\,\,\,\,\,\,\,\,\,\,\, x\in\partial\Omega.
\end{cases}\end{equation}
We will show that
\begin{equation}\label{4-13}
  u-K_{s}(u)\neq t\eta_{s}, \quad\quad \forall \,\, t\geq0, \quad \forall u\in\partial U,
\end{equation}
where $U:=B_{R}(0)\cap P$ with sufficiently large $R>\rho$ (to be determined later). First, observe that, for any $u\in\overline{U}$,
\begin{equation}\label{4-14}
  \left\|(-\Delta)^{s}K_{s}(u)\right\|_{C^{0}(\overline{\Omega})}=\|u\|^{p}_{C^{0}(\overline{\Omega})}\leq R^{p},
\end{equation}
and hence
\begin{equation}\label{4-15}
  \|K_{s}(u)\|_{C^{0,\gamma}(\Omega)}\leq CR^{p}, \quad\quad \forall \,\, 0<\gamma<1,
\end{equation}
thus $K:\, \overline{U}\rightarrow P$ is compact.

\medskip

We use contradiction arguments to prove \eqref{4-13}. Suppose on the contrary that, there exists some $u\in\partial U$ and $t\geq0$ such that
\begin{equation}\label{4-16}
  u-K_{s}(u)=t\eta_{s},
\end{equation}
then one has $\|u\|_{C^{0}(\overline{\Omega})}=R>\rho>0$, $u\in C^{2s}(\Omega)$ satisfies $(-\Delta)^{i}u\in C(\overline{\Omega})$ ($\forall\, i=0,\cdots,s-1$) and the Navier problem
\begin{equation}\label{4-17}\\\begin{cases}
(-\Delta)^{s}u(x)=u^{p}(x)+t, \,\,\,\,\,\,\,\,\, u(x)>0, \,\,\,\,\,\,\,\,\,\, x\in\Omega, \\
u(x)=-\Delta u(x)=\cdots=(-\Delta)^{s-1}u(x)=0, \,\,\,\,\,\,\,\,\,\,\,\, x\in\partial\Omega.
\end{cases}\end{equation}
Choose a constant $C_{1}>\lambda_{1}$, where $\lambda_{1}$ is the first eigenvalue for $(-\Delta)^{s}$ in $\Omega$ with Navier boundary condition and let $\phi$ be the corresponding eigenfunction. Since $u(x)>0$ in $\Omega$ and $p>1$, it is easy to see that, there exists another constant $C_{2}>0$ (say, take $C_{2}=C_{1}^{\frac{p}{p-1}}$) such that
\begin{equation}\label{4-18}
  u^{p}(x)\geq C_{1}u(x)-C_{2}.
\end{equation}
If $t\geq C_{2}$, then we have
\begin{equation}\label{4-19}
  (-\Delta)^{s}u(x)=u^{p}(x)+t\geq C_{1}u(x)-C_{2}+t\geq C_{1}u(x) \quad\quad \text{in} \,\, \Omega.
\end{equation}
Multiplying both side of \eqref{4-19} by the eigenfunction $\phi(x)$, and integrating by parts yield
\begin{eqnarray}\label{4-20}
  C_{1}\int_{\Omega}u(x)\phi(x)\mathrm{d}x&\leq&\int_{\Omega}(-\Delta)^{s}u(x)\cdot\phi(x)\mathrm{d}x=\int_{\Omega}u(x)\cdot(-\Delta)^{s}\phi(x)\mathrm{d}x \\
 \nonumber &=&\lambda_{1}\int_{\Omega}u(x)\phi(x)\mathrm{d}x,
\end{eqnarray}
and hence
\begin{equation}\label{4-21}
  0<(C_{1}-\lambda_{1})\int_{\Omega}u(x)\phi(x)\mathrm{d}x\leq0,
\end{equation}
which is absurd. Thus we must have $0\leq t<C_{2}$. By the a priori estimates in Corollary \ref{cor-n}, we know that
\begin{equation}\label{4-22}
  \|u\|_{L^{\infty}(\overline{\Omega})}\leq C(n,s,p,t,\Omega).
\end{equation}
We can actually show that the above a priori estimates are uniform with respect to $0\leq t<C_{2}$, i.e., for $0\leq t<C_{2}$,
\begin{equation}\label{4-23}
  \|u\|_{L^{\infty}(\overline{\Omega})}\leq C(n,s,p,C_{2},\Omega)=:C_{0}.
\end{equation}
Similar to the uniform estimate \eqref{fe-23} in the proof of Theorem \ref{exisd} in subsection 5.1, note that $\Omega$ is a bounded BCB domain such that the boundary H\"{o}lder estimates for $-\Delta$ holds uniformly on the scaling domain $\Omega_{\rho}$ w.r.t. the scale $\rho\in(0,1]$, through entirely similar blowing-up methods as in the proof of Theorem \ref{prin} in subsection 4.3, we can also derive a contradiction from the Liouville theorem for higher order Lane-Emden equation $(-\Delta)^{s}v(x)=v^{p}(x)$ in $\mathbb{R}^{n}$ (c.f. \cite{CDQ0,CDQ,DQ0,Lin,WX} and the references therein, see also Corollary \ref{cor-g}) and for Navier problem of Lane-Emden equation $(-\Delta)^{s}v(x)=v^{p}(x)$ in $\frac{1}{2^{k}}$-space $\mathcal{C}_{P,S^{n-1}_{2^{-k}}}$ with super poly-harmonic property $(-\Delta)^{i}v(x)\geq0$ in our Theorem \ref{NPDE}. Therefore, the uniform estimates \eqref{4-23} must hold.

\medskip

Now we let $R:=C_{0}+\rho$ and $U:=B_{C_{0}+\rho}(0)\cap P$, then \eqref{4-23} implies
\begin{equation}\label{4-28}
  \|u\|_{L^{\infty}(\overline{\Omega})}\leq C_{0}<C_{0}+\rho,
\end{equation}
which contradicts $u\in\partial U$. This implies that
\begin{equation}\label{4-29}
  u-K(u)\neq t\eta_{s}
\end{equation}
for any $t\geq0$ and $u\in\partial U$ with $U=B_{C_{0}+\rho}(0)\cap P$.

\medskip

From Theorem \ref{L-S}, we deduce that there exists a $u\in\overline{\big(B_{C_{0}+\rho}(0)\cap P\big)\setminus B_{\rho}(0)}$ satisfies
\begin{equation}\label{4-30}
  u=K_{s}(u),
\end{equation}
and hence $\rho\leq\|u\|_{L^{\infty}(\overline{\Omega})}\leq C_{0}+\rho$ solves the higher order Navier problem
\begin{equation}\label{4-31}\\\begin{cases}
(-\Delta)^{s}u(x)=u^{p}(x), \,\,\,\,\,\,\, u(x)>0, \,\,\,\,\,\,\,\,\, x\in\Omega, \\
u(x)=-\Delta u(x)=\cdots=(-\Delta)^{s-1}u(x)=0, \,\,\,\,\,\,\,\,\,\,\,\, x\in\partial\Omega.
\end{cases}\end{equation}
By regularity theory, we can see that $u\in C^{2s}(\Omega)$ and $(-\Delta)^{i}u\in C(\overline{\Omega})$ ($\forall\, i=0,\cdots,s-1$). This concludes our proof of Theorem \ref{exisn}.
\end{proof}


\begin{thebibliography}{99}

\bibitem{A} A. Ancona, {\it On positive harmonic functions in cones and cylinders}, Rev. Mat. Iberoam., \textbf{28} (2012), no. 1, 201-230.

\bibitem{BE} C. Bandle and M. Ess\'{e}n, {\it On Positive Solutions of Emden Equations in Cone-like Domains}, Arch. Rational Mech. Anal., \textbf{112} (1990), 319-338.

\bibitem{BaBo} R. Ba\~nuelos and K. Bogdan, {\it Symmetric stable processes in cones}, Potential Anal., \textbf{21} (2004), no. 3, 263-288.

\bibitem{BHM1} A. Barton, S. Hofmann and S. Mayboroda, {\it The Neumann problem for higher order elliptic equations with symmetric coefficients}, Math. Ann., \textbf{371} (2018), no. 1-2, 297-336.

\bibitem{Ber} J. Bertoin, {\it L\'{e}vy Processes}, Cambridge Tracts in Mathematics, \textbf{121} (1996), Cambridge University Press, Cambridge.

\bibitem{BG} M. F. Bidaut-V\'{e}ron and H. Giacomini, {\it A new dynamical approach of Emden-Fowler equations and systems}, Adv. Differential Equations, \textbf{15} (2010), no. 11-12, 1033-1082.

\bibitem{BP} M. F. Bidaut-V\'{e}ron and S. Pohozaev, {\it Nonexistence results and estimates for some nonlinear elliptic problems}, J. d'Analyse Math., \textbf{84} (2001), 1-49.

\bibitem{BH} J. Bliedtner and W. Hansen, {\it Potential theory. An analytic and probabilistic approach to balayage}, Universitext. Springer-Verlag, Berlin, 1986. xiv+435 pp, ISBN: 3-540-16396-4.

\bibitem{BKN} K. Bogdan, T. Kulczycki and A. Nowak, {\it Gradient estimates for harmonic and $q$-harmonic functions of symmetric stable processes}, Illinois J. Math., \textbf{46} (2002), 541-556.

\bibitem{BM} H. Brezis and F. Merle, {\it Uniform estimates and blow-up behavior for solutions of $-\Delta u=V(x)e^{u}$ in two dimensions}, Commun. PDE, \textbf{16} (1991), 1223-1253.

\bibitem{Buc} C. Bucur, {\it Some observations on the Green function for the ball in the fractional Laplace framework}, Commun. Pure Appl. Anal., \textbf{15} (2016), no. 2, 657-699.

\bibitem{CS} X. Cabr\'{e} and Y. Sire, {\it Nonlinear equations for fractional Laplacians, I: regularity, maximum principles}, Annales de l'Institut Henri Poincar\'{e} C, Analyse Non Lin\'{e}aire, \textbf{31} (2014), 23-53.

\bibitem{CT} X. Cabr\'{e} and J. Tan, {\it Positive solutions of nonlinear problems involving the square root of the Laplacian}, Adv. Math., \textbf{224} (2010), 2052-2093.

\bibitem{CGS} L. Caffarelli, B. Gidas and J. Spruck, {\it Asymptotic symmetry and local behavior of semilinear elliptic equation with critical Sobolev growth}, Comm. Pure Appl. Math., \textbf{42} (1989), 271-297.

\bibitem{CS} L. Caffarelli and L. Silvestre, {\it An extension problem related to the fractional Laplacian}, Commun. PDEs, \textbf{32} (2007), 1245-1260.

\bibitem{CV} L. Caffarelli and L. Vasseur, {\it Drift diffusion equations with fractional diffusion and the quasi-geostrophic equation}, Ann. of Math., \textbf{171}(2010), no. 3, 1903-1930.

\bibitem{CDQ0} D. Cao, W. Dai and G. Qin, {\it Super poly-harmonic properties, Liouville theorems and classification of nonnegative solutions to equations involving higher-order fractional Laplacians}, Trans. Amer. Math. Soc., \textbf{374} (2021), no. 7, 4781-4813.

\bibitem{CPY} D. Cao, S. Peng and S. Yan, {\it Asymptotic behavior of the ground state solutions for H\'{e}non equation}, IMA J. Appl. Math., \textbf{74} (2009), 468-480.

\bibitem{CC} L. Cao and W. Chen, {\it Liouville type theorems for poly-harmonic Navier problems}, Disc. Cont. Dyn. Sys.-A, \textbf{33} (2013), 3937-3955.

\bibitem{CG} S.-Y. A. Chang and M. D. M. Gonz\`{a}lez, {\it Fractional Laplacian in conformal geometry}, Adv. Math., \textbf{226} (2011), 1410-1432.

\bibitem{CY} S.-Y. A. Chang and P. C. Yang, {\it On uniqueness of solutions of $n$-th order differential equations in conformal geometry}, Math. Res. Lett., \textbf{4} (1997), 91-102.

\bibitem{CDQ} W. Chen, W. Dai and G. Qin, {\it Liouville type theorems, a priori estimates and existence of solutions for critical and super-critical order Hardy-H\'{e}non type equations in $\mathbb{R}^n$}, to appear in Math. Z., 36 pp, arXiv: 1808.06609.

\bibitem{CFL} W. Chen, Y. Fang and C. Li, {\it Super poly-harmonic property of solutions for Navier boundary problems on a half space}, J. Funct. Anal., \textbf{265} (2013), 1522-1555.

\bibitem{CL} W. Chen and C. Li, {\it Classification of solutions of some nonlinear elliptic equations}, Duke Math. J., \textbf{63} (1991), no. 3, 615-622.

\bibitem{CL1} W. Chen and C. Li, {\it On Nirenberg and related problems - a necessary and sufficient condition}, Comm. Pure Appl. Math., \textbf{48} (1995), 657-667.

\bibitem{CL4} W. Chen and C. Li, {\it A priori estimates for prescribing scalar curvature equations}, Annals of Math., \textbf{145} (1997), no. 3, 547-564.

\bibitem{CLL} W. Chen, C. Li and Y. Li, {\it A direct method of moving planes for the fractional Laplacian}, Adv. Math., \textbf{308} (2017), 404-437.

\bibitem{CLM} W. Chen, Y. Li and P. Ma, {\it The Fractional Laplacian}, World Scientic Publishing Co. Pte. Ltd., 2020, 344 pp, https://doi.org/10.1142/10550.

\bibitem{CLO} W. Chen, C. Li and B. Ou, {\it Classification of solutions for an integral equation}, Comm. Pure Appl. Math., \textbf{59} (2006), 330-343.

\bibitem{CLZ} W. Chen, Y. Li and R. Zhang, {\it A direct method of moving spheres on fractional order equations}, J. Funct. Anal., \textbf{272} (2017), no. 10, 4131-4157.

\bibitem{CLZC} W. Chen, C. Li, L. Zhang and T. Cheng, {\it A Liouville theorem for $\alpha$-harmonic functions in $\mathbb{R}^{n}_{+}$}, Disc. Contin. Dyn. Syst. - A, \textbf{36} (2016), no. 3, 1721-1736.

\bibitem{CFR} G. Ciraolo, A. Figalli and A. Roncoroni, {\it Symmetry results for critical anisotropic $p$-Laplacian equations in convex cones}, Geom. Funct. Anal., \textbf{30} (2020), no. 3, 770-803.

\bibitem{C} P. Constantin, {\it Euler equations, Navier-Stokes equations and turbulence}, In: Mathematical Foundation of Turbulent Viscous Flows, Vol. 1871 of Lecture Notes in Math. (2006), 1-43, Springer, Berlin.

\bibitem{DD} W. Dai and T. Duyckaerts, {\it Uniform a priori estimates for positive solutions of higher order Lane-Emden equations in $\mathbb{R}^{n}$}, Publicacions Matematiques, \textbf{65} (2021), no. 1, 319-333.

\bibitem{DLQ} W. Dai, Z. Liu and G. Qin, {\it Classification of nonnegative solutions to static Schr\"{o}dinger-Hartree-Maxwell type equations}, SIAM J. Math. Anal., \textbf{53} (2021), no. 2, 1379-1410.

\bibitem{DQ1} W. Dai and G. Qin, {\it Classification of nonnegative classical solutions to third-order equations}, Adv. Math., \textbf{328} (2018), 822-857.

\bibitem{DQ2} W. Dai and G. Qin, {\it Liouville type theorems for elliptic equations with Dirichlet conditions in exterior domains}, J. Differential Equ., \textbf{269} (2020), no. 9, 7231-7252.

\bibitem{DQ3} W. Dai and G. Qin, {\it Liouville type theorem for critical order H\'{e}non-Lane-Emden type equations on a half space and its applications}, J. Funct. Anal., \textbf{281} (2021), no. 10, Paper No. 109227, 37 pp.

\bibitem{DQ0} W. Dai and G. Qin, {\it Liouville type theorems for fractional and higher order H\'{e}non-Hardy type equations via the method of scaling spheres}, Int. Math. Res. Not. IMRN, 2022, 70 pp, DOI: 10.1093/imrn/rnac079, arXiv: 1810.02752.

\bibitem{DQ} W. Dai and G. Qin, {\it Classification of solutions to conformally invariant systems with mixed order and exponentially increasing or nonlocal nonlinearity}, SIAM J. Math. Anal., 2023, 39 pp, DOI: 10.1137/22M1499650, arXiv: 2108.07166.

\bibitem{DQZ} W. Dai, G. Qin and Y. Zhang, {\it Liouville type theorem for higher order H\'{e}non equations on a half space}, Nonlinear Anal. - TMA, \textbf{183} (2019), 284-302.

\bibitem{DFSV} L. Damascelli, A. Farina, B. Sciunzi and E. Valdinoci, {\it Liouville results for $m$-Laplace equations of Lane-Emden-Fowler type}, Ann. Inst. H. Poincar\'{e} C Anal. Non Lin\'{e}aire, \textbf{26} (2009), no. 4, 1099-1119.

\bibitem{DDGW} A. DelaTorre, M. del Pino, M. Gonz\'{a}lez and J. Wei, {\it Delaunay-type singular solutions for the fractional Yamabe problem}, Math. Ann., \textbf{369} (2017), no. 1-2, 597-626.

\bibitem{DS} D. De Silva and O. Savin, {\it A short proof of boundary Harnack inequality}, J. Differential Equ., \textbf{269} (2020), 2419-2429.

\bibitem{DPV} E. Di Nezza, G. Palatucci and E. Valdinoci, {\it Hitchhiker's guide to the fractional Sobolev spaces}, Bull. Sci. Math., \textbf{136} (2012), no. 5, 521-573.

\bibitem{DRV} S. Dipierro, X. Ros-Oton and E. Valdinoci, {\it Nonlocal problems with Neumann boundary conditions}, Rev. Mat. Iberoam., \textbf{33} (2017), no. 2, 377-416.

\bibitem{DSV} S. Dipierro, O. Savin and E. Valdinoci, {\it All functions are locally $s$-harmonic up to a small error}, J. Eur. Math. Soc., \textbf{19} (2017), no. 4, 957-966.

\bibitem{DSS} L. Dupaigne, B. Sirakov and P. Souplet, {\it A Liouville-type theorem for the Lane-Emden equation in a half-space}, Int. Math. Res. Not. IMRN, 2022, no. 12, 9024-9043.

\bibitem{FW} M. M. Fall and T. Weth, {\it Monotonicity and nonexistence results for some fractional elliptic problems in the half-space}, Commun. Contemp. Math., \textbf{18} (2016), no. 1, 1550012, 25 pp.

\bibitem{Farina} A. Farina, {\it Liouville-type theorems for elliptic problems}, Handbook of Differential Equations: Stationary Partial Differential Equations, \textbf{4} (2007), 61-116.

\bibitem{FG} M. Fazly and N. Ghoussoub, {\it On the H\'{e}non-Lane-Emden conjecture}, Discrete Contin. Dyn. Syst. - A, \textbf{34} (2014), no. 6, 2513-2533.

\bibitem{FW} P. Felmer and Y. Wang, {\it Radial symmetry of positive solutions to equations involving the fractional Laplacian}, Commun. Contemp. Math., \textbf{16} (2014), no. 1, Paper No. 1350023, 24 pp.

\bibitem{FR} X. Fern\'{a}ndez-Real and X. Ros-Oton, {\it Regularity Theory for Elliptic PDE}, Zurich Lectures in Advanced Mathematics. European Mathematical Society (EMS), Zurich, 2022. viii+228 pp. ISBN: 978-3-98547-028-0.

\bibitem{FS} A. Figalli and J. Serra, {\it On stable solutions for boundary reactions: a De Giorgi-type result in dimension $4+1$}, Invent. Math., \textbf{219} (2020), no. 1, 153-177.

\bibitem{FKT} R. L. Frank, T. K\"{o}nig and H. Tang, {\it Classification of solutions of an equation related to a conformal $\log$ Sobolev inequality}, Adv. Math., \textbf{375} (2020), 107395, 27 pp.

\bibitem{FLS} R. L. Frank, E. Lenzmann and L. Silvestre, {\it Uniqueness of radial solutions for the fractional Laplacian}, Comm. Pure Appl. Math., \textbf{69} (2016), no. 9, 1671-1726.

\bibitem{GNN} B. Gidas, W. Ni and L. Nirenberg, {\it Symmetry and related properties via maximum principle}, Comm. Math. Phys., \textbf{68} (1979), 209-243.

\bibitem{GS} B. Gidas and J. Spruck, {\it Global and local behavior of positive solutions of nonlinear elliptic equations}, Comm. Pure Appl. Math., \textbf{34} (1981), no. 4, 525-598.

\bibitem{GS1} B. Gidas and J. Spruck, {\it A priori bounds for positive solutions of nonlinear elliptic equations}, Commun. PDE, \textbf{6} (1981), no. 8, 883-901.

\bibitem{GT} D. Gilbarg and N. S. Trudinger, {\it Elliptic partial differential equations of second order}, 2nd Ed., Springer-Verlag, New York, 1983.

\bibitem{GGN} F. Gladiali, M. Grossi and S. Neves, {\it Nonradial solutions for the H\'{e}non equation in $\mathbb{R}^{n}$}, Adv. Math., \textbf{249} (2013), 1-36.

\bibitem{GL} Y. Guo and J. Liu, {\it Liouville-type theorems for polyharmonic equations in $\mathbb{R}^{N}$ and in $\mathbb{R}_{+}^{N}$}, Proc. Roy. Soc. Edinburgh Sect. A: Math., \textbf{138} (2008), no. 2, 339-359.

\bibitem{GW} Z. Guo and F. Wan, {\it Further study of a weighted elliptic equation}, Sci. China Math., \textbf{60} (2017), no. 12, 2391-2406.

\bibitem{HRW} E. Hebey, F. Robert and Y. Wen, {\it Compactness and global estimates for a fourth order equation of critical Sobolev growth arising from conformal geometry}, Commun. Contemp. Math., \textbf{8} (2006), no. 1, 9-65.

\bibitem{LZLH} Y. Lian, K. Zhang, D. Li and G. Hong, {\it Boundary H\"{o}lder regularity for elliptic equations}, J. Math. Pures Appl., \textbf{143} (2020), 311-333.

\bibitem{LL} C. Li and Y. L\"{u}, {\it Maximum principles for Laplacian and fractional Laplacian with critical integrability}, preprint, arXiv: 1905.01782, 27 pp.

\bibitem{LZ2} K. Li and Z. Zhang, {\it Proof of the H\'{e}non-Lane-Emden conjecture in $\mathbb{R}^{3}$}, J. Differential Equ., \textbf{266} (2019), 202-226.

\bibitem{Lin} C. Lin, {\it A classification of solutions of a conformally invariant fourth order equation in $\mathbb{R}^{n}$}, Comment. Math. Helv., \textbf{73} (1998), 206-231.

\bibitem{Li} Y. Y. Li, {\it Remark on some conformally invariant integral equations: the method of moving spheres}, J. European Math. Soc., \textbf{6} (2004), 153-180.

\bibitem{LZ} Y. Y. Li and M. Zhu, {\it Uniqueness theorems through the method of moving spheres}, Duke Math. J., \textbf{80} (1995), 383-417.

\bibitem{LZ1} Y. Y. Li and L. Zhang, {\it Liouville type theorems and Harnack type inequalities for semilinear elliptic equations}, J. d'Analyse Math., \textbf{90} (2003), 27-87.

\bibitem{LB} Y. Li and J. Bao, {\it Fractional Hardy-H\'{e}non equations on exterior domains}, J. Differential Equ., \textbf{266} (2019), no. 2-3, 1153-1175.

\bibitem{LZou} S. Luo and W. Zou, {\it Liouville theorems for integral systems related to fractional Lane-Emden systems in $\mathbb{R}^n_+$}, Diff. Integral Equations, \textbf{29} (2016), no. 11-12, 1107-1138.

\bibitem{MR} P. J.  McKenna and W. Reichel, {\it A priori bounds for semilinear equations and a new class of critical exponents for Lipschitz domains}, J. Funct. Anal., \textbf{244} (2007), 220-246.

\bibitem{Mich} K. Michalik, {\it Sharp estimates of the Green function, the Poisson kernel and the Martin kernel of cones for symmetric stable processes}, Hiroshima Math. J., \text{36} (2006), no. 1, 1-21.

\bibitem{M} E. Mitidieri, {\it Nonexistence of positive solutions of semilinear elliptic systems in $\mathbb{R}^{N}$}, Differential Integral Equations, \textbf{9} (1996), 465-479.

\bibitem{MM1} S. Mayboroda and V. Maz$'$ya, {\it Regularity of solutions to the polyharmonic equation in general domains}, Invent. Math., \textbf{196} (2014), no. 1, 1-68.

\bibitem{MM2} S. Mayboroda and V. Maz$'$ya, {\it Polyharmonic capacity and Wiener test of higher order}, Invent. Math., \textbf{211} (2018), no. 2, 779-853.

\bibitem{MO} X. Ma and Q. Ou, {\it A Liouville theorem for a class semilinear elliptic equations on the Heisenberg group}, Adv. Math., \textbf{413} (2023), Paper No. 108851, 20 pp.

\bibitem{MP} E. Mitidieri and S. I. Pohozaev, {\it A priori estimates and the absence of solutions of nonlinear partial differential equations and inequalities}, Tr. Mat. Inst. Steklova, \textbf{234} (2001), 1-384.

\bibitem{NY} Q. A. Ng\^{o} and D. Ye, {\it Existence and non-existence results for the higher order Hardy-H\'{e}non equations revisited}, J. Math. Pures Appl., \textbf{163} (2022), 265-298.

\bibitem{N} W. Ni, {\it A nonlinear Dirichlet problem on the unit ball and its applications}, Indiana Univ. Math. J., \textbf{31} (1982), 801-807.

\bibitem{Pa} P. Padilla, {\it On some nonlinear elliptic equations}, Thesis, Courant Institute, 1994.

\bibitem{PQS} P. Pol\'{a}\v{c}ik, P. Quittner and P. Souplet, {\it Singularity and decay estimates in superlinear problems via Liouville-type theorems. Part I: Elliptic systems}, Duke Math. J., \textbf{139} (2007), 555-579.

\bibitem{PS} Q. Phan and P. Souplet, {\it Liouville-type theorems and bounds of solutions of Hardy-H\'{e}non equations}, J. Differential Equ., \textbf{252} (2012), 2544-2562.

\bibitem{RW} W. Reichel and T. Weth, {\it A priori bounds and a Liouville theorem on a half-space for higher-order elliptic Dirichlet problems}, Math. Z., \textbf{261} (2009), 805-827.

\bibitem{RZ} W. Reichel and H. Zou, {\it Non-existence results for semilinear cooperative elliptic systems via moving spheres}, J. Differential Equ., \textbf{161} (2000), 219-243.

\bibitem{RS} X. Ros-Oton and J. Serra, {\it The Dirichlet problem for the fractional Laplacian: Regularity up to the boundary}, J. Math. Pures Appl., \textbf{101} (2014), 275-302.

\bibitem{Si} L. Silvestre, {\it Regularity of the obstacle problem for a fractional power of the Laplace operator}, Comm. Pure Appl. Math., \textbf{60} (2007), 67-112.

\bibitem{S} P. Souplet, {\it The proof of the Lane-Emden conjecture in four space dimensions}, Adv. Math., \textbf{221} (2009), no. 5, 1409-1427.

\bibitem{Souto} M. A. S. Souto, {\it A priori estimates and existence of positive solution of nonlinear cooperative elliptic systems}, Differential and Integral Equations, \textbf{8} (1995), no. 5, 1245-1258.

\bibitem{SZ} J. Serrin and H. Zou, {\it Non-existence of positive solutions of Lane-Emden systems}, Diff. Integral Equations, \textbf{9} (1996), no. 4, 635-653.

\bibitem{SX} L. Sun and J. Xiong, {\it Classification theorems for solutions of higher order boundary conformally invariant problems, I}, J. Funct. Anal., \textbf{271} (2016), 3727-3764.

\bibitem{TD} S. Tang and J. Dou, {\it Nonexistence results for a fractional H\'{e}non-Lane-Emden equation on a half-space}, Internat. J. Math., \textbf{26} (2015), no. 13, 1550110, 14 pp.

\bibitem{TTV} S. Terracini, G. Tortone and S. Vita, {\it On $s$-harmonic functions on cones}, Analysis \& PDE, \textbf{11} (2018), no. 7, 1653-1691.

\bibitem{WX} J. Wei and X. Xu, {\it Classification of solutions of higher order conformally invariant equations}, Math. Ann., \textbf{313} (1999), no. 2, 207-228.

\bibitem{ZCCY} R. Zhuo, W. Chen, X. Cui and Z. Yuan, {\it Symmetry and non-existence of solutions for a nonlinear system involving the fractional Laplacian}, Discrete Contin. Dyn. Syst.-A, \textbf{36} (2016), no. 2, 1125-1141.

\end{thebibliography}
\end{document}